\documentclass[11pt,letterpaper]{amsart}
\usepackage[foot]{amsaddr}

\usepackage{mathtools}
\usepackage{amsmath}
\usepackage[T1]{fontenc}
\usepackage[latin9]{inputenc}
\usepackage{geometry}
\usepackage{wrapfig}
\usepackage{multicol}
\usepackage{graphicx}
\usepackage{soul}
\usepackage{xcolor}
\usepackage{amssymb}
\usepackage{placeins}
\usepackage{bbm}
\usepackage{cite}
\usepackage{caption}
\usepackage{enumerate}
\usepackage{afterpage}
\usepackage{enumitem}
\usepackage{bmpsize}
\usepackage{hyperref}
\usepackage{tabu}
\usepackage{enumitem}
\numberwithin{equation}{section}
\usepackage{stmaryrd}
\usepackage{tikz}
\usetikzlibrary{matrix,graphs,arrows,positioning,calc,decorations.markings,shapes.symbols}

\makeatletter
\renewcommand{\email}[2][]{%
  \ifx\emails\@empty\relax\else{\g@addto@macro\emails{,\space}}\fi%
  \@ifnotempty{#1}{\g@addto@macro\emails{\textrm{(#1)}\space}}%
  \g@addto@macro\emails{#2}%
}
\makeatother

\newtheorem{theorem}{Theorem}[section]
\newtheorem{lemma}[theorem]{Lemma}
\newtheorem{proposition}[theorem]{Proposition}

\newtheorem{corollary}[theorem]{Corollary}
{ \theoremstyle{definition}
\newtheorem{definition}[theorem]{Definition}}
{ \theoremstyle{remark}
\newtheorem{remark}[theorem]{Remark}}

\newcommand{\R}{\mathbb{R}}

\newcommand{\ex}{\mathbb{E}}

\newcommand{\pr}{\mathbb{P}}

\newcommand{\weyl}{W^\circ}

\DeclarePairedDelimiter\abs{\lvert}{\rvert}
\DeclareMathOperator{\indic}{\mathbf{1}}

\title{Tightness of Bernoulli Gibbsian line ensembles}
\date{\today}
\author{Evgeni Dimitrov}
\author{Xiang Fang}
\author{Lukas Fesser}
\author{Christian Serio}
\author{Carson Teitler}
\author{Angela Wang}
\author{Weitao Zhu}

\begin{document}

\maketitle

\vspace{-9mm}
\begin{abstract}
A Bernoulli Gibbsian line ensemble $\mathfrak{L} = (L_1, \dots, L_N)$ is the law of the trajectories of $N-1$ independent Bernoulli random walkers $L_1, \dots, L_{N-1}$ with possibly random initial and terminal locations that are conditioned to never cross each other or a given random up-right path $L_N$ (i.e. $L_1 \geq \cdots \geq L_N$). In this paper we investigate the asymptotic behavior of sequences of Bernoulli Gibbsian line ensembles $\mathfrak{L}^N = (L^N_1, \dots, L^N_N)$ when the number of walkers tends to infinity. We prove that if one has mild but uniform control of the one-point marginals of the lowest-indexed (or top) curves $L_1^N$ then the sequence $\mathfrak{L}^N$ is tight in the space of line ensembles. Furthermore, we show that if the top curves $L_1^N$ converge in the finite dimensional sense to the parabolic Airy$_2$ process then $\mathfrak{L}^N$ converge to the parabolic Airy line ensemble.

This project was initiated during the summer research experience for undergraduates (REU) program at Columbia University in 2020.
\end{abstract}

\tableofcontents

%

\vspace{-5mm}
\section{Introduction and main results}\label{Section1}

%
\subsection{Gibbsian line ensembles}\label{Section1.1} In the last several years there has been a significant interest in {\em line ensembles} that satisfy what is known as the {\em Brownian Gibbs property}. A line ensemble is merely a collection of random continuous curves on some interval $\Lambda \subset \mathbb{R}$ (all defined on the same probability space) that are indexed by a set $\Sigma \subset \mathbb{Z}$. In this paper, we will almost exclusively have $\Sigma = \{1, \dots, N\}$ with $N \in \mathbb{N} \cup \{\infty \}$ and if $N = \infty$ we use the convention $\Sigma = \mathbb{N}$. We denote the line ensemble by $\mathcal{L}$ and by $\mathcal{L}_{i}(\omega)(x) :=\mathcal{L}(\omega)(i,x)$ the $i$-th continuous function (or line) in the ensemble, and typically we drop the dependence on $\omega$ from the notation as one does for Brownian motion. We say that a line ensemble $\mathcal{L}$ satisfies the Brownian Gibbs property if it is non-intersecting almost surely, i.e. $\mathcal{L}_i(s) < \mathcal{L}_{i-1}(s)$ for $i = 2, \dots, N$ and $s \in \Lambda$ and it satisfies the following resampling invariance. Suppose we sample $\mathcal{L}$ and fix two times $s, t \in \Lambda$ with $s < t$ and a finite interval $K = \{k_1,k_1 + 1, \dots, k_2\} \subset \Sigma$ with $k_1 \leq k_2$. We can erase the part of the lines $\mathcal{L}_k$ between the points $(s, \mathcal{L}_k(s))$ and $(t, \mathcal{L}_k(t))$ for $k = k_1, \dots, k_2$ and sample independently $k_2 - k_1 + 1$ random curves between these points according to the law of $k_2 - k_1 + 1$ Brownian bridges, which have been conditioned to not cross each other as well as the lines $\mathcal{L}_{k_1 - 1}$ and $\mathcal{L}_{k_2 + 1}$ with the convention that $\mathcal{L}_0 = \infty$ and $\mathcal{L}_{k_2 + 1} = -\infty$ if $k_2 + 1 \not \in \Sigma$. In this way we obtain a new random line ensemble $\mathcal{L}'$, and the essence of the Brownian Gibbs property is that the law of $\mathcal{L}'$ is the same as that of $\mathcal{L}$. The readers can find a precise definition of the Brownian Gibbs property in Definition \ref{DefBGP} but for now they can think of a line ensemble that satisfies the Brownian Gibbs property as $N$ random curves, which locally have the distribution of $N$ avoiding Brownian bridges. 

Part of the interest behind Brownian Gibbsian line ensembles is that they naturally arise in various models in statistical mechanics, integrable probability and mathematical physics. If $N$ is finite, a natural example of a Brownian Gibbsian line ensemble is given by Dyson Brownian motion with $\beta = 2$ (this is the law of $N$ independent one-dimensional Brownian motions all started at the origin and appropriately conditioned to never cross for all positive time). Other important examples of models that satisfy the Brownian Gibbs property include {\em Brownian last passage percolation}, which has been extensively studied recently in \cite{Ham1, Ham2, Ham3, Ham4} and the {\em parabolic Airy line ensemble} $\mathcal{L}^{Airy}$ \cite{Spohn, CorHamA}. The {\em Airy line ensemble} $\mathcal{A}$ was first discovered as a scaling limit of the multi-layer polynuclear growth model in \cite{Spohn}, where its finite dimensional distribution was derived. (The relationship between the Airy and parabolic Airy line ensemble is given by $\mathcal{A}_i(t) = 2^{1/2} \mathcal{L}_i^{Airy}(t) + t^2$ for $i \in \mathbb{N}$.) Subsequently, in \cite{CorHamA} it was shown that the edge of Dyson Brownian motion (or rather a closely related model of {\em Brownian watermelons}) converges uniformly over compacts to the parabolic Airy line ensemble $\mathcal{L}^{Airy}$, see Figure \ref{S1_1}. This stronger notion of convergence was obtained by utilizing the Brownian Gibbs property and the latter has led to the proof of many new and interesting properties of the Airy line ensemble \cite{Ham4, CorHamA, DV18}. Apart from its inherent beautiful structure, the Airy line ensemble plays a distinguished (conjectural) foundational role in the {\em Kardar-Parisi-Zhang (KPZ) universality class} through its relationship to the construction of the {\em Airy sheet} in \cite{DOV18}.

\begin{figure}[h]
\centering
\scalebox{0.70}{\includegraphics{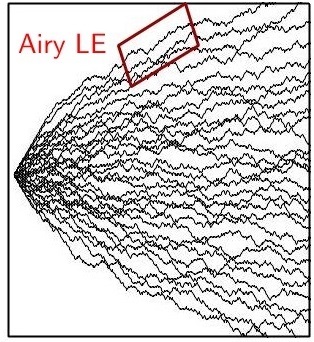}}
\caption{Dyson Brownian motion and the parabolic Airy line ensemble as its edge scaling limit.}
\label{S1_1}
\end{figure}
\vspace{-2.5mm}

The parabolic Airy line ensemble is believed to be a universal scaling limit of not just Dyson Brownian motion but many line ensembles that satisfy a Gibbs property. Recently, it was shown in the preprint \cite{DNV19} that uniform convergence to the parabolic Airy line ensemble holds for sequences of $N$ non-intersecting Bernoulli, geometric, exponential and Poisson random walks started from the origin as $N$ tends to infinity.(We mention that non-intersecting Bernoulli random walkers are equivalent to non-crossing ones after a deterministic shift.) These types of result are reminiscent of Donsker's theorem from classical probability theory, which establishes the convergence of generic random walks to Brownian motion. The difference is that as the number of avoiding walkers is increasing to infinity, one leaves the Gaussian universality class and enters the KPZ universality class. It is worth mentioning that the results in the preprint \cite{DNV19} rely on very precise integrable inputs (exact formulas for the finite dimensional distributions) for the random walkers for each fixed $N$, which are suitable for taking the large $N$ limit -- this is one reason only the packed initial condition is effectively treated. For more general initial conditions, the convergence even in the Bernoulli case, which is arguably the simplest, remains widely open. We also mention that the preprint \cite{DNV19} uses a slightly different notation than what we use in the present paper. Our use of the term Airy line ensemble (denoted by $\mathcal{A}$) agrees with the original definition in \cite{Spohn} and the term parabolic Airy line ensemble (denoted by $\mathcal{L}^{Airy}$) agrees with \cite{CHH19}. On the other hand, the preprint 
\cite{DNV19} calls $\mathcal{A}$ the ``stationary Airy line ensemble'' and $\sqrt{2} \mathcal{L}^{Airy}$ the ``Airy line ensemble''. We have chosen to follow the notation from \cite{CHH19} and not \cite{DNV19} in this paper, as it is more well-established in the field.

The goal of the present paper is to investigate asymptotics of $N$ avoiding Bernoulli random walkers with general (possibly random) initial and terminal conditions in the large $N$ limit. The main questions that motivate our work are:
\begin{enumerate}
\item What are sufficient conditions that ensure that the trajectories of $N$ avoiding Bernoulli random walkers are uniformly tight, meaning that they have uniform weak subsequential limits that are $\mathbb{N}$-indexed line ensembles on $\mathbb{R}$? 
\item What are sufficient conditions that ensure that the trajectories of $N$ avoiding Bernoulli random walkers converge uniformly to the parabolic Airy line ensemble $\mathcal{L}^{Airy}$?  
\end{enumerate}
If $\mathfrak{L}^N = (L^N_1, \dots, L^N_N)$ denotes the trajectories of the $N$ avoiding Bernoulli random walkers (with $L^N_1 \geq L_2^N \geq \cdots \geq L^N_N$) we show that as long as $L^N_1$ under suitable shifts and scales has one-point tight marginals that (roughly) globally approximate an inverted parabola, one can conclude that the whole line ensemble $\mathfrak{L}^N$ under the same shifts and scales is uniformly tight. In other words, having a mild but uniform control of the one-point marginals of the lowest-indexed (or top) curve $L^N_1$ one can conclude that the full line ensemble is tight and moreover any subsequential limit satisfies the Brownian Gibbs property. This is the main result of the paper and appears as Theorem \ref{Thm1}. 

Regarding the second question above, we show that if $L^N_1$ under suitable shifts and scales converges weakly to the {\em parabolic Airy$_2$ process} (the lowest indexed curve in the parabolic Airy line ensemble $\mathcal{L}^{Airy}$) in the finite dimensional sense, then the whole line ensemble $\mathfrak{L}^N$ under the same shifts and scales converges uniformly to the parabolic Airy line ensemble $\mathcal{L}^{Airy}$. The latter result is presented as Corollary \ref{Thm2} in the next section and is a relatively easy consequence of Theorem \ref{Thm1} and the recent characterization result of Brownian Gibbsian line ensembles in \cite{DimMat}.

It is worth pointing out that to establish tightness we do not require actual convergence of the marginals, which makes our approach more general than that of the preprint \cite{DNV19}. In particular, in the preprint \cite{DNV19} the authors assume finite dimensional convergence of $\mathfrak{L}^N$ to the parabolic Airy line ensemble, while our approach does not. In most studied cases for avoiding Bernoulli random walks, such as \cite{BorGor13, GorPet19, Joh05}, one has access to exact formulas that can be used to prove the finite dimensional convergence of $\mathfrak{L}^N$ to the parabolic Airy line ensemble, which makes this a natural assumption to make. Our motivation, however, is to formulate a framework that establishes tightness (or convergence) and relies as little as possible on exact formulas. We do this in the hopes of eventually extending this framework to other models of Gibbsian line ensembles -- ones with general (not necessarily Bernoulli) random walk paths and general (not necessarily avoiding) Gibbs properties. In this sense, the present paper should be thought of as a proof of concept -- showing in (arguably) the simplest case that the global behavior of a sequence of Gibbsian line ensemble can be effectively analyzed using only one-point marginal information for their lowest indexed curves. We mention that since this article has been completed, part of the framework in this paper has been applied to a general class of Gibbsian line ensembles in \cite{DimWu21}.

%
\subsection{Main results}\label{Section1.2}  We begin by giving some necessary definitions, which will be further elaborated in Section \ref{Section2} but will suffice for us to  present the main results of the paper. For $a, b \in \mathbb{Z}$ with $a < b$ we denote by $\llbracket a, b \rrbracket$ the set $\{a, a+1, \dots, b\}$. Given $T_0, T_1 \in \mathbb{Z}$ with $T_0 \leq T_1$ and $N \in \mathbb{N}$ we call a {\em $\llbracket 1, N \rrbracket$-indexed Bernoulli line ensemble on $\llbracket T_0, T_1 \rrbracket$} a random collection of $N$ up-right paths drawn in the region $\llbracket T_0, T_1 \rrbracket \times \mathbb{Z}$ in $\mathbb{Z}^2$ -- see the bottom-right part of Figure \ref{S1_2}. We denote a Bernoulli line ensemble by $\mathfrak{L}$ and $\mathfrak{L}(i,s)$ is the location of the $i$-th up-right path at time $s$ for $(i, s) \in \llbracket 1, N \rrbracket \times \llbracket T_0, T_1 \rrbracket$. For convenience we also denote $L_i(s) = \mathfrak{L}(i,s)$ the $i$-th up-right path in the ensemble and one can think of $L_i$'s as trajectories of Bernoulli random walkers that at each time either stay put or jump by one.

 We say that a Bernoulli line ensemble satisfies the {\em Schur Gibbs property} if it satisfies the following:
\vspace{-3mm}
\begin{enumerate} 
\item With probability $1$ we have $L_1(s) \geq L_2(s) \geq \cdots \geq L_N(s)$ for all $s \in \llbracket T_0, T_1 \rrbracket$.
\item For any $K = \{k_1, k_1 + 1, \dots, k_2 \} \subset \llbracket 1, N - 1 \rrbracket$ and $a,b \in \llbracket T_0, T_1 \rrbracket$ with $a < b$ the conditional law of $L_{k_1}, \dots, L_{k_2}$ in the region $D = \llbracket a , b  \rrbracket \times \mathbb{Z}$, given $\{ \mathfrak{L}(i,s): i \not \in K \mbox{ or } s \not \in \llbracket a+1, b-1\rrbracket \}$ is that of $k_2 - k_1 + 1$ independent Bernoulli random walks that are conditioned to start from $\vec{x} = (L_{k_1}(a), \dots, L_{k_2}(a))$ at time $a$, to end at $\vec{y} = (L_{k_1}(b), \dots, L_{k_2}(b))$ at time $b$ and to never cross each other or the paths $L_{k_1-1}$ or $L_{k_2 + 1}$ in the interval $\llbracket T_0, T_1\rrbracket$ (here $L_0 = \infty$).
\end{enumerate}
In simple words, the above definition states that a Bernoulli line ensemble satisfies the Schur Gibbs property if it is non-crossing and its local distribution is that of avoiding Bernoulli random walk bridges. We mention here that in the above definition the curve $L_N$ plays a special role, since we do not assume that its conditional distribution is that of a Bernoulli bridge conditioned to stay below $L_{N-1}$. Essentially, the curve $L_N$ plays the role of a bottom (random) boundary for our ensemble and a Bernoulli line ensemble satisfying the Schur Gibbs property can be seen to be equivalent to the statement that it is precisely the law of $N-1$ independent Bernoulli bridges that are conditioned to start from some random configuration at time $T_0$, end at some random configuration at time $T_1$ and never cross each other or a given random up-right path $L_N$ in the time interval $\llbracket T_0, T_1 \rrbracket$. We refer to Bernoulli line ensembles that satisfy the Schur Gibbs property as Bernoulli Gibbsian line ensembles. We mention that the name Schur Gibbs property originates from the connection between Bernoulli Gibbsian line ensembles and Schur symmetric polynomials, discussed later in Section \ref{Section9.2}.

A natural context in which Bernoulli Gibbsian line ensembles arise is {\em lozenge tilings} -- see Figure \ref{S1_2} and its caption. To be brief, one can take a finite tileable region in the hexagonal lattice and consider the uniform distribution on all possible tilings of this region with three types of rhombi (also called lozenges). The resulting measure on tilings has a natural Gibbs property, which is that if you freeze the tiling outside of some finite region the tiling inside that region will be conditionally uniform among all possible tilings. For special choices of tileable domains uniform lozenge tilings give rise to Bernoulli line ensembles (with deterministic packed starting and terminal conditions), and the tiling Gibbs property translated to the line ensemble becomes the Schur Gibbs property. In Figure \ref{S1_2} one observes that $L_3$ (which is the bottom-most curve in the ensemble) is not uniformly distributed among all up-right paths that stay below $L_2$ and have the correct endpoints since it needs to stay above the bottom boundary of the tiled region.  
\begin{figure}[h]
\centering
\scalebox{0.70}{\includegraphics{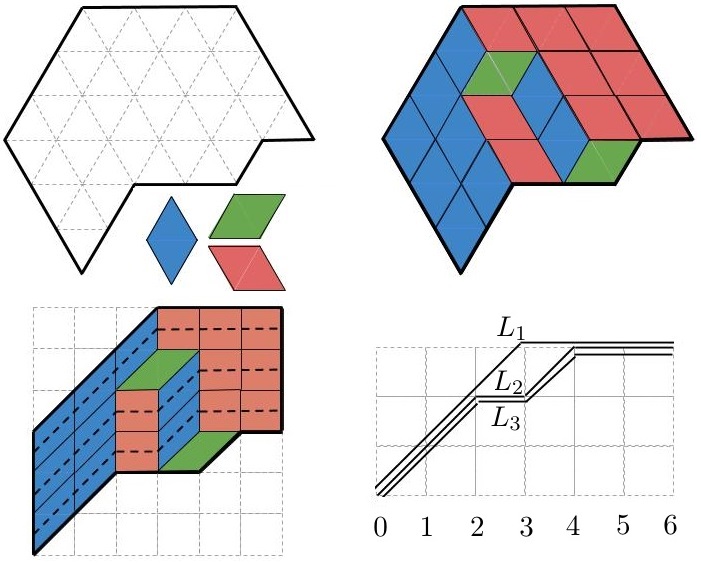}}
\caption{The top-left picture represents a tileable region in the triangular lattice and three types of lozenges. The top-right picture depicts a possible tiling of the region and the bottom-left picture represents the same tiling under an affine transformation. One draws lines through the mid-points of the vertical sides of the vertical rhombi and the squares and this gives rise to a collection of random up-right paths. If one shifts these lines down one obtains a Bernoulli line ensemble -- depicted in the bottom-right picture. If one takes the uniform measure on lozenge tilings the Bernoulli line ensemble one obtains through the above procedure satisfies the Schur Gibbs property.}
\label{S1_2}
\end{figure}

In the remainder of this section we fix a sequence $\mathfrak{L}^N = (L_1^N, \dots, L_N^N)$ of $\llbracket 1, N \rrbracket$-indexed Bernoulli Gibbsian line ensembles on $\llbracket a_N, b_N \rrbracket$ where $a_N \leq 0$ and $b_N \geq 0$ are integers. Our interest is in understanding the asymptotic behavior of $\mathfrak{L}^N$ as $N \rightarrow \infty$ (i.e. when the number of walkers tends to infinity). Below we we list several assumptions on the sequence $\mathfrak{L}^N$, which rely on parameters $\alpha > 0$, $p \in (0,1)$ and $\lambda >0$. The parameter $\alpha$ is related to the fluctuation exponent of the line ensemble and the assumptions below will indicate that $L_1^N(0)$ fluctuates on order $N^{\alpha/2}$. The parameter $p$ is the global slope of the line ensemble, and since we are dealing with Bernoulli walkers the global slope is in $[0,1]$ and we exclude the endpoints to avoid degenerate cases. The parameter $\lambda$ is related to the global curvature of the line ensemble, and the assumptions below will indicate that once the slope is removed the line ensemble approximates the parabola $-\lambda x^2$. We now turn to formulating our assumptions precisely.

{\bf \raggedleft Assumption 1.} We assume that there is a function $\psi: \mathbb{N} \rightarrow (0, \infty)$ such that $\lim_{N \rightarrow \infty} \psi(N) = \infty$ and $a_N < -\psi(N) N^{\alpha}$ while $b_N > \psi(N) N^{\alpha}$.

The significance of Assumption 1 is that the sequence of intervals $[a_N, b_N]$ (on which the line ensemble $\mathfrak{L}^N$ is defined) on scale $N^{\alpha}$ asymptotically covers the entire real line. The nature of $\psi$ is not important and any function converging to infinity along the integers works for our purposes.

{\bf \raggedleft Assumption 2.}  There is a function $\phi: (0, \infty) \rightarrow (0,\infty)$ such that for any $\epsilon > 0$ we have 
\begin{equation}\label{S1globalParabola}
 \sup_{n \in \mathbb{Z}} \limsup_{N \rightarrow \infty} \mathbb{P} \left( \left|N^{-\alpha/2}(L_1^N(n N^{\alpha}) - p n N^{\alpha} + \lambda n^2 N^{\alpha/2}) \right| \geq \phi(\epsilon) \right) \leq \epsilon.
\end{equation}

Let us elaborate on Assumption 2 briefly. If $n = 0$ the statement indicates that $N^{-\alpha/2} L_1(0)$ is a tight sequence of random variables and so $\alpha/2$ is the fluctuation exponent of the ensemble. The transversal exponent is $\alpha$ and is reflected in the way time (the argument in $L_1^N$) is scaled -- it is twice $\alpha/2$ as expected by Brownian scaling. The essence of Assumption 2 is that if one removes a global line with slope $p$ from $L_1^N$ and rescales by $N^{\alpha/2}$ vertically and $N^{\alpha}$ horizontally the resulting curve asymptotically approximates the parabola $-\lambda x^2$. The way the statement is formulated, this approximation needs to happen uniformly over the integers but the choice of $\mathbb{Z}$ is not important. Indeed, one can replace $\mathbb{Z}$ with any subset of $\mathbb{R}$ that has arbitrarily large and small points and the choice of $\mathbb{Z}$ is made for convenience. Equation (\ref{S1globalParabola}) indicates that for each $n \in \mathbb{Z}$ the sequence of random variables $X_n^N = N^{-\alpha/2}(L_1^N(n N^{\alpha}) - p n N^{\alpha} + \lambda n^2 N^{\alpha/2}) $ is tight, but it says a bit more. Namely, it states that if $\mathcal{M}_n$ is the family of all possible subsequential limits of $\{X_n^N\}_{N \geq 1}$ then $\cup_{n \in \mathbb{Z}} \mathcal{M}_n$ is itself a tight family of distributions on $\mathbb{R}$.  A simple case when Assumption 2 is satisfied is when $X_n^N$ converges to the Tracy-Widom distribution for all $n$ as $N \rightarrow \infty$. In this case the family $\cup_{n \in \mathbb{Z}} \mathcal{M}_n$ only contains the Tracy-Widom distribution and so is naturally tight. 

The final thing we need to do is to embed all of our line ensembles $\mathfrak{L}^N$ in the same space. The latter is necessary as we want to talk about tightness and convergence of line ensembles that presently are defined on different state spaces (remember that the number of up-right paths is changing with $N$). We consider $\mathbb{N} \times \mathbb{R}$ with the product topology coming from the discrete topology on $\mathbb{N}$ and the usual topology on $\mathbb{R}$. We let $C(\mathbb{N} \times \mathbb{R})$ be the space of continuous functions on $\mathbb{N} \times \mathbb{R}$ with the topology of uniform convergence over compacts and corresponding Borel $\sigma$-algebra. For each $N \in \mathbb{N}$ we let 
$$f^N_i(s) =  N^{-\alpha/2}(L^N_i(sN^{\alpha}) - p s N^{\alpha} + \lambda s^2 N^{\alpha/2}), \mbox{ for $s\in [-\psi(N) ,\psi(N)]$ and $i = 1,\dots, N$,}$$
and extend $f^N_i$ to $\mathbb{R}$ by setting for $i = 1, \dots, N$
$$f^N_i(s) = f^N_i(-\psi(N)) \mbox{ for $s \leq -\psi(N)$ and } f^N_i(s) = f_i^N(\psi(N)) \mbox{ for $s \geq \psi(N)$}.$$
If $i \geq N+1$ we define $f^N_i(s) = 0$ for $s \in \mathbb{R}$. With the above we have that $\mathcal{L}^N$ defined by 
\begin{equation}\label{DefLineEns}
\mathcal{L}^N(i,s) = \frac{f^N_i(s)  - \lambda s^2}{\sqrt{p(1-p)}}
\end{equation}
 is a random variable taking value in $C(\mathbb{N} \times \mathbb{R})$ and we let $\mathbb{P}_N$ denote its distribution. We remark that the particular extension we chose for $f^N_i$ outside of $[-\psi(N) ,\psi(N)]$ and for $i \geq N+1$ is immaterial since all of our convergence/tightness results are formulated for the topology of uniform convergence over compacts. Consequently, only the behavior of these functions on compact intervals and finite index matters and not what these functions do near infinity, which is where the modification happens as $\lim_{N \rightarrow \infty} \psi(N) = \infty$ by assumption. 

We are now ready to state our main result, whose proof can be found in Section \ref{Section2.4}.
\begin{theorem}\label{Thm1} Under Assumptions 1 and 2 the sequence $\mathbb{P}_N$ is tight. Moreover, if $\mathcal{L}^\infty$ denotes any subsequential limit of $\mathcal{L}^N$ then $\mathcal{L}^\infty$ satisfies the Brownian Gibbs property of Section \ref{Section1.1} (see also Definition \ref{DefPBGP}).
\end{theorem}
\begin{remark} In simple words, Theorem \ref{Thm1} states that if one has a sequence of Bernoulli Gibbsian line ensembles with a mild but uniform control of the one-point marginals of the top curves $L_1^N$ then the entire line ensembles need to be tight. The idea of utilizing the Gibbs property of a line ensemble to improve one-point tightness of the top curve to tightness of the entire curve or even the entire line ensemble has appeared previously in several different contexts. For line ensembles whose underlying path structure is Brownian it first appeared in the seminal work of \cite{CorHamA} and more recently in \cite{CIW19b,CIW19a}. For discrete Gibbsian line ensembles (more general than the one studied in this paper) it appeared in \cite{CD} and for line ensembles related to the inverse gamma directed polymer in \cite{Wu19}.
\end{remark}

Theorem \ref{Thm1} indicates that in order to ensure the existence of subsequential limits for $\mathcal{L}^N$ as in (\ref{DefLineEns}) it suffices to ensure tightness of the one-point marginals of the top curves $L_1^N$ in a sufficiently uniform sense. We next investigate the question of when $\mathcal{L}^N$ converges to the {parabolic Airy line ensemble} $\mathcal{L}^{Airy}$. We let $\mathcal{A} = \{\mathcal{A}_i\}_{i \in \mathbb{N}}$ be the $\mathbb{N}$-indexed Airy line ensemble and $\mathcal{L} = \{\mathcal{L}_i^{Airy}\}_{i \in \mathbb{N}}$ be given by $\mathcal{L}^{Airy}_i(x) = 2^{-1/2}( \mathcal{A}_i(x) - x^2)$ as in \cite[Theorem 3.1]{CorHamA}. In particular, both $\mathcal{A}$ and $\mathcal{L}$ are random variables taking values in the space $C(\mathbb{N} \times \mathbb{R})$, and $\mathcal{A}_1(\cdot)$ is the {\em Airy$_2$ process} while $\mathcal{L}_1^{Airy}(\cdot)$ is the {\em parabolic Airy$_2$ process}. To establish convergence of $\mathcal{L}^N$ to $\mathcal{L}^{Airy}$ we need the following strengthening of Assumption 2.

{\bf \raggedleft Assumption 2'.}  Let $c = \left( \frac{2\lambda^2}{p(1-p)}\right)^{1/3}.$ For any $k \in \mathbb{N}$,  $t_1, \dots, t_k, x_1, \dots, x_k \in \mathbb{R}$ we assume that 
\begin{equation}\label{S1globalParabola}
 \lim_{N \rightarrow \infty} \mathbb{P} \left(  \mathcal{L}^N_1(t_i) \leq x_i \mbox{ for $i = 1, \dots, k$} \right) =  \mathbb{P} \left(  c^{-1/2} \mathcal{L}^{Airy}_1(c t_i) \leq x_i \mbox{ for $i = 1, \dots, k$} \right).
\end{equation}

In plain words, Assumption 2' states that the top curves $\mathcal{L}_1^N(t)$ converge in the finite dimensional sense to $c^{-1/2}\mathcal{L}_1^{Airy}(c t)$. Let us briefly explain why Assumption 2' implies Assumption 2 (and hence we refer to it as a strengthening). Under Assumption 2', we would have that $N^{-\alpha/2}(L_1^N(x N^{\alpha}) - p x N^{\alpha} + \lambda x^2 N^{\alpha/2})$ converge in the finite dimensional sense to $\sqrt{\frac{p(1-p)}{2 c}} \mathcal{A}_1(c x)$. In particlar, for each $n \in \mathbb{Z}$ we have that 
$$\lim_{N \rightarrow \infty} \mathbb{P} \left( \left|N^{-\alpha/2}(L_1^N(n N^{\alpha}) - p n N^{\alpha} + \lambda n^2 N^{\alpha/2}) \right| \geq a \right) = \mathbb{P}\left(\sqrt{\frac{p(1-p)}{2 c}} |\mathcal{A}_1(c n)| \geq a  \right) = $$
$$1 - F_{GUE} \left( a \cdot \sqrt{\frac{2 c}{p(1-p)}} \right)  + F_{GUE} \left( -a \cdot \sqrt{\frac{2 c}{p(1-p)}} \right) ,$$
where we used that $\mathcal{A}_1(x)$ is a stationary process whose one point marginals are given by the Tracy-Widom distribution $F_{GUE}$, \cite{TWPaper}, and that $F_{GUE}$ is diffuse. In particular, given $\epsilon > 0$ we can find $a$ large enough so that the second line above is less than $\epsilon$ and such a choice of $a$ furnishes a function $\phi$ as in Assumption 2.

The next result gives conditions under which $\mathcal{L}^N$ converges to the parabolic Airy line ensemble $\mathcal{L}^{Airy}$. It is proved in Section \ref{Section2.4}
\begin{corollary}\label{Thm2} Under Assumptions 1 and 2' the sequence $\mathcal{L}^N$ converges weakly in the topology of uniform convergence over compacts to the line ensemble $\mathcal{L}^\infty$ defined by 
$$\mathcal{L}^\infty_i(t)  = c^{-1/2}\mathcal{L}_i^{Airy}(ct), \mbox{ for $i \in \mathbb{N}$ and $ t\in \mathbb{R}$, where we recall $c =  \left( \frac{2\lambda^2}{p(1-p)}\right)^{1/3}$. }$$
\end{corollary}
\begin{remark} In plain words, Corollary \ref{Thm2} states that to prove the convergence of a sequence of Bernoulli Gibbsian line ensembles $\mathcal{L}^N$ to the parabolic Airy line ensemble $\mathcal{L}^{Airy}$, it suffices to show that the top curves $\mathcal{L}_1^N$ (one for each ensemble) converge in the finite dimensional sense to the parabolic Airy$_2$ process. We mention that the convergence in Corollary \ref{Thm2} is in the uniform topology over compacts, which is stronger than finite-dimensional convergence. We also mention that in the preprint \cite{DNV19} the conclusion of Corollary \ref{Thm2} was established under the assumption that $\mathcal{L}^N$ converge to $\mathcal{L}^\infty$ in the finite dimensional sense. Simply put, we require as input only the finite dimensional convergence of the top curves, while \cite[Theorem 1.5]{DNV19} requires the finite dimensional convergence of not just the top but {\em all} curves in the line ensemble, which is a much stronger assumption.
\end{remark}

The remainder of the paper is organized as follows. In Section \ref{Section2} we introduce the basic definitions and notation for line ensembles. The main technical result of the paper, Theorem \ref{PropTightGood}, is presented in Section \ref{Section2.3} and Theorem \ref{Thm1} and Corollary \ref{Thm2} are proved in Section \ref{Section2.4} by appealing to it. In Section 3 we prove several statements for Bernoulli random walk bridges, by using a strong coupling result that allows us to compare the latter with Brownian bridges. The proof of Theorem \ref{PropTightGood} is presented in Section \ref{Section4} and is based on three key lemmas. Two of these lemmas are proved in Section \ref{Section5} and the last one in Section \ref{Section6}. The paper ends with Sections \ref{Section8} and \ref{Section9}, where various technical results needed throughout the paper are proved.

\subsection*{Acknowledgments} This project was initiated during the summer REU program at Columbia University in 2020 and we thank the organizer from the Mathematics department, Michael Woodbury, for this wonderful opportunity. E.D. is partially supported by the Minerva Foundation Fellowship.

%
\section{Line ensembles} \label{Section2}
In this section we introduce various definitions and notation that are used throughout the paper. 

%
\subsection{Line ensembles and the Brownian Gibbs property}\label{Section2.1}
In this section we introduce the notions of a {\em line ensemble} and the {\em (partial) Brownian Gibbs property}. Our exposition in this section closely follows that of \cite[Section 2]{CorHamA} and \cite[Section 2]{DimMat}. 

Given two integers $p \leq q$, we let $\llbracket p, q \rrbracket$ denote the set $\{p, p+1, \dots, q\}$. If $p > q$ then $\llbracket p, q \rrbracket = \emptyset$. Given an interval $\Lambda \subset \mathbb{R}$ we endow it with the subspace topology of the usual topology on $\mathbb{R}$. We let $(C(\Lambda), \mathcal{C})$ denote the space of continuous functions $f: \Lambda \rightarrow \mathbb{R}$ with the topology of uniform convergence over compacts, see \cite[Chapter 7, Section 46]{Munkres}, and Borel $\sigma$-algebra $\mathcal{C}$. Given a set $\Sigma \subset \mathbb{Z}$ we endow it with the discrete topology and denote by $\Sigma \times \Lambda$ the set of all pairs $(i,x)$ with $i \in \Sigma$ and $x \in \Lambda$ with the product topology. We also denote by $\left(C (\Sigma \times \Lambda), \mathcal{C}_{\Sigma}\right)$ the space of continuous functions on $\Sigma \times \Lambda$ with the topology of uniform convergence over compact sets and Borel $\sigma$-algebra $\mathcal{C}_{\Sigma}$. Typically, we will take $\Sigma = \llbracket 1, N \rrbracket$ (we use the convention $\Sigma = \mathbb{N}$ if $N = \infty$) and then we write  $\left(C (\Sigma \times \Lambda), \mathcal{C}_{|\Sigma|}\right)$ in place of $\left(C (\Sigma \times \Lambda), \mathcal{C}_{\Sigma}\right)$.

The following defines the notion of a line ensemble.
\begin{definition}\label{CLEDef}
Let $\Sigma \subset \mathbb{Z}$ and $\Lambda \subset \mathbb{R}$ be an interval. A {\em $\Sigma$-indexed line ensemble $\mathcal{L}$} is a random variable defined on a probability space $(\Omega, \mathcal{F}, \mathbb{P})$ that takes values in $\left(C (\Sigma \times \Lambda), \mathcal{C}_{\Sigma}\right)$. Intuitively, $\mathcal{L}$ is a collection of random continuous curves (sometimes referred to as {\em lines}), indexed by $\Sigma$,  each of which maps $\Lambda$ in $\mathbb{R}$. We will often slightly abuse notation and write $\mathcal{L}: \Sigma \times \Lambda \rightarrow \mathbb{R}$, even though it is not $\mathcal{L}$ which is such a function, but $\mathcal{L}(\omega)$ for every $\omega \in \Omega$. For $i \in \Sigma$ we write $\mathcal{L}_i(\omega) = (\mathcal{L}(\omega))(i, \cdot)$ for the curve of index $i$ and note that the latter is a map $\mathcal{L}_i: \Omega \rightarrow C(\Lambda)$, which is $(\mathcal{C}, \mathcal{F})-$measurable. If $a,b \in \Lambda$ satisfy $a < b$ we let $\mathcal{L}_i[a,b]$ denote the restriction of $\mathcal{L}_i$ to $[a,b]$.
\end{definition}

We will require the following result, whose proof is postponed until Section \ref{Section8.1}. In simple terms it states that the space $C (\Sigma \times \Lambda)$ where our random variables $\mathcal{L}$ take value has the structure of a complete, separable metric space. 

\begin{lemma}\label{Polish} Let $\Sigma \subset \mathbb{Z}$ and $\Lambda \subset \mathbb{R}$ be an interval. Suppose that $\{a_n\}_{n = 1}^\infty, \{b_n\}_{n = 1}^\infty$ are sequences of real numbers such that $a_n < b_n$, $[a_n, b_n] \subset \Lambda$, $a_{n+1} \leq a_n$, $b_{n+1} \geq b_n$ and $\cup_{n = 1}^\infty [a_n, b_n] = \Lambda$. For $n \in \mathbb{N}$ let $K_n = \Sigma_n \times [a_n, b_n]$ where $\Sigma_n = \Sigma \cap \llbracket -n, n \rrbracket$. Define $d: C (\Sigma \times \Lambda) \times C (\Sigma \times \Lambda) \rightarrow [0, \infty)$ by
\begin{equation} d (f,g) = \sum_{n=1}^\infty 2^{-n}\min\bigg\{\sup_{(i,t)\in K_n} |f(i,t) - g(i,t)|, \, 1\bigg\}.
\end{equation}
Then $d$ defines a metric on $C (\Sigma \times \Lambda) $ and moreover the metric space topology defined by $d$ is the same as the topology of uniform convergence over compact sets. Furthermore, the metric space $(C (\Sigma \times \Lambda), d)$ is complete and separable.
\end{lemma}

\begin{definition}
Given a sequence $\{ \mathcal{L}^n: n \in \mathbb{N} \}$ of random $\Sigma$-indexed line ensembles we say that $\mathcal{L}^n$ {\em converge weakly} to a line ensemble $\mathcal{L}$, and write $\mathcal{L}^n \implies \mathcal{L}$ if for any bounded continuous function $f: C (\Sigma \times \Lambda) \rightarrow \mathbb{R}$ we have that 
$$\lim_{n \rightarrow \infty} \mathbb{E} \left[ f(\mathcal{L}^n) \right] = \mathbb{E} \left[ f(\mathcal{L}) \right].$$

We also say that $\{ \mathcal{L}^n: n \in \mathbb{N} \}$ is {\em tight} if for any $\epsilon > 0$ there exists a compact set $K \subset C (\Sigma \times \Lambda)$ such that $\mathbb{P}(\mathcal{L}^n \in K) \geq 1- \epsilon$ for all $n \in \mathbb{N}$.

We call a line ensemble {\em non-intersecting} if $\mathbb{P}$-almost surely $\mathcal{L}_i(r) > \mathcal{L}_j(r)$  for all $i < j$ and $r \in \Lambda$.
\end{definition}

We will require the following sufficient condition for tightness of a sequence of line ensembles, which extends \cite[Theorem 7.3]{Billing}. We give a proof in Section \ref{Section8.2}.

\begin{lemma}\label{2Tight}
	Let $\Sigma \subset \mathbb{Z}$ and $\Lambda\subset\mathbb{R}$ be an interval. Suppose that $\{a_n\}_{n = 1}^\infty, \{b_n\}_{n = 1}^\infty$ are sequences of real numbers such that $a_n < b_n$, $[a_n, b_n] \subset \Lambda$, $a_{n+1} \leq a_n$, $b_{n+1} \geq b_n$ and $\cup_{n = 1}^\infty [a_n, b_n] = \Lambda$. Then $\{\mathcal{L}^n\}$ is tight if and only if for every $i\in\Sigma$ we have:
	\begin{enumerate}[label=(\roman*)]
		\item $\lim_{a\to\infty} \limsup_{n\to\infty}\, \pr(|\mathcal{L}^n_i(a_0)|\geq a) = 0$ for some $a_0 \in \Lambda$;
		\item For all $\epsilon>0$ and $k \in \mathbb{N}$,  $\lim_{\delta\to 0} \limsup_{n\to\infty}\, \pr\bigg(\sup_{\substack{x,y\in [a_k,b_k], \\ |x-y|\leq\delta}} |\mathcal{L}^n_i(x) - \mathcal{L}^n_i(y)| \geq \epsilon\bigg) = 0.$
		
	\end{enumerate}
\end{lemma}

We next turn to formulating the Brownian Gibbs property -- we do this in Definition \ref{DefBGP} after introducing some relevant notation and results. If $W_t$ denotes a standard one-dimensional Brownian motion, then the process
$$\tilde{B}(t) =  W_t - t W_1, \hspace{5mm} 0 \leq t \leq 1,$$
is called a {\em Brownian bridge (from $\tilde{B}(0) = 0$ to $\tilde{B}(1) = 0 $) with diffusion parameter $1$.} For brevity we call the latter object a {\em standard Brownian bridge}.

Given $a, b,x,y \in \mathbb{R}$ with $a < b$ we define a random variable on $(C([a,b]), \mathcal{C})$ through
\begin{equation}\label{BBDef}
B(t) = (b-a)^{1/2} \cdot \tilde{B} \left( \frac{t - a}{b-a} \right) + \left(\frac{b-t}{b-a} \right) \cdot x + \left( \frac{t- a}{b-a}\right) \cdot y, 
\end{equation}
and refer to the law of this random variable as a {\em Brownian bridge (from $B(a) = x$ to $B(b) = y$) with diffusion parameter $1$.} Given $k \in \mathbb{N}$ and $\vec{x}, \vec{y} \in \mathbb{R}^k$ we let $\mathbb{P}^{a,b, \vec{x},\vec{y}}_{free}$ denote the law of $k$ independent Brownian bridges $\{B_i: [a,b] \rightarrow \mathbb{R} \}_{i = 1}^k$ from $B_i(a) = x_i$ to $B_i(b) = y_i$ all with diffusion parameter $1$.

We next state a couple of results about Brownian bridges from \cite{CorHamA} for future use.
\begin{lemma}\label{NoTouch} \cite[Corollary 2.9]{CorHamA}. Fix a continuous function $f: [0,1] \rightarrow \mathbb{R}$ such that $f(0) > 0$ and $f(1) > 0$. Let $B$ be a standard Brownian bridge and let $C = \{ B(t) > f(t) \mbox{ for some $t \in [0,1]$}\}$ (crossing) and $T = \{ B(t) = f(t) \mbox{ for some } t\in [0,1]\}$ (touching). Then $\mathbb{P}(T \cap C^c) = 0.$
\end{lemma}
\begin{lemma}\label{Spread} \cite[Corollary 2.10]{CorHamA}. Let $U$ be an open subset of $C([0,1])$, which contains a function $f$ such that $f(0) = f(1) = 0$. If $B:[0,1] \rightarrow \mathbb{R}$ is a standard Brownian bridge then $\mathbb{P}(B[0,1] \subset U) > 0$.
\end{lemma}

The following definition introduces the notion of an $(f,g)$-avoiding Brownian line ensemble, which in simple terms is a collection of $k$ independent Brownian bridges, conditioned on not intersecting each other and staying above the graph of $g$ and below the graph of $f$ for two continuous functions $f$ and $g$.
\begin{definition}\label{DefAvoidingLaw}
Let $k \in \mathbb{N}$ and $\weyl_k$ denote the open Weyl chamber in $\mathbb{R}^{k}$, i.e.
$$\weyl_k = \{ \vec{x} = (x_1, \dots, x_k) \in \mathbb{R}^k: x_1 > x_2 > \cdots > x_k \}.$$
(In \cite{CorHamA} the notation $\mathbb{R}_{>}^k$ was used for this set.)
Let $\vec{x}, \vec{y} \in \weyl_k$, $a,b \in \mathbb{R}$ with $a < b$, and $f: [a,b] \rightarrow (-\infty, \infty]$ and $g: [a,b] \rightarrow [-\infty, \infty)$ be two continuous functions. The latter condition means that either $f: [a,b] \rightarrow \mathbb{R}$ is continuous or $f = \infty$ everywhere, and similarly for $g$. We also assume that $f(t) > g(t)$ for all $t \in[a,b]$, $f(a) > x_1, f(b) > y_1$ and $g(a) < x_k, g(b) < y_k.$

With the above data we define the {\em $(f,g)$-avoiding Brownian line ensemble on the interval $[a,b]$ with entrance data $\vec{x}$ and exit data $\vec{y}$} to be the $\Sigma$-indexed line ensemble $\mathcal{Q}$ with $\Sigma = \llbracket 1, k\rrbracket$ on $\Lambda = [a,b]$ and with the law of $\mathcal{Q}$ equal to $\mathbb{P}^{a,b, \vec{x},\vec{y}}_{free}$ (the law of $k$ independent Brownian bridges $\{B_i: [a,b] \rightarrow \mathbb{R} \}_{i = 1}^k$ from $B_i(a) = x_i$ to $B_i(b) = y_i$) conditioned on the event 
$$E  = \left\{ f(r) > B_1(r) > B_2(r) > \cdots > B_k(r) > g(r) \mbox{ for all $r \in[a,b]$} \right\}.$$ 

It is worth pointing out that $E$ is an open set of positive measure and so we can condition on it in the usual way -- we explain this briefly in the following paragraph.  Let $\left(\Omega, \mathcal{F}, \mathbb{P}\right)$ be a probability space that supports $k$ independent Brownian bridges $\{B_i: [a,b] \rightarrow \mathbb{R} \}_{i = 1}^k$ from $B_i(a) = x_i$ to $B_i(b) = y_i$ all with diffusion parameter $1$. Notice that we can find $\tilde{u}_1, \dots, \tilde{u}_k \in C([0,1])$ and $\epsilon > 0$ (depending on $\vec{x}, \vec{y}, f, g, a, b$) such that $\tilde{u}_i(0) = \tilde{u}_i(1) = 0$ for $i = 1, \dots, k$ and such that if $\tilde{h}_1, \dots, \tilde{h}_k \in C([0,1])$ satisfy $\tilde{h}_i(0) = \tilde{h}_i(1) = 0$ for $i = 1, \dots, k$ and $\sup_{t \in [0,1]}|\tilde{u}_i(t) - \tilde{h}_i(t)| < \epsilon$ then the functions
$$h_i(t) = (b-a)^{1/2} \cdot \tilde{h}_i \left( \frac{t - a}{b-a} \right) + \left(\frac{b-t}{b-a} \right) \cdot x_i + \left( \frac{t- a}{b-a}\right) \cdot y_i,$$ 
satisfy $f(r) > h_1(r) > \cdots > h_k(r) > g(r)$. It follows from Lemma \ref{Spread} that 
$$\mathbb{P}(E) \geq \mathbb{P}\left(\max_{1 \leq i \leq k} \sup_{r \in [0,1]}|\tilde{B}_i(r) - \tilde{u}_i(r)| < \epsilon \right) = \prod_{i = 1}^k \mathbb{P} \left( \sup_{r \in [0,1]}|\tilde{B}_i(r) - \tilde{u}_i(r)| < \epsilon  \right)> 0,$$
 and so we can condition on the event $E$. 

To construct a realization of $\mathcal{Q}$ we proceed as follows. For $\omega \in E$ we define
$$\mathcal{Q}(\omega)(i,r) = B_i(r)(\omega) \mbox{ for $i = 1, \dots, k$ and $r \in [a,b]$}.$$
Observe that for $i \in \{1, \dots, k\}$ and an open set $U \in C([a,b])$ we have that 
$$\mathcal{Q}^{-1}(\{i\} \times U) = \{B_i \in U \} \cap E\; \in\; \mathcal{F},$$
and since the sets $\{i\} \times U$ form an open basis of $C(\llbracket 1, k \rrbracket \times [a,b])$ we conclude that $\mathcal{Q}$ is $\mathcal{F}$-measurable. This implies that the law $\mathcal{Q}$ is indeed well-defined and also it is non-intersecting almost surely.  Also, given measurable subsets $A_1, \dots, A_k$ of $C([a,b])$ we have that 
$$\mathbb{P}(\mathcal{Q}_i \in A_i \mbox{ for $i = 1, \dots, k$} ) = \frac{\mathbb{P}^{a,b, \vec{x},\vec{y}}_{free} \left( \{ B_i \in A_i \mbox{ for $i = 1, \dots, k$}\} \cap E \right) }{\mathbb{P}^{a,b, \vec{x},\vec{y}}_{free}(E)}.$$
We denote the probability distribution of $\mathcal{Q}$ as $\mathbb{P}_{avoid}^{a,b, \vec{x}, \vec{y}, f, g}$ and write $\mathbb{E}_{avoid}^{a,b, \vec{x}, \vec{y}, f, g}$ for the expectation with respect to this measure. 
\end{definition}

The following definition introduces the notion of the Brownian Gibbs property from \cite{CorHamA}.
\begin{definition}\label{DefBGP}
Fix a set $\Sigma = \llbracket 1, N \rrbracket$ with $N \in \mathbb{N}$ or $N = \infty$ and an interval $\Lambda \subset \mathbb{R}$ and let $K = \{k_1, k_1 + 1, \dots, k_2 \} \subset \Sigma$ be finite and $a,b \in \Lambda$ with $a < b$. Set $f = \mathcal{L}_{k_1 - 1}$ and $g = \mathcal{L}_{k_2 + 1}$ with the convention that $f = \infty$ if $k_1 - 1 \not \in \Sigma$ and $g = -\infty$ if $k_2 +1 \not \in \Sigma$. Write $D_{K,a,b} = K \times (a,b)$ and $D_{K,a,b}^c = (\Sigma \times \Lambda) \setminus D_{K,a,b}$. A $\Sigma$-indexed line ensemble $\mathcal{L} : \Sigma \times \Lambda \rightarrow \mathbb{R}$ is said to have the {\em Brownian Gibbs property} if it is non-intersecting and 
$$\mbox{ Law}\left( \mathcal{L}|_{K \times [a,b]} \mbox{ conditional on } \mathcal{L}|_{D^c_{K,a,b}} \right)= \mbox{Law} \left( \mathcal{Q} \right),$$
where $\mathcal{Q}_i = \tilde{\mathcal{Q}}_{i - k_1 + 1}$ and $\tilde{\mathcal{Q}}$ is the $(f,g)$-avoiding Brownian line ensemble on $[a,b]$ with entrance data $(\mathcal{L}_{k_1}(a), \dots, \mathcal{L}_{k_2}(a))$ and exit data $(\mathcal{L}_{k_1}(b), \dots, \mathcal{L}_{k_2}(b))$ from Definition \ref{DefAvoidingLaw}. Note that $\tilde{Q}$ is introduced because, by definition, any such $(f,g)$-avoiding Brownian line ensemble is indexed from $1$ to $k_2 - k_1 + 1$ but we want $\mathcal{Q}$ to be indexed from $k_1$ to $k_2$.

An equivalent way to express the Brownian Gibbs property is as follows. A $\Sigma$-indexed line ensemble $\mathcal{L}$ on $\Lambda$ satisfies the Brownian Gibbs property if and only if it is non-intersecting and for any finite $K = \{k_1, k_1 + 1, \dots, k_2 \} \subset \Sigma$ and $[a,b] \subset \Lambda$ and any bounded Borel-measurable function $F: C(K \times [a,b]) \rightarrow \mathbb{R}$ we have $\mathbb{P}$-almost surely
\begin{equation}\label{BGPTower}
\mathbb{E} \left[ F\left(\mathcal{L}|_{K \times [a,b]} \right)  {\big \vert} \mathcal{F}_{ext} (K \times (a,b))  \right] =\mathbb{E}_{avoid}^{a,b, \vec{x}, \vec{y}, f, g} \bigl[ F(\tilde{\mathcal{Q}}) \bigr],
\end{equation}
where
$$\mathcal{F}_{ext} (K \times (a,b)) = \sigma \left \{ \mathcal{L}_i(s): (i,s) \in D_{K,a,b}^c \right\}$$
is the $\sigma$-algebra generated by the variables in the brackets above, $ \mathcal{L}|_{K \times [a,b]}$ denotes the restriction of $\mathcal{L}$ to the set $K \times [a,b]$, $\vec{x} = (\mathcal{L}_{k_1}(a), \dots, \mathcal{L}_{k_2}(a))$, $\vec{y} = (\mathcal{L}_{k_1}(b), \dots, \mathcal{L}_{k_2}(b))$, $f = \mathcal{L}_{k_1 - 1}[a,b]$ (the restriction of $\mathcal{L}$ to the set $\{k_1 - 1 \} \times [a,b]$) with the convention that $f = \infty$ if $k_1 - 1 \not \in \Sigma$, and $g = \mathcal{L}_{k_2 +1}[a,b]$ with the convention that $g =-\infty$ if $k_2 +1 \not \in \Sigma$. 
\end{definition}
\begin{remark}\label{RemMeas} Let us briefly explain why equation (\ref{BGPTower}) makes sense. Firstly, since $\Sigma \times \Lambda$ is locally compact, we know by \cite[Lemma 46.4]{Munkres} that $\mathcal{L} \rightarrow \mathcal{L}|_{K \times [a,b]}$ is a continuous map from $C(\Sigma \times \Lambda)$ to $C(K \times [a,b])$, so that the left side of (\ref{BGPTower}) is the conditional expectation of a bounded measurable function, and is thus well-defined. A more subtle question is why the right side of (\ref{BGPTower})  is $\mathcal{F}_{ext} (K \times (a,b))$-measurable. This question was resolved in \cite[Lemma 3.4]{DimMat}, where it was shown that the right side is measurable with respect to the $\sigma$-algebra 
$$ \sigma \left\{ \mathcal{L}_i(s) : \mbox{  $i \in K$ and $s \in \{a,b\}$, or $i \in \{k_1 - 1, k_2 +1 \}$ and $s \in [a,b]$} \right\},$$
which in particular implies the measurability with respect to $\mathcal{F}_{ext} (K \times (a,b))$.
\end{remark}

In the present paper it is convenient for us to use the following modified version of the definition above, which we call the partial Brownian Gibbs property -- it was first introduced in \cite{DimMat}. We explain the difference between the two definitions, and why we prefer the second in Remark \ref{RPBGP}.
\begin{definition}\label{DefPBGP}
Fix a set $\Sigma = \llbracket 1 , N \rrbracket$ with $N \in \mathbb{N}$ or $N  = \infty$ and an interval $\Lambda \subset \mathbb{R}$.  A $\Sigma$-indexed line ensemble $\mathcal{L}$ on $\Lambda$ is said to satisfy the {\em partial Brownian Gibbs property} if and only if it is non-intersecting and for any finite $K = \{k_1, k_1 + 1, \dots, k_2 \} \subset \Sigma$ with $k_2 \leq N - 1$ (if $\Sigma \neq \mathbb{N}$), $[a,b] \subset \Lambda$ and any bounded Borel-measurable function $F: C(K \times [a,b]) \rightarrow \mathbb{R}$ we have $\mathbb{P}$-almost surely
\begin{equation}\label{PBGPTower}
\mathbb{E} \left[ F(\mathcal{L}|_{K \times [a,b]}) {\big \vert} \mathcal{F}_{ext} (K \times (a,b))  \right] =\mathbb{E}_{avoid}^{a,b, \vec{x}, \vec{y}, f, g} \bigl[ F(\tilde{\mathcal{Q}}) \bigr],
\end{equation}
where we recall that $D_{K,a,b} = K \times (a,b)$ and $D_{K,a,b}^c = (\Sigma \times \Lambda) \setminus D_{K,a,b}$, and
$$\mathcal{F}_{ext} (K \times (a,b)) = \sigma \left \{ \mathcal{L}_i(s): (i,s) \in D_{K,a,b}^c \right\}$$
is the $\sigma$-algebra generated by the variables in the brackets above, $ \mathcal{L}|_{K \times [a,b]}$ denotes the restriction of $\mathcal{L}$ to the set $K \times [a,b]$, $\vec{x} = (\mathcal{L}_{k_1}(a), \dots, \mathcal{L}_{k_2}(a))$, $\vec{y} = (\mathcal{L}_{k_1}(b), \dots, \mathcal{L}_{k_2}(b))$, $f = \mathcal{L}_{k_1 - 1}[a,b]$ with the convention that $f = \infty$ if $k_1 - 1 \not \in \Sigma$, and $g = \mathcal{L}_{k_2 +1}[a,b]$.
\end{definition}
\begin{remark} \label{BGPDeg}
Observe that if $N = 1$ then the conditions in Definition \ref{DefPBGP} become void, i.e., any line ensemble with one line satisfies the partial Brownian Gibbs property. Also, we mention that (\ref{PBGPTower}) makes sense by the same reason that (\ref{BGPTower}) makes sense, see Remark \ref{RemMeas}.
\end{remark}
\begin{remark}\label{RPBGP}
Definition \ref{DefPBGP} is slightly different from the Brownian Gibbs property of Definition~\ref{DefBGP} as we explain here. Assuming that $\Sigma = \mathbb{N}$ the two definitions are equivalent. However, if $\Sigma = \{1, \dots, N\}$ with $1 \leq N < \infty$ then a line ensemble that satisfies the Brownian Gibbs property also satisfies the partial Brownian Gibbs property, but the reverse need not be true. Specifically, the Brownian Gibbs property allows for the possibility that $k_2 = N$ in Definition \ref{DefPBGP} and in this case the convention is that $g = -\infty$. As the partial Brownian Gibbs property is more general we prefer to work with it and most of the results later in this paper are formulated in terms of it rather than the usual Brownian Gibbs property.
\end{remark}

%
\subsection{Bernoulli Gibbsian line ensembles}\label{Section2.2}
In this section we introduce the notion of a {\em Bernoulli line ensemble} and the {\em Schur Gibbs property}. Our discussion will parallel that of \cite[Section 3.1]{CD}, which in turn goes back to \cite[Section 2.1]{CorHamK}.

\begin{definition}\label{DefDLE}
Let $\Sigma \subset \mathbb{Z}$ and $T_0, T_1 \in \mathbb{Z}$ with $T_0 < T_1$. Consider the set $Y$ of functions $f: \Sigma \times \llbracket T_0, T_1 \rrbracket \rightarrow \mathbb{Z}$ such that $f(j, i+1) - f(j,i) \in \{0, 1\}$ when $j \in \Sigma$ and $i \in\llbracket T_0, T_1 -1 \rrbracket$ and let $\mathcal{D}$ denote the discrete topology on $Y$. We call a function $f: \llbracket T_0, T_1 \rrbracket \rightarrow \mathbb{Z}$ such that $f( i+1) - f(i) \in \{0, 1\}$ when $i \in\llbracket T_0, T_1 -1 \rrbracket$  an {\em up-right path} and elements in $Y$ {\em collections of up-right paths}. 

A $\Sigma$-{\em indexed Bernoulli line ensemble $\mathfrak{L}$ on $\llbracket T_0, T_1 \rrbracket$}  is a random variable defined on a probability space $(\Omega, \mathcal{B}, \mathbb{P})$, taking values in $Y$ such that $\mathfrak{L}$ is a $(\mathcal{B}, \mathcal{D})$-measurable function. 
\end{definition}

\begin{remark} In \cite[Section 3.1]{CD} Bernoulli line ensembles $\mathfrak{L}$ were called {\em discrete line ensembles} in order to distinguish them from the continuous line ensembles from Definition \ref{CLEDef}. In this paper we have opted to use the term Bernoulli line ensembles to emphasize the fact that the functions $f \in Y$ satisfy the property that $f(j, i+1) - f(j,i) \in \{0, 1\}$ when $j \in \Sigma$ and $i \in\llbracket T_0, T_1 -1 \rrbracket$. This condition essentially means that for each $j \in \Sigma$ the function $f(j, \cdot)$ can be thought of as the trajectory of a Bernoulli random walk from time $T_0$ to time $T_1$. As other types of discrete line ensembles, see e.g. \cite{Wu19}, have appeared in the literature we have decided to modify the notation in \cite[Section 3.1]{CD} so as to avoid any ambiguity.
\end{remark}

The way we think of Bernoulli line ensembles is as random collections of up-right paths on the integer lattice, indexed by $\Sigma$ (see Figure \ref{S3_1}). Observe that one can view an up-right path $L$ on $\llbracket T_0, T_1 \rrbracket$ as a continuous curve by linearly interpolating the points $(i, L(i))$. This allows us to define $ (\mathfrak{L}(\omega)) (i, s)$ for non-integer $s \in [T_0,T_1]$ and to view Bernoulli line ensembles as line ensembles in the sense of Definition \ref{CLEDef}. In particular, we can think of $\mathfrak{L}$ as a random variable taking values in $\left(C (\Sigma \times \Lambda), \mathcal{C}_{\Sigma}\right)$ with $\Lambda = [T_0, T_1]$.
\begin{figure}[h]
\centering
\scalebox{0.15}{\includegraphics{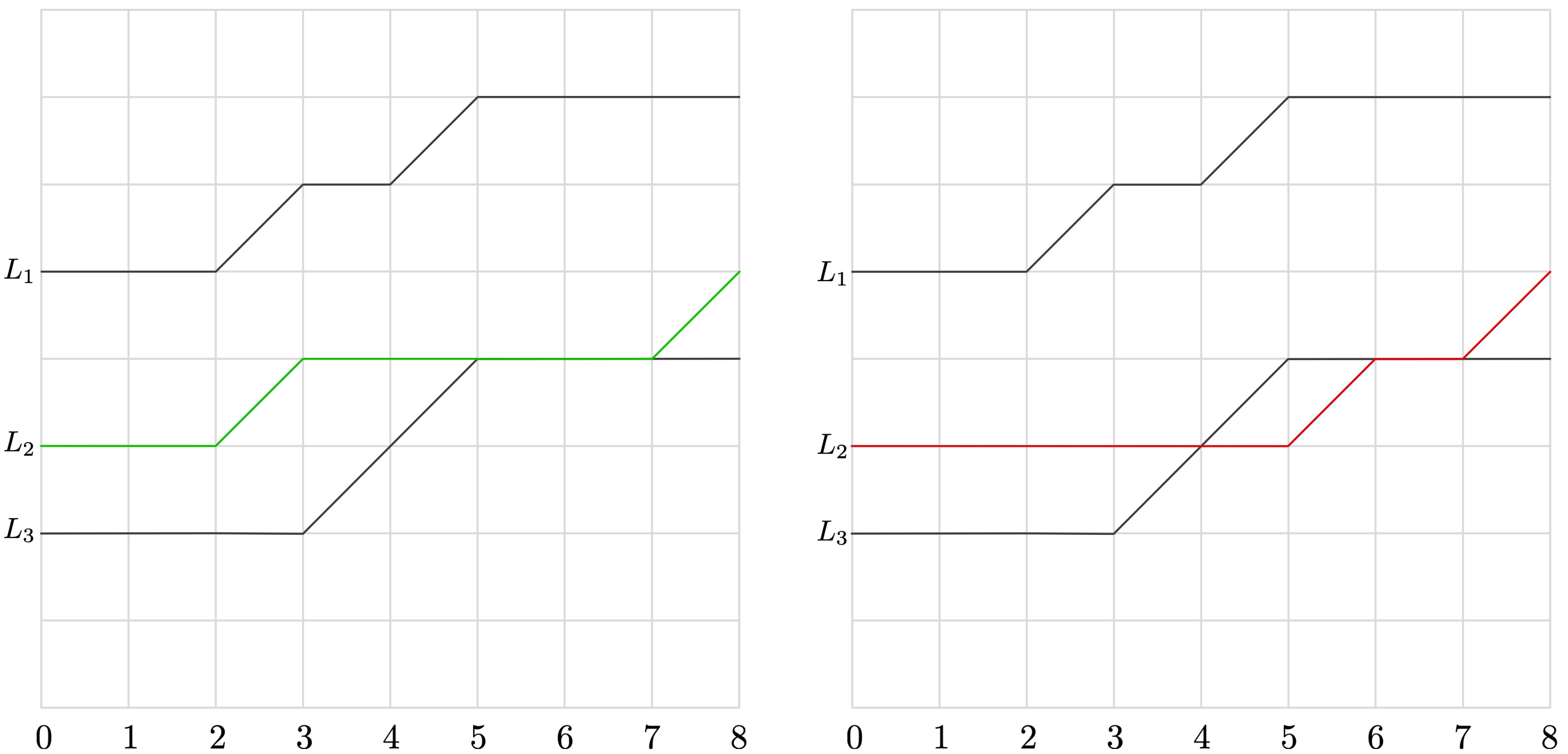}}
\caption{Two samples of $\llbracket 1,3\rrbracket$-indexed Bernoulli line ensembles with $T_0 = 0$ and $T_1 = 8$, with the left ensemble avoiding and the right ensemble nonavoiding.}
\label{S3_1}
\end{figure}
We will often slightly abuse notation and write $\mathfrak{L}: \Sigma \times \llbracket T_0, T_1 \rrbracket \rightarrow \mathbb{Z}$, even though it is not $\mathfrak{L}$ which is such a function, but rather $\mathfrak{L}(\omega)$ for each $\omega \in \Omega$. Furthermore we write $L_i = (\mathfrak{L}(\omega)) (i, \cdot)$ for the index $i \in \Sigma$ path. If $L$ is an up-right path on $\llbracket T_0, T_1 \rrbracket$ and $a, b \in \llbracket T_0, T_1 \rrbracket$ satisfy $a < b$ we let $L\llbracket a, b \rrbracket$ denote the restriction of $L$ to $\llbracket a,b\rrbracket$. \\

Let $t_i, z_i \in \mathbb{Z}$ for $i = 1,2$ be given such that $t_1 < t_2$ and $0 \leq z_2 - z_1 \leq t_2 - t_1$. We denote by $\Omega(t_1,t_2,z_1,z_2)$ the collection of up-right paths that start from $(t_1,z_1)$ and end at $(t_2,z_2)$, by $\mathbb{P}_{Ber}^{t_1,t_2, z_1, z_2}$ the uniform distribution on $\Omega(t_1,t_2,z_1,z_2)$ and write $\mathbb{E}^{t_1,t_2,z_1,z_2}_{Ber}$ for the expectation with respect to this measure. One thinks of the distribution $\mathbb{P}_{Ber}^{t_1,t_2, z_1, z_2}$ as the law of a simple random walk with i.i.d. Bernoulli increments with parameter $p \in (0,1)$ that starts from $z_1$ at time $t_1$ and is conditioned to end in $z_2$ at time $t_2$ -- this interpretation does not depend on the choice of $p \in (0,1)$. Notice that by our assumptions on the parameters the state space $\Omega(t_1,t_2,z_1,z_2)$ is non-empty.  

Given $k \in \mathbb{N}$, $T_0, T_1 \in \mathbb{Z}$ with $T_0 < T_1$ and $\vec{x}, \vec{y} \in \mathbb{Z}^k$ we let $\mathbb{P}^{T_0,T_1, \vec{x},\vec{y}}_{Ber}$ denote the law of $k$ independent Bernoulli bridges $\{B_i: \llbracket T_0, T_1 \rrbracket  \rightarrow \mathbb{Z} \}_{i = 1}^k$ from $B_i(T_0) = x_i$ to $B_i(T_1) = y_i$. Equivalently, this is just $k$ independent random up-right paths $B_i \in \Omega(T_0,T_1,x_i,y_i)$ for $i = 1, \dots, k$ that are uniformly distributed. This measure is well-defined provided that $\Omega(T_0,T_1,x_i,y_i)$ are non-empty for $i = 1, \dots, k$, which holds if $T_1 - T_0 \geq y_i - x_i \geq 0$ for all $i = 1, \dots, k$.

The following definition introduces the notion of an $(f,g)$-avoiding Bernoulli line ensemble, which in simple terms is a collection of $k$ independent Bernoulli bridges, conditioned on not-crossing each other and staying above the graph of $g$ and below the graph of $f$ for two functions $f$ and $g$.
\begin{definition}\label{DefAvoidingLawBer}
Let $k \in \mathbb{N}$ and $\mathfrak{W}_k$ denote the set of signatures of length $k$, i.e.
$$\mathfrak{W}_k = \{ \vec{x} = (x_1, \dots, x_k) \in \mathbb{Z}^k: x_1 \geq  x_2 \geq  \cdots \geq  x_k \}.$$
Let $\vec{x}, \vec{y} \in \mathfrak{W}_k$, $T_0, T_1 \in \mathbb{Z}$ with $T_0 < T_1$, $S\subseteq\llbracket T_0,T_1\rrbracket$, and $f: \llbracket T_0, T_1 \rrbracket \rightarrow (-\infty, \infty]$ and $g: \llbracket T_0, T_1 \rrbracket \rightarrow [-\infty, \infty)$ be two functions. 

With the above data we define the {\em $(f,g;S)$-avoiding Bernoulli line ensemble on the interval $\llbracket T_0, T_1 \rrbracket$ with entrance data $\vec{x}$ and exit data $\vec{y}$} to be the $\Sigma$-indexed Bernoulli line ensemble $\mathfrak{Q}$ with $\Sigma = \llbracket 1, k\rrbracket$ on $\llbracket T_0, T_1 \rrbracket$ and with the law of $\mathfrak{Q}$ equal to $\mathbb{P}^{T_0,T_1, \vec{x},\vec{y}}_{Ber}$ (the law of $k$ independent uniform up-right paths $\{B_i: \llbracket T_0, T_1 \rrbracket \rightarrow \mathbb{R} \}_{i = 1}^k$ from $B_i(T_0) = x_i$ to $B_i(T_1) = y_i$) conditioned on the event 
$$E_S  = \left\{ f(r) \geq B_1(r) \geq B_2(r) \geq \cdots \geq B_k(r) \geq g(r) \mbox{ for all $r \in S$} \right\}.$$ 
The above definition is well-posed if there exist $B_i \in \Omega(T_0,T_1,x_i,y_i)$ for $i = 1, \dots, k$ that satisfy the conditions in $E_S$ (i.e. if the set of such up-right paths is not empty). We will denote by $\Omega_{avoid}(T_0, T_1, \vec{x}, \vec{y}, f,g; S)$ the set of collections of $k$ up-right paths that satisfy the conditions in $E_S$ and then the distribution on $\mathfrak{Q}$ is simply the uniform measure on $\Omega_{avoid}(T_0, T_1, \vec{x}, \vec{y}, f,g; S)$. We denote the probability distribution of $\mathfrak{Q}$ as $\mathbb{P}_{avoid, Ber;S}^{T_0,T_1, \vec{x}, \vec{y}, f, g}$ and write $\mathbb{E}_{avoid, Ber; S}^{T_0, T_1, \vec{x}, \vec{y}, f, g}$ for the expectation with respect to this measure. If $S = \llbracket T_0,T_1\rrbracket$, we write $\Omega_{avoid}(T_0,T_1,\vec{x},\vec{y},f,g)$, $\mathbb{P}_{avoid, Ber}^{T_0,T_1, \vec{x}, \vec{y}, f, g}$, and $\mathbb{E}_{avoid, Ber}^{T_0, T_1, \vec{x}, \vec{y}, f, g}$. If $f=+\infty$ and $g=-\infty$, we write $\Omega_{avoid}(T_0,T_1,\vec{x},\vec{y})$, $\mathbb{P}^{T_0, T_1, \vec{x},\vec{y}}_{avoid, Ber}$, and $\mathbb{E}^{T_0, T_1, \vec{x},\vec{y}}_{avoid, Ber}$.
\end{definition}

It will be useful to formulate simple conditions under which $\Omega_{avoid}(T_0, T_1, \vec{x}, \vec{y}, f,g)$ is non-empty and thus $\mathbb{P}_{avoid, Ber}^{T_0,T_1, \vec{x}, \vec{y}, f, g}$ is well-defined. Note that $\Omega_{avoid}(T_0, T_1, \vec{x}, \vec{y}, f, g; S) \supseteq \Omega_{avoid}(T_0, T_1, \vec{x}, \vec{y}, f,g)$ for any $S\subseteq\llbracket T_0, T_1\rrbracket$, so $\mathbb{P}_{avoid, Ber; S}^{T_0,T_1, \vec{x}, \vec{y}, f, g}$ is also well-defined in this case. We accomplish this in the following lemma, whose proof is postponed until Section \ref{LemmaWDProof}.

\begin{lemma}\label{LemmaWD} Suppose that $k \in \mathbb{N}$ and $T_0, T_1 \in \mathbb{Z}$ with $T_0 < T_1$. Suppose further that 
\begin{enumerate}
\item $\vec{x}, \vec{y} \in \mathfrak{W}_k$ satisfy $T_1 - T_0 \geq y_i - x_i \geq 0$ for $i = 1, \dots, k$,
\item $f : \llbracket T_0, T_1 \rrbracket \rightarrow (-\infty, \infty]$ and $g : \llbracket T_0, T_1 \rrbracket \rightarrow [-\infty, \infty)$ satisfy $f (i+1) = f(i)$ or $f(i+1) = f(i) + 1$, and $g(i+1) = g(i)$ or $g(i+1) = g(i) +1$ for $i = T_0, \dots, T_1 -1$,
\item $f(T_0) \geq x_1, f(T_1) \geq y_1$, $g(T_0) \leq x_k, g(T_1) \leq y_k$ and $f(i) \geq g(i)$ for $i \in \llbracket T_0, T_1 \rrbracket$.
\end{enumerate}
Then the set $\Omega_{avoid}(T_0, T_1, \vec{x}, \vec{y}, f,g)$ from Definition \ref{DefAvoidingLawBer} is non-empty.
\end{lemma}

The following definition introduces the notion of the Schur Gibbs property, which can be thought of a discrete analogue of the partial Brownian Gibbs property in the same way that Bernoulli random walks are discrete analogues of Brownian motion. 
\begin{definition}\label{DefSGP}
Fix a set $\Sigma = \llbracket 1, N \rrbracket$ with $N \in \mathbb{N}$ or $N = \infty$ and $T_0, T_1\in \mathbb{Z}$ with $T_0 < T_1$. A $\Sigma$-indexed Bernoulli line ensemble $\mathfrak{L} : \Sigma \times \llbracket T_0, T_1 \rrbracket \rightarrow \mathbb{Z}$ is said to satisfy the {\em Schur Gibbs property} if it is non-crossing, meaning that 
$$ L_j(i) \geq L_{j+1}(i) \mbox{ for all $j = 1, \dots, N-1$ and $i \in \llbracket T_0, T_1 \rrbracket$},$$
and for any finite $K = \{k_1, k_1 + 1, \dots, k_2 \} \subset \llbracket 1, N - 1 \rrbracket$ and $a,b \in \llbracket T_0, T_1 \rrbracket$ with $a < b$ the following holds.  Suppose that $f, g$ are two up-right paths drawn in $\{ (r,z) \in \mathbb{Z}^2 : a \leq r \leq b\}$ and $\vec{x}, \vec{y} \in \mathfrak{W}_k$ with $k=k_2-k_1+1$ altogether satisfy that $\mathbb{P}(A) > 0$ where $A$ denotes the event $$A =\{ \vec{x} = ({L}_{k_1}(a), \dots, {L}_{k_2}(a)), \vec{y} = ({L}_{k_1}(b), \dots, {L}_{k_2}(b)), L_{k_1-1} \llbracket a,b \rrbracket = f, L_{k_2+1} \llbracket a,b \rrbracket = g \},$$
where if $k_1 = 1$ we adopt the convention $f = \infty = L_0$. Then writing $k = k_2 - k_1 + 1$, we have for any $\{ B_i \in \Omega(a, b, x_i , y_i) \}_{i = 1}^k$ that
\begin{equation}\label{SchurEq}
{\bf 1}_A \cdot \mathbb{P}\left( L_{i + k_1-1}\llbracket a,b \rrbracket = B_{i} \mbox{ for $i = 1, \dots, k$}  \vert \mathcal{F}_{ext}^{Ber} \right) ={\bf 1}_A \cdot  \mathbb{P}_{avoid, Ber}^{a,b, \vec{x}, \vec{y}, f, g} \left( \cap_{i = 1}^k\{ \mathfrak{Q}_i = B_i \} \right),
\end{equation}
where $\mathcal{F}_{ext}^{Ber} =  \sigma(L_{i}(j): (i,j) \in \Sigma \times \llbracket T_0, T_1 \rrbracket \setminus \llbracket k_1, k_2 \rrbracket \times \llbracket a + 1, b-1 \rrbracket )$.
\end{definition}
\begin{remark}\label{RemSGB} In simple words, a Bernoulli line ensemble is said to satisfy the Schur Gibbs property if the distribution of any finite number of consecutive paths, conditioned on their end-points and the paths above and below them is simply the uniform measure on all collection of up-right paths that have the same end-points and do not cross each other or the paths above and below them. 
\end{remark}

\begin{remark}\label{RemSGB2} Observe that in Definition \ref{DefSGP} the index $k_2$ is assumed to be less than or equal to $N-1$, so that if $N < \infty$ the $N$-th path is special and is not conditionally uniform. This is what makes Definition \ref{DefSGP} a discrete analogue of the partial Brownian Gibbs property rather than the usual Brownian Gibbs property. Similarly to the partial Brownian Gibbs property, see Remark \ref{BGPDeg}, if $N = 1$ then the conditions in Definition \ref{DefSGP} become void, i.e., any Bernoulli line ensemble with one line satisfies the Schur Gibbs property. Also we mention that the well-posedness of $\mathbb{P}_{avoid, Ber}^{T_0,T_1, \vec{x}, \vec{y}, f, g}$ in (\ref{SchurEq}) is a consequence of Lemma \ref{LemmaWD} and our assumption that $\mathbb{P}(A) > 0$.
\end{remark}

\begin{remark} In \cite{CD} the authors studied a generalization of the Gibbs property in Definition \ref{DefSGP} depending on a parameter $t \in (0,1)$, which was called the {\em Hall-Littlewood Gibbs property} due to its connection to Hall-Littlewood polynomials \cite{Mac}. The property in Definition \ref{DefSGP} is the $t \rightarrow 0$ limit of the Hall-Littlewood Gibbs property. Since under this $t \rightarrow 0$ limit Hall-Littlewood polynomials degenerate to Schur polynomials we have decided to call the Gibbs property in Definition \ref{DefSGP} the Schur Gibbs property.
\end{remark}
\begin{remark} \label{restrict}  An immediate consequence of Definition \ref{DefSGP} is that if $M \leq N$, we have that the induced law on $\{L_i\}_{i = 1}^M$ also satisfies the Schur Gibbs property as an $\{1,\dots,M\}$-indexed Bernoulli line ensemble on $\llbracket T_0, T_1 \rrbracket$.
\end{remark}

We end this section with the following definition of the term acceptance probability. 
\begin{definition}\label{DefAP} Assume the same notation as in Definition \ref{DefAvoidingLawBer} and suppose that $ T_1 - T_0 \geq y_i -x_i \geq 0$ for $i =1, \dots, k$. We define the {\em acceptance probability } $Z(  T_0, T_1, \vec{x}, \vec{y}, f, g)$ to be the ratio
\begin{equation}\label{EqAP}
Z(  T_0, T_1, \vec{x}, \vec{y}, f, g) = \frac{|\Omega_{avoid}(T_0, T_1, \vec{x}, \vec{y}, f,g)|}{\prod_{i = 1}^k |\Omega(T_0, T_1, x_i, y_i)|}.
\end{equation}
\end{definition}
\begin{remark}\label{RemAP} The quantity $Z(  T_0, T_1, \vec{x}, \vec{y}, f, g)$ is precisely the probability that if $B_i$ are sampled uniformly from $\Omega(T_0, T_1, x_i, y_i)$ for $i =1 , \dots, k$ then the $B_i$ satisfy the condition
$$ E=  \left\{ f(r) \geq B_1(r) \geq B_2(r) \geq \cdots \geq B_k(r) \geq g(r) \mbox{ for all $r \in \llbracket T_0, T_1 \rrbracket$} \right\}.$$
 Let us explain briefly why we call this quantity an acceptance probability. One way to sample $\mathbb{P}_{avoid, Ber}^{T_0,T_1, \vec{x}, \vec{y}, f, g}$ is as follows. Start by sampling a sequence of i.i.d. up-right paths $B^N_i$ uniformly from $\Omega(T_0, T_1, x_i, y_i)$ for $i =1 , \dots, k$ and $N \in \mathbb{N}$. For each $n$ check if $B^n_1, \dots, B^n_k$ satisfy the condition $E$ and let $M$ denote the smallest index that accomplishes this. If $\Omega_{avoid}(T_0, T_1, \vec{x}, \vec{y}, f,g)$ is non-empty then $M$ is geometrically distributed with parameter $Z(  T_0, T_1, \vec{x}, \vec{y}, f, g)$, and in particular $M$ is finite almost surely and $\{B^M_i\}_{i =1}^k$ has distribution $\mathbb{P}_{avoid, Ber}^{T_0,T_1, \vec{x}, \vec{y}, f, g}$. In this sampling procedure we construct a sequence of candidates $\{B^N_i\}_{i =1}^k$ for $N \in \mathbb{N}$ and reject those that fail to satisfy condition $E$, the first candidate that satisfies it is accepted and has law $\mathbb{P}_{avoid, Ber}^{T_0,T_1, \vec{x}, \vec{y}, f, g}$ and the probability that a candidate is accepted is precisely $Z(  T_0, T_1, \vec{x}, \vec{y}, f, g)$, which is why we call it an acceptance probability.
\end{remark}

%
\subsection{Main technical result}\label{Section2.3} In this section we present the main technical result of the paper. We start with the following technical definition.
\begin{definition}\label{Def1} Fix $k \in \mathbb{N}$, $\alpha, \lambda > 0$ and $p \in (0,1)$. Suppose we are given a sequence $\{ T_N \}_{N = 1}^\infty$ with $T_N \in \mathbb{N}$ and that $\{\mathfrak{L}^N\}_{N = 1}^\infty$, $\mathfrak{L}^N = (L^N_1, L^N_2, \dots, L^N_k)$ is a sequence of $\llbracket 1, k \rrbracket$-indexed Bernoulli line ensembles on $ \llbracket -T_N, T_N \rrbracket$. We call the sequence $(\alpha,p,\lambda)$-{\em good} if there exists $N_0 \in \mathbb{N}$ such that 
\begin{itemize}
\item $\mathfrak{L}^N$ satisfies the Schur Gibbs property of Definition \ref{DefSGP} for $N \in \mathbb{N}, N \geq N_0$;  
\item there is a function $\psi: \mathbb{N} \rightarrow (0, \infty)$ such that $\lim_{N \rightarrow \infty} \psi(N) = \infty$ and for each $N \geq N_0$ we have that $ T_N > \psi(N)N^{\alpha}$;
\item  there are functions $\phi_1: \mathbb{Z} \times (0, \infty) \rightarrow (0,\infty)$ and $\phi_2: (0,\infty) \rightarrow \infty$ such that for any $\epsilon > 0$, $n \in \mathbb{Z}$ and $N \geq \phi_1(n, \epsilon)$ we have 
\begin{equation}\label{globalParabola}
 \mathbb{P} \left( \left|N^{-\alpha/2}(L_1^N(n N^{\alpha}) - p n N^{\alpha} + \lambda n^2 N^{\alpha/2}) \right| \geq \phi_2(\epsilon) \right) \leq \epsilon.
\end{equation}
\end{itemize}
\end{definition}
\begin{remark} Let us elaborate on the meaning of Definition \ref{Def1}. In order for a sequence of $\mathfrak{L}^N $ of $\llbracket 1, k \rrbracket$-indexed Bernoulli line ensembles on $ \llbracket -T_N, T_N \rrbracket$ to be $(\alpha,p,\lambda)$-{\em good} we want several conditions to be satisfied. Firstly, we want for all large $N$ the Bernoulli line ensemble $\mathfrak{L}^N$ to satisfy the Schur Gibbs property. The second condition is that while the interval of definition of $\mathfrak{L}^N$ is finite for each $N$ and given by $\llbracket -T_N, T_N \rrbracket$, we want this interval to grow at least with speed $N^{\alpha}$. This property is quantified by the function $\psi$, which can be essentially thought of as an arbitrary unbounded increasing function on $\mathbb{N}$. The third condition is that we want for each $n \in \mathbb{Z}$ the sequence of random variables $N^{-\alpha/2}(L_1^N(n N^{\alpha}) - p n N^{\alpha})$ to be tight but moreover we want globally these random variables to look like the parabola $-\lambda n^2$. This statement is reflected in (\ref{globalParabola}), which provides a certain uniform tightness of the random variables $N^{-\alpha/2}(L_1^N(n N^{\alpha}) - p n N^{\alpha} + \lambda n^2 N^{\alpha/2}) $. A particular case when (\ref{globalParabola}) is satisfied is for example if we know that  for each $n \in \mathbb{Z}$ the random variables $N^{-\alpha/2}(L_1^N(n N^{\alpha}) - p n N^{\alpha} + \lambda n^2 N^{\alpha/2})$ converge to the same random variable $X$. In the applications that we have in mind these random variables would converge to the $1$-point marginals of the Airy$_2$ process that are all given by the same Tracy-Widom distribution (since the Airy$_2$ process is stationary). Equation (\ref{globalParabola}) is a significant relaxation of the requirement that $N^{-\alpha/2}(L_1^N(n N^{\alpha}) - p n N^{\alpha} + \lambda n^2 N^{\alpha/2})$ all converge weakly to the Tracy-Widom distribution -- the convergence requirement is replaced with a mild but uniform control of all subsequential limits.
\end{remark}

The main technical result of the paper is given below and proved in Section \ref{Section4}.
\begin{theorem}\label{PropTightGood}
Fix $k \in \mathbb{N}$ with $k \geq 2$, $\alpha, \lambda > 0$ and $p \in (0,1)$ and let $\mathfrak{L}^N = (L^N_1, L^N_2, \dots, L^N_k)$ be an $(\alpha, p, \lambda)$-good sequence of $\llbracket 1, k \rrbracket$-indexed Bernoulli line ensembles.  Set
$$f^N_i(s) =  N^{-\alpha/2}(L^N_i(sN^{\alpha}) - p s N^{\alpha} + \lambda s^2 N^{\alpha/2}), \mbox{ for $s\in [-\psi(N) ,\psi(N)]$ and $i = 1,\dots, k -1$,}$$
and extend $f^N_i$ to $\mathbb{R}$ by setting for $i = 1, \dots, k - 1$
$$f^N_i(s) = f^N_i(-\psi(N)) \mbox{ for $s \leq -\psi(N)$ and } f^N_i(s) = f_i^N(\psi(N)) \mbox{ for $s \geq \psi(N)$}.$$
Let $\mathbb{P}_N$ denote the law of $\{f^N_i\}_{i = 1}^{k-1}$ and $\tilde{\mathbb{P}}_N$ that of $\{\tilde{f}^N_i\}_{i = 1}^{k-1} := \{(f_i^N(s) - \lambda s^2)/\sqrt{p(1-p)}\}_{i=1}^{k-1}$ both as $\llbracket 1, k-1 \rrbracket$-indexed line ensembles $\mathrm{(}$i.e. as random variables in $(C( \llbracket 1, k -1 \rrbracket \times \mathbb{R}), \mathcal{C}_{k-1}))$. Then
\begin{enumerate}[label=(\roman*)]
	\item The sequences $\mathbb{P}_N$ and $\tilde{\mathbb{P}}_N$ are tight; 
	\item Any subsequential limit $\mathcal{L}^\infty = \{\tilde{f}_i^{\infty}\}_{i = 1}^{k-1}$ of $\tilde{\mathbb{P}}_N$ satisfies the partial Brownian Gibbs property of Definition \ref{DefPBGP}.
\end{enumerate}
\end{theorem}

Roughly, Theorem \ref{PropTightGood} (i) states that if we have a sequence of $\llbracket 1, k \rrbracket$-indexed Bernoulli line ensembles that satisfy the Schur Gibbs property and the top paths of these ensembles under some shift and scaling have tight one-point marginals with a non-trivial parabolic shift, then  under the same shift and scaling the top $k-1$ paths of the line ensemble will be tight. The extension of $f^N_i$ to $\mathbb{R}$ is completely arbitrary and irrelevant for the validity of Theorem \ref{PropTightGood} since the topology on $C( \llbracket 1, k -1 \rrbracket \times \mathbb{R})$ is that of uniform convergence over compacts. Consequently, only the behavior of these functions on compact intervals matters in Theorem \ref{PropTightGood} and not what these functions do near infinity, which is where the modification happens as $\lim_{N \rightarrow \infty} \psi(N) = \infty$  by assumption. The only reason we perform the extension is to embed all Bernoulli line ensembles into the same space $(C( \llbracket 1, k -1 \rrbracket \times \mathbb{R}), \mathcal{C}_{k-1})$.

 We mention that the $k$-th up-right path in the sequence of Bernoulli line ensembles is special and Theorem \ref{PropTightGood} provides no tightness result for it. The reason for this stems from the Schur Gibbs property, see Definition \ref{DefSGP}, which assumes less information for the $k$-th path. In practice, one either has an infinite Bernoulli line ensemble for each $N$ or one has a Bernoulli line ensemble with finite number of paths, which increase with $N$ to infinity. In either of these settings one can use Theorem \ref{PropTightGood} to prove tightness of the full line ensemble, we will see this when we prove Theorem \ref{Thm1} in the next section.

%
\subsection{Proofs of Theorem \ref{Thm1} and Corollary \ref{Thm2}}\label{Section2.4} Here we prove Theorem \ref{Thm1} and Corollary \ref{Thm2}.

\begin{proof}(of Theorem \ref{Thm1}) We use the same notation and assumptions as in the statement of the theorem. For clarity we split the proof into two steps.

{\bf \raggedleft Step 1.} In this step we prove that $\mathcal{L}^N$ is tight. In view of Lemma \ref{2Tight} to establish the tightness of $\mathcal{L}^N$ it suffices to show that for every $k \in \mathbb{N}$  
\begin{enumerate}[label=(\roman*)]
\item $\lim_{a\to\infty} \limsup_{N\to\infty}\, \pr(|\mathcal{L}^N_k(0)|\geq a) = 0$.
\item For all $\epsilon>0$ and $m \in \mathbb{N}$,  $\lim_{\delta\to 0} \limsup_{N\to\infty}\, \pr\bigg(\sup_{\substack{x,y\in [-m,m], \\ |x-y|\leq\delta}} |\mathcal{L}^N_k(x) - \mathcal{L}^N_k(y)| \geq \epsilon\bigg) = 0.$
\end{enumerate}
Let $T_N = \min (-a_N, b_N)$ and for $N \geq k +1$ let $\tilde{\mathfrak{L}}^N = (\tilde{L}_1^N , \tilde{L}_2^N, \dots, \tilde{L}_{k+1}^N)$ denote the $\llbracket 1, k+1\rrbracket$-indexed Bernoulli line ensemble obtained from $\mathfrak{L}^N$ by restriction to the top $k+1$ lines and the interval $\llbracket -T_N, T_N \rrbracket$. In particular, since $\mathfrak{L}^N$ satisfies the Schur Gibbs property we conclude the same is true for $\tilde{\mathfrak{L}}^N$ and moreover Assumptions 1 and 2 in Section \ref{Section1.2} imply that $\{\tilde{\mathfrak{L}}^N\}_{N \geq k+1}$ is an $(\alpha, p, \lambda)$-good in the sense of Definition \ref{Def1}. Specifically, the conditions of Definition \ref{Def1} are satisfied with $N_0 = k+1$, $k$ in the definition equals $k+1$ as above, $\alpha, p, \lambda$ as in the statement of the theorem, $T_N$ as above and $\psi$ as in Assumption 2 in Section \ref{Section1.2}. For the functions $\phi_1, \phi_2$ we may set $\phi_2(\epsilon) = \phi(\epsilon/2)$, where $\phi$ is as in Assumption 2 in Section \ref{Section1.2}, which we recall required that 
$$ \sup_{n \in \mathbb{Z}} \limsup_{N \rightarrow \infty} \mathbb{P} \left( \left|N^{-\alpha/2}(L_1^N(n N^{\alpha}) - p n N^{\alpha} + \lambda n^2 N^{\alpha/2}) \right| \geq \phi(\epsilon) \right) \leq \epsilon.$$
The last equation and the fact that $\phi_2(\epsilon) = \phi(\epsilon/2)$ implies that for each $n \in \mathbb{Z}$ and $\epsilon > 0$ there exists $A(n,\epsilon) \in \mathbb{N}$ such that for $N \geq A(n,\epsilon)$ we have
$$\mathbb{P} \left( \left|N^{-\alpha/2}(L_1^N(n N^{\alpha}) - p n N^{\alpha} + \lambda n^2 N^{\alpha/2}) \right| \geq \phi_2(\epsilon) \right) \leq \epsilon,$$
and then we can set $\phi_1(n,\epsilon) = A(n,\epsilon)$. 

 Since $\{ \tilde{\mathfrak{L}}^N\}_{N \geq k+1}$ is an $(\alpha, p, \lambda)$-good sequence we have from Theorem \ref{PropTightGood} that $\{ \tilde{f}^N_i \}_{i = 1}^{k}$ as in the statement of that theorem for the line ensembles $ \tilde{\mathfrak{L}}^N$ are tight in $(C( \llbracket 1, k   \rrbracket \times \mathbb{R}), \mathcal{C}_{k})$. We may now apply the ``only if'' part of Lemma \ref{2Tight} to $\{\tilde{f}_i^N\}_{i = 1}^k$ and conclude that statements (i) and (ii) from the beginning of this step hold for $\tilde{f}_k^N$, which in turn implies they hold for $\mathcal{L}^N_k$, since by construction $\mathcal{L}_k^N$ and $\tilde{f}_k^N$ have the same distribution.\\

{\bf \raggedleft Step 2.} We next suppose that $\mathcal{L}^\infty$ is any subsequential limit of $\mathcal{L}^N$ and that $n_m \uparrow \infty$ is a sequence such that $\mathcal{L}^{n_m}$ converges weakly to $\mathcal{L}^\infty$. We want to show that $\mathcal{L}^\infty$ satisfies the Brownian Gibbs property. Suppose that $a, b \in \mathbb{R}$ with $a < b$ and $K = \{k_1, k_1+1, \dots, k_2\} \subset \mathbb{N}$ are given. We wish to show that $\mathcal{L}^{\infty}$ is almost surely non-intersecting and for any bounded Borel-measurable function $F: C(K \times [a,b]) \rightarrow \mathbb{R}$ almost surely
\begin{equation}\label{BGPTowerV2S2}
\mathbb{E} \left[ F\left(\mathcal{L}^\infty|_{K \times [a,b]} \right)  {\big \vert} \mathcal{F}_{ext} (K \times (a,b))  \right] =\mathbb{E}_{avoid}^{a,b, \vec{x}, \vec{y}, f, g} \bigl[ F(\tilde{\mathcal{Q}}) \bigr],
\end{equation}
where we use the same notation as in Definition \ref{DefBGP}. In particular, we recall that 
$$\mathcal{F}_{ext} (K \times (a,b)) = \sigma \left \{ \mathcal{L}^\infty_i(s): (i,s) \in D_{K,a,b}^c \right\} \mbox{, with $D_{K,a,b}^c = (\mathbb{N} \times \mathbb{R}) \setminus K \times (a,b).$ } $$

Let $k \geq k_2 +1$ and consider the map $\Pi_{k}: C( \mathbb{N} \times \mathbb{R}) \rightarrow C( \llbracket 1, k  \rrbracket \times \mathbb{R})$ given by $[\Pi_{k}(g)](i,t) = g(i,t)$, which is continuous, and so $\Pi_{k} [\mathcal{L}^{n_m}]$ converge weakly to $\Pi_{k}[\mathcal{L}^\infty]$ as random variables in $C( \llbracket 1, k  \rrbracket \times \mathbb{R})$. If $\{\tilde{f}^N_i \}_{i = 1}^{k}$ are as in Step 1, then we know by construction that the resitrction of $\{\tilde{f}^N_i \}_{i = 1}^{k}$ to $[- \psi(N), \psi(N)]$ has the same distribution as the restriction of $\Pi_{k} [\mathcal{L}^{N}]$ to the same interval. Since $\psi(N) \rightarrow \infty$ by assumption and $\Pi_{k} [\mathcal{L}^{n_m}]$ converge weakly to $\Pi_{k}[\mathcal{L}^\infty]$ we conclude that $\{\tilde{f}^{n_m}_i \}_{i = 1}^{k}$ converge weakly to $\Pi_{k}[\mathcal{L}^\infty]$ (here we used that the topology is that of uniform convergence over compacts). In particular, by the second part of Theorem \ref{PropTightGood} we conclude that $\Pi_{k} [\mathcal{L}^{\infty}]$ satisfies the partial Brownian Gibbs property as a $\llbracket 1, k\rrbracket$-indexed line ensemble on $\mathbb{R}$. The latter implies that $\Pi_k[\mathcal{L}^{\infty}]$ is non-intersecting almost surely and almost surely 
\begin{equation}\label{BGPTowerV2S3}
\mathbb{E} \left[ F\left(\mathcal{L}^\infty|_{K \times [a,b]} \right)  {\big \vert} \tilde{\mathcal{F}}_{ext} (K \times (a,b))  \right] =\mathbb{E}_{avoid}^{a,b, \vec{x}, \vec{y}, f, g} \bigl[ F(\tilde{\mathcal{Q}}) \bigr],
\end{equation}
where
$$\tilde{\mathcal{F}}_{ext} (K \times (a,b)) = \sigma \left \{ \mathcal{L}^\infty_i(s): (i,s) \in \tilde{D}_{K,a,b}^c \right\} \mbox{, with $\tilde{D}_{K,a,b}^c = ( \llbracket 1, k \rrbracket \times \mathbb{R}) \setminus K \times (a,b).$ } $$

Since $\Pi_k[\mathcal{L}^{\infty}]$ is non-intersecting almost surely  and $k \geq k_2 + 1$ was arbitrary we conclude that $\mathcal{L}^{\infty}$ is almost surely non-intersecting. Let $\mathcal{A}$ denote the collection of sets $A$ of the form
$$A = \{ \mathcal{L}^\infty(i_r, x_r) \in B_r \mbox{ for $r = 1, \dots, p$ } \},$$
where $p \in \mathbb{N}$, $B_1, \dots, B_p \in \mathcal{B}(\mathbb{R})$ (the Borel $\sigma$-algebra on $\mathbb{R}$ and $(i_1, x_1), \dots, (i_p, x_p) \in D_{K,a,b}^c$. Since in (\ref{BGPTowerV2S3}) we have that $k \geq k_2 +1$ was arbitrary we conclude that for all $A \in \mathcal{A}$ we have 
$$\mathbb{E} \left[  F\left(\mathcal{L}^\infty|_{K \times [a,b]} \right)   \cdot {\bf 1 }_A \right] = \mathbb{E} \left[  \mathbb{E}_{avoid}^{a,b, \vec{x}, \vec{y}, f, g} \bigl[ F(\tilde{\mathcal{Q}}) \bigr] \cdot {\bf 1 }_A \right].$$
In view of the bounded convergence theorem, we see that the collection of sets $A$ that satisfies the last equation is a $\lambda$-system and as it contains the $\pi$-system $\mathcal{A}$ we conclude by the $\pi-\lambda$ theorem that it contains $\sigma(\mathcal{A})$, which is precisely $  \mathcal{F}_{ext} (K \times (a,b)) $. We may thus conclude (\ref{BGPTowerV2S2}) from the defining properties of conditional expectation and the fact that the right side of (\ref{BGPTowerV2S2}) is $\mathcal{F}_{ext} (K \times (a,b))$-measurable as follows from \cite[Lemma 3.4]{DimMat}. This suffices for the proof.
\end{proof}

\begin{proof}(of Corollary \ref{Thm2}) As explained in Section \ref{Section1.2} we have that Assumption 2' implies Assumption 2 and so by Theorem \ref{Thm1} we know that $\mathcal{L}^N$ is a tight sequence of line ensembles. Let $\mathcal{L}_{sub}^\infty$ be any subsequential limit. We will prove that $\mathcal{L}_{sub}^\infty$ has the same distribution as $\mathcal{L}^\infty$ as in the statement of the theorem. If true, this would imply that $\mathcal{L}^N$ has only one possible subsequential limit (namely $\mathcal{L}^\infty$) which combined with the tightness of $\mathcal{L}^N$ would imply convergence of the sequence to $\mathcal{L}^\infty$.

By Theorem \ref{Thm1} we know that $\mathcal{L}^\infty_{sub}$ satisfies the Brownian Gibbs property and by Assumption 2', we know that $\mathcal{L}^\infty_{sub,1}$ (the top curve of $\mathcal{L}^\infty_{sub}$) has the same distribution as $\mathcal{L}^\infty_1$. In \cite{CorHamA} it was proved that $\mathcal{L}^{Airy}$ satisfies the Brownian Gibbs property and since $\mathcal{L}^\infty_i(t)  = c^{-1/2}\mathcal{L}_i^{Airy}(ct), \mbox{ for $i \in \mathbb{N}$ and $ t\in \mathbb{R}$}$ we conclude that $\mathcal{L}^\infty$ also satisfies the Brownian Gibbs property. To prove the latter one only needs to utilize the fact that if $B_t$ is a standard Brownian motion so is $c^{-1/2} B_{ct}$ -- see e.g. \cite[Lemma 3.5]{DimMat} where a related result is established. Combining all of the above observations, we see that $\mathcal{L}^\infty_{sub}$ and $\mathcal{L}^\infty$ both satisfy the Brownian Gibbs property and have the same top curve distribution, which by \cite[Theorem 1.1]{DimMat} implies that $\mathcal{L}^\infty_{sub}$ and $\mathcal{L}^\infty$ have the same law.
\end{proof}

%
\section{Properties of Bernoulli line ensembles}\label{Section3} In this section we derive several results for Bernoulli line ensembles, which will be used in the proof of Theorem \ref{PropTightGood} in Section \ref{Section4}.

%
\subsection{Monotone coupling lemmas}\label{Section3.1}
 In this section we formulate two lemmas that provide couplings of two Bernoulli line ensembles of non-crossing Bernoulli bridges on the same interval, which depend monotonically on their boundary data. Schematic depictions of the couplings are provided in Figure \ref{fig:MCL}. We postpone the proof of these lemmas until Section \ref{Section8}. 
\begin{figure}[ht]
  \includegraphics[width=0.95\textwidth]{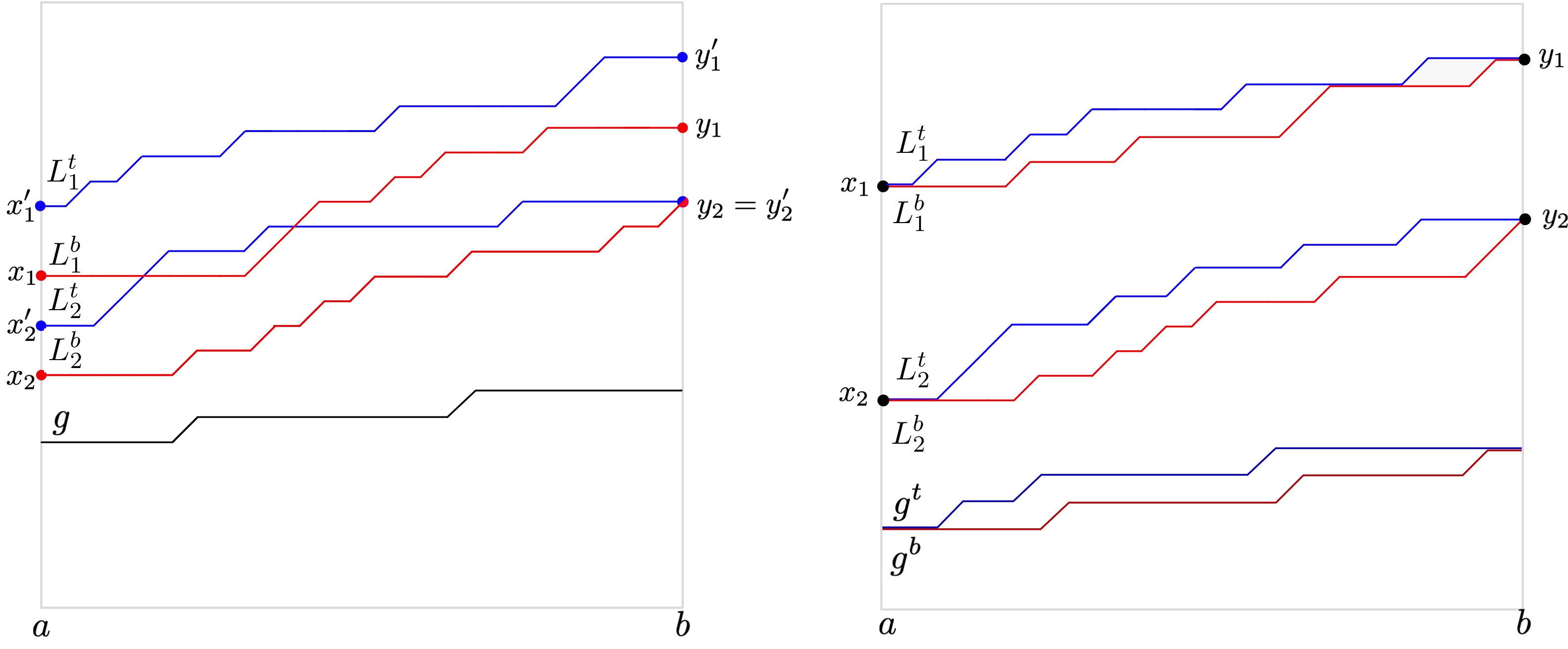}
  \caption{Two diagrammatic depictions of the monotone coupling Lemma \ref{MCLxy} (left part) and Lemma \ref{MCLfg} (right part). Red depicts the lower line ensemble and accompanying entry data, exit data, and bottom bounding curve, while blue depicts that of the higher ensemble.}
  \label{fig:MCL}
\end{figure}

\begin{lemma}\label{MCLxy} Assume the same notation as in Definition \ref{DefAvoidingLawBer}. Fix $k \in \mathbb{N}$, $T_0, T_1 \in \mathbb{Z}$ with $T_0 < T_1$, $S\subseteq\llbracket T_0,T_1\rrbracket$, a function $g: \llbracket T_0, T_1 \rrbracket  \rightarrow [-\infty, \infty)$ as well as $\vec{x}, \vec{y}, \vec{x}\,', \vec{y}\,' \in \mathfrak{W}_k$. Assume that $\Omega_{avoid}(T_0, T_1, \vec{x}, \vec{y}, \infty,g; S)$ and $\Omega_{avoid}(T_0, T_1, \vec{x}', \vec{y}', \infty,g; S)$ are both non-empty. Then there exists a probability space $(\Omega, \mathcal{F}, \mathbb{P})$, which supports two $\llbracket 1, k \rrbracket$-indexed Bernoulli line ensembles $\mathfrak{L}^t$ and $\mathfrak{L}^b$ on $\llbracket T_0, T_1 \rrbracket$ such that the law of $\mathfrak{L}^{t}$ {\big (}resp. $\mathfrak{L}^b${\big )} under $\mathbb{P}$ is given by $\mathbb{P}_{avoid, Ber; S}^{T_0, T_1, \vec{x}\,', \vec{y}\,', \infty, g}$ {\big (}resp. $\mathbb{P}_{avoid, Ber; S}^{T_0, T_1, \vec{x}, \vec{y}, \infty, g}${\big )} and such that $\mathbb{P}$-almost surely we have ${L}_i^t(r) \geq {L}^b_i(r)$ for all $i = 1,\dots, k$ and $r \in \llbracket T_0, T_1 \rrbracket$.
\end{lemma}

\begin{lemma}\label{MCLfg} Assume the same notation as in Definition \ref{DefAvoidingLawBer}. Fix $k \in \mathbb{N}$,  $T_0, T_1 \in \mathbb{Z}$ with $T_0 < T_1$, $S\subseteq\llbracket T_0, T_1\rrbracket$, two functions $g^t, g^b: \llbracket T_0, T_1 \rrbracket \rightarrow [-\infty,\infty)$ and $\vec{x}, \vec{y} \in \mathfrak{W}_k$. We assume that $g^t(r) \geq g^b(r)$ for all $r \in \llbracket T_0, T_1 \rrbracket$ and that $\Omega_{avoid}(T_0, T_1, \vec{x}, \vec{y}, \infty,g^t; S)$ and $\Omega_{avoid}(T_0, T_1, \vec{x}, \vec{y}, \infty,g^b; S)$ are both non-empty. Then there exists a probability space $(\Omega, \mathcal{F}, \mathbb{P})$, which supports two $\llbracket 1, k \rrbracket$-indexed Bernoulli line ensembles $\mathfrak{L}^t$ and $\mathfrak{L}^b$ on $\llbracket T_0, T_1 \rrbracket$ such that the law of $\mathfrak{L}^{t}$ {\big (}resp. $\mathfrak{L}^b${\big )} under $\mathbb{P}$ is given by $\mathbb{P}_{avoid, Ber; S}^{T_0, T_1, \vec{x}, \vec{y}, \infty, g^t}$ {\big (}resp. $\mathbb{P}_{avoid, Ber; S}^{T_0,T_1, \vec{x}, \vec{y}, \infty, g^b}${\big )} and such that $\mathbb{P}$-almost surely we have ${L}_i^t(r) \geq {L}^b_i(r)$ for all $i = 1,\dots, k$ and $r \in \llbracket T_0, T_1 \rrbracket$.
\end{lemma}

In plain words, Lemma \ref{MCLxy} states that one can couple two Bernoulli line ensembles $\mathfrak{L}^{t}$ and $\mathfrak{L}^{b}$ of non-crossing Bernoulli bridges, bounded from below by the same function $g$, in such a way that if all boundary values of $\mathfrak{L}^{t}$ are above the respective boundary values of $\mathfrak{L}^{b}$, then all up-right paths of $\mathfrak{L}^{t}$ are almost surely above the respective up-right paths of $\mathfrak{L}^{b}$. See the left part of Figure \ref{fig:MCL}. Lemma \ref{MCLfg}, states that one can couple two Bernoulli line ensembles $\mathfrak{L}^{t}$ and $\mathfrak{L}^{b}$ that have the same boundary values, but the lower bound $g^t$ of $\mathfrak{L}^{t}$ is above the lower bound $g^b$ of $\mathfrak{L}^{b}$, in such a way that all up-right paths of $\mathfrak{L}^{t}$ are almost surely above the respective up-right paths of $\mathfrak{L}^{b}$. See the right part of Figure \ref{fig:MCL}.

%
\subsection{Properties of Bernoulli and Brownian bridges}\label{Section3.2} In this section we derive several results about Bernoulli bridges, which are random up-right paths that have law $\mathbb{P}_{Ber}^{T_0, T_1, x,y}$ as in Section \ref{Section2.2}, as well as Brownian bridges with law $\mathbb{P}^{T_0,T_1,x,y}_{free}$ as in Section \ref{Section2.1}. Our results will rely on the two monotonicity Lemmas \ref{MCLxy} and \ref{MCLfg} as well as a strong coupling between Bernoulli bridges and Brownian bridges from \cite{CD} -- recalled here as Theorem \ref{KMT}.

If $W_t$ denotes a standard one-dimensional Brownian motion and $\sigma > 0$, then the process
$$B^{\sigma}_t = \sigma (W_t - t W_1), \hspace{5mm} 0 \leq t \leq 1,$$
is called a {\em Brownian bridge (conditioned on $B_0 = 0, B_1 = 0$) with diffusion parameter $\sigma$.} We note that $B^\sigma$ is the unique a.s. continuous Gaussian process on $[0,1]$ with $B_0 = B_1  = 0$, $\ex[B^\sigma_t] = 0$, and
\begin{equation}\label{BBcovar}
\ex[B^\sigma_r B^\sigma_s] = \sigma^2(r\wedge s - rs - sr + sr) = \sigma^2 (r\wedge s - rs).
\end{equation}  
With the above notation we state the strong coupling result we use.
\begin{theorem}\label{KMT}
Let $p \in (0,1)$. There exist constants $0 < C, a, \alpha < \infty$ (depending on $p$) such that for every positive integer $n$, there is a probability space on which are defined a Brownian bridge $B^\sigma$ with diffusion parameter $\sigma = \sqrt{p(1-p)}$ and a family of random paths $\ell^{(n,z)} \in \Omega(0,n, 0, z)$ for $z = 0,\dots,n$ such that $\ell^{(n,z)}$ has law $\mathbb{P}^{0,n,0,z}_{Ber}$ and
\begin{equation}\label{KMTeq}
\mathbb{E}\left[ e^{a \Delta(n,z)} \right] \leq C e^{\alpha (\log n)^2}e^{|z- p n|^2/n}, \mbox{ where $\Delta(n,z):=  \sup_{0 \leq t \leq n} \left| \sqrt{n} B^\sigma_{t/n} + \frac{t}{n}z - \ell^{(n,z)}(t) \right|.$}
\end{equation}
\end{theorem}
\begin{remark} When $p = 1/2$ the above theorem follows (after a trivial affine shift) from \cite[Theorem 6.3]{LF} and the general $p \in (0,1)$ case was done in \cite[Theorem 4.5]{CD}. We mention that a significant generalization of Theorem \ref{KMT} for general random walk bridges has recently been proved in \cite[Theorem 2.3]{DW19}, and in particular the inequality in (\ref{KMTeq}) was shown to hold with $(\log n)^2$ replaced with $\log n$.
\end{remark}

We will use the following simple corollary of Theorem \ref{KMT} to compare Bernoulli bridges with Brownian bridges. We use the same notation as in the theorem.

\begin{corollary}\label{Cheb}
	Fix $p\in (0,1)$, $\beta > 0$, and $A>0$. Suppose $|z-pn| \leq K\sqrt{n}$ for a constant $K>0$. Then for any $\epsilon > 0$, there exists $N$ large enough depending on $p,\epsilon,A,K$ so that for $n\geq N$,
	\[
	\mathbb{P}\Big(\Delta(n,z) \geq An^\beta\Big) < \epsilon.
	\]
\end{corollary}

\begin{proof}
	Applying Chebyshev's inequality and \eqref{KMTeq} gives
	\begin{align*}
	\mathbb{P}\Big(\Delta(n,z) \geq An^\beta\Big) &\leq e^{-An^\beta}\ex\Big[e^{a\Delta(n,z)}\Big] \leq C\exp\Big[-An^\beta + \alpha(\log n)^2 + \frac{|z-pn|^2}{n}\Big]\\
	&\leq C\exp\Big[-An^\beta + \alpha(\log n)^2 + K\Big].
	\end{align*}
	The conclusion is now immediate.
\end{proof}

We also state the following result regarding the distribution of the maximum of a Brownian bridge, which follows from formulas in \cite[Section 12.3]{Dudley}.

\begin{lemma}\label{BBmax}
	Fix $p\in (0,1)$, and let $B^\sigma$ be a Brownian bridge of diffusion parameter $\sigma = \sqrt{p(1-p)}$ on $[0,1]$. Then for any $C,T> 0$ we have
	\begin{equation}\label{BBmaxeq}
	\begin{split}
	\mathbb{P}\left(\max_{s\in[0,T]} B^\sigma_{s/T} \geq C\right) &= \exp\left( - \frac{2C^2}{p(1-p)}\right), \\ \mathbb{P}\left(\max_{s\in[0,T]} \big| B^\sigma_{s/T} \big| \geq C\right) &= 2\sum_{n=1}^\infty (-1)^{n-1} \exp\left(-\frac{2n^2C^2}{p(1-p)}\right).
	\end{split}
	\end{equation}
	In particular,
	\begin{equation}\label{sepBd}
	\mathbb{P}\left(\max_{s\in[0,T]} \big|B^\sigma_{s/T}\big| \geq C\right) \leq 2\exp\left( - \frac{2C^2}{p(1-p)}\right).
	\end{equation}
\end{lemma}

\begin{proof}
	Let $B^1$ be a Brownian bridge with diffusion parameter 1 on $[0,1]$. Then $B^\sigma_t$ has the same distribution as $\sigma B^1_t$. Hence
	\begin{equation*}
	\mathbb{P}\left( \max_{s\in[0,T]} B^\sigma_{s/T} \geq C \right) = \mathbb{P}\left( \max_{t\in[0,1]} B^1_t \geq C/\sigma \right) = e^{-2(C/\sigma)^2} = e^{-2C^2/p(1-p)}.
	\end{equation*}
	The second equality follows from \cite[Proposition 12.3.3]{Dudley}. This proves the first equality in \eqref{BBmaxeq}. Similarly, using \cite[Proposition 12.3.4]{Dudley} we find
	\begin{equation*}
	\mathbb{P}\left( \max_{s\in[0,T]} \big| B^\sigma_{s/T}\big| \geq C \right) = \mathbb{P}\left( \max_{t\in[0,1]} \big| B^1_t\big| \geq C/\sigma \right) = 2\sum_{n=1}^\infty (-1)^{n-1}e^{-2n^2C^2/\sigma^2},
	\end{equation*}
	proving the second inequality in \eqref{BBmaxeq}.
	
	Lastly to prove \eqref{sepBd}, observe that since $B^\sigma_t$ has mean 0, $B^\sigma_t$ and $-B^\sigma_t$ have the same distribution. It follows from the first equality above that
	\begin{equation*}
	\begin{split}
	&\mathbb{P}\left( \max_{s\in[0,T]} \big| B^\sigma_{s/T}\big| \geq C \right) \leq \mathbb{P}\left( \max_{s\in[0,T]}  B^\sigma_{s/T} \geq C \right) + \mathbb{P}\left( \max_{s\in[0,T]}  \big(-B^\sigma_{s/T}\big) \geq C \right) = \\
	&2\,\mathbb{P}\left( \max_{s\in[0,T]}  B^\sigma_{s/T} \geq C \right) = 2e^{-2C^2/p(1-p)}.
	\end{split}
	\end{equation*}
\end{proof}

We state one more lemma about Brownian bridges, which allows us to decompose a bridge on $[0,1]$ into two independent bridges with Gaussian affine shifts meeting at a point in $(0,1)$.

\begin{lemma}\label{2bridges}
	Fix $p\in (0,1)$, $T>0$, $t\in(0,T)$. Let $\xi$ be a Gaussian random variable with mean 0 and variance
	\[
	\ex[\xi^2] = \sigma^2\frac{t}{T}\left(1-\frac{t}{T}\right).
	\]
	Let $B^1,B^2$ be two independent Brownian bridges on $[0,1]$ with diffusion parameters $\sigma \sqrt{t/T}$ and $\sigma \sqrt{(T-t)/T}$ respectively, also independent from $B^\sigma$. Define the process
	\[
	\tilde{B}_{s/T} = \begin{dcases}
	\frac{s}{t}\,\xi + B^1\Big(\frac{s}{t}\Big), & s\leq t,\\
	\frac{T-s}{T-t}\,\xi + B^2\Big(\frac{s-t}{T-t}\Big), & s\geq t,
	\end{dcases}
	\]
	for $s\in [0,T]$. Then $\tilde{B}$ is a Brownian bridge with diffusion parameter $\sigma$.
\end{lemma}

\begin{proof}
	It is clear that the process $\tilde{B}$ is a.s. continuous. Since $\tilde{B}$ is built from three independent zero-centered Gaussian processes, it is itself a zero-centered Gaussian process and thus completely characterized by its covariance. Consequently, to show that $\tilde{B}$ is a Brownian bridge with diffusion parameter $\sigma$, it suffices to show by (\ref{BBcovar}) that if $0\leq r\leq s\leq T$ we have
	\begin{equation}\label{covar}
	\ex[\tilde{B}_{r/T}\tilde{B}_{s/T}] = \sigma^2 \frac{r}{T}\Big(1-\frac{s}{T}\Big).
	\end{equation}
	First assume $s\leq t$ Using the fact that $\xi$ and $B^1_\cdot$ are independent with mean 0, we find
	\begin{equation*}
	\begin{split}
	&\ex[\tilde{B}_{r/T}\tilde{B}_{s/T}] = \frac{rs}{t^2}\cdot\sigma^2\frac{t}{T}\Big(1-\frac{t}{T}\Big) + \sigma^2\frac{t}{T}\cdot\frac{r}{t}\Big(1-\frac{s}{t}\Big) =\\
	&\sigma^2\frac{r}{T}\Big(\frac{s}{t} - \frac{s}{T} + 1 - \frac{s}{t}\Big) = \sigma^2\frac{r}{T}\Big(1-\frac{s}{T}\Big).
	\end{split}
	\end{equation*}
	If $r\geq t$, we compute
	\begin{equation*}
	\begin{split}
	&\ex[\tilde{B}_{r/T}\tilde{B}_{s/T}] = \frac{(T-r)(T-s)}{(T-t)^2}\cdot\sigma^2\frac{t}{T}\Big( 1 - \frac{t}{T}\Big) + \sigma^2\frac{T-t}{T}\cdot\frac{r-t}{T-t}\Big( 1 - \frac{s-t}{T-t}\Big) =\\
	&\frac{\sigma^2(T-s)}{T(T-t)}\left(\frac{t(T-r)}{T} + r-t\right) = \frac{\sigma^2(T-s)}{T(T-t)}\cdot\frac{r(T-t)}{T} = \sigma^2\frac{r}{T}\Big(1-\frac{s}{T}\Big).
	\end{split}
	\end{equation*}
	If $r < t < s$, then since $\xi$, $B^1_\cdot$, and $B^2_\cdot$ are all independent, we have
	\[
	\ex[\tilde{B}_{r/T}\tilde{B}_{s/T}] = \frac{r}{t}\cdot\frac{T-s}{T-t}\cdot\sigma^2\frac{t(T-t)}{T^2} = \sigma^2\frac{r(T-s)}{T^2} = \sigma^2\frac{r}{T}\Big(1-\frac{s}{T}\Big).
	\]
	This proves \eqref{covar} in all cases.
\end{proof}

Below we list four lemmas about Bernoulli bridges. We provide a brief informal explanation of what each result says after it is stated. All four lemmas are proved in a similar fashion. For the first two lemmas one observes that the event whose probability is being estimated is monotone in $\ell$. This allows us by Lemmas \ref{MCLxy} and \ref{MCLfg} to replace $x,y$ in the statements of the lemmas with the extreme values of the ranges specified in each. Once the choice of $x$ and $y$ is fixed one can use our strong coupling results, Theorem \ref{KMT} and Corollary \ref{Cheb}, to reduce each of the lemmas to an analogous one involving a Brownian bridge with some prescribed diffusion parameter. The latter statements are then easily confirmed as one has exact formulas for Brownian bridges, such as Lemma \ref{BBmax}.\\

\begin{lemma}\label{LemmaHalfS4} Fix $p \in (0,1)$, $T \in \mathbb{N}$ and $x, y\in \mathbb{Z}$ such that $T \geq y-x \geq 0$, and suppose that $\ell$ has distribution $\mathbb{P}^{0,T,x,y}_{Ber}$. Let $M_1, M_2 \in \mathbb{R}$ be given. Then we can find $W_0 = W_0(p,M_2 - M_1) \in \mathbb{N}$ such that for $T \geq W_0$, $x \geq M_1 T^{1/2}$, $y \geq pT + M_2 T^{1/2}$ and $s \in [0,T]$ we have
\begin{equation}\label{halfEq1S4}
\mathbb{P}^{0,T,x,y}_{Ber}\left( \ell(s)  \geq \frac{T-s}{T} \cdot M_1 T^{1/2} + \frac{s}{T} \cdot \big(p T + M_2 T^{1/2}\big) - T^{1/4} \right) \geq \frac{1}{3}.
\end{equation}
\end{lemma}
\begin{remark}
If $M_1, M_2 = 0$ then Lemma \ref{LemmaHalfS4} states that if a Bernoulli bridge $\ell$ is started from $(0,x)$ and terminates at $(T,y)$, which are above the straight line of slope $p$, then at any given time $s \in [0,T]$ the probability that $\ell(s)$ goes a modest distance below the straight line of slope $p$ is upper bounded by $ 2/3$.
\end{remark}
\begin{proof}
	Define $A = \lfloor M_1T^{1/2}\rfloor$ and $B = \lfloor pT + M_2 T^{1/2}\rfloor$. Then since $A\leq x$ and $B\leq y$, it follows from Lemma \ref{MCLxy} that there is a probability space with measure $\mathbb{P}_0$ supporting random variables $L_1$ and $L_2$, whose laws under $\mathbb{P}_0$ are $\mathbb{P}^{0,T,A,B}_{Ber}$ and $\mathbb{P}^{0,T,x,y}_{Ber}$ respectively, and $\mathbb{P}_0$-a.s. $L_1\leq L_2$. Thus
	\begin{equation}\label{HalfS4MC}
	\begin{split}
	&\mathbb{P}^{0,T,x,y}_{Ber}\left( \ell(s)  \geq \frac{T-s}{T} \cdot M_1 T^{1/2} + \frac{s}{T} \cdot \big(p T + M_2 T^{1/2}\big) - T^{1/4} \right) =\\
	& \mathbb{P}_0\left( L_2(s)  \geq \frac{T-s}{T} \cdot M_1 T^{1/2} + \frac{s}{T} \cdot \big(p T + M_2 T^{1/2}\big) - T^{1/4} \right) \geq\\
	&\mathbb{P}_0\left( L_1(s)  \geq \frac{T-s}{T} \cdot M_1 T^{1/2} + \frac{s}{T} \cdot \big(p T + M_2 T^{1/2}\big) - T^{1/4} \right) =\\
	&\mathbb{P}^{0,T,A,B}_{Ber}\left( \ell(s)  \geq \frac{T-s}{T} \cdot M_1 T^{1/2} + \frac{s}{T} \cdot \big(p T + M_2 T^{1/2}\big) - T^{1/4} \right).
	\end{split}
	\end{equation}
	Since the uniform distribution on upright paths on $\llbracket 0,T\rrbracket \times \llbracket A,B\rrbracket$ is the same as that on upright paths on $\llbracket 0,T\rrbracket \times \llbracket 0, B-A\rrbracket$ shifted vertically by $A$, the last line of \eqref{HalfS4MC} is equal to
	\[
	\mathbb{P}^{0,T,0,B-A}_{Ber}\left( \ell(s) + A  \geq \frac{T-s}{T} \cdot M_1 T^{1/2} + \frac{s}{T} \cdot \big(p T + M_2 T^{1/2}\big) - T^{1/4} \right).
	\]
	Now we employ the coupling provided by Theorem \ref{KMT}. We have another probability space $(\Omega,\mathcal{F},\mathbb{P})$ supporting a random variable $\ell^{(T,B-A)}$ whose law under $\mathbb{P}$ is $\mathbb{P}^{0,T,0,B-A}_{Ber}$ as well as a Brownian bridge $B^\sigma$ coupled with $\ell^{(T,B-A)}$. We have 
	\begin{equation}\label{HalfS4split}
	\begin{split}
	&\mathbb{P}^{0,T,0,B-A}_{Ber}\left( \ell(s) + A  \geq \frac{T-s}{T} \cdot M_1 T^{1/2} + \frac{s}{T} \cdot \big(p T + M_2 T^{1/2}\big) - T^{1/4} \right) =\\
	& \mathbb{P}\left( \ell^{(T,B-A)}(s) + A \geq \frac{T-s}{T} \cdot M_1 T^{1/2} + \frac{s}{T} \cdot \big(p T + M_2 T^{1/2}\big) - T^{1/4} \right) =\\
	& \mathbb{P}\bigg( \left[\ell^{(T,B-A)}(s) - \sqrt{T} B^\sigma_{s/T} - \frac{s}{T}\cdot(B-A)\right] + \sqrt{T}B^\sigma_{s/T} \geq \\
	&\qquad -A-\frac{s}{T}\cdot(B-A) +\frac{T-s}{T} \cdot M_1 T^{1/2} + \frac{s}{T} \cdot \big(p T + M_2 T^{1/2}\big) - T^{1/4} \bigg).
	\end{split}
	\end{equation}
	From the definitions of $A$ and $B$, we can rewrite the quantity in the last line of \eqref{HalfS4split} and bound by
	\begin{align*}
	&\frac{T-s}{T}\cdot(M_1T^{1/2}-A) + \frac{s}{T}\cdot(pT + M_2T^{1/2} - B) - T^{1/4} \leq \\
	& \frac{T-s}{T} + \frac{s}{T} - T^{1/4} = -T^{1/4} + 1.
	\end{align*}
	Thus the last line of \eqref{HalfS4MC} is bounded below by
	\begin{equation}\label{HalfS4KMT}
	\begin{split}
	& \mathbb{P}\left( \left[\ell^{(T,B-A)}(s) - \sqrt{T} B^\sigma_{s/T} - \frac{s}{T}\cdot(B-A)\right] + \sqrt{T}B^\sigma_{s/T} \geq -T^{1/4} + 1 \right) \geq\\
	& \mathbb{P}\left( \sqrt{T}B^\sigma_{s/T} \geq 0 \quad \mathrm{and} \quad \Delta(T,B-A) < T^{1/4} - 1 \right) \geq\\
	& \mathbb{P}\left( B^\sigma_{s/T} \geq 0 \right) - \mathbb{P}\left( \Delta(T,B-A) \geq T^{1/4} - 1 \right) =\\
	& \frac{1}{2} - \mathbb{P}\left( \Delta(T,B-A) \geq T^{1/4} - 1 \right).
	\end{split}
	\end{equation}
	For the first inequality, we used the fact that the quantity in brackets is bounded in absolute value by $\Delta(T,B-A)$. The second inequality follows by dividing the event $\{B^\sigma_{s/T}\geq 0\}$ into cases and applying subadditivity. Since $|B-A-pT|\leq (|M_2-M_1|+1)\sqrt{T}$, Corollary \ref{Cheb} allows us to choose $W_0$ large enough depending on $p$ and $M_2-M_1$ so that if $T \geq W_0$, then the last line of \eqref{HalfS4KMT} is bounded above by $1/2 - 1/6 = 1/3$. In combination with \eqref{HalfS4MC} this proves \eqref{halfEq1S4}.
\end{proof}

\begin{lemma}\label{LemmaMinFreeS4} Fix $p \in (0,1)$, $T \in \mathbb{N}$ and $y,z\in \mathbb{Z}$ such that $T \geq y,z \geq 0$, and suppose that $\ell_y,\ell_z$ have distributions $\mathbb{P}^{0,T,0,y}_{Ber}$, $\mathbb{P}^{0,T,0,z}_{Ber}$ respectively. Let $M > 0$ and $\epsilon > 0$ be given. Then we can find $W_1=W_1(M,p, \epsilon) \in \mathbb{N}$ and $A=A(M,p, \epsilon) > 0$ such that for $T \geq W_1$, $ y \geq p T -  MT^{1/2}$, $z \leq pT + MT^{1/2}$ we have
\begin{equation}\label{minFree1S4}
\begin{split}
&\mathbb{P}^{0,T,0,y}_{Ber}\left( \inf_{s \in [ 0, T]}\big[ \ell_y(s) -  ps \big] \leq -AT^{1/2} \right) \leq \epsilon, \\ &\mathbb{P}^{0,T,0,z}_{Ber}\,\bigg( \sup_{s \in [ 0, T]}\big[ \ell_z(s) -  ps \big] \geq AT^{1/2} \bigg) \leq \epsilon.
\end{split}
\end{equation}
\end{lemma}
\begin{remark} Roughly, Lemma \ref{LemmaMinFreeS4} states that if a Bernoulli bridge $\ell$ is started from $(0,0)$ and terminates at time $T$ not significantly lower (resp. higher) than the straight line of slope $p$, then the event that $\ell$ goes significantly below (resp. above) the straight line of slope $p$ is very unlikely.
\end{remark}
\begin{proof}
	The two inequalities are proven in essentially the same way. We begin with the first inequality. If $B=\lfloor pT - MT^{1/2}\rfloor$ then it follows from Lemma \ref{MCLxy} that
	\begin{equation}\label{MinFreeS4MC}
	\mathbb{P}^{0,T,0,y}_{Ber}\left( \inf_{s \in [ 0, T]}\big[ \ell_y(s) -  ps \big] \leq -AT^{1/2} \right) \leq \mathbb{P}^{0,T,0,B}_{Ber}\left( \inf_{s \in [ 0, T]}\big[ \ell(s) -  ps \big] \leq -AT^{1/2} \right),
	\end{equation}
	where $\ell$ has law $\mathbb{P}^{0,T,0,B}_{Ber}$. By Theorem \ref{KMT}, there is a probability space $(\Omega,\mathcal{F},\mathbb{P})$ supporting a random variable $\ell^{(T,B)}$ whose law under $\mathbb{P}$ is also $\mathbb{P}^{0,T,0,B}_{Ber}$, and a Brownian bridge $B^\sigma$ with diffusion parameter $\sigma = \sqrt{p(1-p)}$. Therefore
	\begin{equation}\label{MinFreeS4ineq}
	\begin{split}
	&\mathbb{P}^{0,T,0,B}_{Ber}\left( \inf_{s \in [ 0, T]}\big[ \ell(s) -  ps \big] \leq -AT^{1/2} \right) = \mathbb{P}\left( \inf_{s \in [ 0, T]}\big[ \ell^{(T,B)}(s) -  ps \big] \leq -AT^{1/2} \right) \leq\\
	& \mathbb{P}\left( \inf_{s \in [ 0, T]}  \sqrt{T}B^\sigma_{s/T} \leq -\frac{1}{2}AT^{1/2} \right) + \mathbb{P}\left( \sup_{s\in [0,T]} \left|\sqrt{T} B^\sigma_{s/T} + ps - \ell^{(T,B)}(s) \right| \geq \frac{1}{2}AT^{1/2} \right) \leq\\
	& \mathbb{P}\left( \max_{s\in[0,T]} B^\sigma_{s/T} \geq A/2 \right) + \mathbb{P}\left(\Delta(T,B) \geq \frac{1}{2}AT^{1/2} - MT^{1/2} - 1\right). 
	\end{split}
	\end{equation}
	For the first term in the last line, we used the fact that $B^\sigma$ and $-B^\sigma$ have the same distribution. For the second term, we used the fact that
	\begin{align*}
	\sup_{s\in[0,T]}\Big| ps - \frac{s}{T}\cdot B \Big| &\leq \sup_{s\in[0,T]}\Big| ps - \frac{pT-MT^{1/2}}{T}\cdot s \Big| + 1 = MT^{1/2} + 1.
	\end{align*}
	By Lemma \ref{BBmax}, the first term in the last line of \eqref{MinFreeS4ineq} is equal to $e^{-A^2/2p(1-p)}$. If we choose $A \geq \sqrt{2p(1-p)\log(2/\epsilon)}$, then this is $\leq \epsilon/2$. If we also take $A > 2M$, then since $|B-pT| \leq (M+1)\sqrt{T}$, Corollary \ref{Cheb} gives us a $W_1$ large enough depending on $M,p,\epsilon$ so that the second term in the last line of \eqref{MinFreeS4ineq} is also $<\epsilon/2$ for $T\geq W_1$. Adding the two terms and using \eqref{MinFreeS4MC} gives the first inequality in \eqref{minFree1S4}.
	
	If we replace $B$ with $\lceil pT + MT^{1/2} \rceil$ and change signs and inequalities where appropriate, then the same argument proves the second inequality in \eqref{minFree1S4}.
\end{proof}

We need the following definition for our next result. For a function $f \in C([a,b])$ we define its {\em modulus of continuity} for $\delta > 0$ by
\begin{equation}\label{MOCS4}
w(f,\delta) = \sup_{\substack{x,y \in [a,b]\\ |x-y| \leq \delta}} |f(x) - f(y)|.
\end{equation}
\begin{lemma}\label{MOCLemmaS4}Fix $p \in (0,1)$, $T \in \mathbb{N}$ and $y\in \mathbb{Z}$ such that $T \geq y \geq 0$, and suppose that $\ell$ has distribution $\mathbb{P}^{0,T,0,y}_{Ber}$. For each positive $M$, $\epsilon$ and $\eta$, there exist a $\delta(\epsilon, \eta, M) > 0$ and $W_2 = W_2(M, p, \epsilon, \eta) \in \mathbb{N}$ such that  for $T \geq W_2$ and $|y - pT| \leq MT^{1/2}$ we have
\begin{equation}\label{MOCeqS4}
\mathbb{P}^{0,T,0,y}_{Ber}\left( w\big({f^\ell},\delta\big) \geq \epsilon \right) \leq \eta,
\end{equation}
where $f^\ell(u) = T^{-1/2}\big(\ell(uT) - puT\big)$  for $u \in [0,1]$.
\end{lemma}
\begin{remark}
Lemma \ref{MOCLemmaS4} states that if $\ell$ is a Bernoulli bridge that is started from $(0,0)$ and terminates at $(T,y)$ with $y$ close to $pT$ (i.e. with well-behaved endpoints) then the modulus of continuity of $\ell$ is also well-behaved with high probability.
\end{remark}
\begin{proof}
	By Theorem \ref{KMT}, we have a probability measure $\mathbb{P}$ supporting a random variable $\ell^{(T,y)}$ with law $\mathbb{P}^{0,T,0,y}_{Ber}$ as well as a Brownian bridge $B^\sigma$ with diffusion parameter $\sigma = \sqrt{p(1-p)}$. We have
	\begin{equation}\label{MOCKMT}
	\mathbb{P}^{0,T,0,y}_{Ber}\Big( w\big({f^\ell},\delta\big) \geq \epsilon \Big) = \mathbb{P}\Big( w\big(f^{\ell^{(T,y)}},\delta\big) \geq \epsilon \Big),
	\end{equation}
	and
	\begin{equation}\label{MOCbound}
	\begin{split}
	& w\big(f^{\ell^{(T,y)}},\delta\big) = T^{-1/2} \sup_{s,t\in[0,1],\, |s-t|\leq\delta} \Big| \ell^{(T,y)}(sT) - psT - \ell^{(T,y)}(tT) + ptT \Big| \leq\\
	&T^{-1/2} \sup_{s,t \in [0,1], \, |s-t| \leq \delta} \bigg(\left| \sqrt{T}\,B^\sigma_s + sy - psT - \sqrt{T}\,B^\sigma_t - ty + ptT \right| +\\
	&\qquad \qquad \left|\sqrt{T}\,B^\sigma_s + sy - \ell^{(T,y)}(sT)\right| + \left|\sqrt{T}\,B^\sigma_t + ty - \ell^{(T,y)}(tT)\right|\bigg) \leq\\
	& \sup_{s,t \in [0,1], \, |s-t| \leq \delta} \left| B^\sigma_s - B^\sigma_t + T^{-1/2} (y-pT)(s-t)\right| + 2T^{-1/2}\Delta(T,y) \leq\\
	& w\big(B^\sigma,\delta\big) + M\delta + 2T^{-1/2}\Delta(T,y).
	\end{split}
	\end{equation}
	The last line follows from the assumption that $|y-pT|\leq MT^{1/2}$. Now \eqref{MOCKMT} and \eqref{MOCbound} imply
	\begin{equation}\label{MOCsplit}
	\begin{split}
	&\mathbb{P}^{0,T,0,y}_{Ber}\left( w\big(f^{\ell},\delta\big) \geq \epsilon \right) \leq \mathbb{P}\left( w\big(B^\sigma,\delta\big) + M\delta + 2T^{-1/2}\Delta(T,y) \geq \epsilon \right) \leq\\
	&\mathbb{P}\left( w\big(B^\sigma,\delta\big) + M\delta \geq \epsilon/2 \right) + \mathbb{P}\left( \Delta(T,y) \geq \epsilon\, T^{1/2}/4 \right).
	\end{split}
	\end{equation}
	Corollary \ref{Cheb} gives us a $W_2$ large enough depending on $M,p,\epsilon,\eta$ so that the second term in the second line of \ref{MOCsplit} is $\leq\eta/2$ for $T\geq W_2$. Since $B^\sigma$ is a.s. uniformly continuous on the compact interval $[0,1]$, $w(B^\sigma,\delta) \to 0$ as $\delta\to 0$. Thus we can find $\delta_0>0$ small enough depending on $\epsilon,\eta$ so that $w(B^\sigma,\delta_0) < \epsilon/4$ with probability at least $1-\eta/2$. Then with $\delta = \min(\delta_0, \epsilon/4M)$, the first term in the second line of \eqref{MOCsplit} is $\leq\eta/2$ as well. This implies \eqref{MOCeqS4}.
\end{proof}

\begin{lemma}\label{CurveSeparation} Fix $T\in\mathbb{N}$, $p\in (0,1)$, $C,K>0$, and $a,b\in \mathbb{Z}$ such that $\Omega(0,T,a,b)$ is nonempty. Let $\ell_{bot} \in \Omega(0,T,a,b)$ or $\ell_{bot} = -\infty$. Suppose $\vec{x},\vec{y}\in\mathfrak{W}_{k-1}$, $k\geq 2$, are such that $T \geq y_i - x_i \geq 0$ for $1\leq i\leq k-1$. Write $\vec{z} = \vec{y} - \vec{x}$, and suppose that
	\begin{enumerate}[label=(\arabic*)]
		
		\item $x_{k-1} + (z_{k-1}/T)s - \ell_{bot}(s) \geq C\sqrt{T}$ for all $s\in[0,T]$
		
		\item $x_i - x_{i+1} \geq C\sqrt{T}$ and $y_i - y_{i+1} \geq C\sqrt{T}$ for $1\leq i\leq k-2$,
		
		\item $|z_i - pT| \leq K\sqrt{T}$ for $1\leq i\leq k-1$, for a constant $K > 0$.
		
	\end{enumerate}
	Let $\mathfrak{L} = (L_1,\dots,L_{k-1})$ be a line ensemble with law $\mathbb{P}^{0,T,\vec{x},\vec{y}}_{Ber}$, and let $E$ denote the event 
	\[ E=\left\{L_1(s)\geq \cdots \geq L_{k-1}(s)\geq \ell_{bot}(s) \; \mathrm{for} \; s\in [0,T]\right\}.
	\] 
	Then we can find $W_3 = W_3(p,C,K)$ so that for $T\geq W_3$,
	\begin{equation}\label{SepBound1}
	\mathbb{P}^{0,T,\vec{x},\vec{y}}_{Ber}(E) \geq \left(\frac{1}{2} - \sum_{n=1}^\infty (-1)^{n-1} e^{-n^2C^2/8p(1-p)}\right)^{k-1}.
	\end{equation}
	Moreover if $C \geq \sqrt{8p(1-p)\log 3}$, then for $T\geq W_3$ we have
	\begin{equation}\label{SepBound2}
	\mathbb{P}^{0,T,\vec{x},\vec{y}}_{Ber}(E) \geq \left(1 - 3e^{-C^2/8p(1-p)}\right)^{k-1}.
	\end{equation}
\end{lemma}

\begin{remark}
	This lemma states that if $k$ independent Bernoulli bridges are well-separated from each other and $\ell_{bot}$, then there is a positive probability that the curves will cross neither each other nor $\ell_{bot}$. We will use this result to compare curves in an avoiding Bernoulli line ensemble with free Bernoulli bridges.
\end{remark}

\begin{proof}
	Observe that condition (1) simply states that $\ell_{bot}$ lies a distance of at least $C\sqrt{T}$ uniformly below the line segment connecting $x_{k-1}$ and $y_{k-1}$. Thus (1) and (2) imply that $E$ occurs if each curve $L_i$ remains within a distance of $C\sqrt{T}/2$ from the line segment connecting $x_i$ and $y_i$. As in Theorem \ref{KMT}, let $\mathbb{P}_i$ be probability measures supporting $\ell^{(T,z_i)}$ with laws $\mathbb{P}^{0,T,0,z_i}_{Ber}$. Then
	\begin{equation}\label{SepEst}
	\begin{split}
	& \mathbb{P}^{0, T,\vec{x},\vec{y}}_{Ber} (E) \geq \mathbb{P}^{0,T,\vec{x},\vec{y}}_{Ber} \left(\sup_{s\in[0,T]} \big|L_i(s) - x_i - (z_i/T)s\big| \leq C\sqrt{T}/2, \;1\leq i\leq k-1\right) =\\
	&\prod_{i=1}^{k-1}\left[ \mathbb{P}^{0,T,0,z_i}_{Ber} \left(\sup_{s\in[0,T]} \big|L_i(s) - (z_i/T)s\big| \leq C\sqrt{T}/2\right)\right] =\\
	&\prod_{i=1}^{k-1}\left[ 1 - \mathbb{P}_i\left(\sup_{s\in[0,T]} \big|\ell^{(T,z_i)}(s) - (z_i/T)s\big| > C\sqrt{T}/2\right)\right]. \end{split}
	\end{equation}
	In the second line, we used the fact that $L_1,\dots,L_{k-1}$ are independent from each other under $\mathbb{P}^{0,T,0,z_i}_{Ber}$. Let $B^{\sigma,i}$ be the Brownian bridge with diffusion parameter $\sigma = \sqrt{p(1-p)}$ coupled with $\ell^{(T,z_i)}$ given by Theorem \ref{KMT}. Then we have
	\begin{equation}\label{SepCheb}
	\begin{split}
	&\mathbb{P}_i \left(\sup_{s\in[0,T]} \big|\ell^{(T,z_i)}(s) - (z_i/T)s\big| > C\sqrt{T}/2\right) \leq \\
	& \mathbb{P}_i\left(\sup_{s\in[0,T]} |\sqrt{T}B^{\sigma}_{s/T}| > C\sqrt{T}/4\right) + \mathbb{P}_i\left(\Delta(T,z_i) > C\sqrt{T}/4\right).
	\end{split}
	\end{equation}
	By Lemma \ref{BBmax}, the first term in the second line of \eqref{SepCheb} is equal to $2\sum_{n=1}^\infty (-1)^{n-1} e^{-n^2C^2/8p(1-p)}$. Moreover, condition (3) in the hypothesis and Corollary \ref{Cheb} allow us to find $W_3$ depending on $p,C,K$ but not on $i$ so that the last probability in \eqref{SepCheb} is bounded above by $\frac{1}{2} - \sum_{n=1}^\infty (-1)^{n-1} e^{-n^2C^2/8p(1-p)}$ (note that this quantity is positive by (\ref{BBmaxeq})) for $T\geq W_3$. Adding these two terms and referring to \eqref{SepEst} proves \eqref{SepBound1}.
	
	Now suppose $C\geq\sqrt{8p(1-p)\log 3}$. By \eqref{sepBd} in Lemma \ref{BBmax}, the first term in the second line of \eqref{SepCheb} is bounded above by $2e^{-C^2/8p(1-p)}$. After possibly enlarging $W_3$ from above, the second term is $<e^{-C^2/8p(1-p)}$ for $T\geq W_3$. The assumption on $C$ implies that $1-3e^{-C^2/8p(1-p)}\geq 0$, and now combining \eqref{SepCheb} and \eqref{SepEst} proves \eqref{SepBound2}.
\end{proof}

%
\subsection{Properties of avoiding Bernoulli line ensembles}\label{Section3.3}  In this section we derive two results about avoiding Bernoulli line ensembles, which are Bernoulli line ensembles with law $\mathbb{P}_{avoid, Ber;
S}^{T_0,T_1, \vec{x}, \vec{y}, f, g}$ as in Definition \ref{DefAvoidingLawBer}. The lemmas we prove only involve the case when $f(r) = \infty$ for all $r \in \llbracket T_0, T_1 \rrbracket$ and we denote the measure in this case by $\mathbb{P}_{avoid, Ber;S}^{T_0,T_1, \vec{x}, \vec{y}, \infty, g}$. A $\mathbb{P}_{avoid, Ber;S}^{T_0,T_1, \vec{x}, \vec{y}, \infty, g}$-distributed random variable will be denoted by $\mathfrak{Q} = (Q_1, \dots, Q_k)$ where $k$ is the number of up-right paths in the ensemble. As usual, if $g=-\infty$, we write $\mathbb{P}_{avoid, Ber;S}^{T_0,T_1, \vec{x}, \vec{y}}$. Our first result will rely on the two monotonicity Lemmas \ref{MCLxy} and \ref{MCLfg} as well as the strong coupling between Bernoulli bridges and Brownian bridges from Theorem \ref{KMT}, and the further results make use of the material in Section \ref{Section9}.

\begin{lemma}\label{prob19}  Fix $p\in(0,1)$, $k\in\mathbb{N}$. Let $\vec{x},\vec{y}\in\mathfrak{W}_k$ be such that $T \geq y_i - x_i \geq 0$ for $i=1,\dots,k$. Then for any $M,M_1 > 0$ we can find $W_4\in\mathbb{N}$ depending on $p,k,M,M_1$ such that if $T\geq W_4$, $x_k \geq - M_1\sqrt{T}$, and $y_k \geq pT - M_1\sqrt{T}$, then for any $S\subseteq \llbracket 0,T\rrbracket$ we have
	\begin{equation}\label{19ineq}
	\mathbb{P}^{0,T,\vec{x},\vec{y}}_{avoid, Ber; S}\left(Q_k(T/2) - pT/2 \geq M\sqrt{T}\right) \geq \frac{2^{k/2}\big(1-2e^{-4/p(1-p)}\big)^{2k} \exp\left(-\frac{2k(M+M_1+6)^2}{p(1-p)}\right)}{[\pi p(1-p)]^{k/2}}.
	\end{equation}
\end{lemma}

\begin{proof} A sketch of the proof is given in Figure \ref{S3F3} and its caption.
\begin{figure}[h]
\includegraphics[scale=0.5]{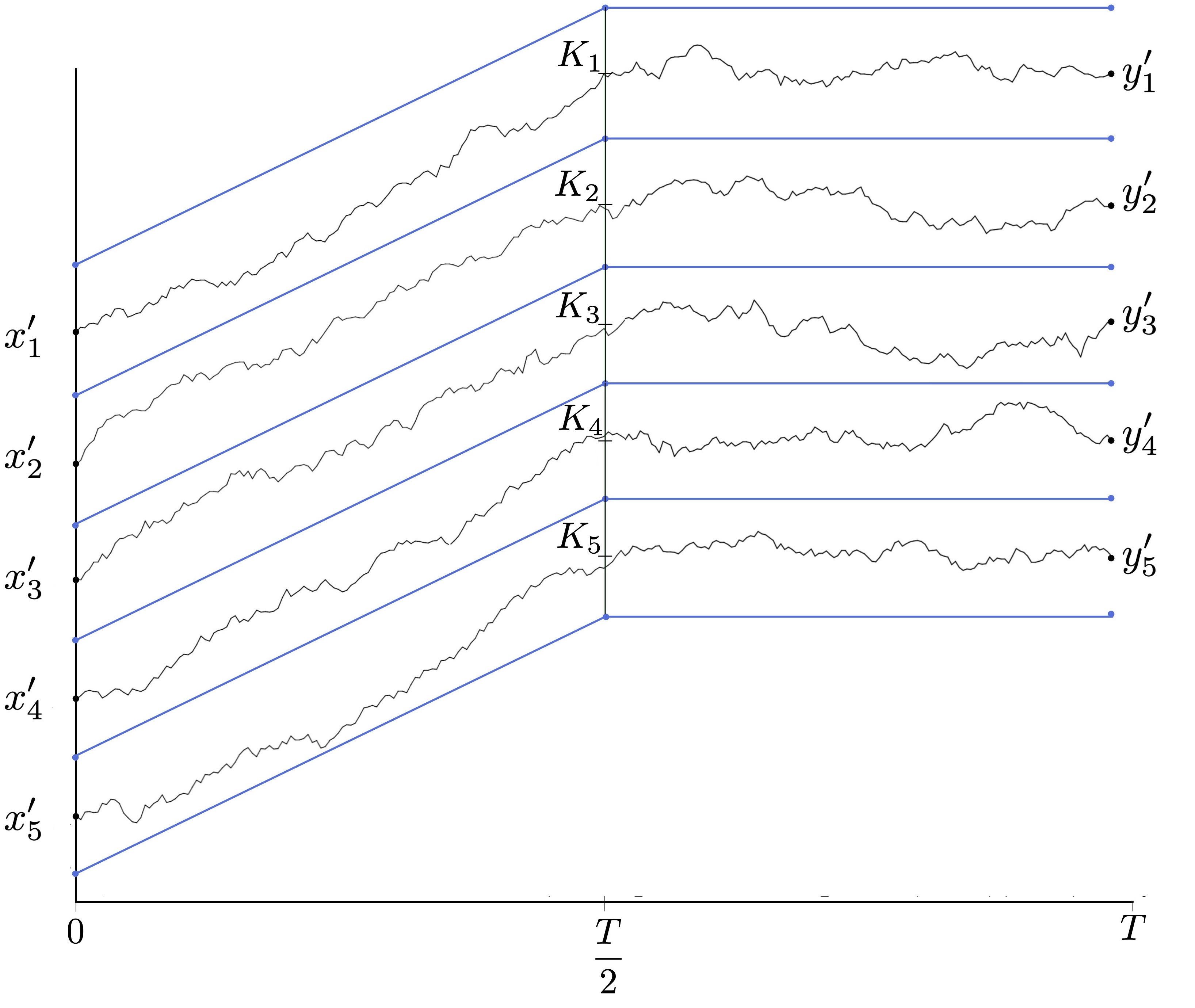}
	\caption{Sketch of the argument for Lemma \ref{prob19}: We use Lemma \ref{MCLxy} to lower the entry and exit data $\vec{x},\vec{y}$ of the curves to $\vec x\,'$ and $\vec y\,'$. We define $E$ to be the event that that the lines in the line ensemble lie in well-separated strips with all the strips high enough so that $E$ is contained in the event we want to lower bound in (\ref{19ineq}). We then use strong coupling with Brownian bridges via Theorem \ref{KMT} and bound the probability of the bridges remaining within the blue windows to lower bound $\mathbb{P}(E)$.}\label{S3F3}
\end{figure}
Define vectors $\vec{x},\vec{y}\in\mathfrak{W}_k$ by
\begin{align*}
x_i' = \lfloor - M_1\sqrt{T} \rfloor - 10(i-1)\lceil\sqrt{T}\rceil, \hspace{2mm} y_i' = \lfloor pT - M_1\sqrt{T}\rfloor - 10(i-1)\lceil\sqrt{T}\rceil.
\end{align*}
Then $x_i'\leq x_k \leq x_i$ and $y_i' \leq y_k \leq y_i$ for $1\leq i\leq k-1$. Thus by Lemma \ref{MCLxy}, we have
\begin{equation*}
\mathbb{P}^{0,T,\vec{x},\vec{y}}_{avoid, Ber; S} \Big(Q_k(T/2) - pT/2 \geq M\sqrt{T}\Big) \geq \mathbb{P}^{0,T,\vec{x}\,',\vec{y}\,'}_{avoid, Ber; S} \Big(Q_k(T/2) - pT/2 \geq M\sqrt{T}\Big).
\end{equation*}
Let us write $K_i = pT/2 + M\sqrt{T}+(10(k-i)-5)\lceil\sqrt{T}\rceil$ for $1\leq i\leq k$. Note $K_i$ is the midpoint of $pT/2 + M\sqrt{T} + 10(k-i-1)\lceil\sqrt{T}\rceil$ and $pT/2 + M\sqrt{T}+10(k-i)\lceil\sqrt{T}\rceil$. Let $E$ denote the event that the following conditions hold for $1\leq i\leq k$:
\begin{enumerate}[label=(\arabic*)]
\item $\left| Q_i(T/2) - K_i \right| \leq 2\lceil\sqrt{T}\rceil$,
\item $\sup_{s\in[0,T/2]} \Big|Q_i(s)-x_i'-\dfrac{K_i-x_i'}{T/2}\,s\Big| \leq 3\sqrt{T}$,
\item $\sup_{s\in[T/2,T]} \Big|Q_i(s)-K_i-\dfrac{y_i'-K_i}{T/2}(s-T/2)\Big| \leq 3\sqrt{T}$.
\end{enumerate}
The first condition implies in particular that $Q_k(T/2)-pT/2 \geq M\sqrt{T}$, and also that $Q_i(T/2)-Q_{i+1}(T/2)\geq 6\sqrt{T}$ for each $i$. The second and third conditions require that each curve $Q_i$ remain within a distance of $3\sqrt{T}$ of the graph of the piecewise linear function on $[0,T]$ passing through the points $(0,x_1')$, $(T/2,K_i)$, and $(T,y_i')$. We observe that
	\[
	\mathbb{P}^{0,T,\vec{x}\,',\vec{y}\,'}_{avoid, Ber; S} \left(Q_k(T/2) - pT/2 \geq M\sqrt{T}\right) \geq \mathbb{P}^{0,T,\vec{x}\,',\vec{y}\,'}_{avoid, Ber; S}(E) \geq \mathbb{P}^{0,T,\vec{x}\,',\vec{y}\,'}_{Ber}(E).
	\]
	The second inequality follows since on $E$ we have $Q_1(s)\geq\cdots\geq Q_k(s)$ for all $s\in\llbracket 0,T \rrbracket$ (here we used that $|\Omega(T_0, T_1, \vec{x}', \vec{y}')| \geq |\Omega_{avoid}(T_0, T_1, \vec{x}', \vec{y}',\infty, -\infty;S)|$ ). Writing $z=y_k'-x_k'$ we have
	\begin{equation}\label{19gibbs}
	\begin{split}
	\mathbb{P}^{0,T,\vec{x}\,',\vec{y}\,'}_{Ber}(E) &= \bigg[\mathbb{P}^{0,T,0,z}_{Ber}\bigg(\left|\ell(T/2)-pT/2-M\sqrt{T}-5\lceil\sqrt{T}\rceil+x_1'\right|\leq 2\lceil\sqrt{T}\rceil \quad\mathrm{and}\\
	&\qquad\qquad \sup_{s\in[0,T/2]}\left|\ell(s) - \frac{K_1-x_1'}{T/2}\,s\right| \leq 3\sqrt{T}\quad\mathrm{and}\\
	&\qquad\qquad \sup_{s\in[T/2,T]}\left|\ell(s)-(K_1-x_1')-\frac{y_1'-K_1}{T/2}(s-T/2)\right| \leq 3\sqrt{T}\bigg) \bigg]^k.
	\end{split}
	\end{equation}

	Let $\mathbb{P}$ be a probability space supporting a random variable $\ell^{(T,z)}$ with law $\mathbb{P}^{0,T,0,z}$ coupled with a Brownian bridge $B^\sigma$ with diffusion parameter $\sigma$, as in Theorem \ref{KMT}. Then the expression on the right in \eqref{19gibbs} being raised to the $k$-th power is bounded below for large enough $T$ by
	\begin{equation}\label{19BB}
	\begin{split}
	& \mathbb{P}^{0,T,0,z}_{Ber}\bigg(\left|\ell(T/2)-pT/2-(M+M_1+5)\sqrt{T}\right|\leq 2\sqrt{T} - 10 \hspace{2mm} \mathrm{and}\\
	&\hspace{2mm}  \sup_{s\in[0,T/2]}\left|\ell(s)-ps-\frac{M+M_1+5}{\sqrt{T}/2}\,s\right| \leq 3\sqrt{T} - 1  \hspace{2mm} \mathrm{and}\\
	&\hspace{2mm}  \sup_{s\in[T/2,T]}\left|\ell(s)-ps-(M+M_1+5)\sqrt{T}+\frac{M+M_1+5}{\sqrt{T}/2}(s-T/2)\right| \leq 3\sqrt{T} - 1 \bigg) \geq\\
	& \mathbb{P}\bigg(\left|\sqrt{T}\,B^\sigma_{1/2} - (M+M_1+5)\sqrt{T}\right|\leq \sqrt{T} \hspace{2mm} \mathrm{and}\\
	&\hspace{2mm} \sup_{s\in[0,T/2]}\left|\sqrt{T}\,B^\sigma_{s/T}-(M+M_1+5)\sqrt{T}\cdot\frac{s}{T/2}\right| \leq 2\sqrt{T}\hspace{2mm} \mathrm{and}\\
	&\hspace{2mm}  \sup_{s\in[T/2,T]}\left|\sqrt{T}\,B^\sigma_{s/T}-(M+M_1+5)\sqrt{T}\cdot\frac{T-s}{T/2}\right| \leq 2\sqrt{T} \bigg) -  \mathbb{P}\left(\Delta(T,z) > \sqrt{T}/2\right).
	\end{split}
	\end{equation}
	Note that $B^\sigma_{1/2}$ is a centered Gaussian random variable with variance $p(1-p)/4 = \sigma^2(1/2)(1-1/2)$. Writing $\xi = B^\sigma_{1/2}$, it follows from Lemma \ref{2bridges} that there exist independent Brownian bridges $B^1,B^2$ with diffusion parameters $\sigma/\sqrt{2}$ so that $B^\sigma_s$ has the same law as $\frac{s}{T/2}\xi + B^1_{2s/T}$ for $s\in[0,T/2]$ and $\frac{T-s}{T/2}\xi + B^2_{(2s-T)/T}$ for $s\in[T/2,T]$. The first term in the last expression in \eqref{19BB} is thus equal to
	\begin{equation}\label{prob19split}
	\begin{split}
	&\mathbb{P}\bigg(|\xi - (M+M_1+5)|\leq 1 \hspace{2mm} \mathrm{and} \sup_{s\in[0,T/2]}\left|B^1_{s/T}-(M+M_1+5-\xi)\cdot\frac{s}{T/2}\right| \leq 2 \hspace{2mm}  \mathrm{and}\\
	&\hspace{2mm}   \sup_{s\in[T/2,T]}\left|B^2_{(2s-T)/T}-(M+M_1+5-\xi)\cdot\frac{T-s}{T/2}\right| \leq 2 \bigg) \geq \\
	& \mathbb{P}\bigg(|\xi - (M+M_1+5)|\leq 1 \hspace{2mm} \mathrm{and} \sup_{s\in[0,T/2]}\big|B^1_{2s/T}\big| \leq 1 \hspace{2mm}  \mathrm{and}\hspace{2mm}  \sup_{s\in[T/2,T]}\big|B^2_{(2s-T)/T}\big| \leq 1 \bigg) =\\
	& \mathbb{P}\Big(|\xi-(M+M_1+5)|\leq 1\Big)\cdot \mathbb{P}\bigg(\sup_{s\in[0,T/2]} \big|B^1_{2s/T}\big|\leq 1\bigg)\cdot \mathbb{P}\bigg(\sup_{s\in[0,T/2]} \big|B^2_{(2s-T)/T}\big|\leq 1\bigg) \geq\\
	& \left(1 \hspace{-0.5mm}-\hspace{-0.5mm}2e^{-4/p(1-p)}\right)^2 \hspace{-2mm} \int_{M+M_1+4}^{M+M_1+6} \frac{e^{-2\xi^2/p(1-p)}d\xi  }{\sqrt{\pi p(1-p)/2}} \geq \frac{2\sqrt{2}\,e^{-2(M+M_1+6)^2/p(1-p)}}{\sqrt{\pi p(1-p)}} \big(1-2e^{-4/p(1-p)}\big)^2.
	\end{split}
	\end{equation}
	In the fourth line, we used the fact that $\xi$, $B^1_\cdot$, and $B^2_\cdot$ are independent, and in the second to last line we used Lemma \ref{BBmax}. Since $|z-pT|\leq (M_1+1)\sqrt{T}$, Lemma \ref{Cheb} allows us to choose $T$ large enough so that $\mathbb{P}(\Delta(T,z) > \sqrt{T}/2)$ is less than 1/2 the expression in the last line of \eqref{prob19split}. Then in view of \eqref{19gibbs} and \eqref{19BB}, we conclude \eqref{19ineq}.
	
\end{proof}

We now state an important weak convergence result, whose proof is presented in Section \ref{Section9} (more specifically see Proposition \ref{WeakConvDistinct}).

\begin{proposition}\label{prob17}
Fix $p,t\in(0,1)$, $k\in\mathbb{N}$, $\vec a, \vec b\in W_k$, where we recall
$$W_k = \{ \vec{x} \in \mathbb{R}^k: x_1 \geq x_2 \geq \cdots \geq x_k \}.$$ 
Suppose $\vec{x}^{T}=(x_{1}^{T},\dots,x_{k}^{T})$ and $\vec{y}^{T}=(y_{1}^{T},\dots,y_{k}^{T})$ are two sequences in $\mathfrak{W}_k$ such that for $i\in \llbracket 1,k\rrbracket$ $$\lim_{T\rightarrow\infty}\frac{x_{i}^{T}}{\sqrt{T}}=a_{i} \quad \text{and} \quad \lim_{T\rightarrow\infty}\frac{y_{i}^{T}-pT}{\sqrt{T}}=b_{i}.$$  Let $(Q_1^T,\dots,Q_k^T)$ have law $\mathbb{P}^{0,T,\vec{x}^{T},\vec{y}^{T}}_{avoid,Ber}$, and define the sequence $\{Z^T\}$ of random $k$-dimensional vectors 
$$Z^{T}=\left(\frac{Q_{1}^T(tT)-ptT}{\sqrt{T}},\dots,\frac{Q_{k}^T(tT)-ptT}{\sqrt{T}}\right).$$
 Then as $T\to\infty$, $Z^{T}$ converges weakly to a random vector $\hat Z$ on $\mathbb{R}^k$ with a probability density $\rho$ supported on $W_k^\circ$.
\end{proposition}

The convergence result in Proposition \ref{prob17} allows us to prove the following lemma, which roughly states that if the entrance and exit data of a sequence of avoiding Bernoulli line ensembles remain in compact windows, then with high probability the curves of the ensemble will remain separated from one another at each point by some small positive distance on scale $\sqrt{T}$. This is how Proposition \ref{prob17} will be used in the main argument in the text, although in Section \ref{Section9} we give a detailed description of the density $\rho$ in Proposition \ref{prob17}.

\begin{lemma}\label{prob 20}
Fix $p,t\in (0,1)$ and $ k\in \mathbb{N}$. Suppose that $\vec x\,^T=(x_1^T,\dots, x_k^T)$, $\vec y\,^T=(y_1^T,\dots , y_k^T)$ are elements of $\mathfrak{W}_k$ such that $T\geq y_i^T-x_i^T\geq 0$ for $i\in \llbracket 1,k\rrbracket$. Then for any $M_1,M_2>0$ and $\epsilon>0$ there exists $W_5\in\mathbb{N}$ and $\delta>0$ depending on $p,k,M_1,M_2$ such that if $T\geq W_5$, $|x_i^T|\leq M_1\sqrt{T}$ and $|y_i^T-pT|\leq M_2\sqrt{T}$, then 
\[
\pr^{0,T,\vec x^T, \vec y^T}_{avoid,Ber}\left(\min_{1\leq i\leq k-1} \big[Q_i(tT)-Q_{i+1}(tT)\big]<\delta\sqrt{T}\right)<\epsilon.
\]
\end{lemma}
\begin{proof}
	We prove the claim by contradiction. Suppose there exist $M_1,M_2,\epsilon>0$ such that for any $W_5\in \mathbb{N}$ and $\delta>0$ there exists some $T\geq W_5$ with
	\[
	\pr^{0,T,\vec x^T,\vec y^T}_{avoid, Ber}\left(\min_{1\leq i\leq k-1}\big[Q_i(tT)-Q_{i+1}(tT)\big]<\delta\sqrt{T}\right)\geq \epsilon.
	\]
	Then we can obtain sequences $T_n$, $\delta_n>0$, $T_n \nearrow \infty$, $\delta_n \searrow 0$, such that for all $n$ we have  
	\[
	\pr^{0,T,\vec x\,^{T_n},\vec y\,^{T_n}}_{avoid, Ber}\left(\min_{1\leq i\leq k-1}\left[\frac{Q_i(tT_n)-Q_{i+1}(tT_n)}{\sqrt{T_n}}\right]<\delta_n\right)\geq \epsilon.
	\]
	Since $|x_i^{T_n}| \leq M_1\sqrt{T_n}$ and $|y_i^{T_n}-pT_n|\leq M_2\sqrt{T_n}$ for $1\leq i\leq k$, the sequences $\{\vec{x}\,^{T_n}/\sqrt{T_n}\}$, $\{(\vec{y}\,^{T_n}-pT_n)/\sqrt{T_n}\}$ are bounded in $\mathbb{R}^k$. It follows that there exist $\vec{x},\vec{y}\in\mathbb{R}^n$ and a subsequence $\{T_{n_m}\}$ such that 
	\[
	\frac{\vec{x}\,^{T_{n_m}}}{\sqrt{T_{n_m}}} \longrightarrow \vec x, \quad
	\frac{\vec{y}\,^{T_{n_m}}-pT_{n_m}}{\sqrt{T_{n_m}}} \longrightarrow \vec y
	\] 
	as $m\to\infty$ (see \cite[Theorem 3.6]{Rudin}). Denote $$Z_i^m =\frac{Q_i(tT_{n_m})-ptT_{n_m}}{\sqrt{T_{n_m}}}.$$ Fix $\delta > 0$ and choose $M$ large enough so that if $m\geq M$ then $\delta_m < \delta$. Then for $m\geq M$ we have
	\begin{equation}\label{prob20inf}
	\epsilon\leq \liminf_{m\to\infty}\pr\left(\min_{1\leq i\leq k-1} \big[Z_i^m-Z_{i+1}^m\big] < \delta_{n_m}\right) \leq \liminf_{m\to\infty}\pr\left(\min_{1\leq i\leq k-1} \big[ Z_i^m-Z_{i+1}^m\big] \leq\delta\right).
	\end{equation}
	Now by Lemma \ref{prob17}, $(Z_1^m,\dots, Z_k^m)$ converges weakly to a random vector $\hat Z$ on $\mathbb{R}^k$ with a density $\rho$. It follows from the portmanteau theorem \cite[Theorem 3.2.11]{Durrett} applied with the closed set $K=[0,\delta]$
	\begin{equation}\label{prob20sup}
	\limsup_{m\to\infty }\pr\left(\min_{1\leq i\leq k-1} \big[Z_i^m-Z_{i+1}^m\big]\in K\right)\leq \pr\left(\min_{1\leq i\leq k-1} \big[\hat Z_i-\hat Z_{i+1}\big]\in K\right).
	\end{equation}
	Combining (\ref{prob20inf}) and (\ref{prob20sup}), we obtain
	\begin{equation}\label{eq:ineq}
	\epsilon\leq \pr\left(0 \leq \min_{1\leq i\leq k-1} \big[\hat Z_i-\hat Z_{i+1}\big]\leq\delta\right)\leq \sum_{i=1}^{k-1}\pr\left(0 \leq \hat Z_i-\hat Z_{i+1}\leq \delta\right).
	\end{equation}
	To conclude the proof, we find a $\delta$ for which \eqref{eq:ineq} cannot hold. For $\tilde\eta\geq 0$ put
	\[
	E_i^{\tilde\eta} = \{\vec{z} \in\mathbb{R}^k : 0 \leq z_i-z_{i+1}\leq\tilde\eta\}.
	\] 
	For each $i\in\llbracket 1, k-1\rrbracket$ and $\eta > 0$, we have
	\begin{equation}\label{eq:integral}
	\mathbb{P}\left(0 \leq \hat{Z}_i - \hat{Z}_{i+1} \leq \eta\right) = \int_{\mathbb{R}^k} \rho\cdot\mathbf{1}_{E_i^\eta}\,dz_1\cdots dz_k.
	\end{equation}
	Clearly $\rho \cdot \mathbf{1}_{E_i^\eta} \to \rho \cdot \mathbf{1}_{E_i^0}$ pointwise as $\eta \to 0$, and $E_i^0 = \{\vec{z} \in\mathbb{R}^k : z_i = z_{i+1}\}$ has Lebesgue measure 0. Thus $\rho\cdot\mathbf{1}_{E_i^\eta} \to 0$ a.e. as $\eta\to 0$. Since $|\rho\cdot\mathbf{1}_{E_i^\eta}| \leq \rho$ and $\rho$ is integrable, the dominated convergence theorem and \eqref{eq:integral} imply that
	\begin{equation*}
	\mathbb{P}\left(0 \leq \hat{Z}_i - \hat{Z}_{i+1} \leq \eta\right) \longrightarrow  0
	\end{equation*}
	as $\eta\to 0$.	Thus for each $i\in\llbracket 1,k-1\rrbracket$ and $\epsilon > 0$ we can find an $\eta_i > 0$ such that $0<\eta \leq \eta_i$ implies $\mathbb{P}(0 \leq \hat{Z}_i - \hat{Z}_{i+1} \leq \eta)<\epsilon/(k-1)$. Putting $\delta=\min_{1\leq i\leq k-1}\eta_i$ we find that 
	\begin{equation*}
	\sum_{i=1}^{k-1}\pr\left( 0 \leq \hat Z_i-\hat Z_{i+1}\leq\delta \right) < \epsilon,
	\end{equation*}
	contradicting \eqref{eq:ineq} for this choice of $\delta$.
	
\end{proof}

%
\section{Proof of Theorem \ref{PropTightGood} }\label{Section4}

The goal of this section is to prove Theorem \ref{PropTightGood}. Throughout this section, we assume that we have fixed $k \in \mathbb{N}$ with $k \geq 2$, $p \in (0,1)$, $\alpha, \lambda > 0$, and
\begin{equation*}
\big\{\mathfrak{L}^N = (L^N_1,L^N_2, \dots, L^N_k)\big\}_{N=1}^{\infty}
\end{equation*}
an $(\alpha,p,\lambda)$-good sequence of $\llbracket 1, k\rrbracket$-indexed Bernoulli line ensembles as in Definition \ref{Def1}, all defined on a probability space with measure $\mathbb{P}$. The proof of Theorem \ref{PropTightGood} depends on three results -- Proposition \ref{PropMain} and Lemmas \ref{PropSup} and \ref{PropSup2}. In these three statements we establish various properties of the sequence of line ensembles $\mathfrak{L}^N$. The constants in these statements depend implicitly on $\alpha$, $p$, $\lambda$, $k$, and the functions $\phi_1, \phi_2, \psi$ from Definition \ref{Def1}, which are fixed throughout. We will not list these dependencies explicitly. The proof of Proposition \ref{PropMain} is given in Section \ref{Section4.1} while the proofs of Lemmas \ref{PropSup} and \ref{PropSup2} are in Section \ref{Section5}. Theorem \ref{PropTightGood} (i) and (ii) are proved in Sections \ref{Section4.2} and \ref{Section4.3} respectively.
%
\subsection{Preliminary results}\label{Section4.1}
The main result in this section is presented as Proposition \ref{PropMain} below. In order to formulate it and some of the lemmas below, it will be convenient to adopt the following notation for any $r > 0$ and $m \in \mathbb{N}$:
\begin{equation}\label{eqsts}
t_m =\lfloor (r+m) N^{\alpha} \rfloor.
\end{equation}
\begin{proposition}\label{PropMain} Let $\mathbb{P}$ be the measure from the beginning of this section. For any $\epsilon > 0$, $r > 0$ there exist $\delta = \delta(\epsilon, r) > 0$ and $N_1 = N_1(\epsilon, r)$ such that for all $N \geq N_1$
	we have 
	$$\mathbb{P}\Big(Z\big( -t_1, t_1, \vec{x}, \vec{y} , \infty,  L^N_{k}\llbracket -t_1, t_1\rrbracket\big) < \delta\Big) < \epsilon,$$
	where $\vec{x} = (L_1^N(-t_1), \dots, L_{k-1}^N(-t_1))$, $\vec{y} = (L_1^N(t_1), \dots, L^N_{k-1}(t_1))$,  $ L^N_{k}\llbracket -t_1, t_1\rrbracket$ is the restriction of $L^N_k$ to the set $\llbracket -t_1, t_1\rrbracket$, and $Z$ is the acceptance probability of Definition \ref{DefAP}. 
\end{proposition}

The general strategy we use to prove Proposition \ref{PropMain} is inspired by the proof of \cite[Proposition 6.5]{CorHamK}. We begin by stating three key lemmas that will be required. The proofs of Lemmas \ref{PropSup} and \ref{PropSup2} are postponed to Section \ref{Section5} and Lemma \ref{LemmaAP1} is proved in Section \ref{Section6}.

\begin{lemma}\label{PropSup}  Let $\mathbb{P}$ be the measure from the beginning of this section. For any $\epsilon > 0$, $r > 0$ there exist $R_1=R_1(\epsilon,r) > 0$ and $N_2= N_2(\epsilon,r)$ such that for $N \geq N_2$ 
	$$\mathbb{P} \left( \sup_{s \in [ -t_3, t_3] }\big[ L^N_1(s) - p s \big] \geq  R_1N^{\alpha/2} \right) < \epsilon.$$
\end{lemma}

\begin{lemma}\label{PropSup2}  Let $\mathbb{P}$ be the measure from the beginning of this section.  For any $\epsilon > 0$, $r > 0$ there exist $R_2=R_2( \epsilon,r) > 0$ and $N_3=N_3(\epsilon,r)$ such that for $N \geq N_3$
	$$\mathbb{P}\left( \inf_{s \in [ -t_3, t_3 ]}\big[L^N_{k-1}(s) - p s \big] \leq - R_2N^{\alpha/2} \right) < \epsilon.$$
\end{lemma}

\begin{lemma}\label{LemmaAP1} Fix $k \in \mathbb{N}$, $k \geq 2$, $p \in (0,1)$, $r, \alpha, M_1, M_2 > 0$ . Suppose that $\ell_{bot}: \llbracket -t_3, t_3 \rrbracket \rightarrow \mathbb{R} \cup \{ - \infty \}$, and $\vec{x}, \vec{y} \in \mathfrak{W}_{k-1}$ are such that $|\Omega_{avoid}(-t_3, t_3, \vec{x}, \vec{y}, \infty, \ell_{bot})| \geq 1$. Suppose further that
	\begin{enumerate}
		\item $\sup_{s \in [- t_3,t_3]}\big[\ell_{bot}(s)  - ps \big]  \leq M_2 (2t_3)^{1/2}$,
		\item  $-pt_3 + M_1 (2t_3)^{1/2} \geq  x_1 \geq  x_{k-1} \geq \max\left(\ell_{bot}(-t_3), -pt_3- M_1 (2t_3)^{1/2}\right),$
		\item $pt_3 + M_1 (2t_3)^{1/2} \geq y_1 \geq y_{k-1} \geq  \max \left( \ell_{bot}(t_3),  p t_3- M_1(2t_3)^{1/2} \right).$
	\end{enumerate}
	Then there exist constants $\tilde{h} > 0$ and $N_4 \in \mathbb{N}$, depending on $ M_1, M_2, p , k, r, \alpha$,  such that for any $\tilde{\epsilon}  > 0$ and $N \geq N_4$ we have
	\begin{equation}\label{eqn60}
	\mathbb{P}^{-t_3, t_3, \vec{x},\vec{y}, \infty, \ell_{bot} }_{avoid, Ber} \Big( Z\big(  -t_1, t_1, \mathfrak{Q}(-t_1) ,\mathfrak{Q}(t_1), \infty,  \ell_{bot}\llbracket -t_1, t_1\rrbracket\big) \leq  \tilde{h} \tilde{\epsilon}   \Big)  \leq \tilde{\epsilon},
	\end{equation}
	where $Z$ is the acceptance probability of Definition \ref{DefAP}, $\ell_{bot}\llbracket -t_1, t_1\rrbracket$ is the vector, whose coordinates match those of $\ell_{bot}$ on $\llbracket -t_1, t_1\rrbracket$ and $\mathfrak{Q}(a) = (Q_1(a), \dots, Q_{k-1}(a))$ is the value of the line ensemble $\mathfrak{Q} = (Q_1, \dots, Q_{k-1})$ whose law is $\mathbb{P}^{-t_3, t_3, \vec{x},\vec{y}, \infty, \ell_{bot} }_{avoid, Ber}$ at location $a$.
\end{lemma}

\begin{proof}[Proof of Proposition \ref{PropMain}] Let $\epsilon > 0$ be given. Define the event
	\begin{equation*}
	\begin{split}
	E_N = &\left\{    L_{k-1}^N(  \pm t_3) \mp pt_3 \geq  - M_1 (2t_3)^{1/2}\right\} \cap \left\{    L_{1}^N(  \pm t_3) \mp pt_3 \leq   M_1 (2t_3)^{1/2}\right\} \cap \\
	& \left\{ \sup_{s \in [ -t_3, t_3]} [{L}^N_{k}(s) - p s ]\leq M_2  (2t_3)^{1/2} \right\}.
	\end{split}
	\end{equation*}
	In view of  Lemmas \ref{PropSup} and \ref{PropSup2} and the fact that $\mathbb{P}$-almost surely $L_1^N(s) \geq L_k^N(s)$ for all $s \in [ -t_3, t_3]$ we can find sufficiently large $M_1, M_2$ and $N_2$ such that for $N \geq N_2$ we have 
	\begin{equation}\label{UBEC}
	\mathbb{P}(E_N^c) <  \epsilon / 2.
	\end{equation}

	Let $\tilde{h}, N_4$ be as in Lemma \ref{LemmaAP1} for the values $M_1, M_2$ as above, the values $\alpha, p, k$ from the beginning of the section and $r$ as in the statement of the proposition. For $\delta = (\epsilon/2) \cdot \tilde{h}$ we denote
	$$V = \Big\{Z\big( -t_1,t_1, \vec{x}, \vec{y} , \infty,  L^N_{k}\llbracket -t_1, t_1\rrbracket\big)< \delta\Big\}$$
	and make the following deduction for $N \geq N_4$
	\begin{equation}\label{boundoneEN}
	\begin{split}
	&\mathbb{P}\big( V \cap E_N \big) =\mathbb{E} \bigg[    \mathbb{E}\Big[{\bf 1}_{E_N} \cdot {\bf 1}_{V} \Big{|} \sigma \big( \mathfrak{L}^N(-t_3),  \mathfrak{L}^N(t_3),L^N_k\llbracket -t_3, t_3\rrbracket   \big)\Big] \bigg] = \\
	&\mathbb{E} \bigg[ {\bf 1}_{E_N} \cdot   \mathbb{E}\Big[ {\bf 1} \{ Z\big( -t_1,t_1, \vec{x}, \vec{y} , \infty,  L^N_{k}\llbracket -t_1, t_1\rrbracket \big) < \delta\}   \Big{|} \sigma \big( \mathfrak{L}^N(-t_3),  \mathfrak{L}^N(t_3),L^N_k\llbracket -t_3, t_3\rrbracket   \big) \Big] \bigg]  = \\
	&\mathbb{E} \left[ {\bf 1}_{E_N} \cdot  \mathbb{E}_{avoid}\left[ {\bf 1} \{ Z\big( -t_1, t_1, \mathfrak{L}(-t_1),\mathfrak{L}(t_1), \infty, {L}^N_k\llbracket -t_1, t_1\rrbracket \big) < \delta\} \right] \right] \leq  \mathbb{E} \left[ {\bf 1}_{E_N} \cdot  \epsilon/2 \right] \leq \epsilon/2.
	\end{split}
	\end{equation}
	In (\ref{boundoneEN}) we have written $\mathbb{E}_{avoid}$ in place of $\mathbb{E}^{-t_3, t_3, \mathfrak{L}^N(-t_3), \mathfrak{L}^N(t_3), \infty, L^N_k\llbracket -t_3, t_3\rrbracket }_{avoid, Ber}$ to ease the notation; in addition, we have that $\mathfrak{L}^N(a) = (L_1^N(a), \dots, L_{k-1}^N(a))$ and $\mathfrak{L}$ on the last line is distributed according to $\mathbb{P}^{-t_3, t_3, \mathfrak{L}^N(-t_3), \mathfrak{L}^N(t_3), \infty, L^N_k\llbracket -t_3, t_3\rrbracket }_{avoid, Ber}$. We elaborate on (\ref{boundoneEN}) in the paragraph below.
	
	The first equality in (\ref{boundoneEN}) follows from the tower property for conditional expectations. The second equality uses the definition of $V$ and the fact that ${\bf 1}_{E_N} $ is $\sigma \big( \mathfrak{L}^N(-t_3),  \mathfrak{L}^N(t_3),L^N_k\llbracket -t_3, t_3\rrbracket   \big)$-measurable and can thus be taken outside of the conditional expectation. The third equality uses the Schur Gibbs property, see Definition \ref{DefSGP}. The first inequality on the third line holds if $N \geq N_4$ and uses Lemma \ref{LemmaAP1} with $\tilde{\epsilon} = \epsilon/2$  as well as the fact that on the event $E_N$ the random variables $\mathfrak{L}^N(-t_3), \mathfrak{L}^N(t_3)$ and $L^N_k \llbracket -t_3, t_3 \rrbracket$ (that play the roles of $\vec{x}, \vec{y}$ and $\ell_{bot}$) satisfy the inequalities 
	\begin{enumerate}
		\item $\sup_{s \in [- t_3,t_3]}\big[L^N_k(s)  - ps \big]  \leq M_2 (2t_3)^{1/2}$,
		\item  $-pt_3 + M_1 (2t_3)^{1/2} \geq  {L}^N_1(-t_3) \geq  {L}^N_{k-1}(-t_3) \geq \max\left(L^N_k(-t_3), -pt_3- M_1 (2t_3)^{1/2}\right),$
		\item $pt_3 + M_1 (2t_3)^{1/2} \geq {L}^N_1(t_3) \geq {L}^N_{k-1}(t_3) \geq  \max \left(L^N_k(t_3),  p t_3- M_1(2t_3)^{1/2} \right).$
	\end{enumerate}
	The last inequality in (\ref{boundoneEN}) is trivial.
	
	Combining (\ref{boundoneEN}) with (\ref{UBEC}), we see that for all $N \geq \max(N_2, N_4)$ we have
	$$\mathbb{P}\left( V  \right) = \mathbb{P}(V \cap E_N) + \mathbb{P}(V \cap E_N^c) \leq \epsilon/2 + \mathbb{P}(E_N^c) < \epsilon,$$
	which proves the proposition.
\end{proof}

%
\subsection{Proof of Theorem \ref{PropTightGood} (i) }\label{Section4.2} Since $\tilde{f}^N_n$ are obtained from $f^N_n$ by subtracting a deterministic continuous function (namely $\lambda s^2$) and rescaling by a constant (namely $\sqrt{p(1-p)}$) we see that $\tilde{\mathbb{P}}_N$ is tight if and only if $\mathbb{P}_N$ is tight and so it suffices to show that $\mathbb{P}_N$ is tight. By Lemma \ref{2Tight}, it suffices to verify the following two conditions for all $i \in \llbracket 1,k-1\rrbracket$, $r>0$, and $\epsilon>0$:
\begin{equation}\label{ThmCond1}
\lim_{a\to\infty} \limsup_{N\to\infty} \pr(|f^N_i(0)|\geq a) = 0 
\end{equation}
\begin{equation}\label{ThmCond2}
\lim_{\delta\to 0} \limsup_{N\to\infty} \pr\left(\sup_{{x,y\in [-r,r], |x-y|\leq\delta}} |f^N_i(x) - f^N_i(y)| \geq \epsilon \right)= 0.
\end{equation}
For the sake of clarity, we will prove these conditions in several steps.\\

\noindent\textbf{Step 1.} In this step we prove \eqref{ThmCond1}. Let $\epsilon > 0$ be given. Then by Lemmas \ref{PropSup} and \ref{PropSup2} we can find $N_2, N_3$ and $R_1, R_2$ such that for $N \geq \max(N_1, N_2)$
\begin{align*}
\pr\left(\sup_{s\in[-t_3,t_3]} [L_1^N(s)-ps]\geq R_1N^{\alpha/2}\right)<\epsilon/2,\\
\pr\left(\inf_{s\in[-t_3,t_3]}[L_{k-1}^N(s)-ps]\leq -R_2 N^{\alpha/2}\right)<\epsilon/2.
\end{align*}
In particular, if we set $R  = \max(R_1, R_2)$ and utilize the fact that $L_1^N(0) \geq \cdots \geq L_{k-1}^N(0)$ we conclude that for any $i \in \llbracket 1, k-1 \rrbracket$ we have
$$\pr\big(|L_i^N(0)|\geq R N^{\alpha/2}\big)\leq \pr\big(L_1^N(0)\geq R_1 N^{\alpha/2}\big) + \pr\big(L_{k-1}^N(0)\leq -R_2N^{\alpha/2}\big) < \epsilon,$$
which implies \eqref{ThmCond1}.\\

\noindent\textbf{Step 2.} In this step we prove (\ref{ThmCond2}). In the sequel we fix $r, \epsilon > 0$ and $i \in \llbracket 1, k-1 \rrbracket$. To prove (\ref{ThmCond2}) it suffices to show that for any $\eta>0$, there exists a $\delta > 0$ and $N_0$ such that $N \geq N_0$ implies 
\begin{equation}\label{S2R1}
\pr\left(\sup_{{x,y\in [-r,r], |x-y|\leq\delta}} |f^N_i(x) - f^N_i(y)| \geq \epsilon\right)<\eta.
\end{equation}

For $\delta > 0$ we define the event
\begin{equation}\label{Adelta}
A^N_\delta =\left\{\sup_{{x,y\in [-t_1,t_1], |x-y|\leq \delta N^\alpha}} \abs*{L_i^N(x)-L_i^N(y)-p(x-y)}\geq \frac{3N^{\alpha/2}\epsilon}4\right\},
\end{equation}
where we recall that $t_1 = \lfloor (r + 1)N^{\alpha} \rfloor$ from (\ref{eqsts}). We claim that there exist $\delta_0 > 0$ and $N_0 \in \mathbb{N}$ such that for $\delta \in (0, \delta_0]$ and $N \geq N_0$ we have
\begin{equation}\label{Abound}
\mathbb{P}(A^N_\delta) < \eta.
\end{equation}
We prove (\ref{Abound}) in the steps below. Here we assume its validity and conclude the proof of (\ref{S2R1}).\\

Observe that if $\delta \in \left (0, \min \left(\delta_0, \epsilon \cdot (8 \lambda r)^{-1} \right) \right)$, where $\lambda$ is as in the statement of the theorem, we have the following tower of inequalities
\begin{equation}\label{Term1}
\begin{split}
&\pr\left(\sup_{{x,y\in [-r,r], |x-y|\leq\delta}} |f^N_i(x) - f^N_i(y)| \geq \epsilon\right) = \\
&\pr\left(\sup_{{x,y\in [-r,r],|x-y|\leq\delta}} \abs*{N^{-\alpha/2}\left(L^N_i(xN^{\alpha}) - L^N_i(yN^\alpha)\right)-p(x-y)N^{\alpha/2}+\lambda(x^2-y^2)} \geq \epsilon\right) \leq \\
&\pr\left(\sup_{{x,y\in [-r,r], |x-y|\leq\delta}} N^{-\alpha/2}\abs*{L_i^N(xN^\alpha)-L_i^N(yN^\alpha)-p(x-y)N^\alpha}+2\lambda r\delta\geq \epsilon\right) \leq \\
& \pr\left(\sup_{{x,y\in [-r,r], |x-y|\leq\delta}} \abs*{L_i^N(xN^\alpha)-L_i^N(yN^\alpha)-p(x-y)N^\alpha}\geq \frac{3N^{\alpha/2}\epsilon}4\right)\leq \mathbb{P}(A^N_\delta) < \eta .
\end{split}
\end{equation}
In (\ref{Term1}) the first equality follows from the definition of $f_i^N$, and the inequality on the second line follows from the inequality $|x^2-y^2|\leq 2r\delta$, which holds for all $x,y \in [-r,r]$ such that $|x-y| \leq \delta$. The inequality in the third line of (\ref{Term1}) follows from our assumption that $\delta <   \epsilon \cdot (8 \lambda r)^{-1}$ and the first inequality on the last line follows from the definition of $A^N_\delta$ in (\ref{Adelta}), and the fact that $t_1 \geq r N^{\alpha}$. The last inequality follows from our assumption that $\delta < \delta_0$ and (\ref{Abound}). From (\ref{Term1}) we get (\ref{S2R1}).\\

{\bf \raggedleft Step 3.} In this step we prove (\ref{Abound}) and fix $\eta > 0$ in the sequel. For $\delta_1, M_1 > 0$ and $N \in \mathbb{N}$ we define the events
\begin{equation}\label{ESetsDef}
\begin{split}
E_1= \hspace{-0.5mm}\left\{\max_{1\leq j \leq k-1}\abs*{L_j^N(\pm t_1) \mp pt_1}\leq M_1N^{\alpha/2} \right\}, E_2=\hspace{-0.5mm}\left\{Z(-t_1, t_1, \vec x, \vec y, \infty, L_{k}^N\llbracket -t_1, t_1 \rrbracket)>\delta_1\right\},
\end{split}
\end{equation}
where we used the same notation as in Proposition \ref{PropMain} (in particular $\vec x = (L^N_1(-t_1), \dots, L_{k-1}^N(-t_1))$ and $\vec y = (L_1^N(t_1), \dots, L_{k-1}^N(t_1))$). Combining Lemmas \ref{PropSup}, \ref{PropSup2} and Proposition \ref{PropMain} we know that we can find $\delta_1 > 0$ sufficiently small, $M_1$ sufficiently large and $\tilde{N} \in \mathbb{N}$ such that for $N \geq \tilde{N}$ we know
\begin{equation}\label{BoundECup}
\mathbb{P} \left(E_1^c \cup E_2^c \right) < \eta/2.
\end{equation}
We claim that we can find $\delta_0 > 0$ and $N_0 \geq \tilde{N}$ such that for $N \geq N_0$ and $\delta \in (0, \delta_0)$ we have
\begin{equation}\label{Abound2}
\mathbb{P}(A^N_\delta \cap E_1 \cap E_2 ) < \eta/2.
\end{equation}
Since
$$\pr(A^N_\delta)= \pr(A^N_\delta\cap E_1\cap E_2)+\pr(A^N_\delta\cap\left(E_1^c\cup E_2^c\right))\leq \pr(A^N_\delta\cap E_1\cap E_2)+\mathbb{P} \left(E_1^c \cup E_2^c \right),$$
we see that (\ref{BoundECup}) and (\ref{Abound2}) together imply (\ref{Abound}).\\

{\bf \raggedleft Step 4.} In this step we prove (\ref{Abound2}). We define the $\sigma$-algebra 
$$\mathcal{F}=\sigma\left(L_{k}^N \llbracket -t_1, t_1 \rrbracket ,L_1^N(\pm t_1 ), L_2^N(\pm t_1 ),\dots, L_{k-1}^N(\pm t_1 )\right).$$
Clearly $E_1, E_2\in \mathcal{F}$, so the indicator random variables $\indic_{E_1}$ and $\indic_{E_2}$ are $\mathcal{F}$-measurable. It follows from the tower property of conditional expectation that
\begin{equation}\label{tower}
\pr\left(A^N_\delta\cap E_1\cap E_2\right)=\ex\left[\indic_{A^N_\delta} \indic_{E_1} \indic_{E_2}\right] =\ex\left[\indic_{E_1} \indic_{E_2}\ex\left[\indic_{A^N_\delta}\mid \mathcal{F}\right]\right]. 
\end{equation}
By the Schur-Gibbs property (see Definition \ref{DefSGP}), we know that $\mathbb{P}$-almost surely
\begin{equation}\label{tower2}
\ex\left[\indic_{A^N_\delta}\mid \mathcal{F}\right]=\ex_{avoid,Ber}^{-t_1,t_1,\vec x, \vec y, \infty, L_{k}^N \llbracket -t_1, t_1 \rrbracket}\left[\indic_{A^N_\delta}\right].
\end{equation}
We now observe that the Radon-Nikodym derivative of $\pr_{avoid,Ber}^{-t_1,t_1,\vec x, \vec y, \infty, L_{k}^N \llbracket -t_1, t_1 \rrbracket}$ with respect to $\pr_{Ber}^{-t_1, t_1,\vec x,\vec y}$ is given by 
\begin{equation}\label{RN}
\frac{d\pr_{avoid,Ber}^{-t_1,t_1,\vec x, \vec y, \infty, L_{k}^N\llbracket -t_1, t_1 \rrbracket } (Q_1, \dots, Q_{k-1})}{d\pr_{Ber}^{-t_1, t_1,\vec x,\vec y}} = \frac{\indic_{\left\{Q_1 \geq \cdots \geq Q_{k-1} \geq Q_k \right\}}}{Z(-t_1,t_1,\vec x, \vec y, \infty,  L_{k}^N\llbracket -t_1, t_1 \rrbracket )},
\end{equation}
where $\mathfrak{Q} = (Q_1, \dots, Q_{k-1})$ is $\pr_{Ber}^{-t_1, t_1,\vec x,\vec y}$-distributed and $Q_k = L_{k}^N\llbracket -t_1, t_1 \rrbracket$. To see this, note that by Definition \ref{DefAvoidingLawBer} we have for any set $A \subset  \prod_{i = 1}^{k-1}\Omega(-t_1, t_1, x_i, y_i )$ that 
\begin{equation*}
\begin{split}
&\pr_{avoid,Ber}^{-t_1,t_1,\vec x, \vec y, \infty, L_{k}^N\llbracket -t_1, t_1 \rrbracket}(A) = \frac{\pr_{Ber}^{-t_1,t_1,\vec x, \vec y}(A\cap\left\{Q_1 \geq \cdots \geq Q_{k-1} \geq Q_k\right\})}{\pr_{Ber}^{-t_1,t_1,\vec x, \vec y}(Q_1\geq\cdots\geq Q_{k-1} \geq Q_k)} = \\
& \frac{\ex_{Ber}^{-t_1,t_1,\vec x, \vec y}\left[\indic_A \cdot  \indic_{\left\{Q_1\geq\cdots\geq Q_{k-1} \geq Q_k \right\}}\right]}{Z(-t_1,t_1,\vec{x},\vec{y}, \infty, L^N_{k} \llbracket -t_1, t_1 \rrbracket)} = \int_A \frac{\indic_{\left\{Q_1\geq\cdots\geq Q_{k-1} \geq Q_k \right\}}}{Z(-t_1,t_1,\vec x, \vec y, \infty, L_{k}^N \llbracket -t_1, t_1 \rrbracket )}\,d\pr_{Ber}^{-t_1,t_1,\vec x, \vec y}.
\end{split}
\end{equation*}
It follows from \eqref{tower}, \eqref{RN}, and the definition of $E_2$ in \ref{ESetsDef} that
\begin{equation}\label{RT1}
\begin{split}
&\pr(A^N_\delta\cap E_1\cap E_2) =\ex\left[\indic_{E_1}\indic_{E_2} \ex_{Ber}^{-t_1, t_1,\vec x,\vec y}\left[\frac{\indic_{B^N_\delta}\cdot \indic_{\left\{Q_1\geq\cdots\geq Q_{k}\right\}}}{Z(-t_1,t_1,\vec x, \vec y, L^N_{k} \llbracket -t_1, t_1 \rrbracket)}\right]\right] \leq \\
&\leq \ex\left[\indic_{E_1}\indic_{E_2}\ex_{Ber}^{-t_1,t_1,\vec x,\vec y}\left[\frac{\indic_{B^N_\delta}}{\delta_1}\right]\right] \leq \frac{1}{\delta_1} \ex\left[ \indic_{E_1} \cdot\pr_{Ber}^{-t_1,t_1,\vec{x}, \vec{y} }(B^N_\delta) \right],
\end{split}
\end{equation}
where $\delta_1$ is as in \eqref{ESetsDef}, and
\begin{equation*}
B^N_\delta = \left\{\sup_{{x,y\in [-t_1,t_1], |x-y|\leq \delta N^\alpha}} \abs*{Q_i(x)-Q_i(y)-p(x-y)}\geq \frac{3N^{\alpha/2}\epsilon}4\right\}.
\end{equation*}
Notice that under $\pr_{Ber}^{-t_1,t_1,\vec{x}, \vec{y} }$ the law of $Q_i$ is precisely $\pr_{Ber}^{-t_1,t_1,x_i, y_i }$, and so we conclude that 
\begin{equation}\label{RT2}
\begin{split}
& \pr_{Ber}^{-t_1,t_1,\vec{x}, \vec{y} }(B^N_\delta) = \pr_{Ber}^{0,2t_1,0, y_i - x_i }\left( \sup_{{x,y\in [0,2t_1], |x-y|\leq \delta N^\alpha}} \abs*{\ell(x)-\ell(y)-p(x-y)}\geq \frac{3N^{\alpha/2}\epsilon}4  \right),
\end{split}
\end{equation}
where $\ell$ has law $\pr_{Ber}^{0,2t_1,0, y_i - x_i }$ (note that in (\ref{RT2}) we implicitly translated the path $\ell$ to the right by $t_1$ and up by $-x_i$, which does not affect the probability in question). Since on the event $E_1$ we know that $|y_i -x_i - 2 p t_1| \leq 2M_1 N^{\alpha}$  we conclude from Lemma \ref{MOCLemmaS4} that we can find $N_0$ and $\delta_0 > 0$ depending on $M_1, r, \alpha$ such that for $N \geq N_0$ and $\delta \in (0, \delta_0)$ we have
\begin{equation}\label{RT3}
\begin{split}
& \indic_{E_1} \cdot \pr_{Ber}^{0,2t_1,0, y_i - x_i }\left( \sup_{{x,y\in [0,2t_1], |x-y|\leq \delta N^\alpha}} \abs*{\ell(x)-\ell(y)-p(x-y)}\geq \frac{3N^{\alpha/2}\epsilon}4  \right) < \frac{\delta_1 \eta}{2}.
\end{split}
\end{equation}
Combining (\ref{RT1}), (\ref{RT2}) and (\ref{RT3}) we conclude (\ref{Abound2}), and hence statement (i) of the theorem.

%
\subsection{Proof of Theorem \ref{PropTightGood} (ii)}\label{Section4.3}

In this section we fix a subsequential limit $\mathcal{L}^\infty = (\tilde{f}_1^\infty,\dots,\tilde{f}_{k-1}^\infty)$ of the sequence $\tilde{\mathbb{P}}_N$ as in the statement of Theorem \ref{PropTightGood}, and we prove that $\mathcal{L}^\infty$ possesses the partial Brownian Gibbs property. Our approach is similar to that in \cite[Sections 5.1 and 5.2]{DimMat}. We first give a definition of measures on scaled free and avoiding Bernoulli random walks. These measures will appear when we apply the Schur Gibbs property to the scaled line ensembles $\{ \tilde{f}_i^N \}_{i = 1}^{k-1}$. 

\begin{definition}\label{scaledRW}
	Let $a,b\in N^{-\alpha}\mathbb{Z}$ with $a<b$ and $x,y\in N^{-\alpha/2}\mathbb{Z}$ satisfy $0\leq y-x \leq (b-a)N^{\alpha/2}$. Let $\ell^{(T,z)}$ denote a random variable with law $\mathbb{P}^{0,T,0,z}_{Ber}$ as before Definition \ref{DefAvoidingLawBer}. We define $\mathbb{P}^{a,b,x,y}_{free,N}$ to be the law of the $C([a,b])$-valued random variable $Y$ given by
	\[
	Y(t) = \frac{x + N^{-\alpha/2}\left[\ell^{((b-a)N^\alpha,\,(y-x)N^{\alpha/2}))}((t-a)N^\alpha) - ptN^\alpha\right]}{\sqrt{p(1-p)}}, \quad t\in [a,b].
	\]
	Now for $i\in\llbracket 1,k\rrbracket$, let $\ell^{(N,z),i}$ denote i.i.d. random variables with laws $\mathbb{P}^{0,N,0,z}_{Ber}$. Let $\vec{x},\vec{y}\in(N^{-\alpha/2}\mathbb{Z})^k$ satisfy $0\leq y_i-x_i \leq (b-a)N^{\alpha/2}$ for $i\in\llbracket 1,k\rrbracket$. We define the $\llbracket 1,k\rrbracket$-indexed line ensemble $\mathcal{Y}^N$ by
	\[
	\mathcal{Y}^N_i(t) = \frac{x_i + N^{-\alpha/2}\left[\ell^{((b-a)N^\alpha,\,(y_i-x_i)N^{\alpha/2})),i}((t-a)N^\alpha) - ptN^\alpha\right]}{\sqrt{p(1-p)}}, \quad i\in\llbracket 1,k\rrbracket, t\in [a,b].
	\]
	We let $\mathbb{P}^{a,b,\vec{x},\vec{y}}_{free,N}$ denote the law of $\mathcal{Y}^N$. Suppose $\vec{x},\vec{y}\in (N^{-\alpha/2}\mathbb{Z})^k\cap W_k$, where 
$$W_k =  \{ \vec{x} = (x_1, \dots, x_k) \in \mathbb{R}^k: x_1 \geq x_2 \geq \cdots \geq x_k \}.$$ 
Suppose further that $f : [a,b] \to (-\infty,\infty]$, $g:[a,b]\to [-\infty,\infty)$ are continuous functions. We define the probability measure $\mathbb{P}^{a,b,\vec{x},\vec{y},f,g}_{avoid,N}$ to be $\mathbb{P}^{a,b,\vec{x},\vec{y}}_{free,N}$ conditioned on the event
	\[
	E = \{f(r) \geq \mathcal{Y}^N_1(r) \geq \cdots \geq \mathcal{Y}^N_k(r) \geq g(r) \mbox{ for } r\in[a,b]\}.
	\]
	This measure is well-defined if $E$ is nonempty.
	
\end{definition}

Next, we state two lemmas whose proofs we give in Section \ref{BGPapp}. The first lemma proves weak convergence of the scaled avoiding random walk measures in Definition \ref{scaledRW}. It states roughly that if the boundary data of these measures converge, then the measures converge weakly to the law of avoiding Brownian bridges with the boundary limiting data, as in Definition \ref{DefAvoidingLaw}.

\begin{lemma}\label{scaledavoidBB}
	Fix $k\in\mathbb{N}$ and $a,b\in\mathbb{R}$ with $a<b$, and let $f:[a-1,b+1]\to(-\infty,\infty]$, $g:[a-1,b+1]\to[-\infty,\infty)$ be continuous functions such that $f(t) > g(t)$ for all $t\in[a-1,b+1]$. Let $\vec{x},\vec{y}\in W_k^\circ$ be such that $f(a) > x_1$, $f(b) > y_1$, $g(a) < x_k$, and $g(b) < y_k$. Let $a_N = \lfloor aN^\alpha\rfloor N^{-\alpha}$ and $b_N = \lceil bN^\alpha\rceil N^{-\alpha}$, and let $f_N : [a-1,b+1]\to(-\infty,\infty]$ and $g_N : [a-1,b+1]\to[-\infty,\infty)$ be continuous functions such that $f_N\to f$ and $g_N\to g$ uniformly on $[a-1,b+1]$. If $f \equiv \infty$ the last statement means that $f_N \equiv \infty$ for all large enough $N$ and if $g \equiv - \infty$ the latter means that $g_N \equiv -\infty$ for all large enough $N$.

Lastly, let $\vec{x}\,^N, \vec{y}\,^N \in (N^{-\alpha/2}\mathbb{Z})^k \cap W_k$, write $\tilde{x}^N_i = (x_i^N - pa_N N^{\alpha/2})/\sqrt{p(1-p)}$, $\tilde{y}^N_i = (y_i^N - pb_N N^{\alpha/2})/\sqrt{p(1-p)}$, and suppose that $\tilde{x}^N_i \to x_i$ and $\tilde{y}^N_i \to y_i$ as $N\to\infty$ for each $i\in\llbracket 1,k\rrbracket$. Then there exists $N_0 \in \mathbb{N}$ so that $\mathbb{P}^{a_N,b_N,\vec{x}\,^N,\vec{y}\,^N,f_N,g_N}_{avoid,N}$ is well-defined for $N\geq N_0$. Moreover, if $\mathcal{Y}^N$ have laws $\mathbb{P}^{a_N,b_N,\vec{x}\,^N,\vec{y}\,^N,f_N,g_N}_{avoid,N}$ and $\mathcal{Z}^N = \mathcal{Y}^N|_{\llbracket 1, k\rrbracket \times[a,b]}$, i.e. $\mathcal{Z}^N$ is a sequence of random variables on $C(\llbracket 1, k \rrbracket \times [a,b])$ obtained by projecting $\mathcal{Y}^N$ to $\llbracket 1, k\rrbracket \times[a,b]$, then the law of $\mathcal{Z}^N$ converges weakly to $\mathbb{P}^{a,b,\vec{x},\vec{y},f,g}_{avoid}$ as $N\to\infty$.
\end{lemma}

The next lemma shows that at any given point, the values of the $k-1$ curves in $\mathcal{L}^\infty$ are each distinct, so that Lemma \ref{scaledavoidBB} may be applied.

\begin{lemma}\label{inftydistinct} Let $\mathcal{L}^\infty$ be as in the beginning of this section. Then for any $s \in \mathbb{R}$ we have $\mathcal{L}^\infty(s) = (\tilde{f}_1^\infty(s),\dots,\tilde{f}_{k-1}^\infty(s)) \in W^\circ_{k-1}$, $\mathbb{P}$-a.s.
\end{lemma}

Using these two lemmas whose proofs are postponed, we now give the proof of Theorem \ref{PropTightGood} (ii).

\begin{proof} (of Theorem \ref{PropTightGood} (ii)) For clarity we split the proof into three steps. \\

{\bf \raggedleft Step 1.} We write $\Sigma = \llbracket 1, k-1 \rrbracket$ and $\tilde{\mathcal{L}}^N = (\tilde{f}^N_1,\dots,\tilde{f}^N_{k-1}) $, where we recall that 
$$\tilde{f}^N_i(s) = N^{-\alpha/2}(L^N_i(sN^\alpha)-psN^\alpha)/\sqrt{p(1-p)}.$$
Since $\mathcal{L}^\infty$ is a weak subsequential limit of $\tilde{\mathcal{L}}^N$ we know there is a subsequence $\{N_m\}_{m = 1}^\infty$ such that $\tilde{\mathcal{L}}^{N_m} \implies \mathcal{L}^\infty$. We will still call the subsequence $\tilde{\mathcal{L}}^N$ to not overburden the notation. By the Skorohod representation theorem \cite[Theorem 6.7]{Billing}, we can also assume that $\tilde{\mathcal{L}}^N$ and $\mathcal{L}^\infty$ are all defined on the same probability space with measure $\mathbb{P}$ and the convergence is happening $\mathbb{P}$-almost surely. Here we are implicitly using Lemma \ref{Polish}  from which we know that the random variables $\tilde{\mathcal{L}}^N$ and $\mathcal{L}^\infty$ take value in a Polish space so that the Skorohod representation theorem is applicable. 

Fix a set $K = \llbracket k_1,k_2\rrbracket \subseteq \llbracket 1, k-2\rrbracket$ and $a,b\in\mathbb{R}$ with $a<b$. We claim that for any bounded Borel-measurable function $F:C(K\times[a,b])\to\mathbb{R}$ we have
\begin{equation}\label{BGPcondex}
\ex[F(\mathcal{L}^\infty|_{K\times[a,b]})\,|\,\mathcal{F}_{ext}(K\times(a,b))] = \ex^{a,b,\vec{x},\vec{y},f,g}_{avoid}[F(\mathcal{Q})],
\end{equation}
where $\vec{x} = (\tilde{f}^\infty_{k_1}(a),\dots,\tilde{f}^\infty_{k_2}(a))$, $\vec{y} = (\tilde{f}^\infty_{k_1}(b),\dots,\tilde{f}^\infty_{k_2}(b))$, $f=\tilde{f}^\infty_{k_1-1}$ (with $\tilde{f}^\infty_0 = +\infty$), $g=\tilde{f}^\infty_{k_2+1}$, the $\sigma$-algebra $\mathcal{F}_{ext}(K\times(a,b))$ is as in Definition \ref{DefBGP}, and $\mathcal{Q}$ has law $\mathbb{P}^{a,b,\vec{x},\vec{y},f,g}_{avoid}$. We mention that by Lemma \ref{inftydistinct} we have $\mathbb{P}$-a.s. that $\tilde{f}_i^{\infty}(x) > \tilde{f}_{i+1}(x)$ for all $i \in \llbracket 1, k - 2 \rrbracket$ so that the right side of (\ref{BGPcondex}) is well-defined. We will prove (\ref{BGPcondex}) in the steps below. Here we assume its validity and conclude the proof of the theorem.\\

We first observe from (\ref{BGPcondex}) that for any bounded Borel-measurable $F_1: C( \llbracket 1, k-2 \rrbracket \times[a,b])\to\mathbb{R}$ and $F_2: C([a,b]) \rightarrow \mathbb{R}$ we have $\mathbb{P}$-almost surely
\begin{equation}\label{BGPfarg}
F_2(\tilde{f}^{\infty}_{k-1}) \cdot \ex[F_1(\mathcal{L}^\infty|_{ \llbracket 1, k-2 \rrbracket \times[a,b]})\,|\,\mathcal{F}_{ext}(\llbracket 1, k-2 \rrbracket  \times(a,b))] = F_2(\tilde{f}^{\infty}_{k-1}) \cdot  \ex^{a,b,\vec{x},\vec{y},f,g}_{avoid}[F_1(\mathcal{Q})].
\end{equation}
Let $\mathcal{H}$ be the class of bounded Borel-measurable functions $H:C( \llbracket 1, k-1 \rrbracket \times[a,b])\to\mathbb{R}$ that satisfy
\begin{equation}\label{BGPfarg2}
 \ex[H(\mathcal{L}^\infty|_{\llbracket 1, k-1 \rrbracket  \times[a,b]})\,|\,\mathcal{F}_{ext}( \llbracket 1, k-2 \rrbracket \times(a,b))] =    \ex^{a,b,\vec{x},\vec{y},f,g}_{avoid}[H(\mathcal{Q},g)],
\end{equation}
where on the right side $(\mathcal{Q},g)$ is the line ensemble with $k-1$ curves, whose top $k-2$ curves agree with $\mathcal{Q}$ and the $k-1$-st one agrees with $g$. From (\ref{BGPfarg}) $\mathcal{H}$ contains functions of the form 
$$H(f_1,\dots, f_{k-1}) = \prod_{i = 1}^{k-1}{\bf 1}\{f_i \in B_i\}$$
for any Borel sets $B_1, \dots, B_{k-1} \subset C([a,b])$. In addition, it is clear from (\ref{BGPfarg2}) that $\mathcal{H}$ is closed under linear combinations (by linearity of conditional expectations) and under bounded monotone limits (by the monotone convergence theorem for conditional expectations). Thus by the monotone class theorem \cite[Theorem 5.2.2]{Durrett}, $\mathcal{H}$ contains all bounded Borel measurable functions $H:C( \llbracket 1, k-1 \rrbracket \times[a,b])\to\mathbb{R}$. 

In particular, setting $H(f_1, \dots, f_{k-1}) = {\bf 1}\{ f_1(s) > \cdots > f_{k-1}(s) \mbox{ for all } s \in [a,b]\}$ in (\ref{BGPfarg2}) we conclude that $\mathcal{L}^\infty|_{\llbracket 1, k-1 \rrbracket  \times[a,b]}$ is non-intersecting $\mathbb{P}$-a.s. for any $a < b$. Taking $a = -m$, $b = m$ and a countable intersection over $m \in \mathbb{Z}$ we conclude that $\mathcal{L}^{\infty}$ is non-intersecting $\mathbb{P}$-a.s. The latter and (\ref{BGPcondex}) together imply that $\mathcal{L}^{\infty}$ satisfies the partial Brownian Gibbs property of Definition \ref{DefPBGP}.\\

\noindent\textbf{Step 2.} In this and the next step we prove (\ref{BGPcondex}). Fix $m\in\mathbb{N}$, $n_1,\dots,n_m\in\Sigma$, $t_1,\dots,t_m\in\mathbb{R}$, and $h_1,\dots,h_m : \mathbb{R}\to\mathbb{R}$ bounded continuous functions. Define $S = \{i\in\llbracket 1,m\rrbracket : n_i \in K, t_i \in [a,b]\}$. In this step we prove that
	\begin{equation}\label{BBcondexsplit}
		\ex\left[\prod_{i=1}^m h_i(\tilde{f}^\infty_{n_i}(t_i))\right] = \ex\left[\prod_{s\notin S} h_s(\tilde{f}^\infty_{n_s}(t_s))\cdot\ex^{a,b,\vec{x},\vec{y},f,g}_{avoid}\left[\prod_{s\in S} h_s(Q_{n_s}(t_s))\right]\right],
	\end{equation}
	where $Q$ denotes a random variable with law $\pr^{a,b,\vec{x},\vec{y},f,g}_{avoid}$. By assumption, we have
	\begin{equation}\label{BGPweak}
		\lim_{N\to\infty}\ex\left[\prod_{i=1}^m h_i(\tilde{f}^N_{n_i}(t_i))\right] = \ex\left[\prod_{i=1}^m h_i(\tilde{f}^\infty_{n_i}(t_i))\right].
	\end{equation}
	We define the sequences $a_N = \lfloor aN^\alpha\rfloor N^{-\alpha}$, $b_N = \lceil bN^\alpha\rceil N^{-\alpha}$, $\vec{x}\,^N = (L_{k_1}^N(a_N),\dots,L_{k_2}^N(a_N))$, $\vec{y}\,^N = (L_{k_1}^N(b_N),\dots,L_{k_2}^N(b_N))$, $f_N = \tilde{f}_{k_1-1}^N$ (where $\tilde{f}^N_0 = +\infty$), $g_N = \tilde{f}_{k_2+1}^N$. Since $a_N \to a$, $b_N\to b$, we may choose $N_0$ sufficiently large so that if $N\geq N_0$, then $t_s < a_N$ or $t_s > b_N$ for all $s\notin S$ with $n_s \in K$. Since the line ensemble $(L_1^N,\dots,L_{k-1}^N)$ in the definition of $\tilde{\mathcal{L}}^N$ satisfies the Schur Gibbs property (see Definition \ref{DefSGP}), we see from Definition \ref{scaledRW} that the law of $\tilde{\mathcal{L}}^N|_{K\times[a,b]}$ conditioned on the $\sigma$-algebra $\mathcal{F} = \sigma\left(\tilde{f}^N_{k_1-1}, \tilde{f}^N_{k_2+1}, \tilde{f}^N_{k_1}(a_N), \tilde{f}^N_{k_1}(b_N),\dots,\tilde{f}^N_{k_2}(a_N),\tilde{f}^N_{k_2}(b_N)\right)$ is (upto a reindexing of the curves) precisely $\mathbb{P}^{a_N,b_N,\vec{x}\,^N,\vec{y}\,^N,f_N,g_N}_{avoid,N}$. Therefore, writing $Z^N$ for a random variable with this law, we have
	\begin{equation}\label{BBschur}
		\ex\left[\prod_{i=1}^m h_i(\tilde{f}^N_{n_i}(t_i))\right] = \ex\left[\prod_{s\notin S} h_s(\tilde{f}^N_{n_s}(t_s))\cdot\ex^{a_N,b_N,\vec{x}\,^N,\vec{y}\,^N,f_N,g_N}_{avoid,N}\left[\prod_{s\in S} h_s(Z^N_{n_s-k_1+1}(t_s))\right]\right].
	\end{equation}
	Now by Lemma \ref{inftydistinct}, we have $\mathbb{P}$-a.s. that $\tilde{f}^\infty_{k_1 -1}(a) > \tilde{f}^\infty_{k_1}(a) > \cdots > \tilde{f}^\infty_{k_2}(a) > \tilde{f}^\infty_{k_2+1}(a)$ and also $\tilde{f}^\infty_{k_1 -1}(b) > \tilde{f}^\infty_{k_1}(b) > \cdots > \tilde{f}^\infty_{k_2}(b) > \tilde{f}^\infty_{k_2+1}(b)$.  In addition, we have by part (i) of Theorem \ref{PropTightGood} that $\mathbb{P}$-almost surely $f_N\to f = \tilde{f}^\infty_{k_1-1}$ and $g_N\to g = \tilde{f}^\infty_{k_2+1}$ uniformly on $[a-1,b+1]\supseteq [a_N,b_N]$, and $(x_i^N - pa_N N^{\alpha/2})/\sqrt{p(1-p)}\to\vec{x}$, $(y_i^N-pb_N N^{\alpha/2})/\sqrt{p(1-p)}\to\vec{y}$ for $i\in\llbracket 1,k-1\rrbracket$. It follows from Lemma \ref{scaledavoidBB} that 
	\begin{equation}\label{BGPNweak}
		\lim_{N\to\infty} \ex^{a_N,b_N,\vec{x}\,^N,\vec{y}\,^N,f_N,g_N}_{avoid,N}\left[\prod_{s\in S} h_s(Z^N_{n_s-k_1+1}(t_s))\right] = \ex^{a,b,\vec{x},\vec{y},f,g}_{avoid}\left[\prod_{s\in S} h_s(Q_{n_s}(t_s))\right].
	\end{equation}
	Lastly, the continuity of the $h_i$ implies that
	\begin{equation}\label{BGPuniform}
		\lim_{N\to\infty}\prod_{s\notin S} h_s(\tilde{f}^N_{n_s}(t_s)) = \prod_{s\notin S} h_s(\tilde{f}^\infty_{n_s}(t_s)).
	\end{equation}
	Combining \eqref{BGPweak}, \eqref{BBschur}, \eqref{BGPNweak}, and \eqref{BGPuniform} with the bounded convergence theorem proves \eqref{BBcondexsplit}.\\
	
	\noindent\textbf{Step 3.} In this step we use \eqref{BBcondexsplit} to prove \eqref{BGPcondex}. The argument below is a standard monotone class argument. For $n\in\mathbb{N}$ we define piecewise linear functions
	\[
	\chi_n(x,r) = \begin{cases}
		0, & x > r + 1/n,\\
		1-n(x-r), & x\in[r,r+1/n],\\
		1, & x < r.
	\end{cases}
	\]
	We fix $m_1,m_2\in\mathbb{N}$, $n^1_1,\dots,n^1_{m_1},n^2_1,\dots,n^2_{m_2}\in\Sigma$, $t^1_1,\dots,t^1_{m_1},t^2_1,\dots,t^2_{m_2}\in\mathbb{R}$, such that $(n^1_i,t^1_i)\notin K\times[a,b]$ and $(n^2_i,t^2_i)\in K\times[a,b]$ for all $i$. Then \eqref{BBcondexsplit} implies that
	\[
	\ex\left[\prod_{i=1}^{m_1} \chi_n(f_{n_i^1}^\infty(t_i^1),a_i)\prod_{i=1}^{m_2}\chi_n(f_{n_i^2}^\infty(t_i^2),b_i)\right] = \ex\left[\prod_{i=1}^{m_1} \chi_n(f_{n_i^1}^\infty(t_i^1),a_i)\ex^{a,b,\vec{x},\vec{y},f,g}_{avoid}\left[\prod_{i=1}^{m_2} \chi_n(Q_{n_i^2}(t_i^2),b_i)\right]\right].
	\]
	Letting $n\to\infty$, we have $\chi_n(x,r)\to \chi(x,r)=\mathbf{1}_{x\leq r}$, and the bounded convergence theorem gives
	\[
	\ex\left[\prod_{i=1}^{m_1} \chi(f_{n_i^1}^\infty(t_i^1),a_i)\prod_{i=1}^{m_2}\chi(f_{n_i^2}^\infty(t_i^2),b_i)\right] = \ex\left[\prod_{i=1}^{m_1} \chi(f_{n_i^1}^\infty(t_i^1),a_i)\ex^{a,b,\vec{x},\vec{y},f,g}_{avoid}\left[\prod_{i=1}^{m_2} \chi(Q_{n_i^2}(t_i^2),b_i)\right]\right].
	\]
	Let $\mathcal{H}$ denote the space of bounded Borel measurable functions $H:C(K\times[a,b])\to\mathbb{R}$ satisfying
	\begin{equation}\label{BGPH}
		\ex\left[\prod_{i=1}^{m_1} \chi(f_{n_i^1}^\infty(t_i^1),a_i)H(\mathcal{L}^\infty|_{K\times[a,b]})\right] = \ex\left[\prod_{i=1}^{m_1} \chi(f_{n_i^1}^\infty(t_i^1),a_i)\ex^{a,b,\vec{x},\vec{y},f,g}_{avoid}\left[H(\mathcal{Q})\right]\right].
	\end{equation}
	The above shows that $\mathcal{H}$ contains all functions $\mathbf{1}_A$ for sets $A$ contained in the $\pi$-system $\mathcal{A}$ consisting of sets of the form
	\[
	\{h\in C(K\times[a,b]) : h(n_i^2,t_i^2) \leq b_i \mbox{ for } i\in\llbracket 1,m_2\rrbracket\}.
	\]
	We note that $\mathcal{H}$ is closed under linear combinations simply by linearity of expectation, and if $H_n\in\mathcal{H}$ are nonnegative bounded measurable functions converging monotonically to a bounded function $H$, then $H\in\mathcal{H}$ by the monotone convergence theorem. Thus by the monotone class theorem \cite[Theorem 5.2.2]{Durrett}, $\mathcal{H}$ contains all bounded $\sigma(\mathcal{A})$-measurable functions. Since the finite dimensional sets in $\mathcal{A}$ generate the full Borel $\sigma$-algebra $\mathcal{C}_K$ (see for instance \cite[Lemma 3.1]{DimMat}), we have in particular that $F\in\mathcal{H}$.
	
	Now let $\mathcal{B}$ denote the collection of sets $B\in\mathcal{F}_{ext}(K\times(a,b))$ such that
	\begin{equation}\label{BGPB}
		\ex[\mathbf{1}_B \cdot F(\mathcal{L}^\infty|_{K\times[a,b]})] = \ex[\mathbf{1}_B \cdot \ex^{a,b,\vec{x},\vec{y},f,g}_{avoid}[F(\mathcal{Q})]].
	\end{equation}
	We observe that $\mathcal{B}$ is a $\lambda$-system. Indeed, since \eqref{BGPH} holds for $H=F$, taking $a_i,b_i\to\infty$ and applying the bounded convergence theorem shows that \eqref{BGPB} holds with $\mathbf{1}_B = 1$. Thus if $B\in\mathcal{B}$ then $B^c\in\mathcal{B}$ since $\mathbf{1}_{B^c} = 1-\mathbf{1}_B$. If $B_i\in\mathcal{B}$, $i\in\mathbb{N}$, are pairwise disjoint and $B=\bigcup_i B_i$, then $\mathbf{1}_B = \sum_i \mathbf{1}_{B_i}$, and it follows from the monotone convergence theorem that $B\in\mathcal{B}$. Moreover, \eqref{BGPH} with $H=F$ implies that $\mathcal{B}$ contains the $\pi$-system $P$ of sets of the form
	\[
	\{h\in C(\Sigma\times\mathbb{R}) : h(n_i,t_i) \leq a_i \mbox{ for } i \in\llbracket 1,m_1\rrbracket, \mbox{ where } (n_i,t_i)\notin K\times(a,b)\}.
	\]
	By the $\pi$-$\lambda$ theorem \cite[Theorem 2.1.6]{Durrett} it follows that $\mathcal{B}$ contains $\sigma(P) = \mathcal{F}_{ext}(K\times(a,b))$. Thus \eqref{BGPB} holds for all $B\in\mathcal{F}_{ext}(K\times(a,b))$. It is proven in \cite[Lemma 3.4]{DimMat} that $\ex^{a,b,\vec{x},\vec{y},f,g}_{avoid}[F(\mathcal{Q})]$ is an $\mathcal{F}_{ext}(K\times(a,b))$-measurable function. Therefore \eqref{BGPcondex} follows from \eqref{BGPB} by the definition of conditional expectation. This suffices for the proof.
	
\end{proof}

%
\section{Bounding the max and min}\label{Section5}

In this section we prove Lemmas \ref{PropSup} and \ref{PropSup2} and we assume the same notation as in the statements of these lemmas. In particular, we assume that $k \in \mathbb{N}$, $k \geq 2$, $p \in (0,1)$, $\alpha, \lambda > 0$ are all fixed and 
\begin{equation*}
	\big\{\mathfrak{L}^N = (L^N_1,L^N_2, \dots, L^N_k)\big\}_{N=1}^{\infty},
\end{equation*}
is an $(\alpha,p,\lambda)$-good sequence of $\llbracket 1, k\rrbracket$-indexed Bernoulli line ensembles as in Definition \ref{Def1} that are all defined on a probability space with measure $\mathbb{P}$. The proof of Lemma \ref{PropSup} is given in Section \ref{Section5.1} and Lemma \ref{PropSup2} is proved in Section \ref{Section5.2}.

%
\subsection{Proof of Lemma \ref{PropSup}}\label{Section5.1}

Our proof of Lemma \ref{PropSup} is similar to that of \cite[Lemma 5.2]{CD}. For clarity we split the proof into three steps. In the first step we introduce some notation that will be required in the proof of the lemma, which is presented in Steps 2 and 3. \\

{\bf \raggedleft Step 1.} We write $s_4 = \lfloor \lceil r+4 \rceil  N^\alpha \rfloor$, $s_3 = \lfloor \lfloor (r+3) \rfloor N^\alpha\rfloor $, so that $s_3 \leq t_3 \leq s_4$, and assume that $N$ is large enough so that $\psi(N)N^{\alpha}$ from Definition \ref{Def1} is at least $s_4$. Notice that such a choice is possible by our assumption that $\mathfrak{L}^N$ is an $(\alpha,p,\lambda)$-good sequence and in particular, we know that $L_i^N$ are defined at $\pm s_4$ for $i \in \llbracket 1, k \rrbracket$. We define events 
$$E(M) = \Big\{\big|L_1^N(-s_4) + ps_4\big| > MN^{\alpha/2}\Big\}, \quad F(M) = \Big\{L_1^N(-s_3) > -ps_3 + MN^{\alpha/2} \Big\},$$
$$G(M) = \Bigg\{\sup_{s\in[0,s_4]} \big[L_1^N(s) - ps \big] \geq (6r+22)(2r+10)^{1/2}(M+1)N^{\alpha/2} \Bigg\}.$$

If $\epsilon > 0$ is as in the statement of the lemma, we note by (\ref{globalParabola}) that we can find $M$ and $\tilde{N}_1$ sufficiently large so that if $N \geq \tilde{N}_1$ we have 
\begin{equation}\label{4.2EFbounds}
	\mathbb{P}(E(M)) < \epsilon/4 \mbox{ and } \mathbb{P}(F(M)) < \epsilon/12.
\end{equation}
In the remainder of this step we show that the event $G(M) \setminus E(M)$ can be written as a {\em countable disjoint union} of certain events, i.e. we show that
\begin{equation}\label{WR3}
	\bigsqcup\limits_{(a,b,s,\ell_{top},\ell_{bot}) \in D(M)}  E(a,b,s,\ell_{top},\ell_{bot}) = G(M) \setminus E(M),
\end{equation}
where the sets $E(a,b,s,\ell_{top},\ell_{bot})$ and $D(M)$ are described below. 

For $a,b,z_1,z_2,z_3 \in\mathbb{Z}$ with $z_1\leq a$, $z_2\leq b$, $s\in\llbracket 0, s_4 \rrbracket$, $\ell_{bot}\in\Omega(-s_4,s,z_1,z_2)$ and $\ell_{top}\in\Omega(s,s_4,b,z_3)$ we define $E(a,b,s,\ell_{top},\ell_{bot})$ to be the event that $L_1^N(-s_4) = a$, $L_1^N(s) = b$, $L_1^N$ agrees with $\ell_{top}$ on $\llbracket s,s_4\rrbracket$, and $L_2^N$ agrees with $\ell_{bot}$ on $\llbracket -s_4,s \rrbracket$.  Let $D(M)$ be the set of tuples $(a,b,s, \ell_{top}, \ell_{bot})$ satisfying
\begin{enumerate}[label=(\arabic*)]
	\item $0\leq s\leq s_4$,
	\item $0\leq b-a \leq s + s_4$, $|a + ps_4| \leq MN^{\alpha/2}$, and $b-ps \geq (6r+22)(2r+10)^{1/2}(M+1)N^{\alpha/2}$,
	\item $z_1\leq a$, $z_2\leq b$, and $\ell_{bot}\in\Omega(-s_4, s, z_1, z_2)$,
	\item $b \leq z_3 \leq b + (s_4 - s)$, and $\ell_{top} \in\Omega(s, s_4, b, z_3)$,
	\item if $s < s' \leq s_4$, $s' \in \mathbb{Z}$ then $\ell_{top}(s') -ps' < (6r+22)(2r+10)^{1/2}(M+1)N^{\alpha/2}$.
\end{enumerate}
It is clear that $D(M)$ is countable. The five conditions above together imply that 
$$\bigcup_{(a,b,s,\ell_{top},\ell_{bot}) \in D(M)}  E(a,b,s,\ell_{top},\ell_{bot}) = G(M) \setminus E(M),$$
and what remains to be shown to prove (\ref{WR3}) is that $E(a,b,s,\ell_{top},\ell_{bot})$ are pairwise disjoint. 

On the intersection of $E(a,b,s,\ell_{top},\ell_{bot})$ and $E(\tilde{a},\tilde{b},\tilde{s},\tilde{\ell}_{top},\tilde{\ell}_{bot})$ we must have $\tilde{a} = L_1^N(-s_4) = a$ so that $a = \tilde{a}$. Furthermore, we have by properties (2) and (5) that $s \geq \tilde{s}$ and $\tilde{s} \geq s$ from which we conclude that $s= \tilde{s}$ and then we conclude $\tilde{b} = b$, $\ell_{top} = \tilde{\ell}_{top}$, $\ell_{bot} = \tilde{\ell}_{bot}$. In summary, if $E(a,b,s,\ell_{top},\ell_{bot})$ and $E(\tilde{a},\tilde{b},\tilde{s},\tilde{\ell}_{top},\tilde{\ell}_{bot})$ have a non-trivial intersection then $(a,b,s,\ell_{top},\ell_{bot}) = (\tilde{a},\tilde{b},\tilde{s},\tilde{\ell}_{top},\tilde{\ell}_{bot})$, which proves (\ref{WR3}).\\

{\raggedleft \bf Step 2.} In this step we prove that  we can find an $N_2$ so that for $N \geq N_2$
\begin{equation}\label{WR1}
	\mathbb{P}\left( \sup_{s \in [0,t_3] }\big[ L^N_1(s) - p s \big] \geq  (6r+22)(2r+10)^{1/2}(M+1)N^{\alpha/2} \right) \leq \mathbb{P}(G(M)) < \epsilon/2.
\end{equation}
A similar argument, which we omit, proves the same inequality with $[-t_3,0]$ in place of $[0,t_3]$ and then the statement of the lemma holds for all $N\geq N_2$, with $R_1 = (6r+22)(2r+10)^{1/2}(M+1)$.

We claim that we can find $\tilde{N}_2 \in \mathbb{N}$ sufficiently large so that if $N \geq \tilde{N}_2$ and $(a,b,s,\ell_{top},\ell_{bot})\in D(M)$ satisfies $\mathbb{P}( E(a,b,s,\ell_{top},\ell_{bot})) > 0$ then we have
\begin{equation}\label{WR7}
	\mathbb{P}^{-s_4,s,a,b,\infty,\ell_{bot}}_{avoid, Ber}\left(\ell(-s_3) > -ps_3 + MN^{\alpha/2}\right) \geq \frac{1}{3}.
\end{equation}
We will prove (\ref{WR7}) in Step 3. For now we assume its validity and conclude the proof of (\ref{WR1}).\\

Let $(a,b,s,\ell_{top},\ell_{bot})\in D(M)$ be such that $\mathbb{P}( E(a,b,s,\ell_{top},\ell_{bot})) > 0$. By the Schur Gibbs property, see Definition \ref{DefSGP}, we have for any $\ell_0\in\Omega(-s_4,s,a,b)$ that
\begin{equation}\label{WR9}
	\mathbb{P}\left(L_1^N\llbracket -s_4,s \rrbracket = \ell_0\,|\,E(a,b,s,\ell_{top}, \ell_{bot}) \right) = \mathbb{P}^{-s_4,s,a,b,\infty,\ell_{bot}}_{avoid, Ber}(\ell = \ell_0),
\end{equation}
where $L_1^N\llbracket -s_4,s \rrbracket$ denotes the restriction of $L_1^N$ to the set $\llbracket -s_4,s \rrbracket$.

Combining (\ref{WR7}) and (\ref{WR9}) we get for $N \geq \tilde{N}_2$
\begin{equation}\label{WR10}
	\begin{split}
		&\mathbb{P}\left( L_1^N(-s_3) > -ps_3 + MN^{\alpha/2} \vert E(a,b,s, \ell_{top}, \ell_{bot})\right) = \\
		&\mathbb{P}^{-s_4,s,a,b,\infty,\ell_{bot}}_{avoid, Ber}\left(\ell(-s_3) > -ps_3 + MN^{\alpha/2}\right) \geq \frac{1}{3}.
	\end{split}
\end{equation}

It follows from (\ref{WR10}) that for $N \geq \tilde{N}_2$ we have
\begin{equation}\label{WR11}
	\begin{split}
		& \epsilon/12 > \mathbb{P}(F(M)) \geq \sum_{\substack{(a,b,s,\ell_{top}, \ell_{bot})\in D(M), \\ \mathbb{P}(E(a,b,s, \ell_{top}, \ell_{bot})) > 0}} \mathbb{P}\left(F(M)\cap E(a,b,s, \ell_{top}, \ell_{bot})\right) = \\
		&\sum_{\substack{(a,b,s,\ell_{top}, \ell_{bot})\in D(M), \\ \mathbb{P}(E(a,b,s, \ell_{top}, \ell_{bot})) > 0}} \hspace{-5mm} \mathbb{P}\left( L_1^N(-s_3) > -ps_3 + MN^{\alpha/2} \vert  E(a,b,s,\ell_{top}, \ell_{bot}) \right )\mathbb{P}\left(E(a,b,s,\ell_{top}, \ell_{bot}) \right) \geq \\
		& \sum_{\substack{(a,b,s,\ell_{top}, \ell_{bot})\in D(M), \\ \mathbb{P}(E(a,b,s, \ell_{top}, \ell_{bot})) > 0}} \frac{1}{3} \cdot \mathbb{P}\left(E(a,b,s,\ell_{top}, \ell_{bot}) \right)  =  \frac{1}{3} \cdot \mathbb{P}(G(M)\setminus E(M)),
	\end{split}
\end{equation}
where in the last equality we used (\ref{WR3}). From (\ref{4.2EFbounds}) and (\ref{WR11}) we have for $N \geq N_2 = \max(\tilde{N}_1, \tilde{N}_2)$  
$$\mathbb{P}(G(M)) \leq \mathbb{P}(E(M)) + \mathbb{P}(G(M)\setminus E(M))< \epsilon/4 + \epsilon/4,$$
which proves (\ref{WR1}).\\

{\bf \raggedleft Step 3.} In this step we prove (\ref{WR7}) and in the sequel we let $(a,b,s,\ell_{top},\ell_{bot})\in D(M)$ be such that $\mathbb{P}( E(a,b,s,\ell_{top},\ell_{bot})) > 0$. We remark that the condition $\mathbb{P}( E(a,b,s,\ell_{top},\ell_{bot})) >0 $ implies that $\Omega_{avoid}({-s_4,s,a,b,\infty,\ell_{bot}})$ is not empty. By Lemma \ref{MCLfg} we know that 
$$\mathbb{P}^{-s_4,s,a,b,\infty,\ell_{bot}}_{avoid, Ber}\left(\ell(-s_3) > -ps_3 + MN^{\alpha/2}\right) \geq \mathbb{P}^{-s_4,s,a,b}_{ Ber}\left(\ell(-s_3) > -ps_3 + MN^{\alpha/2}\right),$$
and so it suffices to show that 
\begin{equation}\label{WT0}
	\mathbb{P}^{-s_4,s,a,b}_{Ber}\left(\ell(-s_3) > -ps_3 + MN^{\alpha/2}\right) \geq \frac{1}{3}.
\end{equation}

One directly observes that
\begin{equation}\label{WT2}
	\begin{split}
		&\mathbb{P}^{-s_4,s,a,b}_{Ber}\left( \ell(-s_3) > -ps_3 + MN^{\alpha/2}\right) = \mathbb{P}^{0,s+s_4,0,b-a}_{Ber}\left(\ell(s_4-s_3) + a \geq -ps_3 + MN^{\alpha/2}\right) \geq\\
		& \mathbb{P}^{0,s+s_4,0,b-a}_{Ber}\left(\ell(s_4-s_3) \geq p(s_4-s_3) + 2MN^{\alpha/2}\right),
	\end{split}
\end{equation}
where the inequality follows from the assumption in (2) that $a+ps_4 \geq -MN^{\alpha/2}$. Moreover, since $b-ps \geq (6r+22)(2r+10)^{1/2}(M+1)N^{\alpha/2}$ and $a+ps_4 \leq MN^{\alpha/2}$, we have 
$$b-a \geq p(s+s_4) + (6r+21)(2r+10)^{1/2}(M+1)N^{\alpha/2} \geq p(s+s_4) + (6r+21)(M+1)(s+s_4)^{1/2}.$$
The second inequality follows since $s+s_4 \leq 2s_4 \leq (2r+10)N^{\alpha}$. 

It follows from Lemma \ref{LemmaHalfS4} with $M_1 = 0$, $M_2 = (6r+21)(M+1)$ that for sufficiently large $N$
\begin{equation}\label{WT3}
	\mathbb{P}^{0,s+s_4,0,b-a}_{Ber}\Big(\ell(s_4-s_3) \geq \frac{s_4-s_3}{s+s_4}[p(s+s_4) + M_2 (s+s_4)^{1/2}] - (s+s_4)^{1/4}\Big) \geq 1/3.
\end{equation}

Note that $\frac{s_4-s_3}{s+s_4} \geq \frac{N^\alpha }{(2r+10)N^\alpha} = \frac{1}{2r+10}$ and so for all $N \in \mathbb{N}$ we have
\begin{equation}\label{WT4}
	\begin{split}
		& \frac{s_4-s_3}{s+s_4}[p(s+s_4) + M_2 (s+s_4)^{1/2}] - (s+s_4)^{1/4} \geq \\
		& p(s_4-s_3) + \frac{(6r+21)(M+1) (s+s_4)^{1/2}}{2r+10}  - (s+s_4)^{1/4} \geq p(s_4-s_3) + 2MN^{\alpha/2}.
	\end{split}
\end{equation}
Combining (\ref{WT2}), (\ref{WT3}) and (\ref{WT4}) we conclude that we can find $\tilde{N}_2 \in \mathbb{N}$ such that if $N \geq \tilde{N}_2$ we have (\ref{WT0}). This suffices for the proof.

%
\subsection{Proof of Lemma \ref{PropSup2}}\label{Section5.2} We mention that the general idea behind the proof of Lemma \ref{PropSup2} has similarities with the proof of \cite[Proposition 2.7]{Ham4}.

We begin by proving the following important lemma, which shows that it is unlikely that the curve $L_{k-1}^N$ falls uniformly very low on a large interval.
\begin{lemma}\label{21}
	Under the same conditions as in Lemma \ref{PropSup2} the following holds. For any $r,\epsilon>0$ there exist $R>0$ and $N_5 \in \mathbb{N}$ such that for all $N\geq N_5$
	\begin{equation}\label{E21}
		\mathbb{P}\left(\sup_{x\in[r,R]} \left(L_{k-1}^N(xN^\alpha) - pxN^\alpha\right) \leq -(\lambda R^2 + \phi_2(\epsilon/16) + 1)N^{\alpha/2}\right) < \epsilon,
	\end{equation}
	where $\lambda, \phi$ are as in the definition of an $(\alpha,p,\lambda)$-good sequence of line ensembles, see Definition \ref{Def1}. The same statement holds if $[r,R]$ is replaced with $[-R,-r]$ and the constants $N_5, R$ depend on $\epsilon, r$ as well as the parameters $\alpha, p, \lambda, k$ and the functions $\phi_2, \psi$ from Definition \ref{Def1}. 
\end{lemma}
\begin{proof}
	Before we go into the proof we give an informal description of the main ideas. The key to this lemma is the parabolic shift implicit in the definition of an $(\alpha,p,\lambda)$-good sequence. This shift requires that the deviation of the top curve $L_1^N$ from the line of slope $p$ to appear roughly parabolic. On the event in equation (\ref{E21}) we have that the $(k-1)$-th curve dips very low uniformly on the interval $[r,R]$ and we will argue that on this event the top $k-2$ curves essentially do not feel the presence of the $(k-1)$-th curve. After a careful analysis using the monotone coupling lemmas from Section \ref{Section3.1} we will see that the latter statement implies that the curve $L_1^N$ behaves like a free bridge between its end-points that have been slightly raised. Consequently, we would expect the midpoint $L_1^N \left( N^{\alpha} (R+r)/2 \right)$ to be close (on scale $N^{\alpha/2}$) to $[L_1^N(rN^{\alpha}) + L^N_1(RN^{\alpha})]/2.$ However, with high probability $[L_1^N(rN^{\alpha}) + L^N_1(RN^{\alpha})]/2$ lies much lower than the inverted parabola $-\lambda(R+r)^2 N^{\alpha/2}/4 $ (due to the concavity of the latter), and so it is very unlikely for $L_1^N \left( N^{\alpha} (R+r)/2 \right)$ to be near it by our assumption. The latter would imply that the event in (\ref{E21}) is itself unlikely, since conditional on it an unlikely event suddenly became likely. 
	
	We proceed to fill in the details of the above sketch of the proof in the following steps. In total there are six steps and we will only prove the statement of the lemma for the interval $[r,R]$, since the argument for $[-R,-r]$ is very similar. \\
	
	\noindent\textbf{Step 1.} We begin by specifying the choice of $R$ in the statement of the lemma, fixing some notation and making a few simplifying assumptions. 
	
	Fix $r , \epsilon > 0$ as in the statement of the lemma. Note that for any $R>r$,
	$$\sup_{r\leq x\leq R} \big(L_{k-1}^N(xN^\alpha) - pxN^\alpha\big) \geq \sup_{\lceil r\rceil \leq x \leq R} \big(L_{k-1}^N(xN^\alpha) - pxN^\alpha\big).$$
	Thus by replacing $r$ with $\lceil r\rceil$, we can assume that $r\in\mathbb{Z}$, which we do in the sequel. Notice that by our assumption that $\mathfrak{L}^N$ is $(\alpha, p,\lambda)$-good we know that (\ref{E21}) holds trivially if $k = 2$ (with the right side of (\ref{E21}) being any number greater than $\epsilon/16$ and in particular $\epsilon$) and so in the sequel we assume that $k \geq 3$. 
	
	Define constants
	\begin{equation}\label{21Cdef}
		C = \sqrt{ 8p(1-p) \log\frac{3}{1-(11/12)^{1/(k-2)}}},
	\end{equation}
	and $R_0 > r$ sufficiently large so that for $R \geq R_0$ and $N \in \mathbb{N}$ we have
	\begin{equation}\label{21Rdef2}
		\frac{\lambda(R-r)^2}{4} \geq  2\phi_2(\epsilon/16) + 2 + k \lceil C  \lceil RN^{\alpha} \rceil -  \lfloor rN^{\alpha} \rfloor \rceil N^{-\alpha/2}.
	\end{equation}
	We define $R = \lceil R_0\rceil + \mathbf{1}_{\lceil R_0\rceil + r\;\mathrm{odd}}$, so that $R\geq R_0$ and the midpoint $(R+r)/2$ are integers. This specifies our choice of $R$ and for convenience we denote $m = (R+r)/2$. 
	
	In the following, we always assume $N$ is large enough so that $\psi(N) > R$, hence $L_i^N$ are defined at $RN^\alpha$ for $1\leq i\leq k$. We may do so by the second condition in the definition of an $(\alpha,p,\lambda)$-good sequence (see Definition \ref{Def1}). \\
	
	With the choice of $R$ as above we define the events
	\begin{equation}\label{21AB}
		\begin{split}
			A &= \left\{L_1^N\left(mN^{\alpha}\right) - pm N^\alpha  + \lambda m^2 N^{\alpha/2} < -\phi_2(\epsilon/16)N^{\alpha/2}\right\},\\
			B &= \left\{\sup_{x\in[r,R]} \left(L_{k-1}^N(xN^\alpha) - pxN^\alpha\right) \leq -[\lambda R^2 + \phi_2(\epsilon/16) + 1] N^{\alpha/2} \right\}.
		\end{split}
	\end{equation}
	The goal of the lemma is to prove that we can find $N_5\in\mathbb{N}$ so that for all $N\geq N_5$
	\begin{equation}\label{21Bbound}
		\mathbb{P}(B) < \epsilon,
	\end{equation}
	which we accomplish in the steps below.\\
	
	\noindent\textbf{Step 2.} In this step we introduce some notation that will be used throughout the next steps. Let $\gamma = \lfloor rN^{\alpha} \rfloor$ and $\Gamma = \lceil RN^{\alpha} \rceil$. We also define the event
	\begin{equation}\label{21x1y1}
		\begin{split}
			& F = \left\{ \sup_{s \in \{\gamma,  \Gamma\}} \left|L_1^N(s) - ps + \lambda s^2N^{-\alpha/2}\right|< [\phi_2(\epsilon/16) + 1] N^{\alpha/2} \right\}.
		\end{split}
	\end{equation}
	In the remainder of this step we show that $F \cap B$ can be written as a \textit{countable disjoint union}
	\begin{equation}\label{21F}
		F \cap B = \bigsqcup_{(\vec{x},\vec{y},\ell_{bot})\in D}  E(\vec{x},\vec{y},\ell_{bot}),
	\end{equation} 
	where the sets $E(\vec{x},\vec{y},\ell_{bot})$ and $D$ are defined below.
	
	For $\vec{x},\vec{y}\in\mathfrak{W}_{k-2}$, $z_1,z_2\in\mathbb{Z}$, and $\ell_{bot}\in\Omega(\gamma, \Gamma ,z_1,z_2)$, let $E(\vec{x},\vec{y},\ell_{bot})$ denote the event that $L_i^N(\gamma) = x_i$ and $L_i^N(\Gamma)=y_i$ for $1\leq i\leq k-2$, and $L_{k-1}^N$ agrees with $\ell_{bot}$ on $[\gamma, \Gamma]$. Let $D$ denote the set of triples $(\vec{x},\vec{y},\ell_{bot})$ satisfying
	\begin{enumerate}[label=(\arabic*)]
		\item $0\leq y_i - x_i \leq \Gamma - \gamma$ for $1\leq i\leq k-2$,
		\item $|x_1 - p\gamma + \lambda \gamma^2 N^{-3\alpha/2}| < \phi_2(\epsilon/16) N^{\alpha/2}$ and $|y_1- p \Gamma  + \lambda \Gamma^2N^{-3\alpha/2}| < \phi_2(\epsilon/16)N^{\alpha/2}$,
		\item $z_1\leq x_{k-2}$, $z_2\leq y_{k-2}$, and $\ell_{bot} \in \Omega(\gamma , \Gamma,z_1,z_2)$,
		\item $\sup_{x \in [r, R]} [\ell_{bot}(x N^{\alpha}) - pxN^{\alpha} ] \leq -[\lambda R^2 + \phi_2(\epsilon/16) + 1]N^{\alpha/2}$.
	\end{enumerate}
	It is clear that $D$ is countable, the events $E(\vec{x},\vec{y},\ell_{bot})$ are pairwise disjoint for different elements in $D$ and (\ref{21F}) is satisfied. \\

	\noindent\textbf{Step 3.} We claim that we can find $\tilde{N}_0$ so that for $N\geq \tilde{N}_0$ we have
	\begin{equation}\label{21AEbound}
		\mathbb{P}(A| E(\vec{x},\vec{y},\ell_{bot})) \geq 1/4
	\end{equation}
	for all $(\vec{x}, \vec{y}, \ell_{bot}) \in D$ such that $\mathbb{P}(E(\vec{x},\vec{y},\ell_{bot})) > 0$. We will prove (\ref{21AEbound}) in the steps below. In this step we assume its validity and conclude the proof of (\ref{21Bbound}). \\
	
	It follows from \eqref{21F} and \eqref{21AEbound} that for $N\geq\tilde{N}_0$ and $\mathbb{P}(F\cap B) > 0$ we have
	\begin{equation*}
		\begin{split}
			&\mathbb{P}(A|F \cap B) = \sum_{(\vec{x},\vec{y},\ell_{bot})\in D, \mathbb{P}(E(\vec{x},\vec{y},\ell_{bot}))) > 0} \frac{\mathbb{P}(A| E(\vec{x},\vec{y},\ell_{bot})\mathbb{P}( E(\vec{x},\vec{y},\ell_{bot}))}{\mathbb{P}(F \cap B)} \geq \\
			& \frac{1}{4}\cdot\frac{\sum_{(\vec{x},\vec{y},\ell_{bot})\in D, \mathbb{P}(E(\vec{x},\vec{y},\ell_{bot}))) > 0} \mathbb{P}( E(\vec{x},\vec{y},\ell_{bot}))}{\mathbb{P}(F\cap B)} = \frac{1}{4}.
		\end{split}
	\end{equation*}
	
	From the third condition in the definition of an $(\alpha,p,\lambda)$-good sequence, see Definition \ref{Def1}, we can find $\tilde{N}_1$ so that $\mathbb{P}(A) < \epsilon/8$ for $N\geq\tilde{N}_1$. Hence if $N \geq \max(\tilde{N}_1, \tilde{N}_2)$ and $\mathbb{P}(F\cap B) > 0$ we have
	\begin{equation}\label{NM1}
		\mathbb{P}(F \cap B) = \frac{\mathbb{P}(A\cap F \cap B)}{\mathbb{P}(A|F \cap B)} \leq 4\mathbb{P}(A) < \epsilon/2.
	\end{equation} 
	Lastly, by the same condition in Definition \ref{Def1} we can find $\tilde{N}_2$ so that for $N\geq\tilde{N}_2$ we have
	\begin{equation}\label{NM2}
		\mathbb{P}(F^c) = 2 \cdot \epsilon/8 = \epsilon/4.
	\end{equation}
	In deriving (\ref{NM2}) we used the fact that $|L_1^N(\gamma)  - L_1^N(rN^{\alpha})| \leq 1$, $|L_1^N(\Gamma)  - L_1^N(RN^{\alpha})| \leq 1$ and $p \in [0,1]$. Combining (\ref{NM1}) and (\ref{NM2}) we conclude that if $N \geq N_5 = \max(\tilde{N}_0, \tilde{N}_1, \tilde{N_2})$
	$$\mathbb{P}(B) \leq \mathbb{P}(F\cap B) + \mathbb{P}(F^c) \leq \epsilon/2 + \epsilon/4 < \epsilon,$$
	which proves (\ref{21Bbound}).\\
	
	\noindent\textbf{Step 4.} In this step we prove \eqref{21AEbound}. We define $\vec{x}\,',\vec{y}\,'\in\mathfrak{W}_{k-2}$ through
	\begin{equation}\label{21xybar}
		\begin{split}
			&x_i' = \overline{x} + (k-1-i)\lceil C\sqrt{T}\,\rceil, \quad y_i' = \overline{y} + (k-1-i)\lceil C\sqrt{T}\,\rceil \mbox{ for $i = 1, \dots, k-2$ with } \\
			&\overline{x} = \lceil  p\gamma  - \lambda \gamma^2 N^{-3\alpha/2} \hspace{-0.5mm}+\hspace{-0.5mm} [\phi_2(\epsilon/16) + 1]N^{\alpha/2} \rceil, \overline{y} = \lceil p \Gamma - \lambda \Gamma^2 N^{-3\alpha/2} \hspace{-0.5mm}+\hspace{-0.5mm} [\phi_2(\epsilon/16) + 1]N^{\alpha/2} \rceil,
		\end{split}
	\end{equation} 
	where $C$ is as in (\ref{21Cdef}) and $T = \Gamma - \gamma$. Note that for any $(\vec{x}, \vec{y}, \ell_{bot}) \in D$ we have
	$$x_i' \geq \overline{x} \geq x_1 \geq x_i \mbox{ and }y_i' \geq \overline{y} \geq y_1 \geq y_i$$
	for each $i = 1, \dots, k-2$. Furthermore, 
	$$x_i' - x_{i+1}' \geq C\sqrt{T} \mbox{ and }y_i' - y_{i+1}' \geq C\sqrt{T}$$
	for all $i = 1, \dots, k-2$ with the convention $x_{k-1}' = \overline{x}$ and $y_{k-1}' = \overline{y}$. \\
	
	We claim that we can find $\tilde{N}_0$ so that for all $N\geq\tilde{N}_0$ and $(\vec{x}, \vec{y}, \ell_{bot}) \in D$ such that $\mathbb{P}(E(\vec{x},\vec{y},\ell_{bot})) > 0$ we have $\prod_{i = 1}^{k-2} |\Omega(\gamma, \Gamma, x_i', y_i')| \geq |\Omega_{avoid}(\gamma, \Gamma, \vec{x}', \vec{y}', \infty, \ell_{bot})| \geq 1$ and moreover we have
	\begin{equation}\label{21 1/3}
		\mathbb{P}^{\gamma, \Gamma,\vec{x}',\vec{y}'}_{Ber} \left(  Q_1\left(mN^{\alpha}\right) - pm N^\alpha  + \lambda m^2 N^{\alpha/2} < -\phi_2(\epsilon/16)N^{\alpha/2}  \right) \geq 1/3,
	\end{equation}
	\begin{equation}\label{21 1/12}
		\mathbb{P}^{\gamma, \Gamma,\vec{x}',\vec{y}'}_{Ber} \left( Q_1 \geq \cdots \geq Q_{k-1} \right) \geq 11/12,
	\end{equation}
	where $\mathfrak{Q} = (Q_1, \dots, Q_{k-2})$ is $\mathbb{P}^{\gamma, \Gamma,\vec{x}',\vec{y}'}_{Ber}$-distributed and we used the convention that $Q_{k-1} = \ell_{bot}$. We prove (\ref{21 1/3}) and (\ref{21 1/12}) in the steps below. In this step we assume their validity and conclude the proof of (\ref{21AEbound}).\\
	
	Observe that for any $(\vec{x}, \vec{y}, \ell_{bot}) \in D$ such that $\mathbb{P}(E(\vec{x},\vec{y},\ell_{bot})) > 0$ we have the following tower of inequalities provided that $N \geq \tilde{N}_0$
	\begin{equation}\label{21xyest}
		\begin{split}
			&\mathbb{P}(A| E(\vec{x},\vec{y},\ell_{bot})) = \mathbb{P}^{\gamma, \Gamma,\vec{x},\vec{y},\infty,\ell_{bot}}_{avoid,Ber} \left( Q_1\left(mN^{\alpha}\right) - pm N^\alpha  + \lambda m^2 N^{\alpha/2} < -\phi_2(\epsilon/16)N^{\alpha/2} \right) \geq  \\
			& \mathbb{P}^{\gamma, \Gamma,\vec{x}',\vec{y}',\infty,\ell_{bot}}_{avoid,Ber} \left( Q_1\left(mN^{\alpha}\right) - pm N^\alpha  + \lambda m^2 N^{\alpha/2} < -\phi_2(\epsilon/16)N^{\alpha/2} \right) = \\
			&\frac{\mathbb{P}^{\gamma, \Gamma,\vec{x}',\vec{y}'}_{Ber} \left( \{ Q_1\left(mN^{\alpha}\right) - pm N^\alpha  + \lambda m^2 N^{\alpha/2} < -\phi_2(\epsilon/16)N^{\alpha/2} \} \cap \{Q_1 \geq \cdots \geq Q_{k-1} \} \right) }{\mathbb{P}^{\gamma, \Gamma,\vec{x}',\vec{y}'}_{Ber} \left( Q_1 \geq \cdots \geq Q_{k-1} \right) }.
		\end{split}
	\end{equation}
	Let us elaborate on (\ref{21xyest}) briefly. The condition that $\mathbb{P}(E(\vec{x},\vec{y},\ell_{bot})) > 0$ is required to ensure that the probabilities on the first line of (\ref{21xyest}) are well-defined and $N \geq \tilde{N}_0$ ensures that all other probabilities are also well-defined. The equality on the first line of (\ref{21xyest}) follows from the definition of $A$ and the Schur Gibbs property, see Definition \ref{DefSGP}, and $\mathfrak{Q} = (Q_1, \dots, Q_{k-2})$ is $\mathbb{P}^{\gamma, \Gamma,\vec{x},\vec{y},\infty,\ell_{bot}}_{avoid,Ber} $-distributed. The inequality in the first line of (\ref{21xyest}) follows from  Lemma \ref{MCLxy}, while the equality in the second line follows from Definition \ref{DefAvoidingLawBer}, and now $\mathfrak{Q} = (Q_1, \dots, Q_{k-2})$ is $\mathbb{P}^{\gamma, \Gamma,\vec{x}',\vec{y}'}_{Ber}$-distributed with the convention that $Q_{k-1} = \ell_{bot}$. 
	
	Combining (\ref{21 1/3}), (\ref{21 1/12}) and (\ref{21xyest}) we conclude that 
	$$\mathbb{P}(A| E(\vec{x},\vec{y},\ell_{bot}))  \geq 1/3- 1/12 = 1/4,$$
	which proves (\ref{21AEbound}).\\
	
	\noindent\textbf{Step 5.} In this step we prove \eqref{21 1/3}. We observe that since $\mathbb{P}(E(\vec{x},\vec{y},\ell_{bot})) > 0$ we know that $|\Omega_{avoid}(\gamma, \Gamma, \vec{x}, \vec{y}, \infty, \ell_{bot})| \geq 1$ and then we conclude from Lemma \ref{LemmaWD} that there exist $\hat{N}_1 \in \mathbb{N}$ such that for $N \geq \hat{N}_1$ we have $|\Omega_{avoid}(\gamma, \Gamma, \vec{x}', \vec{y}', \infty, \ell_{bot})| \geq 1$.
	
	Below $\ell$ will be used for a generic random variable with law $\mathbb{P}^{\cdot,\cdot,\cdot,\cdot}_{Ber}$, where the boundary data changes from line to line. With $\overline{x},\overline{y}$ as in \eqref{21xybar}, write $\overline{z} = \overline{y}-\overline{x}$ and recall that $T = \Gamma - \gamma$. Then
	\begin{equation} \label{21convex}
		\begin{split}
			& \mathbb{P}^{\gamma, \Gamma,x_1',y_1'}_{Ber} \left(\ell\left(mN^{\alpha}\right) - pm N^\alpha  + \lambda m^2 N^{\alpha/2} < -\phi_2(\epsilon/16)N^{\alpha/2}  \right) = \\
			& \mathbb{P}^{0,T,x_1',y_1'}_{Ber} \left(\ell(T/2) - pm N^\alpha  + \lambda m^2 N^{\alpha/2} < -\phi_2(\epsilon/16)N^{\alpha/2}\right) =\\
			& \mathbb{P}^{0,T,\overline{x},\overline{y}}_{Ber}\left(\ell(T/2) - pm N^\alpha  + \lambda m^2 N^{\alpha/2} < -\phi_2(\epsilon/16)N^{\alpha/2} - (k-2)\lceil C\sqrt{T}\rceil\right) \geq\\
			& \mathbb{P}^{0,T,\overline{x},\overline{y}}_{Ber}\left(\ell(T/2)   - \frac{\overline{x} + \overline{y}}{2}  <  \lambda \left(\frac{\gamma^2 + \Gamma^2}{2 N^{3\alpha/2}} \right)-[2\phi_2(\epsilon/16) + 1 +  \lambda m^2]N^{\alpha/2} - k \lceil C\sqrt{T} \rceil\right) = \\
			&\mathbb{P}^{0,T,0,\overline{z}}_{Ber}\left(\ell(T/2) - \overline{z}/2 < \lambda \left(\frac{\gamma^2 + \Gamma^2}{2 N^{3\alpha/2}} \right)-[2\phi_2(\epsilon/16) + 1 +  \lambda m^2]N^{\alpha/2} - k \lceil C\sqrt{T} \rceil\right).
		\end{split}
	\end{equation}
	The equalities in (\ref{21convex}) follow from shifting the boundary data of the curve $\ell$, while the inequality on the third line follows from the definition of $\overline{x},\overline{y}$ as in \eqref{21xybar}.
	
	From our choice of $R$ in Step 1 and the definition of $\gamma, \Gamma$ we know that 
	$$\lambda \frac{\gamma^2+\Gamma^2}{2N^{2\alpha}} -  \lambda m^2 \geq \lambda \frac{(R-r)^2}{4} - \frac{r\lambda}{N^{\alpha}} \geq 2\phi_2(\epsilon/16) + 2 + k \lceil C\sqrt{T} \rceil N^{-\alpha/2}- \frac{r\lambda}{N^{\alpha}} .$$
	$$ .$$
	The last inequality and (\ref{21convex}) imply 
	\begin{equation} \label{22convex}
		\begin{split}
			& \mathbb{P}^{\gamma, \Gamma,x_1',y_1'}_{Ber} \left(\ell\left(mN^{\alpha}\right) - pm N^\alpha  + \lambda m^2 N^{\alpha/2} < -\phi_2(\epsilon/16)N^{\alpha/2}  \right) \geq \\
			&\mathbb{P}^{0,T,0,\overline{z}}_{Ber}\left(\ell(T/2) - \overline{z}/2 < N^{\alpha/2} -  r\lambda N^{-\alpha/2} \right).
		\end{split}
	\end{equation}
	
	Let $\tilde{\mathbb{P}}$ be the probability measure on the space afforded by Theorem \ref{KMT}, supporting a random variable $\ell^{(T,\overline{z})}$ with law $\mathbb{P}^{0,T,0,\overline{z}}_{Ber}$ and a Brownian bridge $B^\sigma$ with diffusion parameter $\sigma = \sqrt{p(1-p)}$. Then the probability in the last line of \eqref{21convex} is equal to
	\begin{equation} \label{23convex}
		\begin{split}
			&\mathbb{P}^{0,T,0,\overline{z}}_{Ber}\left(\ell(T/2) - \overline{z}/2 < N^{\alpha/2} -  r\lambda N^{-\alpha/2} \right) = \tilde{\mathbb{P}} \left( \ell^{(T,\overline{z})}(T/2) - \overline{z}/2 <N^{\alpha/2} -  r\lambda N^{-\alpha/2}    \right) \geq \\
			& \tilde{\mathbb{P}} \left( \sqrt{T}B^\sigma_{1/2} < 0 \mbox{ and }\Delta(T,\overline{z}) <N^{\alpha/2} -  r\lambda N^{-\alpha/2}     \right) \geq \frac{1}{2} - \tilde{\mathbb{P}} \left( \Delta(T,\overline{z}) \geq N^{\alpha/2} -  r\lambda N^{-\alpha/2} \right),
		\end{split}
	\end{equation}
	where we recall that $\Delta(T,\overline{z})$ is as in \eqref{KMTeq}. Since as $N \rightarrow \infty$ we have 
	\begin{equation*}
		T \sim (R-r) N^{\alpha} \mbox{ and }\frac{|\overline{z} - pT|^2}{T} \sim (R+r),
	\end{equation*}
	we conclude from Corollary \ref{Cheb} that there exists $\hat{N}_2 \in \mathbb{N}$ such that if $N \geq \max(\hat{N}_1, \hat{N}_2)$ we have
	\begin{equation} \label{24convex}
		\begin{split}
			\tilde{\mathbb{P}} \left( \Delta(T,\overline{z}) \geq N^{\alpha/2} -  r\lambda N^{-\alpha/2} \right) \leq \frac{1}{6}.
		\end{split}
	\end{equation}
	Combining (\ref{22convex}), (\ref{23convex}) and (\ref{24convex}) we obtain \eqref{21 1/3}. \\

	\noindent\textbf{Step 6.} In this last step, we prove \eqref{21 1/12}. Let $\overline{\ell}_{bot}$ be the straight segment connecting $\overline{x}$ and $\overline{y}$, defined in \eqref{21xybar}. By construction, we have that there is $\hat{N}_3 \in \mathbb{N}$ such that if $N \geq \hat{N}_3$ we have for any $(\vec{x}, \vec{y}, \ell_{bot}) \in D$ that $\ell_{bot}$ lies uniformly below the line segment $\overline{\ell}_{bot}$, which in turn lies at least $C \sqrt{T}$ below the straight segment connecting $x_{k-2}'$ and $y_{k-2}'$. If $\hat{N}_1$ is as in Step 5 we conclude from Lemma \ref{CurveSeparation} that there exists $\hat{N}_4 \in \mathbb{N}$ such that if $N \geq \max(\hat{N}_1, \hat{N}_3, \hat{N}_4)$ and $\mathbb{P}(E(\vec{x},\vec{y},\ell_{bot})) > 0$ 
	\begin{equation}\label{21UP}
		\mathbb{P}^{\gamma, \Gamma,\vec{x}',\vec{y}'}_{Ber} \left( Q_1 \geq \cdots \geq Q_{k-1} \right) \geq \left(1-3e^{-C^2/8p(1-p)}\right)^{k-2} = \frac{11}{12}.
	\end{equation}
	where the condition that $N \geq \hat{N}_1 $ is included to ensure that the probability $\mathbb{P}^{\gamma, \Gamma,\vec{x}',\vec{y}'}_{Ber} $ is well-defined. In deriving (\ref{21UP}) we also used (\ref{21Cdef}), which implies
	$$C = \sqrt{ 8p(1-p) \log\frac{3}{1-(11/12)^{1/(k-2)}}} \geq \sqrt{8p(1-p) \log 3}.$$
	We see that (\ref{21UP}) implies (\ref{21 1/12}), which concludes the proof of the lemma.
\end{proof}

In the remainder of this section we use Lemma \ref{21} to prove Lemma \ref{PropSup2}.
\begin{proof} (of Lemma \ref{PropSup2}) For clarity we split the proof into five steps.\\ 
	
	\noindent\textbf{Step 1.} In this step we specify the choice of $R_2$ in the statement of the lemma and introduce some notation that will be used in the proof of the lemma, which is given in Steps 2-5 below. Throughout we fix $r, \epsilon > 0$. Define the constant
	\begin{equation}\label{4.3Cdef}
		C_1 = \sqrt{16p(1-p)\log\frac{3}{1-2^{-1/(k-1)}}}.
	\end{equation}
	Let $R > r + 3$, $M > 0$ and $\tilde{N}_1 \in \mathbb{N}$ be such that for $N \geq \tilde{N}_1$ we have that the event
	\begin{equation}\label{S5EventB}
\begin{split}
		B = &\left\{ \sup_{x\in [r+3, R] } \big(L^N_{k-1}(xN^\alpha) - pxN^\alpha\big) \geq -MN^{\alpha/2} \right\} \cap  \\
&\left\{ \sup_{x\in [-R,-r-3]} \big(L^N_{k-1}(xN^\alpha) - pxN^\alpha\big) \geq -MN^{\alpha/2} \right\}
\end{split}
	\end{equation}
	satisfies 
	\begin{equation}\label{S5EventBineq}
		\mathbb{P} \left( B\right) \geq 1 - \epsilon/2.
	\end{equation}
	Such a choice of $R, M, \tilde{N}_1$ is possible by Lemma \ref{21}.
	
	Let us set 
	$$s^-_1 = \lceil - R \cdot N^{\alpha} \rceil, \hspace{2mm} s^-_2 =  \lfloor -(r+3) \cdot N^{\alpha} \rfloor, \hspace{2mm} s^+_1 = \lceil (r+3) \cdot N^{\alpha} \rceil, \hspace{2mm} s^+_2 =  \lfloor R \cdot N^{\alpha} \rfloor,$$
	and for $a \in \llbracket s_1^-, s_2^- \rrbracket$ and $b \in \llbracket s_1^+, s_2^+ \rrbracket$ we define $\vec{x}\,',\vec{y}\,' \in \mathfrak{W}_{k-1}$ by
	\begin{equation}\label{S5Boundary}
		\begin{split}
			x_i' &= \lfloor pa - MN^{\alpha/2}\rfloor - (i-1)\lceil C_1(2R)^{1/2}N^{\alpha/2}\rceil,\\
			y_i' &= \lfloor pb - MN^{\alpha/2}\rfloor - (i-1)\lceil C_1(2R)^{1/2} N^{\alpha/2}\rceil,
		\end{split}
	\end{equation}  
	for $i = 1, \dots, k-1$. We will write $\vec{z} = \vec{y}' - \vec{x}'$, and we note that $z_{k-1} \geq p (b-a) -1$ and also $2RN^{\alpha} \geq b-a \geq 2(r+3) N^{\alpha}$. The latter and Lemma \ref{LemmaMinFreeS4} imply that there exists $R_2 > 0$ and $\tilde{N}_2 \in \mathbb{N}$ such that if $N \geq \tilde{N}_2$ we have
	\begin{equation}\label{S5DefR2}
		\mathbb{P}^{0, b - a, 0, z_{k-1}}_{Ber} \left( \inf_{s\in[0,b-a]} \big(\ell(s) - ps\big) \leq -(R_2 - M - C_1 (2R)^{1/2} k) N^{\alpha/2} \right) < \epsilon/4.
	\end{equation}
	This fixes our choice of $R_2$ in the statement of the lemma. \\
	
	With the above choice of $R_2$ we define the event
	\begin{equation}\label{S5EventA}
		A = \left\{\inf_{s \in [ -t_3, t_3 ]}\big[L^N_{k-1}(s) - p s \big] \leq - R_2N^{\alpha/2}\right\},
	\end{equation}
	and then to prove the lemma it suffices to show that there exists $N_4 \in \mathbb{N}$ such that for $N \geq N_4$ 
	\begin{equation}\label{4.3Abound}
		\mathbb{P}(A) < \epsilon
	\end{equation}

	\noindent\textbf{Step 2.} In this step, we prove that the event $B$ from (\ref{S5EventB}) can be written as a \textit{countable disjoint union} of the form
	\begin{equation}\label{4.3disj}
		B = \bigsqcup_{(a,b,\vec{x},\vec{y},\ell_{bot}, \ell_{top}^-, \ell_{top}^+)\in D} E(a,b,\vec{x},\vec{y},\ell_{bot},\ell_{top}^-, \ell_{top}^+),
	\end{equation}
	where the set $D$ and events $E(a,b,\vec{x},\vec{y},\ell_{bot},\ell_{top}^-, \ell_{top}^+)$ are defined below.
	
	For $a \in \llbracket s_1^-, s_2^- \rrbracket$ and $b \in \llbracket s_1^+, s_2^+ \rrbracket$, $\vec{x},\vec{y}\in\mathfrak{W}_{k-1}$, $z_1,z_2,z_1^-, z_2^+\in\mathbb{Z}$, $\ell_{bot} \in \Omega(a, b, z_1, z_2)$, $\ell_{top}^- \in \Omega(s_1^-, a, z_1^-, x_{k-1})$ , $\ell_{top}^+ \in \Omega(b, s_2^+, y_{k-1}, z_2^+)$ we define $E(a,b,\vec{x},\vec{y},\ell_{bot}, \ell_{top}^-, \ell_{top}^+)$ to be the event that $L_i^N(a) = x_i$ and $L_i^N(b) = y_i$ for $1\leq i\leq k-1$, and $L_k^N$ agrees with $\ell_{bot}$ on $\llbracket a,b \rrbracket$, $L_{k-1}^N$ agrees with $\ell_{top}^-$ on $\llbracket s_1^-, a \rrbracket$ and with $\ell_{top}^+$ on $\llbracket b, s_2^+ \rrbracket$.

	We also let $D$ be the collection of tuples $(a,b,\vec{x},\vec{y},\ell_{bot}, \ell_{top}^-, \ell_{top}^+)$ satisfying:
	\begin{enumerate}[label=(\arabic*)]
		\item $a \in \llbracket s_1^-, s_2^- \rrbracket$, $b \in \llbracket s_1^+, s_2^+ \rrbracket$;
		\item $\vec{x}, \vec{y} \in \mathfrak{W}_{k-1}$, $0 \leq y_i - x_i \leq b- a$, $x_{k-1} - pa > - MN^{\alpha/2}$, and $y_{k-1} - pb > - MN^{\alpha/2}$;
		\item if $c \in \llbracket s_1^-, s_2^- \rrbracket \cap (- \infty, a)$ then $\ell_{top}^-(c) - pc \geq -MN^{\alpha/2}$;
		\item if $d \in \llbracket s_1^+, s_2^+ \rrbracket \cap (b, \infty)$ then $\ell_{top}^+(d) - pd\geq -MN^{\alpha/2}$;
		\item $z_1\leq x_{k-1}$, $z_2\leq y_{k-1}$, and $\ell_{bot}\in\Omega(a,b,z_1,z_2)$.
	\end{enumerate} 
	It is clear that $D$ is countable, and that 
	$$B = \bigcup_{(a,b,\vec{x},\vec{y},\ell_{bot})\in D} E(a,b,\vec{x},\vec{y},\ell_{bot}, \ell_{top}^-, \ell_{top}^+),$$
	so to prove (\ref{4.3disj}) it suffices to show that the events $E(a,b,\vec{x},\vec{y},\ell_{bot}, \ell_{top}^-, \ell_{top}^+)$ are pairwise disjoint. Observe that on the intersection of $E(a,b,\vec{x},\vec{y},\ell_{bot}, \ell_{top}^-, \ell_{top}^+)$ and $E(\tilde a,\tilde b,\tilde{\vec{x}},\tilde{\vec{y}},\tilde{\ell}_{bot}, \tilde{\ell}_{top}^-, \tilde{\ell}_{top}^+)$, conditions (2) and (3) imply that $a = \tilde{a}$, while conditions (2) and (4) that $b = \tilde{b}$. Afterwards, we conclude that $\vec{x}=\tilde{\vec{x}}$, $\vec{y} = \tilde{\vec{y}}$, $\ell_{bot} = \tilde{\ell}_{bot}$, $\ell_{top}^- = \tilde{\ell}^-_{top}$ and $\ell_{top}^+ = \tilde{\ell}^+_{top}$, confirming \eqref{4.3disj}.\\

	\noindent\textbf{Step 3.} In this step we prove (\ref{4.3Abound}). We claim that we can find $\tilde{N}_3 \in \mathbb{N}$ such that if $N \geq \tilde{N}_3$ and $(a,b,\vec{x},\vec{y},\ell_{bot}, \ell_{top}^-, \ell_{top}^+) \in D$ is such that $\mathbb{P} \left( E(a,b,\vec{x},\vec{y},\ell_{bot}, \ell_{top}^-, \ell_{top}^+) \right) > 0$ we have
	\begin{equation}\label{4.3AEbound}
		\mathbb{P}(A\,|\,E(a,b,\vec{x},\vec{y},\ell_{bot}, \ell_{top}^-, \ell_{top}^+)) < \epsilon/2.
	\end{equation}
	We will prove (\ref{4.3AEbound}) in the steps below. Here we assume its validity and conclude the proof of (\ref{4.3Abound}).\\
	
	If $N \geq \max ( \tilde{N}_1, \tilde{N}_2, \tilde{N}_3)$ we have in view of \eqref{4.3disj} and \eqref{4.3AEbound} that 
	\begin{equation*}
		\begin{split}
			&\mathbb{P}(A) \leq \mathbb{P}(A \cap B) + \mathbb{P}(B^c) =\mathbb{P}(B^c) +  \sum_{\substack{(a,b,\vec{x},\vec{y},\ell_{bot}, \ell_{top}^-, \ell_{top}^+)\in D\\ \mathbb{P} \left( E(a,b,\vec{x},\vec{y},\ell_{bot}, \ell_{top}^-, \ell_{top}^+) \right) > 0}}  \mathbb{P}(A|E(a,b,\vec{x},\vec{y},\ell_{bot}, \ell_{top}^-, \ell_{top}^+)) \times \\
			&\mathbb{P}(E(a,b,\vec{x},\vec{y},\ell_{bot}, \ell_{top}^-, \ell_{top}^+)) \leq \mathbb{P}(B^c) + \frac{\epsilon}{2}\sum_{\substack{(a,b,\vec{x},\vec{y},\ell_{bot}, \ell_{top}^-, \ell_{top}^+)\in D\\ \mathbb{P} \left( E(a,b,\vec{x},\vec{y},\ell_{bot}, \ell_{top}^-, \ell_{top}^+) \right) > 0}}  \mathbb{P}(E(a,b,\vec{x},\vec{y},\ell_{bot}, \ell_{top}^-, \ell_{top}^+)) =\\
			&\mathbb{P}(B^c) + \frac{\epsilon}{2}\cdot  \mathbb{P}(B) < \epsilon,
		\end{split}
	\end{equation*}
	where in the last inequality we used (\ref{S5EventBineq}). The above inequality clearly implies (\ref{4.3Abound}).\\
	
	\noindent\textbf{Step 4.} In this step we prove \eqref{4.3AEbound}. We claim that there exists $\tilde{N}_4 \in \mathbb{N}$ such that if $N \geq \tilde{N}_4$, $a \in \llbracket s_1^-, s_2^- \rrbracket$ and $b \in \llbracket s_1^+, s_2^+ \rrbracket$ we have that $\prod_{i = 1}^{k-1} |\Omega(a, b,x_i', y_i')| \geq 1$ and
	\begin{equation}\label{S5T1}
		\mathbb{P}^{a, b, \vec{x}', \vec{y}'}_{Ber}(Q_1 \geq \cdots \geq Q_{k-1}) \geq \frac{1}{2},
	\end{equation}
	where $\mathfrak{Q} = (Q_1, \dots, Q_{k-1})$ is $\mathbb{P}^{a, b, \vec{x}', \vec{y}'}_{Ber}$-distributed and we recall that $\vec{x}'$, $\vec{y}'$ were defined in (\ref{S5Boundary}).
	We will prove (\ref{S5T1}) in Step 5 below. Here we assume its validity and conclude the proof of (\ref{4.3AEbound}).\\

	Observe that by condition (2) in Step 2, we have that $x_i'\leq pa - MN^{\alpha/2} \leq x_{k-1} \leq x_i$, and similarly $y_i' \leq pb - MN^{\alpha/2} \leq y_{k-1} \leq y_i$ for $i = 1, \dots, k-1$. From this observation we conclude that if $N \geq \tilde{N}_4$ is sufficiently large and $(a,b,\vec{x},\vec{y},\ell_{bot}, \ell_{top}^-, \ell_{top}^+) \in D$ is such that $\mathbb{P} \left( E(a,b,\vec{x},\vec{y},\ell_{bot}, \ell_{top}^-, \ell_{top}^+) \right) > 0$ we have
	\begin{equation}\label{4.3main}
		\begin{split}
			&\mathbb{P}(A|E(a,b,\vec{x},\vec{y},\ell_{bot},\ell_{top}^-, \ell_{top}^+)) \leq \\
			&\mathbb{P}\left( \inf_{s\in[a, b]} \left(L_{k-1}^N(s) - ps\right) \leq -R_2N^{\alpha/2}\, \big| \, E(a,b,\vec{x},\vec{y},\ell_{bot},\ell_{top}^-, \ell_{top}^+) \right) = \\
			&\mathbb{P}^{a,b, \vec{x}, \vec{y},\infty,\ell_{bot}}_{avoid, Ber} \left( \inf_{s\in[a, b]} \left(Q_{k-1}(s) - ps\right) \leq -R_2N^{\alpha/2} \right) \leq\\
			& \mathbb{P}^{a, b, \vec{x}', \vec{y}'}_{avoid, Ber} \left( \inf_{s\in[a,b]} \big(Q_{k-1}(s) - ps\big) \leq -R_2N^{\alpha/2} \right) = \\
			&\frac{\mathbb{P}^{a, b, \vec{x}', \vec{y}'}_{Ber} \left( \{ \inf_{s\in[a,b]} \big(Q_{k-1}(s) - ps\big) \leq -R_2 N^{\alpha/2} \}  \cap \{ Q_1 \geq \cdots \geq Q_{k-1} \} \right)}{\mathbb{P}^{a, b, \vec{x}\,', \vec{y}\,'}_{Ber}(Q_1 \geq \cdots \geq Q_{k-1})} \leq \\
			&\frac{\mathbb{P}^{a, b, \vec{x}', \vec{y}'}_{Ber} \left( \inf_{s\in[a,b]} \big(Q_{k-1}(s) - ps\big) \leq -R_2 N^{\alpha/2} \right)}{\mathbb{P}^{a, b, \vec{x}\,', \vec{y}\,'}_{Ber}(Q_1 \geq \cdots \geq Q_{k-1})}.
		\end{split}
	\end{equation}
	Let us elaborate on (\ref{4.3main}) briefly. The first inequality in (\ref{4.3main}) follows from the definition of $A$ and the fact that $a \leq -t_3$ while $b \geq t_3$ by construction. The condition $\mathbb{P} \left( E(a,b,\vec{x},\vec{y},\ell_{bot}, \ell_{top}^-, \ell_{top}^+) \right) > 0$ ensures that the first three probabilities in (\ref{4.3main}) are all well-defined. The equality on the second line follows from the Schur Gibbs property and the inequality on the third line follows from Lemmas \ref{MCLxy} and \ref{MCLfg} since $x_i' \leq x_i$ and $y_i' \leq y_i$ by construction. To ensure that the probability in the fourth line is well-defined (and hence Lemmas \ref{MCLxy} and \ref{MCLfg} are applicable) it suffices to assume that $N \geq \tilde{N}_4$, in view of Lemma \ref{LemmaWD}. The equality on the fourth line follows from the definition of $\mathbb{P}^{a, b, \vec{x}', \vec{y}'}_{avoid, Ber} $, see Definition \ref{DefAvoidingLawBer} and the last inequality is trivial.\\
	
	By our choice of $R_2$, see (\ref{S5DefR2}), we know that there is $\tilde{N}_5 \in \mathbb{N}$ such that if $N \geq \tilde{N}_5$
	\begin{equation}\label{4.3main2}
		\begin{split}
			&\mathbb{P}^{a, b, \vec{x}', \vec{y}'}_{Ber} \left( \inf_{s\in[a,b]} \big(Q_{k-1}(s) - ps\big) \leq -R_2 N^{\alpha/2} \right) = \\
			&\mathbb{P}^{0, b-a, 0, z_{k-1}}_{Ber} \left( \inf_{s\in[0,b-a]} \big(\ell(s) - ps\big) \leq -R_2 N^{\alpha/2} - x_{k-1}'\right) \leq \\
			&\mathbb{P}^{0, b-a, 0, z_{k-1}}_{Ber} \left( \inf_{s\in[0,b-a]} \big(\ell(s) - ps\big) \leq -(R_2 - M - C_1(2R)^{1/2} k) N^{\alpha/2} \right) < \epsilon/4.
		\end{split}
	\end{equation}
	Combining (\ref{S5T1}), (\ref{4.3main}) and (\ref{4.3main}) we conclude that for $N \geq \tilde{N}_3 = \max(\tilde{N}_4, \tilde{N}_5)$ we have
	$$\mathbb{P}(A|E(a,b,\vec{x},\vec{y},\ell_{bot},\ell_{top}^-, \ell_{top}^+)) < 2 \cdot \epsilon/4 = \epsilon/2,$$
	which implies (\ref{4.3AEbound}).\\
	
	\noindent\textbf{Step 5.} In this final step we prove (\ref{S5T1}). Set $T = b -a $ and note that by our assumption that $a \in \llbracket  s_1^-, s_2^- \rrbracket$ and $b \in \llbracket s_1^+, s_2^+\rrbracket$ we know that $(2r+6) N^{\alpha} \leq T \leq 2R N^{\alpha}$. This implies that $1 + C_1 (2R)^{1/2} N^{\alpha/2} \geq x_i' - x_{i+1}' \geq C_1 \sqrt{T}$ and likewise for $y_i'$. It follows from Lemma \ref{CurveSeparation}, applied with $\ell_{bot} = -\infty$ that there is $\tilde{N_4} \in \mathbb{N}$ such that if $N \geq \tilde{N}_4$ we have $T \geq y_i' - x_i' \geq 0$ for all $i$ so that $\prod_{i = 1}^{k-1} |\Omega(a, b,x_i', y_i')| \geq 1$ and moreover
\begin{equation}\label{4.3avoid}
\begin{split}
& \mathbb{P}^{a, b, \vec{x}', \vec{y}'}_{Ber}(Q_1 \geq \cdots \geq Q_{k-1}) = \mathbb{P}^{0, b-a, \vec{x}', \vec{y}'}_{Ber}(Q_1 \geq \cdots \geq Q_{k-1}) \geq\\ & \left(1 - 3e^{-C_1^2/8p(1-p)}\right)^{k-1} \geq 1/2
\end{split}
\end{equation}
 In deriving (\ref{4.3avoid}) we also used (\ref{4.3Cdef}), which implies
$$C_1 = \sqrt{16p(1-p)\log\frac{3}{1-2^{-1/(k-1)}}} \geq \sqrt{8p(1-p) \log 3}. $$
Equation (\ref{4.3avoid}) clearly implies (\ref{S5T1}) and this concludes the proof of the lemma.
\end{proof}

%
\section{Lower bounds on the acceptance probability}\label{Section6} We prove Lemma \ref{LemmaAP1} in Section \ref{sect61} by using Lemma \ref{LemmaAP2}, whose proof is presented in Section \ref{sect62}.

%
\subsection{Proof of Lemma \ref{LemmaAP1}}\label{sect61} Throughout this section we assume the same notation as in Lemma \ref{LemmaAP1}, i.e., we assume that we have fixed $k \in \mathbb{N}$, $p \in (0,1)$, $M_1, M_2 > 0$, $\ell_{bot}: \llbracket -t_3, t_3 \rrbracket \rightarrow \mathbb{R} \cup \{ - \infty \}$, and $\vec{x}, \vec{y} \in \mathfrak{W}_{k-1}$ such that $|\Omega_{avoid}(-t_3, t_3, \vec{x}, \vec{y}, \infty, \ell_{bot})| \geq 1$. We also assume that
\begin{enumerate}
	\item $\sup_{s \in [- t_3,t_3]}\big[\ell_{bot}(s)  - ps \big]  \leq M_2 (2t_3)^{1/2}$,
	\item  $-pt_3 + M_1 (2t_3)^{1/2} \geq  x_1 \geq  x_{k-1} \geq \max\left(\ell_{bot}(-t_3), -pt_3 - M_1 (2t_3)^{1/2}\right),$
	\item $pt_3 + M_1 (2t_3)^{1/2} \geq y_1 \geq y_{k-1} \geq  \max \left( \ell_{bot}(t_3),  p t_3 - M_1(2t_3)^{1/2} \right).$
\end{enumerate}

\begin{definition}\label{TildeDef}
	We write $S = \llbracket -t_3,-t_1\rrbracket\cup \llbracket t_1,t_3\rrbracket$, and we denote by $\mathfrak{Q} = (Q_1,\dots,Q_{k-1})$ and $\tilde{\mathfrak{Q}} = (\tilde{Q}_1, \dots, \tilde{Q}_{k-1})$ the $\llbracket 1, k-1 \rrbracket$-indexed line ensembles which are uniformly distributed on $\Omega_{avoid}(-t_3,t_3,\vec{x},\vec{y},\infty, \ell_{bot})$ and $\Omega_{avoid}(-t_3, t_3, \vec{x}, \vec{y}, \infty, \ell_{bot};S)$ respectively. We let $\mathbb{P}_{\mathfrak{Q}}$ and $\mathbb{P}_{\tilde{\mathfrak{Q}}}$ denote these uniform measures.
\end{definition}
In other words, $\tilde{\mathfrak{Q}}$ has the law of $k-1$ independent Bernoulli bridges that have been conditioned on not-crossing each other on the set $S$ and also staying above the graph of $\ell_{bot}$ but only on the intervals $\llbracket-t_3, -t_1\rrbracket$ and $\llbracket t_1, t_3\rrbracket$. The latter restriction means that the lines are allowed to cross on $\llbracket -t_1+1,t_1-1\rrbracket$, and  $\tilde{Q}_{k-1}$ is allowed to dip below $\ell_{bot}$ on $\llbracket -t_1+1,t_1-1\rrbracket$ as well.

\begin{lemma}\label{LemmaAP2} There exists $N_5 \in \mathbb{N}$ and constants $g,h > 0$ such that for $N \geq N_5$ we have
	\begin{equation}\label{eqn57}
		\mathbb{P}_{\tilde{\mathfrak{Q}}} \left( Z\big(  -t_1, t_1, \tilde{\mathfrak{Q}}(-t_1) , \tilde{\mathfrak{Q}}(t_1), \ell_{bot}\llbracket -t_1, t_1\rrbracket\big)\geq g    \right) \geq h.
	\end{equation}
\end{lemma}
We will prove Lemma \ref{LemmaAP2} in Section \ref{sect62}. In the remainder of this section, we give the proof of Lemma \ref{LemmaAP1} using Lemma \ref{LemmaAP2}. The proof begins by evaluating the Radon-Nidokym derivative between $\pr_{\mathfrak{Q}'}$ and $\pr_{\tilde{\mathfrak{Q}}'}$. We then use this Radon-Nikodym derivative to transition between $\tilde{\mathfrak{Q}}$ in Lemma \ref{LemmaAP2} which ignores $\ell_{bot}$ on $\llbracket -(t_1-1),t_1-1\rrbracket$ and $\mathfrak{Q}$ in Lemma \ref{LemmaAP1} which avoids $\ell_{bot}$ everywhere. 

\begin{proof}[Proof of Lemma \ref{LemmaAP1}]
	
	Let us denote by $\pr_{\mathfrak{Q}'}$ and $\pr_{\tilde{\mathfrak{Q}}'}$ the measures on $\llbracket 1, k-1\rrbracket$-indexed Bernoulli line ensembles $\mathfrak{Q}'$, $\tilde{\mathfrak{Q}}'$ on the set $S$ in Definition \ref{TildeDef} induced by the restrictions of the measures $\mathbb{P}_{\mathfrak{Q}}$, $\mathbb{P}_{\tilde{\mathfrak{Q}}}$ to $S$. Also let us write $\Omega_a(\cdot)$ for $\Omega_{avoid}(\cdot)$ for simplicity, and denote by $\Omega_a(S)$ the set of elements of $\Omega_{avoid}(-t_3,t_3,\tilde{\mathfrak{Q}}(-t_3),\tilde{\mathfrak{Q}}(t_3))$ restricted to $S$. We claim the Radon-Nikodym derivative between these two restricted measures on elements $\mathfrak B = (B_1, \dots, B_{k-1})$ of $\Omega_a(S)$ is given by 
\begin{equation}\label{propRadon}
\frac{d\pr_{\mathfrak{Q}'}}{d\pr_{\tilde{\mathfrak{Q}}'}}\left(\mathfrak B\right) = \frac{\mathbb{P}_{\mathfrak{Q}'}(\mathfrak{B})}{\mathbb{P}_{\tilde{\mathfrak{Q}}'}(\mathfrak{B})} = (Z')^{-1} Z\left(-t_1,t_1,\mathfrak{B}\left(-t_1\right),\mathfrak{B}\left(t_1\right),\ell_{bot}\llbracket -t_1,t_1\rrbracket\right),
\end{equation}
with $Z' = \ex_{\tilde{\mathfrak{Q}}'}\left[Z\left(-t_1, t_1, \mathfrak B(-t_1), \mathfrak B(t_1), \ell_{bot}\llbracket -t_1, t_1\rrbracket\right)\right]$. The first equality holds simply because the measures are discrete. To prove the second equality, observe that
	\begin{equation}
		\begin{split}
			\pr_{\mathfrak{Q}'}(\mathfrak{B}) &= \frac{|\Omega_a(-t_1,t_1,\mathfrak{B}(-t_1),\mathfrak{B}(t_1),\ell_{bot}\llbracket -t_1,t_1\rrbracket)|}{|\Omega_a(-t_3,t_3,{\mathfrak{Q}}(-t_3),{\mathfrak{Q}}(t_3),\ell_{bot})|}, 
			\\
			\pr_{\tilde{\mathfrak{Q}}'}(\mathfrak{B}) &= \frac{\prod_{i = 1}^{k-1}|\Omega(-t_1,t_1,B_i(-t_1),B_i(t_1))|}{|\Omega_a(-t_3,t_3,\tilde{\mathfrak{Q}}(-t_3),\tilde{\mathfrak{Q}}(t_3),\ell_{bot};S)|}
		\end{split}
	\end{equation}
	These identities follow from the restriction, and the fact that the measures are uniform. Then from Definition \ref{DefAP} we know
	\[
	Z(-t_1,t_1,\mathfrak{B}(-t_1),\mathfrak{B}(t_1),\ell_{bot}) = \frac{|\Omega_a(-t_1,t_1,\mathfrak{B}(-t_1),\mathfrak{B}(t_1),\ell_{bot}\llbracket-t_1,t_1\rrbracket)|}{\prod_{i = 1}^{k-1}|\Omega(-t_1,t_1, B_i(-t_1),B_i(t_1))|}
	\]
	and hence
	\begin{equation*}
		\begin{split}
			Z' =&\sum_{\mathfrak{B}\in \Omega_a(S)}\frac{\prod_{i = 1}^{k-1}|\Omega(-t_1,t_1,B_i(-t_1),B_i(t_1))|}{|\Omega_a(-t_3,t_3,\tilde{\mathfrak{Q}}(-t_3),\tilde{\mathfrak{Q}}(t_3),\ell_{bot};S)|}\cdot\frac{|\Omega_a(-t_1,t_1,\mathfrak B(-t_1),\mathfrak{B}(t_1),\ell_{bot})|}{\prod_{i = 1}^{k-1}|\Omega(-t_1,t_1, B_i(-t_1),B_i(t_1))|}=\\ 
			&\frac{\sum_{\mathfrak{B}\in\Omega_a(S)}|\Omega_a(-t_1,t_1,\mathfrak B(-t_1),\mathfrak B(t_1),\ell_{bot})|}{|\Omega_a(-t_3,t_3,\tilde{\mathfrak{Q}}(-t_3),\tilde{\mathfrak{Q}}(t_3),\ell_{bot};S)|} = \frac{|\Omega_a(-t_3,t_3,{\mathfrak{Q}}(-t_3),{\mathfrak{Q}}(t_3),\ell_{bot})|}{|\Omega_a(-t_3,t_3,\tilde{\mathfrak{Q}}(-t_3),\tilde{\mathfrak{Q}}(t_3),\ell_{bot};S)|}.
		\end{split}
	\end{equation*}
	Comparing the above identities proves the second equality in \eqref{propRadon}.
	
	Now note that $Z\left(-t_1, t_1, \mathfrak B(-t_1), \mathfrak B(t_1), \ell_{bot}\llbracket -t_1, t_1\rrbracket\right)$ is a deterministic function of $\left((\mathfrak B(-t_1), \mathfrak B(t_1)\right)$. In fact, the law of $\left((\mathfrak B(-t_1), \mathfrak B(t_1)\right)$ under $\pr_{\tilde{\mathfrak{Q}}'}$ is the same as that of $\big(\tilde{\mathfrak{Q}}(-t_1), \tilde{\mathfrak{Q}}(t_1)\big)$ by way of the restriction. It follows from Lemma \ref{LemmaAP2} that
	\begin{align*}
		Z' &= \ex_{\tilde{\mathfrak{Q}}'}\left[Z\left(-t_1, t_1, \mathfrak B(-t_1), \mathfrak B(t_1), \ell_{bot}\llbracket -t_1, t_1\rrbracket\right)\right]\\
		&= \ex_{\tilde{\mathfrak{Q}}}\left[Z\left(-t_1, t_1, \mathfrak{Q}(-t_1), \mathfrak{Q}(t_1), \ell_{bot}\llbracket -t_1, t_1\rrbracket\right)\right]\geq gh,
	\end{align*}	
	which gives us 
	\begin{equation}
		\label{Zineq} (Z')^{-1}\leq \frac{1}{gh}.
	\end{equation}
	Similarly,  the law of $\left(\mathfrak B(-t_1), \mathfrak{B}(t_1)\right)$ under $\pr_{\mathfrak{Q}'}$ is the same as that of $\left(\mathfrak{Q}(-t_1), \mathfrak{Q}(t_1)\right)$ under $\pr_{\mathfrak{Q}}$. Hence
	\begin{equation}\label{LBswap}
		\begin{split}
			&\pr_{\mathfrak{Q}}\Big(Z(-t_1, t_1, \mathfrak{Q}(-t_1), \mathfrak{Q}(t_1), \ell_{bot}\llbracket -t_1, t_1\rrbracket)\leq gh\tilde \epsilon\Big)=\\
			&\qquad\pr_{\mathfrak{Q}'}\Big(Z\left(-t_1, t_1, \mathfrak B(-t_1), \mathfrak B(t_1), \ell_{bot}\llbracket -t_1, t_1\rrbracket\right)\leq gh\tilde \epsilon\Big).
		\end{split}
	\end{equation}
	Now let us write $E=\{Z\left(-t_1, t_1, \mathfrak B(-t_1), \mathfrak B(t_1), \ell_{bot}\llbracket -t_1, t_1\rrbracket\right)\leq gh\tilde\epsilon\}\subset \Omega_a(S)$. Then according to \eqref{propRadon}, we have
	\[
	\pr_{\mathfrak{Q}'}(E)=\int_{\Omega_a(S)} \indic_E\, d\pr_{\mathfrak{Q}'} = (Z')^{-1}\int_{\Omega_a(S)}\indic_E \cdot\, Z\left(-t_1, t_1, \mathfrak B(-t_1), \mathfrak B(t_1), \ell_{bot}\llbracket -t_1, t_1\rrbracket\right)\, d\pr_{\tilde{\mathfrak{Q}}'}(\mathfrak{B}).
	\]
	From the definition of $E$, the inequality \eqref{Zineq}, and the fact that $\mathbf{1}_E \leq 1$, it follows that
	\[
	\pr_{\mathfrak{Q}'}(E)\leq (Z')^{-1}\int_{\Omega_a(S)} \indic_E\cdot\, gh\tilde{\epsilon}\, d\pr_{\tilde{\mathfrak{Q}}'} \leq \frac{1}{gh}\int_{\Omega_a(S)} gh\tilde\epsilon\, d\pr_{\tilde{\mathfrak{Q}}'}\leq \tilde{\epsilon}.
	\]
	In combination with \eqref{LBswap}, this proves \eqref{eqn60} with $\tilde{h} = gh$.
	
\end{proof}

%
\subsection{Proof of Lemma \ref{LemmaAP2}} \label{sect62} In this section, we prove Lemma \ref{LemmaAP2}. We first state and prove two auxiliary lemmas necessary for the proof. The first lemma establishes a set of conditions under which we have the desired lower bound on the acceptance probability. 

\begin{lemma}\label{LemmaBP1} Let $\epsilon > 0$ and $V^{top} > 0$ be given such that $V^{top} > M_2 + 6 (k-1) \epsilon$. Suppose further that $\vec{a}, \vec{b} \in \mathfrak{W}_{k-1}$ are such that 
\begin{enumerate}
\item $V^{top} (2t_3)^{1/2} \geq a_1 + p t_1 \geq a_{k-1} + pt_1 \geq (M_2 + 2 \epsilon) (2t_3)^{1/2}$;
\item $V^{top} (2t_3)^{1/2} \geq b_1 - p t_1 \geq b_{k-1} - pt_1 \geq (M_2 + 2 \epsilon) (2t_3)^{1/2}$; 
\item $a_i -a_{i+1} \geq 3\epsilon (2t_3)^{1/2}$ and $b_{i} - b_{i+1} \geq 3 \epsilon (2t_3)^{1/2}$ for $i = 1, \dots, k-2$.
\end{enumerate}
Then we can find $g = g(\epsilon, V^{top}, M_2) > 0$ and $N_6 \in \mathbb{N}$ such that for all $N \geq N_6$ we have 
\begin{equation}\label{eqnRT}
Z\big(  -t_1, t_1, \vec{a} ,\vec{b}, \ell_{bot}\llbracket -t_1, t_1\rrbracket\big) \geq g.
\end{equation}
\end{lemma}

\begin{proof}Observe by the rightmost inequalities in conditions (1) and (2) in the hypothesis, as well as condition (1) in Lemma \ref{LemmaAP1}, that $\ell_{bot}$ lies a distance of at least $2\epsilon(2t_3)^{1/2} \geq 2\epsilon(2t_1)^{1/2}$ uniformly below the line segment connecting $a_{k-1}$ and $b_{k-1}$. Also note that (1) and (2) imply $|b_i-a_i-2pt_1| \leq (V^{top} - M_2-2\epsilon)(2t_3)^{1/2}$ for each $i$. Lastly noting (3), we see that the conditions of Lemma \ref{CurveSeparation} are satisfied with $C = 2\epsilon$. This implies \eqref{eqnRT}, with
	\[
	g = \left( \frac{1}{2} - \sum_{n=1}^\infty (-1)^{n-1} e^{-\epsilon^2n^2/2p(1-p)}\right)^{k-1}.
	\]
	
\end{proof}

\noindent The next lemma helps us derive the lower bound $h$ in \eqref{eqn57}.

\begin{lemma}\label{LemmaBP2} For any $R > 0$ we can find $V_1^t, V_1^b \geq M_2 + R$, $h_1 > 0$ and $N_7 \in \mathbb{N}$ (depending on $R$) such that if $N \geq N_7$ we have
\begin{equation}\label{eqnRT2}
\mathbb{P}_{\tilde{\mathfrak{Q}}} \left(  (2t_3)^{1/2} V_1^t \geq \tilde{Q}_1(\pm t_2) \mp p t_2 \geq \tilde{Q}_{k-1}(\pm t_2) \mp p t_2 \geq (2t_3)^{1/2} V_1^b  \right) \geq h_1.
\end{equation}	
\end{lemma}

\begin{proof}
	We first define the constants $V_1^b$ and $h_1$, as well as two other constants $C$ and $K_1$ to be used in the proof. We put
	\begin{equation}\label{5.10const}
	\begin{split}
	C &= \sqrt{8p(1-p)\log\frac{3}{1-(11/12)^{1/(k-2)}}},\\
	V_1^b &= M_1 + Ck + M_2 + R, \qquad K_1 = (4r+10)V_1^b,\\
	h_1 &=  \frac{2^{k/2-5}\big(1-2e^{-4/p(1-p)}\big)^{2k}}{(\pi p(1-p))^{k/2}}\,\exp\left(-\frac{2k(K_1+M_1+6)^2}{p(1-p)}\right).
	\end{split}
	\end{equation}
	Note in particular that $V_1^b > M_2 + R$. We will fix $V_1^t > V_1^b$ in Step 3 below depending on $h_1$. We will prove in the following steps that for these choices of  $V_1^b, V_1^t, h_1$, we can find $N_7$ so that for $N\geq N_7$ we have
	\begin{align}
	\mathbb{P}_{\tilde{\mathfrak{Q}}}\left(\tilde{Q}_{k-1}(\pm t_2) \mp pt_2 \geq (2t_3)^{1/2}V_1^b\right) &\geq 2h_1, \label{5.10bound1}\\
	\mathbb{P}_{\tilde{\mathfrak{Q}}}\left(\tilde{Q}_1(\pm t_2) \mp pt_2 > (2t_3)^{1/2}V_1^t\right) &\leq h_1. \label{5.10bound2}
	\end{align}
	Assuming the validity of the claim, we then observe that the probability in \eqref{eqnRT2} is bounded below by $2h_1 - h_1 = h_1$, proving the lemma. We will prove \eqref{5.10bound1} and \eqref{5.10bound2} in three steps.\\
	
	\noindent\textbf{Step 1.} In this step we prove that there exists $N_7$ so that \eqref{5.10bound1} holds for $N\geq N_7$, assuming results from Step 2 below. We condition on the value of $\tilde{\mathfrak{Q}}$ at 0 and divide $\tilde{\mathfrak{Q}}$ into two independent line ensembles on $[-t_3,0]$ and $[0,t_3]$. Observe by Lemma \ref{MCLfg} that
	\begin{equation}\label{5.10MC}
	\mathbb{P}_{\tilde{\mathfrak{Q}}}\left(\tilde{Q}_{k-1}(\pm t_2) \mp pt_2 \geq (2t_3)^{1/2}V_1^b\right) \geq \mathbb{P}^{-t_3,t_3,\vec{x},\vec{y}}_{avoid, Ber; S}\left(\tilde{Q}_{k-1}(\pm t_2) \mp pt_2 \geq (2t_3)^{1/2}V_1^b\right).
	\end{equation}
	With $K_1$ as in \eqref{5.10const}, we define events
	\[
	E_{\vec{z}} = \left\{\big(\tilde{Q}_1(0),\dots,\tilde{Q}_{k-1}(0)\big) = \vec{z}\right\}, \quad X = \left\{ \vec{z}\in\mathfrak{W}_{k-1} : z_{k-1} \geq K_1(2t_3)^{1/2} \mbox { and } \mathbb{P}^{-t_3,t_3,\vec{x},\vec{y}}_{avoid,Ber; S}(E_{\vec{z}}) > 0\right\},
	\]
	and $E = \bigsqcup_{\vec{z} \in X} E_{\vec{z}}$. By Lemma \ref{LemmaWD}, we can choose $\tilde{N}_0$ large enough depending on $M_1,C,k,M_2,R$ so that $X$ is non-empty for $N\geq\tilde{N}_0$. By Lemma \ref{prob19} we can find $\tilde{N}_1$ so that
	\begin{equation}\label{5.10Ebound}
	\mathbb{P}^{-t_3,t_3,\vec{x},\vec{y}}_{avoid, Ber; S}(E) \geq \mathbb{P}^{-t_3,t_3,\vec{x},\vec{y}}_{avoid,Ber; S}\left(\tilde{Q}_{k-1}(0) \geq K_1(2t_3)^{1/2}\right) \geq A\exp\left(-\frac{2k(K_1+M_1+6)^2}{p(1-p)}\right)
	\end{equation}
	for $N\geq\tilde{N}_1$, where $A = A(p,k)$ is a constant given explicitly in \eqref{19ineq}.
	
	Now let $\tilde{Q}_i^1$ and $\tilde{Q}_i^2$ denote the restrictions of $\tilde{Q}_i$ to $[-t_3,0]$ and $[0,t_3]$ respectively for $1\leq i\leq k-1$, and write $S_1 = S\cap\llbracket -t_3,0\rrbracket$, $S_2 = S\cap\llbracket 0, t_3\rrbracket$. We observe that if $\vec{z}\in X$, then
	\begin{equation}\label{5.10split}
	\mathbb{P}^{-t_3,t_3,\vec{x},\vec{y}}_{avoid,Ber;S}\left(\tilde{Q}^1_{k-1} = \ell_1, \tilde{Q}^2_{k-1} = \ell_2 \, |\, E_{\vec{z}}\right) = \mathbb{P}^{-t_3,0,\vec{x},\vec{z}}_{avoid,Ber;S_1}(\ell_1)\cdot\mathbb{P}^{0,t_3,\vec{z},\vec{y}}_{avoid,Ber;S_2}(\ell_2).
	\end{equation}
	In Step 2, we will find $\tilde{N}_2$ so that for $N\geq\tilde{N}_2$ we have
	\begin{equation}\label{5.10fourth}
	\begin{split}
	&\mathbb{P}^{-t_3,0,\vec{x},\vec{z}}_{avoid,Ber;S_1}\left(\tilde{Q}^1_{k-1}(-t_2) + pt_2 \geq (2t_3)^{1/2}V_1^b\right) \geq \frac{1}{4},\\
	&\mathbb{P}^{0,t_3,\vec{x},\vec{z}}_{avoid,Ber;S_2}\left(\tilde{Q}^2_{k-1}(t_2) - pt_{2} \geq (2t_3)^{1/2}V_1^b\right) \geq \frac{1}{4}.
	\end{split}
	\end{equation}
	Using \eqref{5.10Ebound}, \eqref{5.10split}, and \eqref{5.10fourth}, we conclude that
	\[
	\mathbb{P}^{-t_3,t_3,\vec{x},\vec{y}}_{avoid, Ber;S}\left(\tilde{Q}_{k-1}(\pm t_2) \mp pt_2 \geq (2t_3)^{1/2}V_1^b\right) \geq \frac{A}{16}\exp\left(-\frac{2k(K_1+M_1+6)^2}{p(1-p)}\right)
	\]
	for $N\geq N_7 = \max(\tilde{N}_0,\tilde{N}_1,\tilde{N}_2)$. In combination with \eqref{5.10MC}, this proves \eqref{5.10bound1} with $h_1 = A/16$ as in \eqref{5.10const}.\\
	
	\noindent\textbf{Step 2.} In this step, we prove the inequalities in \eqref{5.10fourth} from Step 1, using Lemma \ref{LemmaHalfS4}. Let us define vectors $\vec{x}\,', \vec{z}\,', \vec{y}\,'$ by
	\begin{align*}
	x_i' &= \lfloor -pt_3 - M_1(2t_3)^{1/2}\rfloor - (i-1)\lceil C(2t_3)^{1/2}\rceil,\\
	z_i' &= \lfloor K_1(2t_3)^{1/2}\rfloor - (i-1)\lceil C(2t_3)^{1/2}\rceil,\\
	y_i' &= \lfloor pt_3 - M_1(2t_3)^{1/2}\rfloor - (i-1)\lceil C(2t_3)^{1/2}\rceil.
	\end{align*}
	Note that $x_i' \leq x_{k-1} \leq x_i$ and $x_i' - x_{i+1}' \geq C(2t_3)^{1/2}$ for $1\leq i\leq k - 1$, and likewise for $z_i',y_i'$. By Lemma \ref{MCLxy} we have
	\begin{equation}\label{5.10separate}
	\begin{split}
	&\mathbb{P}^{-t_3,0,\vec{x},\vec{z}}_{avoid,Ber;S_1}  \hspace{-1mm} \left(\tilde{Q}^1_{k-1}(-t_2) + pt_2 \geq (2t_3)^{1/2}V_1^b\right)  \hspace{-1mm}\geq  \hspace{-1mm}\mathbb{P}^{-t_3,0,\vec{x}\,',\vec{z}\,'}_{avoid,Ber;S_1} \hspace{-1mm} \left(\tilde{Q}^1_{k-1}(-t_2) + pt_2 \geq (2t_3)^{1/2}V_1^b\right)  \\ 
	& \geq \mathbb{P}^{-t_3,0,x_{k-1}',z_{k-1}'}_{Ber}\left(\ell_1(-t_2) + pt_2 \geq (2t_3)^{1/2}V_1^b\right) - \left( 1 - \mathbb{P}^{-t_3,t_3,\vec{x}\,',\vec{z}\,'}_{Ber}\left(\tilde{Q}^1_1 \geq \cdots \geq \tilde{Q}_{k-1}^1\right)\right).
	\end{split}
	\end{equation} 
	To bound the first term on the second line, first note that $x_{k-1}' \geq -pt_3 - (M_1+C(k-1))(2t_3)^{1/2}$ and $z_{k-1}' \geq K_1(2t_3)^{1/2} - C(k-1)(2t_3)^{1/2}$ for sufficiently large $N$. Let us write $\tilde{x},\tilde{z}$ for these two lower bounds. Then by Lemma \ref{LemmaHalfS4}, we have an $\tilde{N}_3$ so that for $N\geq\tilde{N}_3$,
	\begin{equation}\label{5.10third1}
	\mathbb{P}^{-t_3,0,x_{k-1}',z_{k-1}'}_{Ber}\left(\ell_1(-t_2) \geq \frac{t_2}{t_3}\,\tilde{x} + \frac{t_3-t_2}{t_3}\,\tilde{z} - (2t_3)^{1/4}\right) \geq \frac{1}{3}.
	\end{equation}  
	Moreover, as long as $\tilde{N}_3^\alpha > 2$, we have for $N\geq\tilde{N}_3^\alpha$ that
	\begin{equation}\label{2r+5}
	\frac{t_3-t_2}{t_3} \geq 1 - \frac{(r+2)N^\alpha}{(r+3)N^\alpha - 1} > 1-\frac{r+2}{r+5/2} = \frac{1}{2r+5}.
	\end{equation}
	It follows from our choice of $V_1^b$ and $K_1 = 2(2r+5)V_1^b$ in \eqref{5.10const}, as well as \eqref{2r+5}, that 
	\begin{align*}
	&\frac{t_2}{t_3}\,\tilde{x} + \frac{t_3-t_2}{t_3}\,\tilde{z} - (2t_3)^{1/4} = -pt_2 - C(k-1)(2t_3)^{1/2} - \frac{t_2}{t_3}\,M_1(2t_3)^{1/2} + \\
& \frac{t_3-t_2}{t_3}\,K_1(2t_3)^{1/2} - (2t_3)^{1/4} \geq -pt_2 - Ck(2t_3)^{1/2} - M_1(2t_3)^{1/2} + \frac{1}{2r+5}\,K_1(2t_3)^{1/2}\\
& =  -pt_2 + (M_1 + Ck + 2(M_2+R))(2t_3)^{1/2}>  -pt_2 + (2t_3)^{1/2}V_1^b.
	\end{align*}
	For the first inequality, we used the fact that $t_2/t_3 < 1$, and we assumed that $\tilde{N}_3$ is sufficiently large so that $C(k-1)(2t_3)^{1/2} + (2t_3)^{1/4} \leq Ck(2t_3)^{1/2}$ for $N\geq\tilde{N}_3$. Using \eqref{5.10third1}, we conclude for $N\geq\tilde{N}_3$
	\begin{equation}\label{5.10third2}
	\mathbb{P}^{-t_3,0,x_{k-1}',z_{k-1}'}_{Ber}\left(\ell_1(-t_2) + pt_2 \geq (2t_3)^{1/2}V_1^b\right) \geq \frac{1}{3}.
	\end{equation}
	Since $|z_i'-x_i'-pt_2| \leq (K_1+M_1+1)(2t_2)^{1/2}$, we have by Lemma \ref{CurveSeparation} and our choice of $C$ that the second probability in the second line of \eqref{5.10separate} is bounded below by
	\[
	\left(1-3e^{-C^2/8p(1-p)}\right)^{k-1} \geq 11/12
	\]
	for $N$ larger than some $\tilde{N}_4$. It follows from \eqref{5.10separate} and \eqref{5.10third2} that for $N\geq\tilde{N}_2 = \max(\tilde{N}_3,\tilde{N}_4)$,
	\begin{equation*}
	\mathbb{P}^{-t_3,0,\vec{x},\vec{z}}_{avoid,Ber;S_1}\left(\tilde{Q}^1_{k-1}(-t_2) + pt_2 \geq (2t_3)^{1/2}V_1^b\right) \geq \frac{1}{3} - \frac{1}{12} = \frac{1}{4},
	\end{equation*}
	proving the first inequality in \eqref{5.10fourth}. The second inequality is proven similarly.
	\\
	
	\noindent\textbf{Step 3.} In this last step, we fix $V_1^t$ and prove that we can enlarge $N_7$ from Step 1 so that \eqref{5.10bound2} holds for $N\geq N_7$. Let $C$ be as in \eqref{5.10const}, and define vectors $\vec{x}\,'', \vec{y}\,''\in\mathfrak{W}_{k-1}$ by
	\begin{align*}
	x_i'' &= \lceil -pt_3 + M_1(2t_3)^{1/2}\rceil + (k-i)\lceil C(2t_3)^{1/2}\rceil,\\
	y_i'' &= \lceil pt_3 + M_1(2t_3)^{1/2}\rceil + (k-i)\lceil C(2t_3)^{1/2}\rceil.
	\end{align*}
	Note that $x_i'' \geq x_1 \geq x_i$ and $x_i''-x_{i+1}'' \geq C(2t_3)^{1/2}$, and likewise for $y_i''$. Moreover, $\ell_{bot}$ lies a distance of at least $C(2t_3)^{1/2}$ uniformly below the line segment connecting $x_{k-1}''$ and $y_{k-1}''$. By Lemma \ref{MCLxy} we have
	\begin{align*}
	&\mathbb{P}_{\tilde{\mathfrak{Q}}}\left(\tilde{Q}_1(\pm t_2) \mp pt_2 > (2t_3)^{1/2}V_1^t\right) \leq \mathbb{P}^{-t_3,t_3,\vec{x}\,'',\vec{y}\,'',\infty,\ell_{bot}}_{avoid,Ber;S}\left(\sup_{s\in[-t_3,t_3]} \big[\tilde{Q}_1(s)-ps\big] \geq (2t_3)^{1/2}V_1^t\right)\leq\\
	& \frac{\mathbb{P}^{-t_3,t_3,x_1'',y_1''}_{Ber}\left(\sup_{s\in[-t_3,t_3]} \big[\tilde{L}_1(s)-ps\big] \geq (2t_3)^{1/2}V_1^t\right)}{\mathbb{P}^{-t_3,t_3,\vec{x}\,'',\vec{y}\,''}_{Ber}\left(\tilde{L}_1\geq\cdots\geq\tilde{L}_{k-1}\geq\ell_{bot}\right)}.
	\end{align*}
	In the numerator in the second line, we used the fact that the curves $\tilde{L}_1,\dots,\tilde{L}_{k-1}$ are independent under $\mathbb{P}^{-t_3,t_3,x_1'',y_1''}_{Ber}$, and the event in the parentheses depends only on $\tilde{L}_1$. By Lemma \ref{LemmaMinFreeS4}, since $\min(x_1'' + pt_3, \, y_1'' - pt_3) \leq (M_1+C(k-1))(2t_3)^{1/2}$, we can choose $V_1^t > V_1^b$ as well as $\tilde{N}_5$ large enough so that the numerator is bounded above by $h_1/2$ for $N\geq\tilde{N}_5$. Since $|y_i'' - x_i'' - 2pt_3| \leq 1$, our choice of $C$ and Lemma \ref{CurveSeparation} give a $\tilde{N}_6$ so that the denominator is at least $11/12$ for $N\geq\tilde{N}_6$. This gives an upper bound of $12/11\cdot h_1/2< h_1$ in the above as long as $N_7\geq\max(\tilde{N}_5,\tilde{N}_6)$, which concludes the proof of \eqref{5.10bound2}.

\end{proof}

We are now equipped to prove Lemma \ref{LemmaAP2}. Let us put for convenience
\begin{equation}\label{t12}
t_{12} = \left\lfloor \frac{t_1+t_2}{2}\right\rfloor.
\end{equation}

\begin{proof}(of Lemma \ref{LemmaAP2}) We first introduce some notation to be used in the proof. Let $S$ be as in Definition \ref{TildeDef}. For $\vec{c}, \vec{d} \in \mathfrak{W}_{k-1}$, let us write $\tilde{S} = \llbracket -t_2,-t_1\rrbracket \cup \llbracket t_1,t_2\rrbracket$, $\tilde{\Omega}(\vec{c},\vec{d}) = \Omega_{avoid}(-t_2, t_2, \vec{c}, \vec{d}, \infty, \ell_{bot}; \tilde{S})$. For $s\in\tilde{S}$ we define events
	\begin{equation}
	\begin{split}
	&A(\vec{c}, \vec{d},s) = \left\{\tilde{\mathfrak{Q}} \in \tilde\Omega(\vec{c},\vec{d}): \tilde Q_{k-1}(\pm s) \mp ps \geq (M_2 + 1) (2t_3)^{1/2} \right\}, \\
	&B(\vec{c},\vec{d},V^{top},s) = \left\{ \tilde{\mathfrak{Q}} \in \tilde\Omega(\vec{c},\vec{d}): \tilde Q_{1}(\pm s) \mp ps \leq V^{top} (2t_3)^{1/2} \right\},\\
	&C(\vec{c}, \vec{d}, \epsilon, s) = \left\{ \tilde{\mathfrak{Q}} \in \tilde\Omega(\vec{c},\vec{d}): \min_{1\leq i\leq k-2, \, \varsigma \in \{-1, 1\}} \big[\tilde Q_{i}(\varsigma s) - \tilde Q_{i+1}(\varsigma s)\big] \geq 3\epsilon (2t_3)^{1/2} \right\},\\
	&D(\vec{c},\vec{d},V^{top},\epsilon,s) = A(\vec{c}, \vec{d},s) \cap B(\vec{c},\vec{d},V^{top},s) \cap C(\vec{c}, \vec{d}, \epsilon, s).
	\end{split}
	\end{equation}
	Here, $\epsilon$ and $V^{top}$ are constants which we will specify later. By Lemma \ref{LemmaBP1}, for all $(\vec{c},\vec{d})$ and $N$ sufficiently large we have
\begin{equation}\label{S6R1} 
D(\vec{c},\vec{d},V^{top},\epsilon,s)  \subset \left\{Z\left(  -t_1, t_1, \mathfrak{Q}(-t_1), \mathfrak{Q}(t_1), \ell_{bot}\llbracket -t_1, t_1\rrbracket\right) > g\right\}
\end{equation}
for some $g$ depending on $\epsilon,V^{top},M_2$. The above gives all the notation we require. 

We now turn to the proof of the lemma, which split is into several steps.\\
	
	{\bf \raggedleft Step 1.} In this step, we show that there exist $R > 0$ and $\bar{N}_0$ sufficiently large so that if $c_{k-1} + pt_2 \geq (2t_3)^{1/2} (M_2 + R)$ and $d_{k-1} - pt_2 \geq (2t_3)^{1/2} (M_2 + R)$, then for all $s\in\tilde{S}$ and $N\geq\bar{N}_0$ we have
	\begin{equation}\label{6.2step1}
	\begin{split}
	\mathbb{P}^{-t_2,t_2,\vec{c},\vec{d},\infty,\ell_{bot}}_{avoid,Ber; \tilde S}\big(A(\vec{c},\vec{d},s)\big) \geq  \frac{19}{20} \quad \mathrm{and} \quad \mathbb{P}^{-t_2,t_2,\vec{c},\vec{d}}_{avoid,Ber;\tilde S}\left(Q_{k-1}|_{\tilde S} \geq \ell_{bot}|_{\tilde S}\right) \geq \frac{99}{100}.
	\end{split}
	\end{equation} 
	Let us begin with the first inequality. We observe via Lemma \ref{MCLfg} that
	\begin{equation}\label{6.2step1MC}
	\mathbb{P}^{-t_2,t_2,\vec{c},\vec{d},\infty,\ell_{bot}}_{avoid,Ber; \tilde S}\big(A(\vec{c},\vec{d},s)\big) \geq \mathbb{P}^{-t_2,t_2,\vec{c},\vec{d}}_{avoid,Ber; \tilde S}\big(A(\vec{c},\vec{d},s)\big).
	\end{equation}
	Now define the constant
	\begin{equation}\label{6.2C}
	C = \sqrt{8p(1-p)\log\frac{3}{1-(199/200)^{1/(k-1)}}}
	\end{equation}
	and vectors $\vec{c}\,', \vec{d}\,' \in \mathfrak{W}_k$ by
	\begin{align*}
	c_i' &= \lfloor -pt_2 + (M_2+R)(2t_3)^{1/2}\rfloor - (i-1)\lceil C(2t_2)^{1/2}\rceil,\\
	d_i' &= \lfloor pt_2 + (M_2+R)(2t_3)^{1/2}\rfloor - (i-1)\lceil C(2t_2)^{1/2}\rceil.
	\end{align*}
	Then by Lemma \ref{MCLxy} we have
	\begin{equation}\label{6.2step1split}
	\begin{split}
	&\mathbb{P}^{-t_2,t_2,\vec{c},\vec{d}}_{avoid,Ber; \tilde S}\big(A(\vec{c},\vec{d},s)\big) \geq \mathbb{P}^{-t_2, t_2, \vec{c}\,', \vec{d}\,'}_{avoid, Ber; \tilde S}(A(\vec{c}\,', \vec{d}\,',s)) \geq \\ 
	&\mathbb{P}^{-t_2, t_2, c_{k-1}', d_{k-1}'}_{Ber}\hspace{-1mm}\left(\inf_{s\in \tilde S}\big[\ell(s) - ps\big] \geq (M_2+1)(2t_3)^{1/2}\right)\hspace{-1mm} - \hspace{-1mm} \left( \hspace{-0.5mm}1 - \mathbb{P}^{-t_2, t_2, \vec{c}\,', \vec{d}\,'}_{Ber}\left(L_1 \geq \cdots \geq L_{k-1}\right) \hspace{-0.5mm} \right).
	\end{split}
	\end{equation}
	By Lemma \ref{CurveSeparation} and our choice of $C$, we can find $\tilde{N}_0$ so that $\mathbb{P}^{-t_2, t_2, \vec{c}\,', \vec{d}\,'}_{Ber}(L_1 \geq \cdots \geq L_{k-1})>199/200 > 39/40$ for $N\geq\tilde{N}_0$. Writing $z = d_{k-1}' - c_{k-1}'$, the term in the second line of \eqref{6.2step1split} is equal to
	\begin{align*}
	&\mathbb{P}^{-t_2, t_2, 0, z}_{Ber}\Big(\inf_{s\in \tilde S}\big[\ell(s) + c_{k-1}' - ps\big]  \geq (M_2 + 1)(2t_3)^{1/2}\Big) \geq \\
	& \mathbb{P}^{0, 2t_2, 0, z}_{Ber}\Big(\inf_{s\in [0,2t_2]}\big[\ell(s) - ps\big] \geq (-R+Ck+1)(2t_3)^{1/2}\Big).
	\end{align*}
	In the second line, we used the estimate $c_{k-1}' \geq -pt_2 + (M_2+R-Ck)(2t_3)^{1/2}$. Now by Lemma \ref{LemmaMinFreeS4}, we can choose $R$ large enough depending on $C,k,M_2,p$ so that this probability is greater than $39/40$ for $N$ greater than some $\tilde{N}_1$. This gives a lower bound in \eqref{6.2step1split} of $39/40-1/40 = 19/20$ for $N\geq\max(\tilde{N}_0,\tilde{N}_1)$, and in combination with \eqref{6.2step1MC} this proves the first inequality in \eqref{6.2step1}.
	
	We prove the second inequality in \eqref{6.2step1} similarly. Note that since $\ell_{bot}(s) \leq ps + M_2(2t_3)^{1/2}$ on $[-t_3,t_3]$ by assumption, we have
	\begin{equation} \label{sigma}
	\begin{split}
	&\mathbb{P}^{-t_2,t_2,\vec{c},\vec{d}}_{avoid,Ber;\tilde S}\left(\tilde Q_{k-1}|_{\tilde S} \geq \ell_{bot}|_{\tilde S}\right) \geq \mathbb{P}^{-t_2, t_2, \vec{c},\vec{d}}_{avoid, Ber; \tilde S}\left(\inf_{s\in[-t_2, t_2]} \big[Q_{k-1}(s) - ps\big] \geq M_2(2t_3)^{1/2}\right) \geq\\
	&\mathbb{P}^{-t_2, t_2, \vec{c}\,',\vec{d}\,'}_{avoid, Ber;\tilde S}\left(\inf_{s\in[-t_2, t_2]} \big[\tilde Q_{k-1}(s) - ps\big] \geq M_2(2t_3)^{1/2}\right) \geq\\
	&\mathbb{P}^{0, 2t_2, 0, z}_{Ber}\hspace{-1mm}\left(\inf_{s\in[0, 2t_2]} \big[\ell(s) - ps\big] \geq -(R-Ck)(2t_3)^{1/2} \hspace{-1mm} \right) - \left(\hspace{-0.5mm} 1 - \mathbb{P}^{-t_2,t_2,\vec{c}\,',\vec{d}\,'}_{Ber}(\tilde L_1\geq \cdots \geq \tilde L_{k-1})\hspace{-0.5mm}\right).
	\end{split}
	\end{equation}
	We enlarge $R$ if necessary so that the probability in the third line of \eqref{sigma} is $>199/200$ for $N\geq\tilde{N}_2$ by Lemma \ref{LemmaMinFreeS4}, and \ref{CurveSeparation} implies as above that the second expression in the last line of \eqref{sigma} is $>-1/200$ for $N\geq\tilde{N}_3$. This gives us a lower bound of $199/200 - 1/200 = 99/100$ for $N\geq\tilde{N}_0 = \max(\tilde{N}_2,\tilde{N}_3)$ as desired. This proves the two inequalities in \eqref{6.2step1} for $N\geq\bar{N}_0 = \max(\tilde{N}_0,\tilde{N}_1,\tilde{N}_2,\tilde{N}_3)$.\\
	
	{\bf \raggedleft Step 2.} In this step we fix $R$ sufficiently large so that $R >C$ from (\ref{6.2C}) and  the inequalities in (\ref{6.2step1}) both hold for this choice of $R$. Our work from Step 1 ensures that such a choice for $R$ is possible. Let $V_1^t, V_1^b$, and $h_1$ be as in Lemma \ref{LemmaBP2}  for this choice of $R$. Define the set
	\begin{equation}\label{6.2E}
	\begin{split}
	E = \big\{ \vec{c}, \vec{d} \in \mathfrak{W}_{k-1} : &\; (2t_3)^{1/2} V_1^t \geq \max(c_1 + p t_2 d_1 - pt_2) \mbox{ and }\\
	&\; \min(c_{k-1} + p t_2, d_{k-1} - pt_2)  \geq (2t_3)^{1/2} V_1^b \big\}.
	\end{split}
	\end{equation}
	We show in this step that there exists $V^{top} \geq M_2 + 6(k-1)$ and $\bar{N}_1$ such that for all $(\vec{c}, \vec{d}) \in E$, $s\in\tilde{S}$, and $N\geq\bar{N}_1$ we have
	\begin{equation}\label{6.2step2}
	\mathbb{P}^{-t_2,t_2,\vec{c},\vec{d},\infty,\ell_{bot}}_{avoid,Ber;\tilde S}\big(B(\vec{c},\vec{d},V^{top},s)\big) \geq  \frac{19}{20}.
	\end{equation}
	Let $C$ be as in \eqref{6.2C}, and define $\vec{c}\,'', \vec{d}\,'' \in \mathfrak{W}_{k-1}$ by
	\begin{align*}
	c_i'' &= \lceil -pt_2 + (2t_3)^{1/2} V_1^t\rceil + (k-1-i)\lceil C(2t_2)^{1/2}\rceil,\\
	d_i'' &= \lceil pt_2 + (2t_3)^{1/2} V_1^t\rceil + (k-1-i)\lceil C(2t_2)^{1/2}\rceil.
	\end{align*}
	Then $c_i'' \geq c_1 \geq c_i$ and $c_i'' - c_{i+1}'' \geq C(2t_2)^{1/2}$ for each $i$, and likewise for $d_i''$. By Lemma \ref{MCLxy}, the left hand side of \eqref{6.2step2} is bounded below by
	\begin{equation}\label{6.2step3split}
	\begin{split}
	&\mathbb{P}^{-t_2,t_2,\vec{c}\,'', \vec{d}\,'', \infty,\ell_{bot}}_{avoid, Ber;\tilde S}\left(\sup_{s\in \tilde S}\big[\tilde Q_1(s) - ps\big] \leq V^{top}(2t_3)^{1/2}\right) \geq\\
	& \mathbb{P}^{0,2t_2,0,z'}_{Ber}\left(\sup_{s\in[-t_2,t_2]}\big[\ell(s) - ps\big] \leq (V^{top}-V_1^t-Ck)(2t_3)^{1/2}\right) -\\ &\qquad\qquad \left(1 - \mathbb{P}^{-t_2,t_2,\vec{c}\,'', \vec{d}\,'', \infty,\ell_{bot}}_{Ber}\left(L_1\geq\cdots\geq L_{k-1}\geq \ell_{bot}\right)\right).
	\end{split}
	\end{equation}
	In the last line, we have written $z' = d_1'' - c_1''$, and we used the fact that $c_1'' \leq -pt_2 + (V_1^t + Ck)(2t_3)^{1/2}$. By Lemma \ref{LemmaMinFreeS4}, we can find $V^{top}$ large enough depending on $V_1^t,C,k,p$ so that the probability in the third line of \eqref{6.2step3split} is at least $39/40$ for $N\geq\tilde{N}_4$. On the other hand, the above observations regarding $\vec{c}\,''$, $\vec{d}\,''$, and $\ell_{bot}$, as well as the fact that $|d_1'' - c_1'' - 2pt_2| \leq 1$, allow us to conclude from Lemma \ref{CurveSeparation} that the probability in the last line of \eqref{6.2step3split} is at least $39/40$ for $N\geq\tilde{N}_5$.  In applying Lemma \ref{CurveSeparation} we used the fact that $V_1^b \geq M_2+R$, which implies that $\ell_{bot}$ lies a distance of at least $R(2t_3)^{1/2}$ (and hence $C (2t_3)^{1/2}$ as $R > C$ by construction) uniformly below the line segment connecting $c_{k-1}''$ and $d_{k-1}''$. We thus obtain a lower bound of $39/40 - 1/40 = 19/20$ in \eqref{6.2step3split} for $\bar{N}_1 = \max(\tilde{N}_4,\tilde{N}_5)$, which proves (\ref{6.2step2}) as desired.\\
	
	{\bf \raggedleft Step 3.} In this step, we show that with $E$, $V_1^t$, and $V_1^b$ as in Step 2, there exist $\epsilon > 0$ sufficiently small and $\bar{N}_2$ such that for all $(\vec{c}, \vec{d}) \in E$ and $N\geq\bar{N}_2$, we have
	\begin{equation}\label{LemmaBP2Step3}
	\mathbb{P}^{-t_2,t_2,\vec{c},\vec{d},\infty,\ell_{bot}}_{avoid,Ber;\tilde S}\big(D(\vec{c},\vec{d},V^{top},\epsilon,t_{12}) \big) \geq \frac{1}{2}.
	\end{equation}
	We claim that this follows if we find $\tilde{N}_6$ so that for $N\geq\tilde{N}_6$,
	\begin{equation}\label{6.2step3cond}
	\mathbb{P}^{-t_2,t_2,\vec{c},\vec{d}}_{avoid,Ber;\tilde S}\big(C(\vec{c},\vec{d},\epsilon,t_{12})\,|\,A(\vec{c},\vec{d},t_1) \cap B(\vec{c},\vec{d},V^{top},t_1)\big) \geq \frac{9}{10}.
	\end{equation}
	To see this, note that \eqref{6.2step1} and \eqref{6.2step2} imply that for $N\geq\max(\bar{N}_0,\bar{N}_1)$,
	\[
	\mathbb{P}^{-t_2,t_2,\vec{c},\vec{d}}_{avoid,Ber;\tilde S}\big(A(\vec{c},\vec{d},t_1) \cap B(\vec{c},\vec{d},V^{top},t_1)\big) \geq \left( \frac{19}{20} - \frac{1}{20} \right) \cdot \frac{99}{100}  > \frac{4}{5},
	\]
	and then \eqref{6.2step3cond} and the second inequality in \eqref{6.2step1} imply that for $N\geq \bar{N}_2 = \max(\bar{N}_0,\bar{N}_1,\tilde{N}_6)$,
	\[
	\mathbb{P}^{-t_2,t_2,\vec{c},\vec{d},\infty,\ell_{bot}}_{avoid,Ber;\tilde S}\big(A(\vec{c},\vec{d},t_1) \cap B(\vec{c},\vec{d},V^{top},t_1) \cap C(\vec{c},\vec{d},\epsilon,t_{12})\big) > \frac{9}{10}\cdot\frac{4}{5}  - \frac{1}{99} > \frac{1}{2},
	\]
	which gives (\ref{LemmaBP2Step3}) once we recall the definition of $D(\vec{c},\vec{d},V^{top},\epsilon,t_{12})$.\\

	In the remainder of this step, we verify \eqref{6.2step3cond}. Observe that $A(\vec{c},\vec{d},t_1) \cap B(\vec{c},\vec{d},V^{top},t_1)$ can be written as a countable disjoint union: 
	\begin{equation}\label{6.2step3disj}
	A(\vec{c},\vec{d},t_1) \cap B(\vec{c},\vec{d},V^{top},t_1) = \bigsqcup_{(\vec{a},\vec{b})\in I} F(\vec{a},\vec{b}).
	\end{equation}
	Here, for $\vec{a},\vec{b}\in\mathfrak{W}_{k-1}$, $F(\vec{a},\vec{b})$ is the event that $\mathfrak{Q}(-t_1) = \vec{a}$ and $\mathfrak{Q}(t_1) = \vec{b}$, and $I$ is the collection of pairs $(\vec{a},\vec{b})$ satisfying
	\begin{enumerate}[label = (\arabic*)]
		
		\item $ 0 \leq \min(a_i - c_i,\, d_i - b_i) \leq t_2 - t_1$ and $0\leq b_i-a_i \leq 2t_1$ for $1\leq i\leq k-1$,
		
		\item $\min(a_{k-1} + pt_1,\, b_{k-1} - pt_1) \geq (M_2+1)(2t_3)^{1/2}$,
		
		\item $\max(a_1 + pt_1,\, b_1 - pt_1) \leq V^{top}(2t_3)^{1/2}$.
		
	\end{enumerate}
	Now let $\mathfrak{Q}^1 = (Q^1_1,\dots,Q^1_{k-1})$ and $\mathfrak{Q}^2 = (Q^2_2,\dots,Q^2_{k-1})$ denote the restrictions of $\tilde{\mathfrak{Q}}$ to $\llbracket -t_2,-t_1\rrbracket$ and $\llbracket t_1,t_2\rrbracket$ respectively. Then we observe that
	\begin{equation}\label{6.2step3ind}
	\begin{split}
	&\mathbb{P}^{-t_2,t_2,\vec{c},\vec{d}}_{avoid, Ber; \tilde S}\left(\mathfrak{Q}^1 = \mathfrak{B}^1, \mathfrak{Q}^2 = \mathfrak{B}^2\,\big|\,F(\vec{a},\vec{b})\right) = \mathbb{P}^{-t_2,-t_1,\vec{c},\vec{a}}_{avoid, Ber}\left(\mathfrak{Q}^1 = \mathfrak{B}^1\right) \cdot \mathbb{P}^{t_1,t_2,\vec{b},\vec{d}}_{avoid, Ber}\left(\mathfrak{Q}^2 = \mathfrak{B}^2\right).
	\end{split}
	\end{equation}
	We also let $\tilde{I} = \{(\vec{a},\vec{b})\in I : \pr^{-t_2,t_2,\vec c, \vec d}_{avoid,Ber;\tilde{S}}(F(\vec{a},\vec{b})) > 0\}$, and we choose $\tilde{N}_7$ so that $\tilde{I}$ is nonempty for $N\geq\tilde{N}_7$ using Lemma \ref{LemmaWD}. We now fix $(\vec{a},\vec{b})$ and argue that we can choose $\epsilon > 0$ small enough and $\tilde{N}_8$ so that for $N\geq\tilde{N}_8$,
	\begin{equation}\label{6.2step3 9/10}
	\pr^{-t_2,t_2,\vec c, \vec d}_{avoid,Ber;\tilde{S}}\left(C(\vec{c},\vec{d},\epsilon,t_{12})\,\big|\, F(\vec{a},\vec{b})\right) \geq \frac{9}{10}.
	\end{equation}
	Then using \eqref{6.2step3 9/10} and \eqref{6.2step3disj} and summing over $\tilde{I}$ proves \eqref{6.2step3cond} for $N\geq\tilde{N}_6 = \max(\tilde{N}_7,\tilde{N}_8)$.
	
	To prove \eqref{6.2step3 9/10}, we first show that we can find $\delta > 0$ and $\tilde{N}_7$ so that
	\begin{equation}\label{6.2step3right}
	\mathbb{P}^{-t_2,-t_1,\vec{c},\vec{a}}_{avoid, Ber}\left(\max_{1\leq i\leq k-2} \big[Q^1_i(-t_{12}) - Q^1_{i+1}(-t_{12})\big] \geq \delta(2t_3)^{1/2}\right) \geq \frac{3}{\sqrt{10}}
	\end{equation}
	for $N\geq\tilde{N}_7$. We prove this inequality using Lemma \ref{prob 20}. In order to apply this result, we first observe that since $|-t_{12}+\frac{1}{2}(t_1+t_2)| \leq 1$ by \eqref{t12}, we have
	\begin{equation}\label{Lt1Set}
	0 \leq Q^1_i(-t_{12})-Q^1_i(-\tfrac{1}{2}(t_1+t_2)) \leq 1.
	\end{equation}
	Now applying Lemma \ref{prob 20} with $M_1 = V_1^t$, $M_2 = V^{top}$, we obtain $\tilde{N}_7$ and $\delta>0$ such that if $N \geq \tilde{N}_7$, then
	\[
	\pr^{-t_2,-t_1,\vec c, \vec a}_{avoid,Ber}\left(\min_{1\leq i\leq k-1} \big[Q^1_i(-\tfrac{1}{2}(t_1+t_2))-Q_{i+1}^1(-\tfrac{1}{2}(t_1+t_2))\big] < \delta(t_2-t_1)^{1/2}\right) < 1 - \frac{3}{\sqrt{10}}.
	\]
	Together with \eqref{Lt1Set} and the fact that $t_3/4 < t_2-t_1$, this implies that
	\begin{equation}\label{step3delta/2}
	\pr^{-t_2,-t_1,\vec c, \vec a}_{avoid,Ber}\left(\min_{1\leq i\leq k-1} \big[Q_i^1(-t_{12})-Q_{i+1}^1(-t_{12})\big]<(\delta/2)(2t_3)^{1/2}-1\right) < 1 - \frac{3}{\sqrt{10}}
	\end{equation}
	for $N\geq \tilde{N}_7$. Now we observe that as long as $\tilde{N}_7^\alpha \geq \frac{1+8/\delta^2}{r+3}$, then $(\delta/4)(2t_3)^{1/2} \leq (\delta/2)(2t_2)^{1/2} - 1$ for $N\geq\tilde{N}_7$. This implies \eqref{6.2step3right}. A similar argument gives us a $\tilde{\delta}>0$ such that
	\[
	\pr^{-t_2,-t_1,\vec c, \vec a}_{avoid,Ber}\left(\min_{1\leq i\leq k-1} \big[Q_i(-t_{12})-Q_{i+1}(-t_{12})\big]<(\tilde\delta/4)(2t_3)^{1/2}\right)<1 - \frac{3}{\sqrt{10}}
	\]
	for $N\geq\tilde{N}_7$. Then putting $\epsilon = \min(\delta,\tilde{\delta})/12$ and using \eqref{6.2step3ind}, we obtain \eqref{6.2step3 9/10} for $N\geq\tilde{N}_7$.\\
	
	{\bf \raggedleft Step 4.} In this step, we find $\bar{N}_3$ so that
	\begin{equation}\label{6.2step4}
	\mathbb{P}^{-t_2,t_2,\vec{c},\vec{d},\infty,\ell_{bot}}_{avoid,Ber;\tilde S}\big(D(\vec{c},\vec{d},V^{top},\epsilon,t_1) \big) \geq \frac{1}{2}\left(\frac{1}{2} - \sum_{n=1}^\infty (-1)^{n-1} e^{-\epsilon^2 n^2/2p(1-p)}\right)^{k-1}
	\end{equation}
	for $N\geq\bar{N}_3$. We will find $\tilde{N}_9$ so that for $N\geq\tilde{N}_9$,
	\begin{equation}\label{6.2step4sep}
\begin{split}
&\mathbb{P}^{-t_2,t_2,\vec{c},\vec{d},\infty,\ell_{bot}}_{avoid,Ber;\tilde S}\left(D(\vec{c},\vec{d},V^{top},\epsilon,t_1) \,\big|\,D(\vec c, \vec d, V^{top}, \epsilon,t_{12})\right) \geq \\
& \left(\frac{1}{2} - \sum_{n=1}^\infty (-1)^{n-1} e^{-\epsilon^2 n^2/2p(1-p)}\right)^{k-1}.
\end{split}
	\end{equation}
	Then \eqref{LemmaBP2Step3} implies \eqref{6.2step4} for $N\geq\bar{N}_3 = \max(\bar{N}_2,\tilde{N}_9)$.
	
	To prove \eqref{6.2step4sep} we first observe that we can write
	\begin{equation}\label{6.2step4disj}
	D(\vec c, \vec d, V^{top}, \epsilon,t_{12}) = \bigsqcup_{(\vec{a},\vec{b})\in J} G(\vec{a},\vec{b}).
	\end{equation}
	Here, for $\vec{a},\vec{b}\in\mathfrak{W}_{k-1}$, $G(\vec{a},\vec{b})$ is the event that $\mathfrak{Q}(-t_{12}) = \vec{a}$ and $\mathfrak{Q}(t_{12}) = \vec{b}$, and $J$ is the collection of $(\vec{a},\vec{b})$ satisfying
	\begin{enumerate}[label = (\arabic*)]
		
		\item $ 0 \leq \min(a_i - c_i,\, d_i - b_i) \leq t_2 - t_{12}$ and $0\leq b_i-a_i \leq 2t_{12}$ for $1\leq i\leq k-1$,
		
		\item $\min(a_{k-1} + pt_1,\, b_{k-1} - pt_1) \geq (M_2+1)(2t_3)^{1/2}$,
		
		\item $\max(a_1 + pt_1,\, b_1 - pt_1) \leq V^{top}(2t_3)^{1/2}$,
		
		\item $\min(a_i-a_{i+1}, \, b_i-b_{i+1}) \geq 3\epsilon(2t_3)^{1/2}$ for $1\leq i\leq k-2$.
		
	\end{enumerate}
	We let $\tilde{J} = \{(\vec{a},\vec{b})\in J : \mathbb{P}^{-t_2,t_2,\vec{c},\vec{d},\infty,\ell_{bot}}_{avoid,Ber;\tilde S}(G(\vec{a},\vec{b})) > 0\}$, and we take $\tilde{N}_9$ large enough by Lemma \ref{LemmaWD} so that $\tilde{J}\neq\varnothing$. We also let $\tilde{D}(V^{top},\epsilon,t_1)$ denote the set consisting of elements of $D(\vec{c},\vec{d},V^{top},\epsilon,t_1)$ restricted to $\llbracket -t_{12},t_{12}\rrbracket$. Then for $(\vec{a},\vec{b})\in\tilde{J}$ we have
	\begin{equation}\label{6.2step4cond}
	\begin{split}
	&\mathbb{P}^{-t_2,t_2,\vec{c},\vec{d},\infty,\ell_{bot}}_{avoid,Ber;\tilde S}\left(D(\vec{c},\vec{d},V^{top},\epsilon,t_1) \,\big|\,G(\vec{a},\vec{b})\right) = \mathbb{P}^{-t_{12},t_{12},\vec{a},\vec{b},\infty,\ell_{bot}}_{avoid,Ber;\tilde S}\left(\tilde D(V^{top},\epsilon,t_1)\right) \geq \\
	&\mathbb{P}^{-t_{12},t_{12},\vec{a},\vec{b}}_{Ber}\left(\tilde D(V^{top},\epsilon,t_1)\cap \{L_1\geq\cdots\geq L_{k-1}\geq\ell_{bot}\}\right).
	\end{split}
	\end{equation}
	We observe that the event in the second line of \eqref{6.2step4cond} occurs as long as each curve $L_i$ remains within a distance of $\epsilon(2t_3)^{1/2}$ from the straight line segment connecting $a_i$ and $b_i$ on $[-t_{12},t_{12}]$, for $1\leq i\leq k-2$. By the argument in the proof of Lemma \ref{CurveSeparation}, we can enlarge $\tilde{N}_9$ so that the probability of this event is bounded below by the expression on the right in \eqref{6.2step4sep} for $N\geq\tilde{N}_9$. Then using \eqref{6.2step4cond} and \eqref{6.2step4disj} and summing over $\tilde{J}$ implies \eqref{6.2step4sep}.\\
	
	{\bf \raggedleft Step 5.} In this last step, we complete the proof of the lemma, fixing the constants $g$ and $h$ as well as $N_5$. Let $g=g(\epsilon,V^{top},M_2)$ be as in Lemma \ref{LemmaBP1} for the choices of $\epsilon,V^{top}$ in Steps 2 and 3, let
	\[
	h = \frac{h_1}{2}\left(\frac{1}{2} - \sum_{n=1}^\infty (-1)^{n-1} e^{-\epsilon^2 n^2/2p(1-p)}\right)^{k-1}
	\]
	with $h_1$ as in Step 2, and let $N_5 = \max(\bar{N}_0,\bar{N}_1,\bar{N}_2,\bar{N}_3,N_7)$, with $N_7$ as in Lemma \ref{LemmaBP2}. In the following we assume that $N\geq N_5$. By \eqref{6.2step4} we have that if $(\vec{c},\vec{d})\in E$ and $N\geq N_5$, then
	$$\mathbb{P}_{avoid, Ber;\tilde{S}}^{-t_2, t_2, \vec{c}, \vec{d}, \infty, \ell_{bot}} ( H) \geq \frac{h}{h_1},$$
	where $H$ is the event that
	
	\begin{enumerate}
		\item $V^{top} (2t_3)^{1/2} \geq \tilde Q_1(-t_1) + p t_1 \geq \tilde Q_{k-1}(-t_1) + pt_1 \geq (M_2 + 1) (2t_2)^{1/2}$,
		\item $V^{top} (2t_3)^{1/2} \geq \tilde Q_1(t_1) - p t_1 \geq \tilde Q_{k-1}(t_1) - pt_1 \geq (M_2 + 1) (2t_3)^{1/2}$,
		\item $\tilde Q_i(-t_1) - \tilde Q_{i+1}(-t_1) \geq 3\epsilon (2t_2)^{1/2}$ and $\tilde Q_i(t_1) - \tilde Q_{i+1}(t_1)  \geq 3 \epsilon (2t_2)^{1/2}$ for $i = 1, \dots, k-2$.
	\end{enumerate}
	Let $Y$ denote the event appearing in \eqref{eqnRT2}. Then we can write $Y = \bigsqcup_{(\vec{c},\vec{d})\in E} Y(\vec{c},\vec{d})$, where $Y(\vec{c},\vec{d})$ is the event that $\tilde{\mathfrak{Q}}(-t_2) = \vec{c}$, $\tilde{\mathfrak{Q}}(t_2) = \vec{d}$, and $E$ is defined in Step 2. If $\tilde{E} = \{(\vec{c},\vec{d})\in E : \mathbb{P}_{\tilde{\mathfrak{Q}}}(Y(\vec{c},\vec{d})) > 0\}$, we can assume by Lemma \ref{LemmaWD} that $N_5$ is large enough so that $\tilde{E}\neq\varnothing$. It follows from Lemma \ref{LemmaBP2} that $\mathbb{P}_{\tilde{\mathfrak Q}}(Y) \geq h_1$. We conclude from the definition of $\mathbb{P}_{\tilde{\mathfrak{Q}}}$ that for all $N\geq N_5$,
	\begin{align*}
	&\mathbb{P}_{\tilde{\mathfrak{Q}}}(H) \geq \mathbb{P}_{\tilde{\mathfrak{Q}}}(H\cap Y) = \sum_{(\vec{c},\vec{d})\in \tilde E} \mathbb{P}_{\tilde{\mathfrak{Q}}}(Y(\vec{c},\vec{d}))\cdot \mathbb{P}_{\tilde{\mathfrak{Q}}}(H\,|\,Y(\vec{c},\vec{d})) =\\
	&\sum_{(\vec{c},\vec{d})\in \tilde E} \mathbb{P}_{\tilde{\mathfrak{Q}}}(Y(\vec{c},\vec{d}))\cdot \mathbb{P}^{-t_2,t_2,\vec{c},\vec{d},\infty,\ell_{bot}}_{avoid,Ber;\tilde S}(H) \geq \frac{h}{h_1}\sum_{(\vec{c},\vec{d})\in \tilde E} \mathbb{P}_{\tilde{\mathfrak{Q}}}(Y(\vec{c},\vec{d})) = \frac{h}{h_1}\,\mathbb{P}_{\tilde{\mathfrak{Q}}}(Y) \geq h.
	\end{align*}
	Now Lemma \ref{LemmaBP1} implies \eqref{eqn57}, completing the proof.
\end{proof}

%

\section{Appendix A} \label{Section8}

In this section we prove Lemmas \ref{Polish}, \ref{2Tight},  \ref{MCLxy} and \ref{MCLfg}. 

%
\subsection{Proof of Lemma \ref{Polish}}\label{Section8.1} We adopt the same notation as in the statement of Lemma \ref{Polish} and proceed with its proof. \\

Observe that the sets $K_1\subset K_2\subset\cdots\subset\Sigma\times\Lambda$ are compact, they cover $\Sigma\times\Lambda$, and any compact subset $K$ of $\Sigma\times\Lambda$ is contained in all $K_n$ for sufficiently large $n$. To see this last fact, let $\pi_1,\pi_2$ denote the canonical projection maps of $\Sigma\times\Lambda$ onto $\Sigma$ and $\Lambda$ respectively. Since these maps are continuous, $\pi_1(K)$ and $\pi_2(K)$ are compact in $\Sigma$ and $\Lambda$. This implies that $\pi_1(K)$ is finite, so it is contained in $\Sigma_{n_1} = \Sigma\cap\llbracket -n_1,n_1\rrbracket$ for some $n_1$. On the other hand, $\pi_2(K)$ is closed and bounded in $\mathbb{R}$, thus contained in some closed interval $[\alpha,\beta]\subseteq\Lambda$. Since $a_n\searrow a$ and $b_n\nearrow b$, we can choose $n_2$ large enough so that $\pi_2(K)\subseteq[\alpha,\beta]\subseteq[a_{n_2},b_{n_2}]$. Then taking $n=\max(n_1,n_2)$, we have $K \subseteq \pi_1(K) \times \pi_2(K) \subseteq \Sigma_n \times [a_n,b_n] = K_n$.

We now split the proof into several steps.\\

\noindent\textbf{Step 1.} In this step, we show that the function $d$ defined in the statement of the lemma is a metric. For each $n$ and $f,g\in C(\Sigma\times\Lambda)$, we define
\[
d_n(f,g) = \sup_{(i,t)\in K_n} |f(i,t)-g(i,t)|,\quad d_n'(f,g) = \min\{d_n(f,g), 1\} 
\]
Then we have
\[
d(f,g) = \sum_{n=1}^\infty 2^{-n} d_n'(f,g).
\]
Clearly each $d_n$ is nonnegative and satisfies the triangle inequality, and it is then easy to see that the same properties hold for $d_n'$. Furthermore, $d_n'\leq 1$, so $d$ is well-defined and $d(f,g) \in[0,1]$. Observe that $d$ is nonnegative, and if $f=g$, then each $d_n'(f,g)=0$, so the sum $d(f,g)$ is 0. Conversely, if $f\neq g$, then since the $K_n$ cover $\Sigma\times\Lambda$, we can choose $n$ large enough so that $K_n$ contains an $x$ with $f(x)\neq g(x)$. Then $d_n'(f,g)\neq 0$, and hence $d(f,g)\neq 0$. Lastly, the triangle inequality holds for $d$ since it holds for each $d_n'$.\\

\noindent\textbf{Step 2.} Now we prove that the topology $\tau_d$ on $C(\Sigma\times\Lambda)$ induced by $d$ is the same as the topology of uniform convergence over compacts, which we denote by $\tau_c$. Recall that $\tau_c$ is generated by the basis consisting of sets
\[
B_K(f,\epsilon) = \Big\{g\in C(\Sigma\times\Lambda) : \sup_{(i,t)\in K} |f(i,t) - g(i,t)| < \epsilon \Big\},
\]
for $K\subset\Sigma\times\Lambda$ compact, $f\in C(\Sigma\times\Lambda)$, and $\epsilon>0$, and $\tau_d$ is generated by sets of the form $B^d_\epsilon(f) = \{g:d(f,g) < \epsilon\}$. 

We first show that $\tau_d \subseteq \tau_c$. It suffices to prove that every set $B_\epsilon^d(f)$ is a union of sets $B_K(f,\epsilon)$. First, choose $\epsilon>0$ and $f\in C(\Sigma\times\Lambda)$. Let $g\in B^d_\epsilon(f)$. We will find a basis element $A_g$ of $\tau_c$ such that $g\in A_g\subset B^d_\epsilon(f)$. Let $\delta = d(f,g) < \epsilon$, and choose $n$ large enough so that $\sum_{k>n} 2^{-k} < \frac{\epsilon-\delta}{2}$. Define $A_g = B_{K_n}(g,\frac{\epsilon-\delta}{n})$, and suppose $h\in A_g$. Then since $K_m\subseteq K_n$ for $m\leq n$, we have
\begin{align*}
d(f,h) &\leq d(f,g) + d(g,h) \leq \delta + \sum_{k=1}^n 2^{-k}d_n(g,h) + \sum_{k>n} 2^{-k} < \delta + \frac{\epsilon-\delta}{2} + \frac{\epsilon-\delta}{2} = \epsilon.
\end{align*}
Therefore $g\in A_g\subset B^d_\epsilon(f)$. Then we can write
\[
B^d_\epsilon(f) = \bigcup_{g\in B^d_\epsilon(f)} A_g,
\]
a union of basis elements of $\tau_c$.

We now prove conversely that $\tau_c\subseteq\tau_d$. Let $K\subset\Sigma\times\Lambda$ be compact, $f\in C(\Sigma\times\Lambda)$, and $\epsilon>0$. Choose $n$ so that $K\subset K_n$, and let $g\in B_K(f,\epsilon)$ and $\delta = \sup_{x\in K} |f(x)-g(x)| < \epsilon$. If $d(g,h) < 2^{-n}(\epsilon-\delta)$, then $d_n'(g,h) \leq 2^n d(g,h) < \epsilon-\delta$, hence $d_n(g,h) < \epsilon-\delta$, assuming without loss of generality that $\epsilon \leq 1$. It follows that
\begin{align*}
\sup_{x\in K} |f(x)-h(x)| &\leq \delta + \sup_{x\in K} |g(x)-h(x)| \leq \delta + d_n(g,h) < \delta + \epsilon-\delta = \epsilon.
\end{align*}
Thus $g\in B^d_{2^{-n}(\epsilon-\delta)}(g) \subset B_K(f,\epsilon)$, proving that $B_K(f,\epsilon)\in\tau_d$ by the same argument as above. We conclude that $\tau_d = \tau_c$.\\

\noindent\textbf{Step 3.} In this step, we show that $(C(\Sigma\times\Lambda), d)$ is a complete metric space. Let $\{f_n\}_{n\geq 1}$ be Cauchy with respect to $d$. Then we claim that $\{f_n\}$ must be Cauchy with respect to $d_n'$, on each $K_n$. This follows from the observation that $ d_n'(f_\ell, f_m) \leq 2^n d(f_\ell, f_m)$. Thus $\{f_n\}$ is Cauchy with respect to the uniform metric on each $K_n$, and hence converges uniformly to a continuous limit $f^{K_n}$ on each $K_n$ (see \cite[Theorem 7.15]{Rudin}). Since the pointwise limit must be unique at each $x\in \Sigma\times\Lambda$, we have $f^{K_n}(x) = f^{K_m}(x)$ if $x\in K_n\cap K_m$. Since $\cup_n K_n = \Sigma\times\Lambda$, we obtain a well-defined function $f$ on all of $\Sigma\times\Lambda$ given by $f(x)=\lim_{n\to\infty} f^{K_n}(x)$. We have $f\in C(\Sigma\times\Lambda)$ since $f|_{K_n} = f^{K_n}$ is continuous on $K_n$ for all $n$. Moreover, if $K\subset\Sigma\times\Lambda$ is compact and $n$ is large enough so that $K\subset K_n$, then because $f_n \to f^{K_n} = f|_{K_n}$ uniformly on $K_n$, we have $f_n \to f^{K_n}|_K = f|_K$ uniformly on $K$. That is, for any $K\subset\Sigma\times\Lambda$ compact and $\epsilon>0$, we have $f_n \in B_K(f,\epsilon)$ for all sufficiently large $n$. Therefore $f_n \to f$ in $\tau_c$, and equivalently in the metric $d$ by Step 2.\\

\noindent\textbf{Step 4.} Lastly, we prove separability by adapting the arguments from \cite[Example 1.3]{Billing}. For each pair of positive integers $n,k$, let $D_{n,k}$ be the subcollection of $C(\Sigma\times\Lambda)$ consisting of polygonal functions that are piecewise linear on $\{j\}\times I_{n,k,i}$ for each $j\in\Sigma_n$ and each subinterval 
\[
I_{n,k,i} = \big[a_n+\tfrac{i-1}{k}(b_n-a_n), \, a_n+\tfrac{i}{k}(b_n-a_n)\big], \quad 1\leq i\leq k,
\] 
taking rational values at the endpoints of these subintervals, and extended constantly to all of $\Lambda$. Then $D = \cup_{n,k} D_{n,k}$ is countable, and we claim that it is dense in $\tau_c$. To see this, let $K\subset\Sigma\times\Lambda$ be compact, $f\in C(\Sigma\times\Lambda)$, and $\epsilon>0$, and choose $n$ so that $K\subset K_n$. Since $f$ is uniformly continuous on $K_n$, we can choose $k$ large enough so that for $0\leq i\leq k$, if $t\in I_{n,k,i}$, then 
\[
\big|f(j,t) - f(j, a_n + \tfrac{i}{k}(b_n-a_n))\big| < \epsilon/2
\]
for all $j\in\Sigma_n$. Using that $\mathbb{Q}$ is dense in $\mathbb{R}$ we can choose $g\in \cup_k D_{n,k}$ with $|g(j,a_n + \frac{i}{k}(b_n-a_n)) - f(j,a_n + \frac{i}{k}(b_n-a_n))| < \epsilon/2$. Then we have 
\begin{align*}
	\big|f(j,t) - g(j, a_n + \tfrac{i-1}{k}(b_n-a_n))\big| < \epsilon \quad \mathrm{and} \quad \big|f(j,t) - g(j, a_n + \tfrac{i}{k}(b_n-a_n))\big| < \epsilon.
\end{align*}
Since $g(j,t)$ is a convex combination of $g(j, a_n + \tfrac{i-1}{k}(b_n-a_n))$ and $g(j, a_n + \tfrac{i}{k}(b_n-a_n))$, we get
\[
|f(j,t) - g(j,t)| < \epsilon
\]
as well. In summary,
\[
\sup_{(j,t)\in K} |f(j,t)-g(j,t)| \leq \sup_{(j,t)\in K_n} |f(j,t)-g(j,t)| < \epsilon,
\] 
so $g\in B_K(f,\epsilon)$. This proves that $D$ is a countable dense subset of $C(\Sigma\times\Lambda)$.

%
\subsection{Proof of Lemma \ref{2Tight}}\label{Section8.2}

We first prove two lemmas that will be used in the proof of Lemma \ref{2Tight}. The first result allows us to identify the space $C(\Sigma\times\Lambda)$ with a product of copies of $C(\Lambda)$. In the following, we assume the notation of Lemma \ref{2Tight}.

\begin{lemma}\label{ProdTop}
	Let $\pi_i: C (\Sigma \times \Lambda) \rightarrow C(\Lambda)$, $i \in \Sigma$, be the projection maps given by
	$\pi_i(F)(x) = F(i, x)$ for $x \in \Lambda$. Then the $\pi_i$ are continuous. Endow the space $\prod_{i\in\Sigma} C(\Lambda)$ with the product topology induced by the topology of uniform convergence over compacts on $C(\Lambda)$. Then the mapping
	\begin{align*}
		F : C(\Sigma\times\Lambda) \longrightarrow \prod_{i\in\Sigma} C(\Lambda), \quad f\mapsto (\pi_i(f))_{i\in\Sigma}
	\end{align*}
	is a homeomorphism.
\end{lemma}

\begin{proof}
	We first prove that the $\pi_i$ are continuous. We know $C(\Sigma\times\Lambda)$ is metrizable by Lemma \ref{Polish}, and by a similar argument so is $C(\Lambda)$ (take $\Sigma = \{0\}$ in Lemma \ref{Polish}). Consequently, it suffices to assume that $f_n\to f$ in $C(\Sigma\times\Lambda)$ and show that $\pi_i(f_n)\to \pi_i(f)$ in $C(\Lambda)$. Let $K$ be compact in $\Lambda$. Then $\{i\}\times K$ is compact in $\Sigma\times\Lambda$, and $f_n\to f$ uniformly on $\{i\}\times K$ by assumption, so we have $\pi_i(f_n)|_K = f_n|_{\{i\}\times K} \to f|_{\{i\}\times K} = \pi_i(f)|_K$ uniformly on $K$. Since $K$ was arbitrary, we conclude that $\pi_i(f_n) \to \pi_i(f)$ in $C(\Lambda)$ as desired. 
	
We now observe that $F$ is invertible. If $(f_i)_{i\in\Sigma} \in \prod_{i\in\Sigma} C(\Lambda)$, then the function $f$ defined by $f(i,\cdot) = f_i(\cdot)$ is in $C(\Sigma\times\Lambda)$, since $\Sigma$ has the discrete topology. This gives a well-defined inverse for $F$. It suffices to prove that $F$ and $F^{-1}$ are open maps.
	
We first show that $F$ sends each basis element $B_K(f,\epsilon)$ of $C(\Sigma\times\Lambda)$ to a basis element in $\prod_{i\in\Sigma} C(\Lambda)$. Note that a basis for the product topology is given by products $\prod_{i\in\Sigma} B_{K_i}(f_i,\epsilon)$, where at most finitely many of the $K_i$ are nonempty. Here, we use the convention that $B_{\varnothing}(f_i,\epsilon) = C(\Lambda)$. Let $\pi_\Sigma,\pi_\Lambda$ denote the canonical projections of $\Sigma\times\Lambda$ onto $\Sigma,\Lambda$. The continuity of $\pi_\Sigma$ implies that if $K\subset\Sigma\times\Lambda$ is compact, then $\pi_\Sigma(K)$ is compact in $\Sigma$, hence finite. Observe that the set $K\cap(\{i\}\times\Lambda)$ is an intersection of a compact set with a closed set and is hence compact in $\Sigma\times\Lambda$. Therefore the sets $K_i = \pi_\Lambda(K\cap(\{i\}\times\Lambda))$ are compact in $\Lambda$ for each $i\in\Sigma$ since $\pi_\Lambda$ is continuous. We observe that $F(B_K(f,\epsilon)) = \prod_{i\in\Sigma} U_i$, where
	\[
	U_i = B_{K_i}(\pi_i(f),\epsilon), \quad\mathrm{if} \quad i \in \pi_\Sigma(K),
	\]
	and $U_i = C(\Lambda)$ otherwise. Since $\pi_\Sigma(K)$ is finite and the $K_i$ are compact, we see that $F(B_K(f,\epsilon))$ is a basis element in the product topology as claimed.
	
	Lastly, we show that $F^{-1}$ sends each basis element $U = \prod_{i\in\Sigma} B_{K_i}(f_i,\epsilon)$ for the product topology to a set of the form $B_K(f,\epsilon)$. We have $K_i=\varnothing$ for all but finitely many $i$. Write $f = F^{-1}((f_i)_{i\in\Sigma})$ and $K=\cup_{i \in \Sigma} ( \{i\} \times K_i)$. Notice that $K$ is compact in $\Sigma\times\Lambda$ as a finite union of compact sets (each of $\{i\} \times K_i$ is compact by Tychonoff's theorem, \cite[Theorem 37.3]{Munkres}). Moreover, one has
	\[
	F^{-1}(U) = B_K(f,\epsilon),
	\]
which proves that $F^{-1}$ is also an open map.
	
\end{proof}

We next prove a lemma which states that a sequence of line ensembles is tight if and only if all individual curves form tight sequences.

\begin{lemma}\label{ProjTight}
	Suppose that $\{\mathcal{L}^n\}_{n\geq 1}$ is a sequence of $\Sigma$-indexed line ensembles on $\Lambda$, and let $X_i^n = \pi_i(\mathcal{L}^n)$. Then the $X_i^n$ are $C(\Lambda)$-valued random variables on $(\Omega,\mathcal{F},\mathbb{P})$, and $\{\mathcal{L}^n\}$ is tight if and only if for each $i \in \Sigma$ the sequence $\{X_i^n\}_{n\geq 1}$ is tight.
	
\end{lemma}

\begin{proof}

The fact that the $X_i^n$ are random variables follows from the continuity of the $\pi_i$ in Lemma \ref{ProdTop} and \cite[Theorem 1.3.4]{Durrett}. First suppose the sequence $\{\mathcal{L}^n\}$ is tight. By Lemma \ref{Polish}, $C(\Sigma\times\Lambda)$ is a Polish space, so it follows from Prohorov's theorem, \cite[Theorem 5.1]{Billing}, that $\{\mathcal{L}^n\}$ is relatively compact. That is, every subsequence $\{\mathcal{L}^{n_k}\}$ has a further subsequence $\{\mathcal{L}^{n_{k_\ell}}\}$ converging weakly to some $\mathcal{L}$. Then for each $i\in\Sigma$, since $\pi_i$ is continuous by the above, the subsequence $\{\pi_i(\mathcal{L}^{n_{k_\ell}})\}$ of $\{\pi_i(\mathcal{L}^{n_k})\}$ converges weakly to $\pi_i(\mathcal{L})$ by the Continuous mapping theorem, \cite[Theorem 2.7]{Billing}. Thus every subsequence of $\{\pi_i(\mathcal{L}^n)\}$ has a convergent subsequence. Since $C(\Lambda)$ is a Polish space (apply Lemma \ref{Polish} with $\Sigma = \{0\}$), Prohorov's theorem, \cite[Theorem 5.2]{Billing}, implies $\{\pi_i(\mathcal{L}^n)\}$ is tight.

Conversely, suppose $\{X_i^n\}$ is tight for all $i\in\Sigma$. Then given $\epsilon > 0$, we can find compact sets $K_i\subset C(\Lambda)$ such that
\[
\mathbb{P}(X_i^n \notin K_i) \leq \epsilon/2^i
\]
for each $i\in\Sigma$. By Tychonoff's theorem, \cite[Theorem 37.3]{Munkres}, the product $\tilde{K} = \prod_{i\in\Sigma} K_i$ is compact in $\prod_{i\in\Sigma} C(\Lambda)$. We have
\begin{equation}\label{tychonoff}
\mathbb{P}\big((X_i^n)_{i\in\Sigma} \notin \tilde{K} \big) \leq \sum_{i\in\Sigma} \mathbb{P}(X_i^n \notin K_i) \leq \sum_{i=1}^\infty \epsilon/2^i = \epsilon.
\end{equation}
By Lemma \ref{ProdTop}, we have a homeomorphism $G : \prod_{i\in\Sigma} C(\Lambda) \to C(\Sigma\times\Lambda)$. We observe that $G((X_i^n)_{i\in\Sigma}) = \mathcal{L}^n$, and $K = G(\tilde{K})$ is compact in $C(\Sigma\times\Lambda)$. Thus $\mathcal{L}^n \in K$ if and only if $(X_i^n)_{i\in\Sigma} \in \tilde{K}$, and it follows from \eqref{tychonoff} that
\[
\mathbb{P}(\mathcal{L}^n \in K) \geq 1 - \epsilon.
\]
This proves that $\{\mathcal{L}^n\}$ is tight.

\end{proof}

We are now ready to prove Lemma \ref{2Tight}.

\begin{proof} (of Lemma \ref{2Tight}) By a direct extension of \cite[Theorem 7.3]{Billing}, a sequence $\{P_n\}_{n \geq 1}$ of probability measures on $C([u,v])$ with the uniform topology is tight if and only if the following conditions hold:
\begin{equation}\label{S8iop}
\begin{split}
&\mbox{ for some $w \in [u,v]$ we have }\lim_{a\to\infty} \limsup_{n\to\infty} P_n(|x(w)|\geq a) = 0, \\
&\lim_{\delta\to 0} \limsup_{n\to\infty} P_n\Big(\sup_{|s-t|\leq\delta} |x(s)-x(t)| \geq \epsilon\Big) = 0 \quad \textrm{for all}\;\epsilon>0.
\end{split}
\end{equation}
If $\{ \mathcal{L}^n\}_{n \geq 1}$ is tight we conclude the sequence $\{\mathcal{L}_i^n|_{[a_m,b_m]}\}_{n \geq 1}$ is tight for every $m\geq 1$. This is because the projection map is continuous. The last two statements prove the ``only if'' part of the lemma. In the remainder we focus on the ``if'' part, i.e. proving that $\{ \mathcal{L}^n\}_{n \geq 1}$ is tight, given that conditions (i) and (ii) in the statement of the lemma are satisfied.\\

 Fix $i\in\Sigma$. By Lemma \ref{ProjTight}, it suffices to show that the sequence $\{\mathcal{L}_i^n\}_{n\geq 1}$ of $C(\Lambda)$-valued random variables is tight. From (\ref{S8iop}) we see that conditions (i) and (ii)  in the lemma imply that the sequence $\{\mathcal{L}_i^n|_{[a_m,b_m]}\}_{n \geq 1}$ is tight for every $m\geq 1$. Let $\pi_m : C(\Lambda) \to C([a_m,b_m])$ denote the map $f \mapsto f|_{[a_m,b_m]}$ and note that $\pi_m$ is continuous. It follows from \cite[Theorem 1.3.4]{Durrett} that $\pi_m(\mathcal{L}^n) = \mathcal{L}_i^n|_{[a_m,b_m]}$ is a $C([a_m,b_m])$-valued random variable. Tightness of the sequence implies that for any $\epsilon > 0$, we can find compact sets $K_m\subset C([a_m,b_m])$ so that
\[
\mathbb{P}\big(\pi_m(\mathcal{L}_i^n) \notin K_m \big) \leq \epsilon/2^m
\]
for each $m\geq 1$. Writing $K = \cap_{m=1}^\infty \pi_m^{-1}(K_m)$, it follows that
\[
\mathbb{P}\big(\mathcal{L}^n_i \in K\big) \geq 1 - \sum_{m=1}^\infty \epsilon/2^m = 1 - \epsilon.
\]
To conclude tightness of $\{\mathcal{L}_i^n\}$, it suffices to prove that $K = \cap_{m=1}^\infty \pi_m^{-1}(K_m)$ is sequentially compact in $C(\Lambda)$. We argue by diagonalization. Let $\{f_n\}$ be a sequence in $K$, so that $f_n|_{[a_m,b_m]} \in K_m$ for every $m,n$. Since $K_1$ is compact, there is a sequence $\{n_{1,k}\}$ of natural numbers such that the subsequence $\{f_{n_{1,k}}|_{[a_1,b_1]}\}_k$ converges in $C([a_1,b_1])$. Since $K_2$ is compact, we can take a further subsequence $\{n_{2,k}\}$ of $\{n_{1,k}\}$ so that $\{f_{n_{2,k}}|_{[a_2,b_2]}\}_k$ converges in $C([a_2,b_2])$. Continuing in this manner, we obtain sequences $\{n_{1,k}\} \supseteq \{n_{2,k}\} \supseteq\cdots$ so that $\{f_{n_{m,k}}|_{[a_m,b_m]}\}_k$ converges in $C([a_m,b_m])$ for all $m$. Writing $n_k = n_{k,k}$, it follows that the sequence $\{f_{n_k}\}$ converges uniformly on each $[a_m,b_m]$. If $K$ is any compact subset of $C(\Lambda)$, then $K\subset [a_m,b_m]$ for some $m$, and hence $\{f_{n_k}\}$ converges uniformly on $K$. Therefore $\{f_{n_k}\}$ is a convergent subsequence of $\{f_n\}$.

\end{proof}

%
\subsection{Proof of Lemma \ref{LemmaWD}}\label{LemmaWDProof} We adopt the same notation as in the statement of Lemma \ref{LemmaWD} and proceed with its proof.\\

We first construct a candidate $\mathfrak{B}$ and then we prove that $\mathfrak{B} \in \Omega_{avoid}(T_0,T_1,\vec x,\vec y, f,g)$. Denote $B_0=f$ and $B_{k+1}=g$ with $x_0=f(T_0)$ and $y_0=f(T_1)$. By Condition (3) of Lemma \ref{LemmaWD} we know $x_0\geq x_1$ and $y_0\geq y_1$. We define inductively $B_j$ for $j = 1,\dots, k$ as follows (recall that $B_0 = f$). Assuming that $B_{j-1}$ has been constructed we let $B_j(T_0)=x_j$ and then for $i\in \llbracket T_0, T_{1}-1\rrbracket$ we define
\begin{equation}\label{MaxEns}
B_j(i+1)=\begin{cases}
B_j(i)+1 & \text{if } B_j(i)+1\leq \min\{B_{j-1}(i+1),y_j\}\\
B_j(i) &\text{else.}
\end{cases}
\end{equation} 
This gives our candidate $\mathfrak{B} = (B_1, \dots, B_k)$. In order to verify that this candidate ensemble $\mathfrak{B}$ is an element of $\Omega_{avoid}(T_0,T_1,\vec x,\vec y,f,g)$, three properties must be ensured:
\begin{equation}\label{nonEmpCond}
\begin{split}
\text{(a) } &\mathfrak{B}(T_0)=\vec x\text{ and }\mathfrak{B}(T_1)=\vec y\\
\text{(b) } &f(i)\geq B_1(i)\geq \cdots \geq B_k(i)\geq g(i)\text{ for all }i\in \llbracket T_0,T_1\rrbracket\\
\text{(c) } &B_{j}(i+1)-B_j(i)\in \{0,1\}\text{ for all }i\in \llbracket T_0,T_1-1\rrbracket\text{ and }j\in \llbracket 1,k\rrbracket
\end{split}
\end{equation}
 Property (c) follows directly from our definition in (\ref{MaxEns}). We split the proof of (a) and (b) above into three steps.\\

{\bf \raggedleft Step 1.} In this step we prove that for each $j = 1, \dots, k$ that $B_{j-1}(i) \geq B_j(i)$ for $i \in \llbracket T_0, T_1\rrbracket$. If $j = 1$ and $f \equiv \infty$ there is nothing to prove, so we may assume that either $ j \geq 2$ or $j = 1$ and $f$ is an up-right path -- the proofs in these cases are the same.  Suppose that for some $i \in \llbracket T_0, T_1 - 1\rrbracket$ we have that $B_{j}(i) \leq B_{j-1}(i)$ then we know by construction that $B_j(i+1) = B_{j}(i)$ or $B_{j}(i) +1$. In the former case, we trivially get
$$B_{j}(i+1) = B_j(i) \leq B_{j-1}(i) \leq B_{j-1}(i+1),$$
where the last inequality used that $B_{j-1}$ is an up-right path. If $B_j(i+1) =B_{j}(i) +1$ from (\ref{MaxEns}) we see that $B_{j}(i) +1 \leq B_{j-1}(i+1)$ and so we again conclude that $B_{j}(i+1) \leq  B_{j-1}(i+1)$. By assumption we know that $B_j(T_0) = x_j \leq x_{j-1} = B_{j-1}(T_0),$ and so by inducting on $i$ from $T_0$ to $T_1$ we conclude that $B_{j-1}(i) \geq B_j(i)$ for $i \in \llbracket T_0, T_1\rrbracket$ and $j = 1, \dots, k$. To summarize, we have proved that for $i\in \llbracket T_0,T_1\rrbracket$
\begin{equation}\label{NonEmpPartialIneq}
 f(i)\geq B_1(i)\geq \cdots \geq B_k(i).
\end{equation}

{\bf \raggedleft Step 2.} In this step we prove (a). By construction we already know that $\mathfrak{B}(T_0) = \vec{x}$ and so we only need to prove that $\mathfrak{B}(T_1)=\vec y$. We will show this claim inductively on $j$: we trivially know the claim is true for $j=0$, since $y_0=f(T_1)$ is given. Then suppose that $B_j(T_1)=y_j$ holds up to $j=n-1$. We seek to prove that $B_n(T_1) = y_n$. Notice that by construction we know that $B_n(i) \leq y_n$ for all $i \in \llbracket T_0, T_1\rrbracket$ and so we only need to show that $B_n(T_1) \geq y_n$.

Suppose first that $B_n(i +1) = B_n(i) + 1$ for all $i \in \llbracket T_0, T_1 - 1\rrbracket$. Then we know that $B_n(T_1) = x_n + (T_1 - T_0) \geq y_n$ by assumption (1) in Lemma \ref{LemmaWD}, and so we are done. Conversely, there is an $i_0 \in \llbracket T_0, T_1 - 1\rrbracket$ such that $B_n(i_0+1) = B_n(i_0)$ and we can take $i_0$ to be the largest index in $\llbracket T_0, T_1 - 1\rrbracket$ satisfying this condition. Observe that by (\ref{MaxEns}) we must have that either $B_n(i_0) \geq y_n$ or $  B_n(i_0) \geq B_{n-1}(i_0 + 1) $. In the former case, we see that since $B_n$ is an up-right path we must have $B_n(T_1) \geq B_n(i_0) \geq y_n$ and again we are done. Thus we only need to consider the case when $B_{n-1}(i_0 + 1) \leq B_n(i_0).$ By the maximality of $i_0$ we know that $B_n(i+1) = B_n(i) + 1$ for $i = i_0 + 1, \dots,T_1$ and so we see that 
$$B_n(T_1) = B_n(i_0+1) + (T_1 - i_0 - 1) = B_n(i_0) + (T_1 - i_0 - 1) \geq $$
$$B_{n-1}(i_0 + 1) + (T_1 - i_0 - 1) \geq B_{n-1} (T_1) = y_{n-1} \geq y_n.$$
Overall, we conclude in all cases that $B_n(T_1) \geq y_n$ which concludes the proof of (a).\\

{\bf \raggedleft Step 3.} In this step we prove (b), and in view of (\ref{NonEmpPartialIneq}) we see that it suffices to show that $B_{k}(i)\geq g(i)$ for all $i$. If $g \equiv -\infty$ there is nothing to prove and so we may assume that $g$ is an up-right path.

 Suppose that $g(i)>B_k(i)$ for some $i\in \llbracket T_0,T_1\rrbracket$. Since $g(T_0)\leq B_k(T_0)=x_k$ by Condition (3) in Lemma \ref{LemmaWD}, we know that there exists some point $i_0$ such that $g(i_0)=B_k(i_0)$ and $g(i_0+1)>B_k(i_0+1)$. In particular, since $g$ and $B_k$ can each only increase by $1$, this implies $B_k(i_0)=B_k(i_0+1)$ and $g(i_0 + 1) = g(i_0)  + 1$. This implies either $B_k(i_0)=y_k$ or $B_k(i_0)+1>B_{k-1}(i_0+1)$. If $B_k(i_0)=y_k$ then by assumption (3) of Lemma \ref{LemmaWD} we conclude
$$y_k \geq g(T_1) \geq g(i_0+1) = g(i_0) +1 = B_k(i_0) + 1\geq  y_k + 1,$$
which is an obvious contradiction. 
		
Therefore, it must be the case that $B_k(i_0)+1>B_{k-1}(i_0+1)$ and then we conclude that $B_{k-1}(i_0 + 1 ) = B_{k-1}(i_0 ) = B_k(i_0)$ in view of (\ref{NonEmpPartialIneq}). By the same argument we see that $B_{k-1}(i_0 + 1 ) = B_{k-1}(i_0 )$ can only occur if $B_{k-2}(i_0 + 1 ) = B_{k-2}(i_0 ) = B_{k-1}(i_0 )$ and iterating this $k$ times we conclude that $B_0(i_0+1) = B_0(i_0) = B_1(i_0) = \cdots = B_k(i_0) = g(i_0) = g(i_0+1) - 1$. But then $g(i_0 +1) > f(i_0+1)$, which contradicts condition (3) in Lemma \ref{LemmaWD}. The contradiction arose from our assumption that $g(i)>B_k(i)$ for some $i\in \llbracket T_0,T_1\rrbracket$ and so no such $i$ exists, proving (b).

%

\subsection{Proof of Lemmas \ref{MCLxy} and \ref{MCLfg}}

We will prove the following lemma, of which the two lemmas are immediate consequences. In particular, Lemma \ref{MCLxy} is the special case when $g^b = g^t$, and Lemma \ref{MCLfg} is the case when $\vec{x} = \vec{x}\,'$ and $\vec{y} = \vec{y}\,'$. We argue in analogy to \cite[Lemma 5.6]{DimMat}.

\begin{lemma}
	Fix $k \in \mathbb{N}$, $T_0, T_1 \in \mathbb{Z}$ with $T_0 < T_1$, $S\subseteq\llbracket T_0, T_1\rrbracket$, and two functions $g^b, g^t: \llbracket T_0, T_1 \rrbracket  \rightarrow [-\infty, \infty)$ with $g^b\leq g^t$ on $S$. Also fix $\vec{x}, \vec{y}, \vec{x}\,', \vec{y}\,' \in \mathfrak{W}_k$ such that $x_i\leq x_i'$, $y_i\leq y_i'$ for $1\leq i\leq k$. Assume that $\Omega_{avoid}(T_0, T_1, \vec{x}, \vec{y}, \infty,g^b; S)$ and $\Omega_{avoid}(T_0, T_1, \vec{x}\,', \vec{y}\,', \infty,g^t; S)$ are both non-empty. Then there exists a probability space $(\Omega, \mathcal{F}, \mathbb{P})$, which supports two $\llbracket 1, k \rrbracket$-indexed Bernoulli line ensembles $\mathfrak{L}^t$ and $\mathfrak{L}^b$ on $\llbracket T_0, T_1 \rrbracket$ such that the law of $\mathfrak{L}^{t}$ {\big (}resp. $\mathfrak{L}^b${\big )} under $\mathbb{P}$ is given by $\mathbb{P}_{avoid, Ber; S}^{T_0, T_1, \vec{x}\,', \vec{y}\,', \infty, g^t}$ {\big (}resp. $\mathbb{P}_{avoid, Ber; S}^{T_0, T_1, \vec{x}, \vec{y}, \infty, g^b}${\big )} and such that $\mathbb{P}$-almost surely we have ${L}_i^t(r) \geq {L}^b_i(r)$ for all $i = 1,\dots, k$ and $r \in \llbracket T_0, T_1 \rrbracket$.
\end{lemma}

\begin{proof} Throughout the proof, we will write $\Omega_{a,S}$ to mean $\Omega_{avoid}(T_0,T_1,\vec{x},\vec{y},\infty,g^b;S)$ and $\Omega_{a,S}'$ to mean $\Omega_{avoid}(T_0,T_1,\vec{x}\,',\vec{y}\,',\infty,g^t;S)$. We split the proof into two steps.\\
	
	\noindent\textbf{Step 1.} We first aim to construct a Markov chain $(X^n,Y^n)_{n\geq 0}$, with $X^n\in \Omega_{a,S}$, $Y^n\in \Omega_{a,S}'$, with initial distribution given by
	\begin{align*}
	& X^0_i(t) = \min(x_i+t-T_0, \, y_i), \qquad Y^0_i(t) = \min(x_i'+t-T_0, \, y_i'),
	\end{align*}
	for $t\in\llbracket T_0, T_1\rrbracket$ and $1\leq i\leq k$. First observe that we do in fact have $X^0 \in \Omega_{a,S}$, since $X_i^0(T_0) = x_i$, $X_i^0(T_1) = y_i$, $X_i^0(t) \leq \min(x_{i-1} + t - T_0, \, y_{i-1}) = X_{i-1}^0(t)$, and $X_k^0(t) \geq x_i + t - T_0 \geq g^b(T_0) + t - T_0 \geq g^b(t)$. We also note here that $X^0$ is \textit{maximal} on the entire space $\Omega(T_0,T_1,\vec{x},\vec{y})$, in the sense that for any $Z\in\Omega(T_0,T_1,\vec{x},\vec{y})$, we have $Z_i(t) \leq X^0_i(t)$ for all $t\in\llbracket T_0, T_1 \rrbracket$. In particular, $X^0$ is maximal on $\Omega_{a,S}$. Likewise, we see that $Y^0$ is maximal on $\Omega'_{a,S}$.
	
	We want the chain $(X^n,Y^n)$ to have the following properties: 
	\begin{enumerate}[label=(\arabic*)]
		
		\item $(X^n)_{n\geq 0}$ and $(Y^n)_{n\geq 0}$ are both Markov in their own filtrations,
		
		\item $(X^n)$ is irreducible and aperiodic, with invariant distribution $\mathbb{P}_{avoid,Ber;S}^{T_0,T_1,\vec{x},\vec{y},\infty,g^b}$,
		
		\item $(Y^n)$ is irreducible and aperiodic, with invariant distribution $\mathbb{P}_{avoid,Ber;S}^{T_0,T_1,\vec{x}\,',\vec{y}\,',\infty,g^t}$,
		
		\item $X^n_i\leq Y^n_i$ on $\llbracket T_0, T_1 \rrbracket$ for all $n\geq 0$ and $1\leq i \leq k$.
		
	\end{enumerate}
	
	\noindent This will allow us to conclude convergence of $X^n$ and $Y^n$ to these two uniform measures.
	
	We specify the dynamics of $(X^n, Y^n)$ as follows. At time $n$, we uniformly sample a triple $(i,t,z)\in \llbracket 1, k\rrbracket \times \llbracket T_0, T_1\rrbracket \times \llbracket x_k,y_1'-1\rrbracket$. We also flip a fair coin, with $\mathbb{P}(\textrm{heads})=\mathbb{P}(\textrm{tails})=1/2$. We update $X^n$ and $Y^n$ using the following procedure. If $j\neq i$, we leave $X_j,Y_j$ unchanged, and for all points $s\neq t$, we set $X^{n+1}_i(s) = X^n_i(s)$. If $T_0 < t < T_1$, $X^n_i(t-1)=z$, and $X^n_i(t+1)=z+1$ (note that this implies $X^n_i(t)\in\{z,z+1\}$), we consider two cases. If $t \in S$, then we set
	\[
	X^{n+1}_i(t) = \begin{cases}
	z+1, & \textrm{if heads},\\
	z, & \textrm{if tails},
	\end{cases}
	\]
	assuming this does not cause $X^{n+1}_i(t)$ to fall below $X^n_{i+1}(t)$, with the convention that $X^n_{k+1} = g^b$. If $t\notin S$, we perform the same update regardless of whether it results in a crossing. In all other cases, we leave $X^{n+1}_i(t)=X^n_i(t)$. We update $Y^n$ using the same rule, with $g^t$ in place of $g^b$.
	
	We first observe that $X^n$ and $Y^n$ are in fact non-crossing on $S$ for all $n$. Note $X^0$ is non-crossing, and if $X^n$ is non-crossing, then the only way $X^{n+1}$ could be crossing on $S$ is if the update were to push $X^{n+1}_i(t)$ below $X^n_{i+1}(t)$ for some $i,t$ with $t\in S$. But any update of this form is suppressed, so it follows by induction that $X^n\in\Omega_{a,S}$ for all $n$. Similarly, we see that $Y^n\in\Omega_{a,S}'$.
	
	It is easy to see that $(X^n,Y^n)$ is a Markov chain, since at each time $n$, the value of $(X^{n+1},Y^{n+1})$ depends only on the current state $(X^n,Y^n)$, and not on the time $n$ or any of the states prior to time $n$. Moreover, the value of $X^{n+1}$ depends only on the state $X^n$, not on $Y^n$, so $(X^n)$ is a Markov chain in its own filtration. The same applies to $(Y^n)$. This proves the property (1) above.
	
	We now argue that $(X^n)$ and $(Y^n)$ are irreducible. Fix any $Z \in \Omega_{a;S}$. As observed above, we have $Z_i \leq X^0_i$ on $\llbracket T_0,T_1\rrbracket$ for all $i$. We argue that we can reach the state $Z$ starting from $X^0$ in some finite number of steps with positive probability. Due to the maximality of $X^0$, we only need to move the paths downward. If we do this starting with the bottom path, then there is no danger of the paths $X_i$ crossing on $S$, or of $X_k$ crossing $g^b$ on $S$. To ensure that $X^n_k = Z_k$, we successively sample triples $(k,t,z)$ as follows. We initialize $t = T_0 + 1$. If $X_k^n(t) = Z_k(t)$, we increment $t$ by 1. Otherwise, we have $X_k^n(t) > Z_k(t)$, so we set $z = X_k^n(t) - 1$ and flip tails. This may or may not push $X_k(t)$ downwards by 1. We then increment $t$ and repeat this process. If $t$ reaches $T_1 - 1$, then at the increment we reset $t = T_0 + 1$. After finitely many steps, $X_k$ will agree with $Z_k$ on all of $\llbracket T_0, T_1 \rrbracket$. We then repeat this process for $X^n_i$ and $Z_i$, with $i$ descending. Since each of these samples and flips has positive probability, and this process terminates in finitely many steps, the probability of transitioning from $X^n$ to $Z$ after some number of steps is positive. The same reasoning applies to show that $(Y^n)$ is irreducible.
	
	To see that the chains are aperiodic, simply observe that if we sample a triple $(i,T_0,z)$ or $(i,T_1,z)$, then the states of both chains will be unchanged.
	
	To see that the uniform measure $\mathbb{P}_{avoid,Ber;S}^{T_0,T_1,\vec{x},\vec{y},\infty,g^b}$ on $\Omega_{a,S}$ is invariant for $(X^n)$, fix any $\omega\in\Omega_{a,S}$. For simplicity, write $\mu$ for the uniform measure. Then for all $\tau\in\Omega_{a,S}$, we have $\mu(\tau) = 1/|\Omega_{a,S}|$. Hence
	\begin{align*}
	& \sum_{\tau\in\Omega_{a,S}} \mu(\tau)\mathbb{P}(X^{n+1} = \omega\,|\,X^n = \tau) = \frac{1}{|\Omega_{a,S}|}\sum_{\tau\in\Omega_{a,S}} \mathbb{P}(X^{n+1} = \omega\,|\,X^n = \tau) = \\
	& \frac{1}{|\Omega_{a,S}|}\sum_{\tau\in\Omega_{a,S}} \mathbb{P}(X^{n+1} = \tau\,|\,X^n = \omega) = \frac{1}{|\Omega_{a,S}|}\cdot 1 = \mu(\omega).
	\end{align*}
	The second equality is clear if $\tau=\omega$. Otherwise, note that $\mathbb{P}(X_{n+1} = \omega\,|\,X_n = \tau) \neq 0$ if and only if $\tau$ and $\omega$ differ only in one indexed path (say the $i$th) at one point $t$, where $|\tau_i(t)-\omega_i(t)|=1$, and this condition is also equivalent to $\mathbb{P}(X^{n+1} = \tau\,|\,X^n = \omega) \neq 0$. If $X^n=\tau$, there is exactly one choice of triple $(i,t,z)$ and one coin flip which will ensure $X^{n+1}_i(t)=\omega(t)$, i.e., $X^{n+1}=\omega$. Conversely, if $X^n=\omega$, there is one triple and one coin flip which will ensure $X^{n+1}=\tau$. Since the triples are sampled uniformly and the coin flips are fair, these two conditional probabilities are in fact equal. This proves (2), and an analogous argument proves (3).
	
	Lastly, we argue that $X^n_i\leq Y^n_i$ on $\llbracket T_0, T_1 \rrbracket$ for all $n\geq 0$ and $1\leq i\leq k$. This is of course true at $n=0$. Suppose it holds at some $n\geq 0$, and suppose that we sample a triple $(i,t,z)$. Then the update rule can only change the values of the $X_i^n(t)$ and $Y_i^n(t)$. Notice that the values can change by at most 1, and if $Y^n_i(t) - X^n_i(t) = 1$, then the only way the ordering could be violated is if $Y_i$ were lowered and $X_i$ were raised at the next update. But this is impossible, since a coin flip of heads can only raise or leave fixed both curves, and tails can only lower or leave fixed both curves. Thus it suffices to assume $X^n_i(t) = Y^n_i(t)$.
	
	There are two cases to consider that violate the ordering of $X^{n+1}_i(t)$ and $Y^{n+1}_i(t)$. Either (i) $X_i(t)$ is raised but $Y_i(t)$ is left fixed, or (ii) $Y_i(t)$ is lowered yet $X_i(t)$ is left fixed. These can only occur if the curves exhibit one of two specific shapes on $\llbracket t-1, t+1\rrbracket$. For $X_i(t)$ to be raised, we must have $X^n_i(t-1) = X^n_i(t) = X^n_i(t+1) - 1$, and for $Y_i(t)$ to be lowered, we must have $Y^n_i(t-1) - 1 = Y^n_i(t) = Y^n_i(t+1)$. From the assumptions that $X^n_i(t) = Y^n_i(t)$, and $X^n_i \leq Y^n_i$, we observe that both of these requirements force the other curve to exhibit the same shape on $\llbracket t-1, t+1\rrbracket$. Then the update rule will be the same for both curves for either coin flip, proving that both (i) and (ii) are impossible. \\
	
	\noindent\textbf{Step 2.} It follows from (2) and (3) and \cite[Theorem 1.8.3]{Norris} that $(X^n)_{n\geq 0}$ and $(Y^n)_{n\geq 0}$ converge weakly to $\mathbb{P}_{avoid,Ber;S}^{T_0,T_1,\vec{x},\vec{y},\infty,g^b}$ and $\mathbb{P}_{avoid,Ber;S}^{T_0,T_1,\vec{x}\,',\vec{y}\,',\infty,g^t}$ respectively. In particular, $(X^n)$ and $(Y^n)$ are tight, so $(X^n,Y^n)_{n\geq 0}$ is tight as well. By Prohorov's theorem, it follows that $(X^n,Y^n)$ is relatively compact. Let $(n_m)$ be a sequence such that $(X^{n_m},Y^{n_m})$ converges weakly. Then by the Skorohod representation theorem \cite[Theorem 6.7]{Billing}, it follows that there exists a probability space $(\Omega,\mathcal{F},\mathbb{P})$ supporting random variables $\mathfrak{X}^n$, $\mathfrak{Y}^n$ and $\mathfrak{X},\mathfrak{Y}$ taking values in $\Omega_{a,S},\Omega_{a,S}'$ respectively, such that
	\begin{enumerate}[label=(\arabic*)]
		
		\item The law of $(\mathfrak{X}^n,\mathfrak{Y}^n)$ under $\mathbb{P}$ is the same as that of $(X^n,Y^n)$,
		
		\item $\mathfrak{X}^n(\omega) \longrightarrow \mathfrak{X}(\omega)$ for all $\omega\in\Omega$,
		
		\item $\mathfrak{Y}^n(\omega) \longrightarrow \mathfrak{Y}(\omega)$ for all $\omega\in\Omega$.
		
	\end{enumerate}
	
	In particular, (1) implies that $\mathfrak{X}^{n_m}$ has the same law as $X^{n_m}$, which converges weakly to $\mathbb{P}_{avoid,Ber;S}^{T_0,T_1,\vec{x},\vec{y},\infty,g^b}$. It follows from (2) and the uniqueness of limits that $\mathfrak{X}$ has law $\mathbb{P}_{avoid,Ber;S}^{T_0,T_1,\vec{x},\vec{y},\infty,g^b}$. Similarly, $\mathfrak{Y}$ has law $\mathbb{P}_{avoid,Ber;S}^{T_0,T_1,\vec{x}\,',\vec{y}\,',\infty,g^t}$. Moreover, condition (4) in Step 1 implies that $\mathfrak{X}^n(i, \cdot) \leq \mathfrak{Y}^n(i, \cdot) $, $\mathbb{P}$-a.s., so $\mathfrak{X}(i, \cdot)\leq \mathfrak{Y}(i, \cdot)$ for $1\leq i\leq k$, $\mathbb{P}$-a.s. Thus we can take $\mathfrak{L}^b = \mathfrak{X}$ and $\mathfrak{L}^t = \mathfrak{Y}$.
	
\end{proof}

%

\subsection{Proof of Lemmas \ref{scaledavoidBB} and \ref{inftydistinct}}\label{BGPapp} In this section we use the same notation as in Section \ref{Section4.3}. We first prove Lemma \ref{scaledavoidBB}. We will use the following lemma, which proves an analogous convergence result for a single rescaled Bernoulli random walk.

\begin{lemma}\label{scaledRWbb}
	Let $x,y,a,b\in \mathbb{R}$ with $a<b$, and let $a_N,b_N\in N^{-\alpha}\mathbb{Z}$, $x^N,y^N \in N^{-\alpha/2}\mathbb{Z}$ be sequences with $a_N \leq a$, $b_N\geq b$, and $|y^N-x^N| \leq (b_N-a_N)N^{\alpha/2}$. Suppose $a_N\to a$, $b_N\to b$. Write $\tilde{x}^N = (x^N - pa_NN^{\alpha/2})/\sqrt{p(1-p)}$, $\tilde{y}^N = (y^N - pb_NN^{\alpha/2})/\sqrt{p(1-p)}$, and assume $\tilde{x}^N \to x$, $\tilde{y}^N\to y$ as $N\to\infty$. Let $Y^N$ be a sequence of random variables with laws $\mathbb{P}^{a_N,b_N,x^N,y^N}_{free,N}$, and let $Z^N = Y^N|_{[a,b]}$. Then the law of $Z^N$ converges weakly to $\mathbb{P}^{a,b,x,y}_{free}$ as $N\to\infty$.
\end{lemma}

\begin{proof}
	Let us write $z^N = (y^N-x^N)N^{\alpha/2}$ and $T_N = (b_N-a_N)N^\alpha$. Let $\tilde{B}$ be a standard Brownian bridge on $[0,1]$, and define random variables $B^N$, $B$ taking values in $C([a_N,b_N])$, $C([a,b])$ respectively via
	\begin{align*}
	B^N(t) &= \sqrt{b_N-a_N}\cdot\tilde{B}\left(\frac{t-a_N}{b_N-a_N}\right) + \frac{t-a_N}{b_N-a_N}\cdot \tilde{y}^N + \frac{b_N-t}{b_N-a_N}\cdot \tilde{x}^N,\\
	B(t) &= \sqrt{b-a}\cdot \tilde{B}\left(\frac{t-a}{b-a}\right) + \frac{t-a}{b-a}\cdot y + \frac{b-t}{b-a}\cdot x.
	\end{align*}
	We observe that $B$ has law $\mathbb{P}^{a,b,x,y}_{free}$ and $B^N\implies B$ as $N\to\infty$. By \cite[Theorem 3.1]{Billing}, to show that $Z^N\implies B$, it suffices to find a sequence of probability spaces supporting $Y^N,B^N$ so that
	\begin{equation}\label{scaledRWweak}
	\rho(B^N,Y^N) = \sup_{t\in[a_N,b_N]} |B^N(t) - Y^N(t)| \implies 0 \quad \mathrm{as} \quad N\to\infty.
	\end{equation}
	It follows from Theorem \ref{KMT} that for each $N\in\mathbb{N}$ there is a probability space supporting $B^N$ and $Y^N$, as well as constants $C,a',\alpha' > 0$, such that
	\begin{equation}\label{scaledRWcheb}
	\ex\left[e^{a'\Delta(N,x^N,y^N)}\right] \leq Ce^{\alpha'\log N} e^{|z^N-pT_N|^2/N^\alpha},
	\end{equation}
	where $\Delta(N,x^N,y^N) = \sqrt{p(1-p)}\,N^{\alpha/2}\rho(B^N,Y^N)$. Since $(z^N - pT_N)N^{-\alpha/2} \to \sqrt{p(1-p)}\,(y-x)$ by assumption, there exist $N_0\in\mathbb{N}$ and $A>0$ so that $|z-pT_N| \leq AN^{\alpha/2}$ for $N\geq N_0$. Then for $\epsilon > 0$ and $N\geq N_0$, Chebyshev's inequality and \eqref{scaledRWcheb} give
	\[
	\mathbb{P}(\rho(B^N,Y^N) > \epsilon) \leq C e^{-a'\epsilon\sqrt{p(1-p)}\, N^{\alpha/2}}e^{\alpha'\log N} e^{A^2}.
	\]
	The right hand side tends to 0 as $N\to\infty$, implying \eqref{scaledRWweak}.
\end{proof}

We now give the proof of Lemma \ref{scaledavoidBB}.

\begin{proof}(of Lemma \ref{scaledavoidBB}) We prove the two statements of the lemma in two steps.\\
	
	\noindent\textbf{Step 1. } In this step we fix $N_0 \in \mathbb{N}$ so that $\mathbb{P}^{a_N,b_N,\vec{x}\,^N,\vec{y}\,^N,f_N,g_N}_{avoid,N}$ is well-defined for $N\geq N_0$. Observe that we can choose $\epsilon > 0$ and continuous functions $h_1,\dots,h_k : [a,b]\to\mathbb{R}$ depending on $a,b,\vec{x},\vec{y},f,g$ with $h_i(a) = x_i$, $h_i(b)=y_i$ for $i\in\llbracket 1,k\rrbracket$, such that if $u_i:[a,b]\to\mathbb{R}$ are continuous functions with $\rho(u_i,h_i) = \sup_{x\in[a,b]} |u_i(x)-h_i(x)| < \epsilon$, then
	\begin{equation}\label{scaledavoidwindow}
	f(x) - \epsilon > u_1(x) + \epsilon > u_1(x) - \epsilon > \cdots > u_k(x) + \epsilon > u_k(x) - \epsilon > g(x) + \epsilon
	\end{equation}
	for all $x\in[a,b]$. By Lemma \ref{Spread}, we have
	\begin{equation}\label{BBwindow}
	\mathbb{P}^{a,b,\vec{x},\vec{y}}_{free}(\rho(\mathcal{Q}_i,h_i) < \epsilon \mbox{ for } i\in\llbracket 1,k\rrbracket) > 0.
	\end{equation}
	Since $y_i^N - x_i^N - p(b_N-a_N)N^{\alpha/2} \to \sqrt{p(1-p)}\,(y_i-x_i)$ as $N\to\infty$ for $i\in\llbracket 1,k\rrbracket$ and $p<1$, we can find $N_1\in\mathbb{N}$ so that for $N\geq N_1$, $|y_i^N-x_i^N| \leq (b_N-a_N)N^{\alpha/2}$. It follows from Lemma \ref{scaledRWbb} that if $\tilde{\mathcal{Y}}^N$ have laws $\mathbb{P}^{a_N,b_N,\vec{x}\,^N,\vec{y}\,^N}_{free,N}$ for $N\geq N_1$ and $\tilde{\mathcal{Z}}^N = \tilde{\mathcal{Y}}^N|_{\llbracket 1, k \rrbracket \times[a,b]}$, then the law of $\tilde{\mathcal{Z}}^N$ converges weakly to $\mathbb{P}^{a,b,\vec{x},\vec{y}}_{free}$. In view of \eqref{BBwindow} we can then find $N_2$ so that if $N \geq \max(N_1,N_2)$ then
\begin{equation}\label{BBwindow2}
	\mathbb{P}^{a_N,b_N,\vec{x}\,^N,\vec{y}\,^N}_{free,N}(\rho(\tilde{\mathcal{Y}}^N_i,h_i) < \epsilon \mbox{ for } i\in\llbracket 1,k\rrbracket) > 0.
\end{equation}
	We now choose $N_3$ so that $\sup_{x\in[a-1,b+1]}|f(x)-f_N(x)| < \epsilon/4$ and $\sup_{x\in[a-1,b+1]}|g(x)-g_N(x)| < \epsilon/4$. If $f=\infty$ (resp. $g=-\infty$), we interpret this to mean that $f_N=\infty$ (resp. $g_N=-\infty$). We take $N_4$ large enough so that if $N\geq N_4$ and $|x-y|\leq N^{-\alpha/2}$ then $|f(x)-f(y)|<\epsilon/4$ and $|g(x)-g(y)|<\epsilon/4$. Lastly, we choose $N_5$ so that $N_5^{-\alpha} < \epsilon/4$. Then for $N\geq N_0 = \max(N_1,N_2,N_3,N_4,N_5)$, we have using (\ref{scaledavoidwindow}) that 
\begin{equation}\label{BBwindow3}
	\{\rho(\tilde{\mathcal{Y}}^N_i,h_i) < \epsilon \mbox{ for } i\in\llbracket 1,k\rrbracket\} \subset \{f_N \geq \mathcal{Y}^N_1 \geq \cdots \geq \mathcal{Y}^N_k \geq g_N \mbox{ on } [a_N,b_N]\}.
\end{equation}
	By \eqref{BBwindow2} and \eqref{BBwindow3} we conclude that 
$$ \mathbb{P}_{free,N}^{a_N, b_N, \vec{x}^{N}, \vec{y}^N} \left( \{f_N \geq \mathcal{Y}^N_1 \geq \cdots \geq \mathcal{Y}^N_k \geq g_N \mbox{ on } [a_N,b_N]\} \right) > 0,$$
which implies that $\mathbb{P}^{a_N,b_N,\vec{x}\,^N,\vec{y}\,^N,f_N,g_N}_{avoid,N}$ is well-defined.\\
	
	\noindent\textbf{Step 2. } In this step we prove that $\mathcal{Z}^N \implies \mathbb{P}^{a,b,\vec{x},\vec{y},f,g}_{avoid}$, with $\mathcal{Z}^N$ defined in the statement of the lemma. We write $\Sigma = \llbracket 1,k\rrbracket$, $\Lambda = [a,b]$, and $\Lambda_N = [a_N,b_N]$. It suffices to show that for any bounded continuous function $F : C(\Sigma\times\Lambda)\to\mathbb{R}$ we have
	\begin{equation}\label{scaledavoidweak}
	\lim_{N\to\infty} \ex[F(\mathcal{Z}^N)] = \ex[F(\mathcal{Q})],
	\end{equation}
	where $\mathcal{Q}$ has law $\mathbb{P}^{a,b,\vec{x},\vec{y},f,g}_{avoid}$.
	
	We define the functions $H_{f,g} : C(\Sigma\times\Lambda)\to\mathbb{R}$ and $H^N_{f,g}:C(\Sigma\times\Lambda_N)\to\mathbb{R}$ by
	\begin{align*}
	H_{f,g}(\mathcal{L}) &= \mathbf{1}\{f > \mathcal{L}_1 > \cdots > \mathcal{L}_k > g \mbox{ on } \Lambda\},\\
	H^N_{f,g}(\mathcal{L}^N) &= \mathbf{1}\{f \geq \mathcal{L}^N_1 \geq \cdots \geq \mathcal{L}^N_k \geq g \mbox{ on } \Lambda_N\}.
	\end{align*}
	Then we observe that for $N\geq N_0$,
	\begin{equation}\label{scaledavoidcond}
	\ex[F(\mathcal{Z}^N)] = \frac{\ex[F(\mathcal{L}^N|_{\Sigma\times[a,b]})H_{f,g}^N(\mathcal{L}^N)]}{\ex[H_{f,g}^N(\mathcal{L}^N)]},
	\end{equation}
	where $\mathcal{L}^N$ has law $\mathbb{P}^{a_N,b_N,\vec{x}\,^N,\vec{y}\,^N}_{free,N}$. By our choice of $N_0$ in Step 1, the denominator in \eqref{scaledavoidcond} is positive for all $N\geq N_0$. Similarly, we have
	\begin{equation}\label{BBavoidcond}
	\ex[F(\mathcal{Q})] = \frac{\ex[F(\mathcal{L})H_{f,g}(\mathcal{L})]}{\ex[H_{f,g}(\mathcal{L})]}, 
	\end{equation}
	where $\mathcal{L}$ has law $\mathbb{P}^{a,b,\vec{x},\vec{y}}_{free}$. From \eqref{scaledavoidcond} and \eqref{BBavoidcond}, we see that to prove \eqref{scaledavoidweak} it suffices to show that for any bounded continuous function $F:C(\Sigma\times\Lambda)\to\mathbb{R}$,
	\begin{equation}\label{scaledBBex}
	\lim_{N\to\infty}\ex[F(\mathcal{L}^N|_{\Sigma\times[a,b]})H_{f,g}^N(\mathcal{L}^N)] = \ex[F(\mathcal{L})H_{f,g}(\mathcal{L})].
	\end{equation}
	By Lemma \ref{scaledRWbb}, $\mathcal{L}^N|_{\Sigma\times[a,b]} \implies \mathcal{L}$ as $N\to\infty$. Since $C(\Sigma\times\Lambda)$ is separable, the Skorohod representation theorem \cite[Theorem 6.7]{Billing} gives a probability space $(\Omega,\mathcal{F},\mathbb{P})$ supporting $C(\Sigma\times\Lambda_N)$-valued random variables $\mathcal{L}^N$ with laws $\mathbb{P}^{a_N,b_N,\vec{x}\,^N,\vec{y}\,^N}_{free,N}$ and a $C(\Sigma\times\Lambda)$-valued random variable $\mathcal{L}$ with law $\mathbb{P}^{a,b,\vec{x},\vec{y}}_{free}$ such that $\mathcal{L}^N|_{\Sigma\times[a,b]}\to\mathcal{L}$ uniformly on compact sets, pointwise on $\Omega$. Here we rely on the fact that $a_N,b_N$ are respectively the largest element of $N^{-\alpha}\mathbb{Z}$ less than $a$ and the smallest element greater than $b$, so that $\mathcal{L}^N|_{\Sigma\times[a,b]}$ uniquely determines $\mathcal{L}^N$ on $[a_N,b_N]$.
	
	Define the events
	\begin{align*}
	E_1 &= \{\omega : f > \mathcal{L}_1(\omega) > \cdots > \mathcal{L}_k(\omega) > g \mbox{ on } [a,b]\},\\
	E_2 &= \{\omega : \mathcal{L}_i(\omega)(r) < \mathcal{L}_{i+1}(\omega)(r) \mbox{ for some } i\in\llbracket 0,k\rrbracket \mbox{ and } r\in[a,b]\},
	\end{align*}
	where in the definition of $E_2$ we use the convention $\mathcal{L}_0 = f$, $\mathcal{L}_{k+1} = g$. The continuity of $F$ implies that $F(\mathcal{L}^N|_{\Sigma\times[a,b]})H^N_{f_N,g_N}(\mathcal{L}^N) \to F(\mathcal{L})$ on the event $E_1$, and $F(\mathcal{L}^N|_{\Sigma\times[a,b]})H^N_{f_N,g_N}(\mathcal{L}^N)\to 0$ on the event $E_2$. By Lemma \ref{NoTouch} we have $\mathbb{P}(E_1 \cup E_2) = 1$, so $\mathbb{P}$-a.s. we have $F(\mathcal{L}^N|_{\Sigma\times[a,b]})H^N_{f_N,g_N}(\mathcal{L}^N) \to F(\mathcal{L})H_{f,g}(\mathcal{L})$. The bounded convergence theorem then implies \eqref{scaledBBex}, completing the proof of \eqref{scaledavoidweak}.
\end{proof}

We now state two lemmas about Brownian bridges which will be used in the proof of Lemma \ref{inftydistinct}. The first lemma shows that a Brownian bridge started at 0 almost surely becomes negative somewhere on its domain.

\begin{lemma}\label{BBcross0}
	Fix any $T>0$ and $y\in\mathbb{R}$, and let $Q$ denote a random variable with law $\mathbb{P}^{0,T,0,y}_{free}$. Define the event $C = \{\inf_{s\in[0,T]} Q(s) < 0\}$. Then $\mathbb{P}^{0,T,0,y}_{free}(C) = 1$.
\end{lemma}

\begin{proof}
	Let $B$ denote a standard Brownian bridge on $[0,1]$, and let 
	\[
	\tilde{B}_s = B_{s/T} + \frac{sy}{T}, \quad \mathrm{for}\,s\in[0,T].
	\]
	Then $\tilde{B}$ has the law of $Q$. Consider the stopping time $\tau = \inf\{s>0 : \tilde{B}_s < 0 \}$. We will argue that $\tau = 0$ a.s, which implies the conclusion of the lemma since $\{\tau = 0\} \subset C$. We observe that since $\tilde{B}$ is a.s. continuous and $\mathbb{Q}$ is dense in $\mathbb{R}$,
	\[
	\{\tau = 0\} = \cap_{n \in \mathbb{N}}\,\cup_{s\in(0,1/n)\cap\mathbb{Q}} \{\tilde{B}_s < 0 \} \in \cap_{n \in \mathbb{N}} \sigma(\tilde{B}_s : s<1/n).
	\]
	Here, $\sigma(\tilde{B}_s : s<\epsilon)$ denotes the $\sigma$-algebra generated by $\tilde{B}_s$ for $s<\epsilon$. We used the fact that for a fixed $\epsilon$, each set $\{\tilde{B}_s < 0\}$ for $s\in (0,\epsilon)\cap\mathbb{Q}$ is contained in this $\sigma$-algebra, and thus so is their countable union. It follows from Blumenthal's 0-1 law \cite[Theorem 7.2.3]{Durrett} that $\mathbb{P}(\tau = 0) \in \{0,1\}$. To complete the proof, it suffices to show that $\mathbb{P}(\tau = 0) > 0$. By \eqref{BBcovar}, $B_{s/T}$ is distributed normally with mean 0 and variance $\sigma^2 = (s/T)(1-s/T)$. We observe that for any $s\in(0,T)$, 
	\[
	\mathbb{P}(\tau\leq s) \geq \mathbb{P}(B_{s/T} < -sy/T) = \mathbb{P}\left(\sigma\mathcal{N}(0,1) > (s/T)y\right) = \mathbb{P}\left(\mathcal{N}(0,1) > y\sqrt{s/(T-s)}\right).
	\]
	As $s\to 0$, the probability on the right tends to $\mathbb{P}(\mathcal{N}(0,1) > 0) = 1/2$. Since $\{\tau = 0\} = \bigcap_{n=1}^\infty \{\tau \leq 1/n\}$ and $\{\tau\leq 1/(n+1)\} \subset \{\tau \leq 1/n\}$, we conclude that
	\[
	\mathbb{P}(\tau = 0) = \lim_{n\to\infty} \mathbb{P}(\tau \leq 1/n) \geq 1/2.
	\]
	Therefore $\mathbb{P}(\tau = 0) = 1$.
	
\end{proof}	

The second lemma shows that a difference of two independent Brownian bridges is another Brownian bridge.

\begin{lemma}\label{BBDif}
	Let $a,b,x_1,y_1,x_2,y_2\in \R$ with $a<b$. Let $B_1(t)$, $B_2(t)$ be independent Brownian bridges from on $[a,b]$ from $x_1$ to $y_1$ and from $x_2$ to $y_2$ respectively, as defined in (\ref{BBDef}). If $B(t)=B_1(t)-B_2(t)$ for $t\in[a,b]$, then $2^{-1/2}B$ is itself a Brownian bridge on $[a,b]$ from $2^{-1/2}(x_1 -x_2)$ to $2^{-1/2}(y_1 - y_2)$.
\end{lemma}
\begin{proof}
	By definition, for $i = 1,2$ we have
	\begin{equation*}
		B_i(t)=(b-a)^{1/2}\cdot \tilde{B_i}\left(\frac{t-a}{b-a}\right)+\left(\frac{b-t}{b-a}\right)\cdot x_i+\left(\frac{t-a}{b-a}\right)\cdot y_i,
	\end{equation*} 
	with $\tilde{B_i}(t)=W^i_t-tW^i_1$ for independent Brownian motions $W^1$ and $W^2$. We have
	\begin{equation}\label{BBdifeq}
	  B(t)=(b-a)^{1/2}\cdot(\tilde{B_1}-\tilde{B_2})\left(\frac{t-a}{b-a}\right)+\left(\frac{b-t}{b-a}\right)\cdot (x_1-x_2)+\left(\frac{t-a}{b-a}\right)\cdot (y_1-y_2).
	\end{equation}
	Note that the process $\tilde{B}_1 - \tilde{B}_2$ is a linear combination of continuous Gaussian mean 0 processes, so it is a continuous Gaussian mean 0 process, and is thus characterized by its covariance. Since $\tilde{B}_1(\cdot)$ and $\tilde B_2(\cdot)$ are both Gaussian with mean 0 and the covariance $\min(s,t)$, their difference $\tilde{B}_1(\cdot) - \tilde{B}_2(\cdot)$ is also Gaussian with mean $0$ and covariance $2\min(s,t)$. This implies that $2^{-1/2}(\tilde{B_1}-\tilde{B_2})$ is itself a Brownian bridge $\tilde B$ on $[a,b]$, and hence equation \ref{BBdifeq} can be rewritten  \begin{equation*}
		2^{-1/2} B(t) =(b-a)^{1/2}\cdot\tilde{B}\left(\frac{t-a}{b-a}\right)+\left(\frac{b-t}{b-a}\right)\cdot 2^{-1/2} (x_1-x_2)+\left(\frac{t-a}{b-a}\right)\cdot 2^{-1/2} (y_1-y_2).
	\end{equation*}
	This is a Brownian bridge on $[a,b]$ from $2^{-1/2}(x_1 -x_2)$ to $2^{-1/2}(y_1 - y_2)$ as desired.
\end{proof}

To conclude this section, we prove Lemma \ref{inftydistinct}.

\begin{proof}(of Lemma \ref{inftydistinct})
Suppose that $\mathcal{L}^\infty$ is a subsequential limit of $(\tilde{f}_1^N, \dots, \tilde{f}_{k-1}^N)$. By possibly passing to a subsequence we may assume that $(\tilde{f}_1^N, \dots, \tilde{f}_{k-1}^N) \implies \mathcal{L}^\infty$. We will still call the subsequence $(\tilde{f}_1^N, \dots, \tilde{f}_{k-1}^N)$ to not overburden the notation. By the Skorohod representation theorem \cite[Theorem 6.7]{Billing}, we can also assume that $(\tilde{f}_1^N, \dots, \tilde{f}_{k-1}^N)$ and $\mathcal{L}^\infty$ are all defined on the same probability space with measure $\mathbb{P}$ and the convergence is happening $\mathbb{P}$-almost surely. Here we are implicitly using Lemma \ref{Polish}  from which we know that the random variables $(\tilde{f}_1^N, \dots, \tilde{f}_{k-1}^N)$ and $\mathcal{L}^\infty$ take value in a Polish space so that the Skorohod representation theorem is applicable. 

Let us denote the random variables with laws $(\tilde{f}_1^N, \dots, \tilde{f}_{k-1}^N) $ by $\mathcal{X}^N$ and the one with law $\mathcal{L}^\infty$ by $\mathcal{X}$ and so $\mathcal{X}^N \rightarrow \mathcal{X}$ almost surely w.r.t. $\mathbb{P}$. In particular, $\mathcal{X}^N(s) \to \mathcal{X}(s)$ for any $s \in \mathbb{R}$. Recall that $f^N_i(s) = N^{-\alpha/2}(L^N_i(sN^\alpha) - psN^\alpha) + \lambda s^2$, so $\mathcal{X}^N_i(s) = N^{-\alpha/2}(\mathcal{L}^N_i(sN^\alpha) - psN^\alpha)/\sqrt{p(1-p)}$, where $\mathcal{L}^N$ has the law of $L^N$.
	
	Suppose that $\mathcal{X}_i(s) = \mathcal{X}_{i+1}(s)$ for some $i\in\llbracket 1,k-2\rrbracket$. Then we have $\mathcal{X}_i^N(s) - \mathcal{X}^N_{i+1}(s) \to 0$, i.e., $N^{-\alpha/2}(\mathcal{L}_i^N(sN^\alpha) - \mathcal{L}^N_{i+1}(sN^\alpha))\to 0$ as $N\to\infty$. Let us write $a = \lfloor sN^\alpha\rfloor N^{-\alpha}$, $b = \lceil (s+2)N^\alpha\rceil N^{-\alpha}$ and $x^N = \mathcal{L}_i^N(aN^\alpha) - \mathcal{L}_{i+1}^N(aN^\alpha)$, $y^N = \mathcal{L}_i^N(bN^\alpha) - \mathcal{L}_{i+1}^N(bN^\alpha)$. Then $N^{-\alpha/2}x^N\to 0$. If $Q_i,Q_{i+1}$ are independent Bernoulli bridges with laws $\mathbb{P}^{a,b,\mathcal{L}_i^N(aN^\alpha),\mathcal{L}_i^N(bN^\alpha)}_{Ber}$ and $\mathbb{P}^{a,b,\mathcal{L}_{i+1}^N(aN^\alpha),\mathcal{L}_{i+1}^N(bN^\alpha)}_{Ber}$, then $\ell=Q_i - Q_{i+1}$ is a random walk bridge taking values in $\{-1,0,1\}$, from $(a,x^N)$ to $(b,y^N)$. Let us denote the law of $N^{-\alpha/2}\ell/\sqrt{p(1-p)}$ by $\mathbb{P}^{a,b,x^N,y^N}_{di\!f\!f}$.
	
	By Lemma \ref{scaledRWbb}, $(x^N+N^{-\alpha/2}Q_{i+1}-ptN^\alpha)/\sqrt{p(1-p)}$ and $(x^N+N^{-\alpha/2}Q_i-ptN^\alpha)/\sqrt{p(1-p)}$ converge weakly to the law of two Brownian bridges $B^i$ from $\mathcal{L}_i^\infty(s)$ to $\mathcal{L}_i^\infty(s+2)$ and $B^{i+1}$ from $\mathcal{L}_{i+1}^\infty(s)$ to $\mathcal{L}_{i+1}^\infty(s+2)$ respectively. Consequently, their difference $N^{-\alpha/2}\ell/\sqrt{p(1-p)}$ converges weakly to the difference of two independent Brownian bridges, $B^1-B^2$. 
	By Lemma \ref{BBDif}, this difference is equal to $2^{1/2} B$, where $B$ is a Brownian bridge $B$ on $[s,s+2]$ from 0 to $2^{-1/2} y$, where $y = \mathcal{L}_i^\infty(s+2) - \mathcal{L}_{i+1}^\infty(s+2).$ In other words, $B$ has law $\mathbb{P}^{s,s+2,0,2^{-1/2}y}_{free}$. Therefore $\mathbb{P}^{a,b,x^N,y^N}_{dif\!f}$ converges weakly to $\mathbb{P}^{s,s+2,0,y}_{free}$. With probability one, $\min_{t\in[s,s+2]} B_t < 0$ by Lemma \ref{BBcross0}. Thus given $\delta > 0$, we can choose $N$ large enough so that the probability of $N^{-\alpha/2}\ell/\sqrt{p(1-p)}$, or equivalently $\ell$, remaining above 0 on $[a,b]$ is less than $\delta$. Thus for large enough $N$ we have
	\begin{equation}\label{collideAP}
		\begin{split}
		&\mathbb{P}\left(f_i^\infty(s) = f_{i+1}^\infty(s)\right) \leq \mathbb{P}\bigg(\mathbb{P}^{a,b,x^N,y^N}_{dif\!f}\bigg(\sup_{s\in[a,b]}\ell(s) \geq 0\bigg) < \delta\bigg) \leq\\
		& \mathbb{P}\left(Z(a,b,\mathcal{L}^N(aN^\alpha),\mathcal{L}^N(bN^\alpha),\infty,\mathcal{L}^N_k) < \delta\right).
		\end{split}
	\end{equation}
	Here, $Z$ denotes the acceptance probability of Definition \ref{DefAP}. This is the probability that $k-1$ independent Bernoulli bridges $Q_1,\dots,Q_{k-1}$ on $[a,b]$ with entrance and exit data $\mathcal{L}^N(a)$ and $\mathcal{L}^N(b)$ do not cross one another or $\mathcal{L}_k^N$. The last inequality follows because $\ell$ has the law of the difference of $Q_i$ and $Q_{i+1}$, and the acceptance probability is bounded above by the probability that $Q_i$ and $Q_{i+1}$ do not cross, i.e., that $Q_i - Q_{i+1} \geq 0$. By Proposition \ref{PropMain}, given $\epsilon > 0$ we can choose $\delta$ so that the probability on the right in \eqref{collideAP} is $<\epsilon$. We conclude that
	\[
	\mathbb{P}\left(f_i^\infty(s) = f_{i+1}^\infty(s)\right) = 0.
	\]
\end{proof}

%
\section{Appendix B} \label{Section9}

The goal of this section is to prove Proposition \ref{prob17}, which roughly states that if the boundary data of an avoiding Bernoulli line ensemble converges then the fixed time distribution of the ensemble converges weakly to a random vector with density $\rho$. In the process of the proof we will identify this limiting density $\rho$.

Throughout this section we fix $k \in \mathbb{N}$ and consider sequences of $\llbracket 1,k\rrbracket$-indexed line ensembles with distribution given by $\mathbb{P}^{0,T,\vec{x},\vec{y}}_{avoid,Ber}$ in the sense of Definition \ref{DefAvoidingLawBer}. Recall that this is just the law of $k$ independent Bernoulli random walks that have been conditioned to start from $\vec{x}=(x_{1},\dots,x_{k})$ at time $0$ and end at $\vec{y}=(y_1,\cdots,y_{k})$ at time $T$ and are conditioned on never crossing. Here $\vec{x}$, $\vec{y}\in\mathfrak{W}_{k}$ satisfy $T\geq y_{i}-x_{i}\geq 0$ for $i=1,\dots,k$, which by Lemma \ref{LemmaWD} ensures the well-posedness of $\mathbb{P}^{0,T,\vec{x},\vec{y}}_{avoid,Ber}$. 

 In Section \ref{Section9.1}, we introduce some definitions and formulate the precise statements of the two results we want to prove as Propositions \ref{S9PropH} and \ref{WeakConvDistinct}. In Section \ref{Section9.2}, we introduce some basic results about skew Schur polynomials and express the fixed time distribution of avoiding Bernoulli line ensembles through these polynomials in Lemma \ref{BerDist}. In Sections \ref{Section9.3} and \ref{Section9.4}, we prove Propositions \ref{S9PropH} and \ref{WeakConvDistinct} for an important special case. In Section \ref{Section9.5} we introduce some notations and results about multi-indices and multivariate functions which paves the way for the full proofs of Propositions \ref{S9PropH} and \ref{WeakConvDistinct} in that section and Section \ref{Section9.6}. 

%
\subsection{Weak convergence}{\label{Section9.1}}
We start by recalling and introducing some helpful notation. Recall,
$$W_k = \{ \vec{x} \in \mathbb{R}^k: x_1 \geq x_2 \geq \cdots \geq x_k \}, \hspace{5mm} W_k^{\circ} = \{\vec{x} \in \mathbb{R}^k: x_1 > x_2 > \cdots > x_k \}.$$
\begin{definition}\label{DefScaled} Here we recall the scaling from Proposition \ref{prob17}. We fix $p,t\in(0,1)$, and $\vec{a},\vec{b}\in {W}_{k}$. Suppose that $\vec{x}^{T}=(x_{1}^{T},\dots,x_{k}^{T})$ and $\vec{y}^{T}=(y_{1}^{T},\dots,y_{k}^{T})$ are two sequences of $k$-dimensional vectors in $\mathfrak{W}_{k}$ such that $$\lim_{T\rightarrow\infty}\frac{x_{i}^{T}}{\sqrt{T}}=a_{i} \text{ and } \lim_{T\rightarrow\infty}\frac{y_{i}^{T}-pT}{\sqrt{T}}=b_{i}$$ for $i=1,\dots,k$. Define the sequence of random $k$-dimensional vectors $Z^{T}$ by 
	\begin{align}{\label{ZTandL}}
		\begin{split}
		Z^{T}= (Z_1^T, \dots, Z_k^T) =  \left(\frac{L^T_{1}(tT)-ptT}{\sqrt{T}},\dots,\frac{L^T_{k}(tT)-ptT}{\sqrt{T}}\right),
	\end{split}
	\end{align}
	where $(L^T_{1},\dots,L^T_{k})$ is $\mathbb{P}^{0,T,\vec{x}^{T},\vec{y}^{T}}_{avoid,Ber}$-distributed.
\end{definition}

We next define a class of functions that will be used to express the limiting density $\rho$ in Proposition \ref{prob17}. These functions depend on two vectors $\vec{a}, \vec{b} \in W_k$ as well as parameters $p, t \in (0,1)$ through the quantities
\begin{equation}\label{CWeights}
c_{1}(p,t)=\frac{1}{p(1-p)t}, \quad c_{2}(p,t)=\frac{1}{p(1-p)(1-t)}, \quad c_{3}(p,t)=\frac{1}{2p(1-p)t(1-t)}.
\end{equation}
Suppose the vectors $\vec{a}$ and $\vec{b}$ have the following form
\begin{equation}{\label{Block}}
\begin{split}
	&\vec{a}=(a_{1},\dots,a_{k})=(\underbrace{\alpha_{1},\dots,\alpha_{1}}_{m_{1}},\dots,\underbrace{\alpha_{p},\dots,\alpha_{p}}_{m_{p}})\\
	&\vec{b}=(b_{1},\dots,b_{k})=(\underbrace{\beta_{1},\dots,\beta_{1}}_{n_{1}},\dots,\underbrace{\beta_{q},\dots,\beta_{q}}_{n_{q}})
\end{split}
\end{equation}
where $\alpha_{1}>\alpha_{2}>\cdots>\alpha_{p}$, $\beta_{1}>\beta_{2}>\cdots>\beta_{q}$ and $\sum_{i=1}^{p}m_{i}=\sum_{i=1}^{q}n_{i}=k$. We denote $\vec{m}=(m_{1},\cdots,m_{p})$, $\vec{n}=(n_{1},\cdots,n_{q})$ and define two determinants $\varphi(\vec{a},\vec{z},\vec{m})$ and $\psi(\vec{b},\vec{z},\vec{n})$  as follows
\begin{equation}{\label{TwoDet}}
\begin{split}
	\varphi(\vec{a},\vec{z},\vec{m})= \det
	\left[ \begin{array}{ccc}
		\left((c_{1}(t,p)z_{j})^{i-1}e^{c_{1}(t,p)\alpha_{1}z_{j}}\right)_{\substack{i=1,\dots,m_{1}\\j=1,\dots,k}}\\
	\vdots\\
	\left((c_{1}(t,p)z_{j})^{i-1}e^{c_{1}(t,p)\alpha_{p}z_{j}}\right)_{\substack{i=1,\dots,m_{p} \\j=1,\dots,k}}
	\end{array}
	\right]
\\
	\psi(\vec{b},\vec{z},\vec{n})= \det
	\left[ \begin{array}{ccc}
		\left((c_{2}(t,p)z_{j})^{i-1}e^{c_{2}(t,p)\beta_{1}z_{j}}\right)_{\substack{i=1,\dots,n_{1}\\j=1,\dots,k}}\\
	\vdots\\
	((c_{2}(t,p)z_{j})^{i-1}e^{c_{2}(t,p)\beta_{q}z_{j}})_{\substack{i=1,\dots, n_{q} \\j=1,\dots,k}}
	\end{array}
	\right].
\end{split}
\end{equation}

Then we define the function 
\begin{align}{\label{PreDensity17}}
	H(\vec{z})=\varphi(\vec{a},\vec{z},\vec{m}) \cdot \psi(\vec{b},\vec{z},\vec{n}) \cdot \prod_{i=1}^{k}e^{-c_{3}(t,p)z_{i}^{2}}.
\end{align}
The function $H$ implicitly depends on $p, t, k, \vec{a}, \vec{b}$ but we will not reflect this dependence in the notation. The following result summarizes the properties we will require from $H(\vec{z})$. Its proof can be found in Sections \ref{Section9.3} and \ref{Section9.5}.

\begin{proposition}\label{S9PropH} Fix $p,t \in(0,1)$ and $\vec{a}, \vec{b} \in W_k$ and let $H(\vec{z})$ be as in (\ref{PreDensity17}). Then we have: 
\begin{enumerate}
\item $H(\vec{z}) \geq 0$ for all $\vec{z} \in W_k$.
\item $H(\vec{z}) = 0$ for $\vec{z} \in W_k \setminus W_k^{\circ}$ and $H(\vec{z}) > 0$ for $\vec{z} \in W_k^{\circ}$.
\item $Z_c := \int_{W_k} H(\vec{z}) d\vec{z}\in (0, \infty)$, where $d\vec{z}$ stands for the usual Lebesgue measure.
\end{enumerate}
\end{proposition}

In view of Proposition \ref{S9PropH} we know that the function
\begin{equation}\label{defRho}
\rho(\vec{z}) = \rho(z_1, \dots, z_k) = Z_c^{-1} \cdot {\bf 1}_{\{z_1 > z_2 > \cdots > z_k \}} \cdot H(\vec{z}),
\end{equation}
defines a density on $\mathbb{R}^k$. This is the limiting density in Proposition \ref{prob17}. 
We end this section by stating the main convergence statement we want to establish. 
\begin{proposition}{\label{WeakConvDistinct}}
Assume the same notation as in the Definition \ref{DefScaled}. Then the random vectors $Z^{T}$ converge weakly to $\rho$ as in (\ref{defRho}) as $T \rightarrow \infty$.	
\end{proposition}
The way the proof of the above two propositions is organized in the remainder of the section is as follows. We first prove Proposition \ref{S9PropH} and Proposition \ref{WeakConvDistinct} for the case when $\vec{a}, \vec{b} \in W_k^{\circ}$ -- this is done in Sections \ref{Section9.3} and \ref{Section9.4} respectively. Afterwards we will prove Proposition \ref{S9PropH} for vectors $\vec{a}$, $\vec{b}$ that have the form in (\ref{Block}) in Section \ref{Section9.5} and then use Proposition \ref{WeakConvDistinct} for the case $\vec{a}, \vec{b} \in W_k^{\circ}$ and the monotone coupling Lemma \ref{MCLxy} to prove Proposition \ref{WeakConvDistinct} in the general case in Section \ref{Section9.6}.

%
\subsection{Skew Schur polynomials}{\label{Section9.2}}
In this section we give some definitions and elementary results regarding skew Schur polynomials, which are mainly based on \cite[Chapter 1]{Mac}. Afterwards we explain how the fixed time distribution of an avoiding Bernoulli line ensemble is expressible in terms of these skew Schur polynomials. 
\begin{definition}{\label{DefPar}} \emph{Partitions, skew diagrams, interlacing, conjugation}
\begin{enumerate}
	\item A \emph{partition} is an infinite sequence $\lambda=(\lambda_{1}, \lambda_{2}, \dots, \lambda_{r}, \dots)$ of non-negative integers in decreasing order $\lambda_{1}\geq \lambda_{2}\geq \cdots\geq \lambda_{r}\geq \cdots$ and containing only finitely many non-zero terms. The non-zero $\lambda_{i}$ are called \emph{parts} of $\lambda$, the number of parts is called the \emph{length} of the partition $\lambda$, denoted by $l(\lambda)$, and the sum of the parts is the \emph{weight} of $\lambda$, denoted by $|\lambda|$.
	\item A partition $\lambda$ is graphically represented by a Young diagram that has $\lambda_1$ left-justified boxes on the top row, $\lambda_2$ boxes on the second row and so on. Suppose $\lambda$ and $\mu$ are two partitions, we write $\lambda\supset\mu$ if $\lambda_{i}\geq \mu_{i}$ for all $i\in \mathbb{N}$. We call the set-theoretic difference of the two Young diagrams of $\lambda$ and $\mu$ a {\em skew diagram} and denote it by $\lambda / \mu$.
	\item Partitions $\lambda=(\lambda_{1}, \lambda_{2},\cdots)$ and $\mu=(\mu_{1}, \mu_{2},\cdots)$ are call \emph{interlaced}, denoted by $\mu\preceq \lambda$, if $\lambda_1\geq \mu_1\geq \lambda_2\geq \mu_2\geq\cdots$.
	\item The \emph{conjugate} of a partition $\lambda$ is the partition $\lambda^{\prime}$ such that
	\begin{equation*}
		\lambda^{\prime}_{i}=\max_{j\geq 1}\{j:\lambda_{j}\geq i\}
	\end{equation*}
	In particular, $\lambda_{1}^{\prime}=l(\lambda)$, $\lambda_{1}=l(\lambda^{\prime})$ and notice that $\lambda^{\prime\prime}=\lambda$. For example, the conjugate of $(5441)$ is $(43331)$.
\end{enumerate}
\end{definition}

According to Definition \ref{DefPar}, we directly get that if $\mu\subset \lambda$ then $l(\lambda)\geq l(\mu)$ and $l(\lambda^{\prime})\geq l(\mu^{\prime})$. Also, $\mu\preceq\lambda$ implies $\mu\subset\lambda$. Also as explained in \cite[pp. 5]{Mac} we have that if $\mu\preceq\lambda$ are interlaced, then $\lambda^{\prime}_{i}-\mu^{\prime}_{i}=0\text{ or }1$ for every $i\geq 1$.

\begin{definition} \emph{Elementary Symmetric Functions.}{\label{ElemSymFunc}} For each integer $r\geq 0$, the $r$-th \emph{elementary symmetric function} $e_{r}$ is the sum of all products of $r$ distinct variables $x_i$, so that $e_0=1$ and
	\begin{align}
		e_{r}=\sum_{i_1<i_2<\cdots<i_r}x_{i_1}x_{i_2}\cdots x_{i_r}
	\end{align}
	for $r\geq 1$. For $r<0$, we define $e_r$ to be zero. In particular, when $x_1=x_2=\cdots=x_n=1$, $x_{n+1}=x_{n+2}=\cdots=0$, $e_r$ is just the binomial coefficient when $0\leq r\leq n$:
	 \begin{equation*}
	 	e_{r}(1^{n})=\binom{n}{r}
	 \end{equation*}
	 	and $e_r=0$ when $r>n$.
\end{definition}

Next, we introduce Skew Schur Polynomial based on \cite[Chapter 1, (5.5), (5.11), (5.12)]{Mac}.
\begin{definition} \emph{Skew Schur Polynomial, Jacob-Trudi Formula}{\label{DefSkewSchurPoly}}
\begin{enumerate}
	\item Suppose $\mu\subset\lambda$ are partitions. If $\mu\preceq\lambda$ are interlaced, then the \emph{skew Schur polynomial} $s_{\lambda/\mu}$ with single variable $x$ is defined by $s_{\lambda/\mu}(x)=x^{|\lambda-\mu|}$. Otherwise, we define $s_{\lambda/\mu}(x)=0$.
	\item Suppose $\mu\subset\lambda$ are two partitions, define the \emph{skew Schur polynomial} $s_{\lambda/\mu}$ with respect to variables $x_1, x_2, \cdots, x_{n}$ by
	\begin{align}\label{S9Split}
		s_{\lambda/\mu}(x_1,\cdots,x_n)=\sum_{(\nu)}\prod_{i=1}^{n}s_{\nu^{i}/\nu^{i-1}}(x_i)=\sum_{(\nu)}\prod_{i=1}^{n}x_{i}^{|\nu^{i}-\nu^{i-1}|}
	\end{align}
	summed over all sequences $(\nu)=(\nu^{0},\nu^{1},\cdots,\nu^{n})$ of partitions such that $\nu^{0}=\mu$, $\nu^{n}=\lambda$ and $\nu^{0}\preceq\nu^{1}\preceq\cdots\preceq\nu^{n}$. In particular, when $x_1=x_2=\cdots=x_{n}=1$, the skew Schur polynomial is just the number of such sequences of interlaced partitions $(\nu)$.	 This definition also implies the following \emph{branching relation} of skew Schur polynomials:
	\begin{equation}{\label{Branch}}
		\begin{split}
			s_{\kappa/\mu}(x_1,\dots, x_n)=\sum_{\lambda}s_{\kappa/\lambda}(x_1, \dots, x_m)\cdot s_{\lambda/\mu}(x_{m+1}, \dots, x_n),
		\end{split}
	\end{equation}
	where $1 \leq m \leq n$ and $s_{\kappa/\lambda}(\varnothing) = {\bf 1}\{ \kappa = \lambda\}$. 
	\item We also have the following \emph{Jacob-Trudi Formula}\cite[Chapter 1, (5.5)]{Mac} for the skew Schur polynomial:
	\begin{equation}{\label{J-TFormula}}
		s_{\lambda/\mu}=\det\left(e_{\lambda_{i}^{\prime}-\mu_{j}^{\prime}-i+j}\right)_{1\leq i,j\leq m}
	\end{equation}
	where $m\geq l(\lambda^{\prime})$, and $e_r$ is the elementary symmetric function in Definition \ref{ElemSymFunc}.
\end{enumerate}
\end{definition}

Based on the above preparation, we are ready to state the following lemma giving the distribution of avoiding Bernoulli line ensembles at time $\lfloor tT \rfloor$.
\begin{lemma}{\label{BerDist}}
Assume the same notation as in Definition \ref{DefScaled}, denote $m=\lfloor tT \rfloor$, $n=T-\lfloor tT \rfloor$ and assume $m,n \in \mathbb{N}$. Then, the avoiding Bernoulli line ensemble at time $m$ has the following distribution: 
\begin{align}{\label{ProbMassFunc}}
\begin{split}
&\mathbb{P}_{avoid, Ber}^{0,T,\vec{x}^{T},\vec{y}^{T}}(L^T_{1}(m) = \lambda_{1}, \cdots, L^T_{k}(m) = \lambda_{k})= \\
&\frac{\det\left(e_{\lambda_{i}-x_{j}^{T}-i+j}(1^m)\right)_{1\leq i,j\leq k}\cdot \det\left(e_{y_{i}^T-\lambda_{j}-i+j}(1^n)\right)_{1\leq i,j\leq k}}{\det\left(e_{y_{i}^T-x_{j}^T-i+j}(1^{m+n})\right)_{1\leq i,j\leq k}},
\end{split}
\end{align}
where $\lambda_{1}\geq\lambda_{2}\geq\cdots\geq\lambda_{k}$ are integers. 
\end{lemma}
\begin{proof} Notice that if we shift all $x_i, y_i$ and $\lambda_i$ by the same integer, both sides of (\ref{ProbMassFunc}) stay the same and so we may assume that all of these quantities are positive by adding the same large integer to all the coordinates. We then let $\kappa$ be the partition with parts $\kappa_i = y_i^T$, $\mu$ be the partition with $\mu_i = x_i^T$ and $\lambda$ be the partition with parts $\lambda_i$ for $i = 1, \dots, k$. All three partitions have length $k$. In view of  (\ref{J-TFormula}) we see that the right side of (\ref{ProbMassFunc}) is precisely 
\begin{equation}\label{S9P1E1}
\mbox{ RHS of (\ref{ProbMassFunc})} = \frac{s_{\lambda^{\prime}/\mu^{\prime}}(1^{m})\cdot s_{\kappa^{\prime}/\lambda^{\prime}}(1^{n})}{s_{\kappa^{\prime}/\mu^{\prime}}(1^T)}.
\end{equation}

Let $\Omega(0,T,\vec{x}^T, \vec{y}^T)$ be the set of all avoiding Bernoulli line ensembles from $\vec{x}^T$ to $\vec{y}^T$, and analogously define $\Omega(0,m,\vec{x}^T, \lambda)$ and $\Omega(0,n,\lambda, \vec{y}^T)$. Then we get by the uniformity of the measure $\mathbb{P}_{avoid, Ber}^{0,T,\vec{x}^{T},\vec{y}^{T}}$ that 
\begin{equation}\label{S9P1E2}
\mbox{ LHS of (\ref{ProbMassFunc})} = \frac{|\Omega(0,m,\vec{x}^T, \lambda)|\cdot|\Omega(0,n,\lambda, \vec{y}^T)|}{|\Omega(0,T,\vec{x}^T, \vec{y}^T)|}.
\end{equation}

Let us define the set 
\[TB_{\kappa/\mu}^T:=\{(\lambda^0,\dots,\lambda^T)\mid \lambda^0=\mu, \lambda^T=\kappa, \lambda^i\preceq\lambda^{i+1}\text{ for }i=0,\cdots,T-1\}.\]
Notice that by the definition of $\Omega(0,T,\vec{x}^T, \vec{y}^T)$ the map $f: \Omega(0,T,\vec{x}^T, \vec{y}^T) \rightarrow TB_{\kappa/\mu}^T$ given by 
$$f(\mathfrak{L}) = (\mathfrak{L}(\cdot, 0), \dots, \mathfrak{L}(\cdot, T)),$$
where $\mathfrak{L}(\cdot, j)$ stands for the partition with parts $\mathfrak{L}(i, j)$ for $i = 1,\dots, k$ defines a bijection between $\Omega(0,T,\vec{x}^T, \vec{y}^T)$ and $TB_{\kappa/\mu}^T$ and so we conclude that 
$$|\Omega(0,T,\vec{x}^T, \vec{y}^T)| = | TB_{\kappa/\mu}^T| = s_{\kappa^{\prime}/\mu^{\prime}}(1^T),$$
where in the last equality we used (\ref{S9Split}). Applying the same argument to $\Omega(0,m,\vec{x}^T, \lambda)$ and $\Omega(0,n,\lambda, \vec{y}^T)$ we conclude 
\begin{equation}\label{S9P1E3}
|\Omega(0,m,\vec{x}^T, \lambda)| = s_{\lambda^{\prime}/\mu^{\prime}}(1^{m}), \hspace{2mm}|\Omega(0,n,\lambda, \vec{y}^T)| = s_{\kappa^{\prime}/\lambda^{\prime}}(1^{n}),  \hspace{2mm} |\Omega(0,T,\vec{x}^T, \vec{y}^T)| = s_{\kappa^{\prime}/\mu^{\prime}}(1^T).
\end{equation}
Combining (\ref{S9P1E1}), (\ref{S9P1E2}) and (\ref{S9P1E3}) gives (\ref{ProbMassFunc}).
\end{proof}

%
\subsection{Proof of Proposition \ref{S9PropH} for $\vec{a}, \vec{b} \in W_k^{\circ}$}{\label{Section9.3}} In this section we prove a few technical results we will need later as well as Proposition \ref{S9PropH} for the case when $\vec{a}, \vec{b}$ have distinct entries.

\begin{lemma}{\label{Limh}}
Suppose that $p\in(0,1)$ and $R>0$ are given. Suppose that $x\in[-R,R]$ and $N=pn+\sqrt{n}x \in [0, n]$ is an integer. Then 
\begin{equation}
\begin{split}{\label{AsympForh}}
&e_{N}(1^{n})=(\sqrt{2\pi})^{-1}\cdot \exp\left(-\frac{x^2}{2(1-p)p}\right)\cdot \exp\left(N\log\left(\frac{1-p}{p}\right)\right)\cdot \exp\left(O(n^{-1/2})\right) \cdot \\
& \exp\left(-n\log(1-p)-(1/2)\log n-(1/2)\log\left(p(1-p)\right)\right)
\end{split}
\end{equation}
where the constant in the big $O$ notation depends on $p$ and $R$ alone. Moreover, there exist positive constants $C,c>0$ depending on $p$ alone such that for all large enough $n\in\mathbb{N}$ and $N\in[0,n]$,
\begin{equation}{\label{ebounded}}
	e_{N}(1^n)\leq C\cdot \exp\left(N\log\frac{1-p}{p}-n\log(1-p)-(1/2)\log n\right)\cdot \exp\left(-cn^{-1}(N-pn)^2\right).
\end{equation}
\end{lemma}
\begin{remark}
Notice that when $R>0$ is fixed and $N\in[pn-R\sqrt{n},pn+R\sqrt{n}]$ we have $N \in [0,n]$ for all large enough $n$ so that our insistence that $N \in [0,n]$ in the first part of Lemma \ref{Limh} does not affect the asymptotics. The second part of the lemma, equation (\ref{ebounded}), also trivially holds if $N \not \in [0, n]$ since $e_{N}(1^{n}) = 0$ in this case by Definition \ref{ElemSymFunc}.
\end{remark}
\begin{proof}
For clarity the proof is split into several steps.\\
\textbf{Step 1.} In this step we prove (\ref{AsympForh}). Using Definition \ref{ElemSymFunc} we obtain
\begin{align}{\label{ElemSymFactorial}}
	e_{N}(1^n)=\frac{n!}{N!(n-N)!}
\end{align}

We have the following formula \cite{Rob} for $n\geq 1$
\begin{equation}{\label{Stirling}}
	n!=\sqrt{2\pi n}n^ne^{-n}e^{r_{n}}\text{, where }\frac{1}{12n+1}<r_{n}<\frac{1}{12n}
\end{equation}
Applying (\ref{Stirling}) to equation (\ref{ElemSymFactorial}) gives
\begin{equation}{\label{hStirling}}
\begin{split}
	& e_{N}(1^n)=\frac{\exp\left( (n+1/2)\log n-(N+1/2)\log N-(n-N+1/2)\log(n-N)+O\left(n^{-1}\right) \right)}{\sqrt{2\pi}} \\
	& =  (\sqrt{2\pi})^{-1}\cdot \exp\left((n+1/2)\log n -(N+1/2)\log \frac{N}{pn}-(n-N+1/2)\log \frac{n-N}{(1-p)n}\right) \cdot \\
	& \exp\left(-(N+1/2)\log(pn)-(n-N+1/2)\log((1-p)n)+O\left(n^{-1}\right)\right).
\end{split}
\end{equation}

Denote $\Delta=\sqrt{n}x=O\left(n^{1/2}\right)$, and we now use the Taylor expansion of the logarithm and the expression for $N$ to get
\begin{equation*}
	\log \frac{N}{pn}=\log\left(1+\frac{\Delta}{pn}\right) = \frac{\Delta}{pn}-\frac{1}{2}\frac{\Delta^2}{p^2 n^2}+O\left(n^{-3/2}\right)
\end{equation*}
Analogously, we have 
\begin{equation*}
\log \frac{n-N}{(1-p)n} =\log\left(1-\frac{\Delta}{(1-p)n}\right)=-\frac{\Delta}{(1-p)n}-\frac{1}{2}\frac{\Delta^2}{(1-p)^2 n^2}+O\left(n^{-3/2}\right)	
\end{equation*}

Plugging the two equations above into equation (\ref{hStirling}) we get
\begin{equation}{\label{eN}}
\begin{split}
	& e_{N}(1^n)= (\sqrt{2\pi})^{-1}\cdot \exp\left(-(N+1/2)\left[\frac{\Delta}{pn}-\frac{1}{2}\frac{\Delta^2}{p^2 n^2}+O\left(n^{-3/2}\right)\right]\right) \cdot \\
	&  \exp\left(-(n-N+1/2)\left[-\frac{\Delta}{(1-p)n}-\frac{1}{2}\frac{\Delta^2}{(1-p)^2 n^2}+O\left(n^{-3/2}\right)\right]\right) \cdot \\
	& \exp\left((n+1/2)\log n-(N+1/2)\log (pn)-(n-N+1/2)\log((1-p)n)+O\left(n^{-1}\right)\right)
\end{split}
\end{equation}

We next observe that 
\begin{equation}{\label{eNComp}}
	\begin{split}
		&-\frac{\Delta(N+1/2)}{pn}+\frac{(n-N+1/2)\Delta}{(1-p)n}=-\frac{\Delta^{2}}{p(1-p)n}+O\left(n^{-1/2}\right)\\
		 &\frac{\Delta^2(N+1/2)}{2n^2 p^2}+\frac{\Delta^2 (n-N+1/2)}{2(1-p)^2 n^2}=\frac{\Delta^2}{2p(1-p)n}+O\left(n^{-1/2}\right)\\
		&(n+1/2)\log n-(N+1/2)\log (pn)-(n-N+1/2)\log((1-p)n)=\\
		&N\log \frac{1-p}{p}-\frac{1}{2}\log p(1-p)-\frac{1}{2}\log n-n\log(1-p) 
	\end{split}
\end{equation}
Plugging (\ref{eNComp}) into (\ref{eN}) we arrive at (\ref{AsympForh}).\\
\textbf{Step 2.} In this step we prove (\ref{ebounded}). If $N=0$ or $n$ we know that $e_{N}(1^n)=1$ and then (\ref{ebounded}) is easily seen to hold with $C=1$ and any $c\in\left(0,\min(-\log p, -\log(1-p))\right)$. Thus it suffices to consider the case when $N\in[1,n-1]$ and in the sequel we also assume that $n\geq 2$.

Combining (\ref{ElemSymFactorial}) and (\ref{Stirling}) we conclude that 
\begin{equation}{\label{Inequlh}}
	e_{N}(1^n)\leq \exp\left((n+1/2)\log n-(N+1/2)\log N-(n-N+1/2)\log (n-N)\right)
\end{equation}
From (\ref{Inequlh}) we get for all large enough $n$ that
\begin{equation*}
\begin{split}
	&\phi_{n}:=\log\left[e_{N}(1^n)\cdot\exp\left(-N\log((1-p)/p)+n\log(1-p)+(1/2)\log n\right)\right]\\
	&\leq (n+\frac{1}{2})\log n-(N+\frac{1}{2})\log N-(n-N+\frac{1}{2})\log (n-N)-N\log\frac{1-p}{p}+n\log(1-p)+\frac{1}{2}\log n\\
	& =(n+1/2)\log n-(N+1/2)\log \frac{N}{pn}-(N+1/2)\log (pn)-(n-N+1/2)\log\frac{n-N}{(1-p)n}\\
	& -(n-N+1/2)\log((1-p)n)-N\log\frac{1-p}{p}+n\log(1-p)+(1/2)\log n\\
	&=-(N+1/2)\log\frac{N}{pn}-(n-N+1/2)\log\frac{n-N}{(1-p)n}-(1/2)\log \left(p(1-p)\right)\\
	& =-(pn+\Delta+1/2)\log \left(1+\frac{\Delta}{pn}\right)-((1-p)n-\Delta+1/2)\log\left(1-\frac{\Delta}{(1-p)n}\right)-\frac{1}{2}\log \left(p(1-p)\right)\\
	&\leq C_{1}+\psi_{n}(\Delta)
\end{split}
\end{equation*}
where $C_{1}>0$ is sufficiently large depending on $p$ alone and
\begin{equation}
	\psi_{n}(s)=-(pn+s+1/2)\log\left(1+\frac{s}{pn}\right)-((1-p)n-s+1/2)\log\left(1-\frac{s}{(1-p)n}\right)
\end{equation}
where $s\in [-pn+1,(1-p)n-1]$. We claim that we can find positive constants $C_2>0$ and $c>0$ such that for all $n$ sufficiently large and $s\in[-pn+1,(1-p)n-1]$ we have
\begin{equation}{\label{psinbounded}}
	\psi_{n}(s)\leq C_2-cn^{-1}s^{2}
\end{equation}
We prove (\ref{psinbounded}) in Step 3 below. For now we assume its validity and conclude the proof of (\ref{ebounded}).
In view of $\phi_{n}\leq C_1+\psi_{n}(s)$ and (\ref{psinbounded}) we know that
\begin{equation*}
	e_{N}(1^n)\leq \exp\left(C_1+C_2+N\log((1-p)/p)-n\log(1-p)-(1/2)\log n\right)\cdot\exp(-cn^{-1}(N-pn)^2),
\end{equation*}
which proves (\ref{ebounded}) with $C=e^{C_1+C_2}$.\\
\textbf{Step 3.} In this step we prove (\ref{psinbounded}). A direct computation gives
\begin{equation}{\label{CompDerivative}}
\begin{split}
	& \psi_{n}^{\prime}(s)=-\log\left(1+\frac{t}{pn}\right)+\log\left(1-\frac{t}{(1-p)n}\right)+\frac{1}{2}\cdot\frac{1}{pn+t}+\frac{1}{2}\cdot\frac{1}{(1-p)n-t}\\
	& \psi_{n}^{\prime\prime}(s)=\frac{(n+1)\cdot s^2+(2p-1)n(n+1)\cdot s+p(p-1)n^2(n+1)+(1/2)n^2}{(pn+s)^2((1-p)n-s)^2}
\end{split}
\end{equation}
Notice that the numerator of $\psi_{n}^{\prime\prime}(s)$ is a quadratic function with min at $x_{min}=-\frac{(2p-1)n(n+1)}{2(n+1)}=(-p+1/2)n$, which is the midpoint of the interval $[-pn+1,(1-p)n-1]$. Consequently, the numerator reaches its maximum at either one of the two endpoints of the interval $[-pn+1,(1-p)n-1]$. The denominator is the square of a parabola that reaches its minimum also at the endpoints of the interval $[-pn+1,(1-p)n-1]$. Therefore, we conclude that 
\begin{equation}{\label{psi2primebounded}}
\begin{split}
	& \psi_{n}^{\prime\prime}(s)\leq \psi_{n}^{\prime\prime}(-pn+1)=\psi_{n}^{\prime\prime}((1-p)n-1)=\frac{-\frac{1}{2}n^2+1}{(n-1)^2}=-\frac{1}{2} -\frac{1}{n-1}+\frac{1}{2}\cdot\frac{1}{(n-1)^2}\\
	& \leq -\frac{1}{2}\cdot\frac{1}{n-1}\leq -\frac{1}{2n}=-2cn^{-1}
\end{split}
\end{equation}
where $c=1/4$. Next, we prove (\ref{psinbounded}) under two cases when $s\in[-pn+1,0]$ and $s\in [0,(1-p)n-1]$, respectively. \\
$1^{\circ}$. When $s\in[-pn+1,0]$, by the fundamental theorem of calculus and (\ref{psi2primebounded}) we get 
\begin{equation*}
	\psi_{n}^{\prime}(s)=\psi_{n}^{\prime}(0)-\int_{s}^{0}\psi_{n}^{\prime\prime}(y)dy\geq \psi_{n}^{\prime}(0)-(-s)(-2cn^{-1})=\frac{2p-1}{2p(1-p)n}-2cn^{-1}s,
\end{equation*}
and a second application of the same argument yields for $s\in[-pn+1,0]$
\begin{equation*}
	\psi_{n}(s)=\psi_{n}(0)-\int_{s}^{0}\psi_{n}^{\prime}(y)dy\leq -\int_{s}^{0}\left(\frac{2p-1}{2p(1-p)n}-2cn^{-1}y\right) dy=\frac{(2p-1)s}{2p(1-p)n}-cn^{-1}s^2,
\end{equation*}
When $p\leq 1/2$, $\frac{(2p-1)s}{2p(1-p)n}\leq \frac{(2p-1)pn}{2p(1-p)n}=\frac{1-2p}{2(1-p)}$, so (\ref{psinbounded}) gets proved with $C_{2}=\frac{1-2p}{2(1-p)}$. When $p>1/2$, (\ref{psinbounded}) gets proved $C_{2}=0$.\\
$2^{\circ}$. When $s\in[0,(1-p)n-1]$, using the fundamental theorem of calculus and (\ref{psi2primebounded}) we get 
\begin{equation*}
	\psi_{n}^{\prime}(s)=\psi_{n}^{\prime}(0)+\int_{0}^{s}\psi_{n}^{\prime\prime}(y)dy\leq =\frac{2p-1}{2p(1-p)n}-2cn^{-1}s,
\end{equation*}
and a second application of the same argument yields for $s\in[0,(1-p)n-1]$
\begin{equation*}
	\psi_{n}(s)=\psi_{n}(0)+\int_{0}^{s}\psi_{n}^{\prime}(y)dy\leq \frac{(2p-1)s}{2p(1-p)n}-cn^{-1}s^2,
\end{equation*}
When $p\geq 1/2$, $\frac{(2p-1)s}{2p(1-p)n}\leq \frac{(2p-1)(1-p)n}{2p(1-p)n}=\frac{2p-1}{2p}$, so (\ref{psinbounded}) gets proved with $C_{2}=\frac{2p-1}{2p}$. When $p<1/2$, (\ref{psinbounded}) gets proved $C_{2}=0$. Combining cases $1^{\circ}$ and $2^{\circ}$ we complete the proof.
\end{proof}

\begin{lemma}{\label{ALambdaBLambda}}
	Assume the same notation as in Definition \ref{DefScaled}. Fix $\vec{z}\in\mathbb{R}^{k}$ such that $z_1>\cdots>z_k$. Suppose that $T_{0}\in\mathbb{N}$ is sufficiently large so that for $T\geq T_{0}$ we have
$$z_{k}\sqrt{T}+ptT\geq a_1\sqrt{T}+k+1 \text{ and } b_{k}\sqrt{T}+pT\geq z_{1}\sqrt{T}+ptT+k+1,$$
and define $\lambda_i^T = \lfloor z_{i}\sqrt{T}+ptT \rfloor$ for $i = 1, \dots, k$ (to ease notation we suppress the dependence of $\lambda$ on $T$ in what follows). Setting $m = \lfloor tT \rfloor $ and $n = T - m$ define
\begin{align}{\label{DefALambda}}
	A_\lambda(T) =\det\left(e_{\lambda_{i}-x_{j}^{T}-i+j}(1^m)\right)_{1\leq i,j\leq k}\cdot \det\left(e_{y_{i}^T-\lambda_{j}-i+j}(1^n)\right)_{1\leq i,j\leq k},
\end{align} 
\begin{align}{\label{DefBLambda}}
\begin{split}
	& B_{\lambda}(T)=(\sqrt{2\pi})^{k}\cdot\exp\left(kT\log(1-p)+k\log T +(k/2)\log(p(1-p))\right) \cdot\\
	&  \exp\left(-\log\left(\frac{1-p}{p}\right)\sum_{i=1}^{k}(y_{i}^{T}-x_{i}^{T})\right)\cdot A_{\lambda}(T)
\end{split}
\end{align}
We claim that
\begin{align}{\label{BLambda}}
\begin{split}
&\lim_{T\rightarrow\infty} B_{\lambda}(T)=(2\pi)^{-k/2}\cdot\exp(-(k/2)\log(p(1-p))-(k/2)\log(t(1-t))) \cdot  \\
&\det\left[e^{c_1(t,p)a_i z_j}\right]_{i,j=1}^{k} \cdot \det\left[e^{c_2(t,p)b_i z_j}\right]_{i,j=1}^{k} \cdot\prod_{i=1}^{k}\exp\left(-\frac{c_1(t,p)a_i^{2}+c_2(t,p)b_i^2}{2}\right).
\end{split}
\end{align}
\end{lemma}
\begin{proof}
Let us write 
$$A^1_\lambda  = \det\left(e_{\lambda_{i}-x_{j}^{T}-i+j}(1^m)\right)_{1\leq i,j\leq k}, \hspace{2mm} A_\lambda^2 =  \det\left(e_{y_{i}^T-\lambda_{j}-i+j}(1^n)\right)_{1\leq i,j\leq k} \mbox{, and }$$
$$A^3_\lambda = \det\left(e_{y_{i}^T-x_{j}^T-i+j}(1^{m+n})\right)_{1\leq i,j\leq k}.$$
Then from Lemma \ref{Limh} we have
\begin{align}{\label{SkewSchur1}}
	\begin{split}
		& A^1_\lambda=\det\left[\exp\left(-\frac{(\lambda_{i}-x_{j}^{T}+j-i-pm)^{2}}{2(1-p)pm}\right)\exp\left(O\left( T^{-1/2}\right)\right)\right]\cdot (\sqrt{2\pi})^{-k}\cdot\\
		& \exp\left(-km\log(1-p)-(k/2)\log m-(k/2)\log(p(1-p))+\log\left(\frac{1-p}{p}\right)\sum_{i=1}^{k}(\lambda_{i}-x^{T}_{i})\right)
	\end{split}
\end{align}
\begin{align}{\label{SkewSchur2}}
	\begin{split}
		&A^2_\lambda =\det\left[\exp\left(-\frac{(y_i^{T}-\lambda_{j}+j-i-pn)^{2}}{2(1-p)pn}\right)\exp\left(O\left( T^{-1/2}\right)\right)\right]\cdot (\sqrt{2\pi})^{-k}\cdot\\
		& \exp\left(-kn\log(1-p)-(k/2)\log n-(k/2)\log(p(1-p))+\log\left(\frac{1-p}{p}\right)\sum_{i=1}^{k}(y_{i}^{T}-\lambda_{i})\right)
	\end{split}
\end{align}
\begin{align}{\label{SkewSchur3}}
	\begin{split}
		&A^3_\lambda =\det\left[\exp\left(-\frac{(y_{i}^{T}-x_{j}^{T}+j-i-pT)^{2}}{2(1-p)pT}\right)\exp\left(O\left( T^{-1/2}\right)\right)\right]\cdot (\sqrt{2\pi})^{-k}\cdot\\
		& \exp\left(-kT\log(1-p)-(k/2)\log T-(k/2)\log(p(1-p))+\log\left(\frac{1-p}{p}\right)\sum_{i=1}^{k}(y_{i}^{T}-x^{T}_{i})\right)
	\end{split}
\end{align}
where the constants in the big $O$ notation are uniform as $z_{i}$ vary over compact subsets of $\mathbb{R}$. Combining (\ref{SkewSchur2}), (\ref{SkewSchur1}) and (\ref{DefBLambda}) we see that
\begin{align}{\label{BLambdafinal}}
	\begin{split}
		B_{\lambda}&(T)=(2\pi)^{-k/2}\cdot\exp(-(k/2)\log(p(1-p))-(k/2)\log(t(1-t))+O(T^{-1})) \cdot \\
		& \det\left[\exp\left(-\frac{(z_i-a_j)^2}{2p(1-p)t}+O(T^{-1/2})\right)\right]\cdot\det\left[\exp\left(-\frac{(b_i-z_j)^2}{2p(1-p)(1-t)}+O(T^{-1/2})\right)\right]
	\end{split}
\end{align}
Taking the limit $T\rightarrow\infty$ in (\ref{BLambdafinal}), and using the identities
\begin{align}\label{DetIdentities}
	\begin{split}
		\det\left[\exp\left(-\frac{(z_{i}-a_j)^2}{2p(1-p)t}\right)\right]=\det\left[e^{c_1(t,p)a_i z_j}\right]_{i,j=1}^{k}\cdot \prod_{i=1}^{k}\exp\left(-\frac{c_1(t,p)}{2}(a_i^2+z_i^2)\right)\text{, and}\\
		\det\left[\exp\left(-\frac{(b_{i}-z_j)^2}{2p(1-p)(1-t)}\right)\right]=\det\left[e^{c_2(t,p)b_i z_j}\right]_{i,j=1}^{k}\cdot \prod_{i=1}^{k}\exp\left(-\frac{c_2(t,p)}{2}(b_i^2+z_i^2)\right)
	\end{split}
\end{align}
we get (\ref{BLambda}).
\end{proof}

\begin{lemma}{\label{NonVanish}}
	Suppose the vector $\vec{m}=(m_1,\dots,m_p)$ satisfies $k=\sum_{i=1}^{p}m_{i}$, and $\alpha_1>\alpha_2>\cdots>\alpha_p$. Then the following determinant
\[ U= \det
	\left[ \begin{array}{ccc}
		(z_{j}^{i-1}e^{\alpha_{1}z_{j}})_{\substack{i=1,\dots,m_{1}\\j=1,\dots,k}}\\
	\vdots\\
	(z_{j}^{i-1}e^{\alpha_{p}z_{j}})_{\substack{i=1,\dots,m_{p}\\j=1,\dots,k}}
	\end{array}
	\right]
\]
is non-zero for any $\vec{z} = (z_{1},\dots,z_{k}) \in \mathbb{R}^k$ whose entries are distinct.
\end{lemma}
\begin{proof} We claim that, the following equation with respect to $z$ over $\mathbb{R}$
$$(\xi_{1}+\xi_{2}z+\cdots+\xi_{m_1}z^{m_{i}-1})e^{\alpha_{1}z}+\cdots + (\xi_{m_{1}+\cdots+m_{p-1}+1}+\cdots+\xi_{k}z^{m_{p}-1})e^{\alpha_{p}z}=0$$ has at most $(k-1)$ distinct roots, where $(\xi_{1},\dots,\xi_{k})\in\mathbb{R}^{k}$ is non-zero.\\

Denote the rows of the matrix in the definition of $U$ by $v_1, \dots, v_k$. If the above claim holds, we can conclude that we cannot find non-zero $(\xi_{1},\cdots,\xi_{k})\in\mathbb{R}^{k}$ such that $\xi_{1}v_{1}+\cdots+\xi_{k}v_{k}=0$. Thus, the $k$ row vectors of the determinant are linearly independent and the determinant is non-zero. Thus it suffices to prove the claim, and we do it by induction on $k$.\\
$1^{\circ}$. If $k=2$, the equation is $(\xi_{1}+\xi_{2}z)e^{\alpha_{1}z}=0$ or $\xi_{1}e^{\alpha_{1}z}+\xi_{2}e^{\alpha_{2}z}=0$, where $\xi_{1},\xi_{2}\in\mathbb{R}$ cannot be zero at the same time. Then, it's easy to see that the equation has at most $1$ root in two scenarios.\\
$2^{\circ}$. Suppose the claim holds for $k\leq n$.\\
$3^{\circ}$. When $k=n+1$, we have the equation $$(\xi_{1}+\xi_{2}z+\cdots+\xi_{m_1}z^{m_{i}-1})e^{\alpha_{1}z}+\cdots(\xi_{m_{1}+\cdots+m_{p-1}+1}+\cdots+\xi_{k}z^{m_{p}-1})e^{\alpha_{p}z}=0$$ but now $\sum_{i=1}^{p}m_{i}=n+1$. WLOG, suppose $(\xi_{1},\dots,\xi_{m_{1}})$ has a non-zero element and $\xi_{\ell}$ is the first non-zero element. Notice that the above equation has the same roots as the following one:
$$F(z)=(\xi_{\ell}z^{\ell-1}+\cdots+\xi_{m_1}z^{m_1-1})+\cdots+(\xi_{m_1+\cdots+m_{p-1}+1}+\cdots+\xi_{k}z^{m_{p}-1})e^{(\alpha_{p}-\alpha_{1})z}=0$$
Assume it has at least $(n+1)$ distinct roots $\eta_{1}<\eta_{2}<\cdots<\eta_{n+1}$. Then $F^{\prime}(z)=0$ has at least $n$ distinct roots $\delta_{1}<\cdots<\delta_{n}$ such that $\eta_{1}<\delta_{1}<\eta_{2}<\cdots<\delta_{n}<\eta_{n+1}$, by Rolle's Theorem. Actually, $F^{\prime}(z)=(\xi_{\ell}(\ell-1))z^{\ell-2}+\cdots+\xi_{m_1}(m_1-1)z^{m_{1}-2})+\cdots+(\xi_{m_1+\cdots+m_{p-1}+1}^{\prime}+\cdots+\xi_{k}^{\prime}z^{m_{p}-1})e^{(\alpha_{p}-\alpha_{1})z}=0$
where $\xi_{i}^{\prime}$, $i=m_1+1,\cdots,k$ are coefficients that can be calculated. This equation has at most $(m_1-1)+m_2+\cdots+m_{p}-1=n-1$ roots by $2^{\circ}$, which leads to a contradiction. Therefore, our claim holds and we have proved the lemma.
\end{proof}

\begin{proof}(of Proposition \ref{S9PropH} when $\vec{a}, \vec{b} \in W_k^{\circ}$) Let us fix $\vec{z} \in W_k^{\circ}$, and define $\lambda^T$ as in Lemma \ref{ALambdaBLambda}. We also let $x_i^T$ and $y_i^T$ be sequences of integers such that 
 $$\lim_{T\rightarrow\infty}\frac{x_{i}^{T}}{\sqrt{T}}=a_{i} \text{ and } \lim_{T\rightarrow\infty}\frac{y_{i}^{T}-pT}{\sqrt{T}}=b_{i}$$ for $i=1,\dots,k$. 
In view of the Jacobi-Trudi formula (\ref{J-TFormula}) we know that $B_\lambda(T)$ as in Lemma \ref{ALambdaBLambda} are non-negative and from (\ref{BLambda}) they converge as $T$ tends to infinity to 
\begin{align*}
\begin{split}
&(2\pi)^{-k/2}\cdot\exp(-(k/2)\log(p(1-p))-(k/2)\log(t(1-t))) \cdot  \\
&\det\left[e^{c_1(t,p)a_i z_j}\right]_{i,j=1}^{k} \cdot \det\left[e^{c_2(t,p)b_i z_j}\right]_{i,j=1}^{k} \cdot\prod_{i=1}^{k}\exp\left(-\frac{c_1(t,p)a_i^{2}+c_2(t,p)b_i^2}{2}\right).
\end{split}
\end{align*}
On the other hand, we have that when the entries of $\vec{a}, \vec{b}$ are distinct 
$$H(\vec{z}) = \det\left[e^{c_1(t,p)a_i z_j}\right]_{i,j=1}^{k} \cdot \det\left[e^{c_2(t,p)b_i z_j}\right]_{i,j=1}^{k} \cdot \prod_{i=1}^{k}e^{-c_{3}(t,p)z_{i}^{2}}.$$
The last two statements imply that $H(\vec{z}) \geq 0$ and from Lemma \ref{NonVanish} we have $H(\vec{z}) \neq 0$ so that $H(\vec{z}) > 0$ for $\vec{z} \in W_k^{\circ}$. If $\vec{z} \in W_k \setminus W_k^{\circ}$ then $z_i = z_j$ for some $i \neq j$ and then we see that $H(\vec{z}) = 0$ since the matrices in determinants in the equation above for $H(\vec{z})$ have $i$-th and $j$-th column that are equal, which makes the determinant vanish. This proves the first two statements in the proposition. 

To prove the third statement observe that by the continuity, non-negativity of $H(\vec{z})$ and the fact that it is strictly positive in the open set $W_k^{\circ}$ we know that $Z_c \in (0, \infty]$ and so we only need to prove that $Z_c < \infty$. Using the formula $$\det\left[A_{i,j}\right]_{i,j=1}^{k}=\sum_{\sigma\in S_{k}}(-1)^{\sigma}\cdot\prod_{i=1}^{k}A_{i,\sigma(i)}$$ and the triangle inequality we see that 
\begin{equation}{\label{Det1bounded2}}
	\begin{split}
		& \left|\det\left[e^{c_1(t,p)a_i z_j}\right]_{i,j=1}^{k}\right|\leq\sum_{\sigma\in S_k}\prod_{j=1}^{k}e^{c_1(t,p)a_{\sigma(j)}z_{j}}\leq \sum_{\sigma\in S_k}\prod_{j=1}^{k}e^{c_1(t,p)\left(\sum_{i=1}^{k}|a_i|\right)\cdot|z_j|}\\
		& \leq k! \cdot \prod_{i=1}^{k}e^{C_1|z_j|}\text{, where $C_{1}=\sum_{i=1}^{k}c_1(t,p)|a_i|$}
	\end{split}
\end{equation}
Analogously, define the constant $C_2=\sum_{i=1}^{k}c_2(t,p)|b_i|$ and we have
\begin{equation}{\label{Det2bounded2}}
	\left|\det\left[e^{c_2(t,p)b_i z_j}\right]_{i,j=1}^{k}\right|\leq k! \cdot \prod_{i=1}^{k}e^{C_{2}|z_{j}|}
\end{equation} 
Using (\ref{Det1bounded2}) and (\ref{Det2bounded2}) we get
\begin{align}{\label{Dominate}}
	|H(\vec{z})|\leq (k!)^2\cdot\prod_{i=1}^{k}e^{C|z_i|-c_3(t,p)z_i^2}
\end{align} where $C=C_{1}+C_{2}$. Since the right side of (\ref{Dominate}) is integrable (because of the square in the exponential) we conclude that $H(\vec{z})$ is also integrable by domination and so $Z_c < \infty$ as desired.
\end{proof}

%
\subsection{Proof of Proposition \ref{WeakConvDistinct} for $\vec{a}, \vec{b} \in W_k^{\circ}$}\label{Section9.4} For clarity we split the proof into several steps.\\
{\raggedleft \textbf{Step 1.}} In this step we prove that $Z_c$ from Proposition \ref{S9PropH} in the case when $\vec{a}, \vec{b}$ have distinct entries satisfies the equation
\begin{equation}\label{ZCform}
Z_c = (2\pi)^{\frac{k}{2}}(p(1-p)t(1-t))^{\frac{k}{2}}\cdot e^{\frac{c_{1}(t,p)}{2}\sum_{i=1}^{k}a_{i}^{2}}\cdot e^{\frac{c_{2}(t,p)}{2}\sum_{i=1}^{k}b_{i}^{2}}\det\left[e^{-\frac{1}{2p(1-p)}(b_{i}-a_{j})^{2}}\right]_{i,j=1}^{k}.
\end{equation}

Let $B_\lambda(T)$ be as in Lemma \ref{ALambdaBLambda} for $\lambda \in \mathfrak{W}_k$, with $\vec{x}^T, \vec{y}^T$ as in the statement of the proposition. It follows from Lemma \ref{BerDist} that 
\begin{align}{\label{BSum}}
\begin{split}
& \sum_{\lambda\in\mathfrak{W}_{k}}\frac{B_{\lambda}(T)}{T^{k/2}}=(\sqrt{2\pi})^k\cdot\exp(kT\log(1-p)+(k/2)\log T+(k/2)\log p(1-p))\cdot\\ &  \exp\left(-\log\left(\frac{1-p}{p}\right)\sum_{i=1}^{k}(y_{i}^{T}-x_{i}^{T})\right)\cdot \det\left(e_{y_{i}^T-x_{j}^T-i+j}(1^{m+n})\right)_{1\leq i,j\leq k},
\end{split}
\end{align}
where we recall that $m = \lfloor t T \rfloor$ and $n = T - m$. Taking the $T \rightarrow \infty$ limit in (\ref{BSum}) and using (\ref{SkewSchur3}) we obtain 
\begin{align}{\label{LimSum}}
	\lim_{T\rightarrow\infty}\sum_{\lambda\in\mathfrak{W}_{k}}\frac{B_{\lambda}(T)}{T^{k/2}}=\det\left[e^{-\frac{1}{2p(1-p)}(b_i-a_j)^2}\right]_{i,j=1}^{k}.
\end{align}

For $\lambda\in\mathfrak{W}_{k}$ and $T\in\mathbb{N}$ we define $Q_{\lambda}(T)$ to be the cube $[\lambda_{1}T^{-1/2}-pt\sqrt{T},(\lambda_{1}+1)T^{-1/2}-pt\sqrt{T})\times \cdots\times[\lambda_{k}T^{-1/2}-pt\sqrt{T},(\lambda_{k}+1)T^{-1/2}-pt\sqrt{T})$ and note that $Q_\lambda(T)$ has Lebesgue measure $T^{-k/2}$. In addition, we define the step functions $f_{T}$ through
\begin{align}
	f_{T}(\vec{z})=\sum_{\lambda\in\mathfrak{W}_{k}}B_{\lambda}(T)\cdot\mathbf{1}_{Q_{\lambda}(T)}(\vec{z})
\end{align}
and observe that 
\begin{align}{\label{IntSimpleFunc}}
	\sum_{\lambda\in\mathfrak{W}_{k}}\frac{B_{\lambda}(T)}{T^{k/2}}=\int_{\mathbb{R}^{k}}f_{T}(\vec{z})d\vec{z}
\end{align}
where $d\vec{z}$ represents the usual Lebesgue measure on $\mathbb{R}^{k}$.

In view of (\ref{BLambda}) we know that for almost every $\vec{z}=(z_1,\cdots,z_k)\in\mathbb{R}^{k}$ we have 
\begin{align}{\label{PointwiseConvfT}}
\begin{split}
&\lim_{T\rightarrow\infty}f_{T}(\vec{z})={\bf 1}_{\{z_1 > \cdots > z_k \}} \cdot H(z) \cdot(2\pi p(1-p)t(1-t))^{-\frac{k}{2}}\cdot \\
&\prod_{i=1}^{k}\exp\left(-\frac{c_1(t,p)a_i^2+c_2(t,p)b_i^2}{2}\right).
\end{split}
\end{align}
We claim that there exists a non-negative integrable function $g$ on $\mathbb{R}^{k}$ such that if $T$ is large enough 
\begin{align}{\label{Dominatefg}}
	|f_{T}(z_1,\dots,z_k)|\leq|g(z_1,\dots,z_k)|
\end{align}
We will prove (\ref{Dominatefg}) in Step 2 below. For now we assume its validity and conclude the proof of (\ref{ZCform}).\\

From (\ref{PointwiseConvfT}) and the dominated convergence theorem with dominating function $g$ as in (\ref{Dominatefg}) we know that
\begin{equation}
	\begin{split}{\label{LimIntfT}}
	\lim_{T\rightarrow\infty}\int_{\mathbb{R}^{k}}f_T(\vec{z})d\vec{z}=\int_{W_k}H(\vec{z})(2\pi p(1-p)t(1-t))^{-\frac{k}{2}}\prod_{i=1}^{k}\exp\left(-\frac{c_{1}(t,p)a_i^2+c_2(t,p)b_i^2}{2}\right)d\vec{z}.
\end{split}
\end{equation}
Combining (\ref{LimIntfT}), (\ref{IntSimpleFunc})  and (\ref{LimSum}) we conclude that
\begin{equation}
	\begin{split}
		\det\left[e^{-\frac{1}{2p(1-p)}(b_i-a_j)^2}\right]_{i,j=1}^{k}=\int_{W_k}H(\vec{z})\cdot(2\pi p(1-p)t(1-t))^{-\frac{k}{2}}\cdot\prod_{i=1}^{k}e^{-\frac{c_1(t,p)a_i^2+c_2(t,p)b_i^2}{2}}d\vec{z}.
	\end{split}
\end{equation}
which clearly establishes (\ref{ZCform}).\\

{\raggedleft \textbf{Step 2. }}In this step we demonstrate an integrable function $g$ that satisfies (\ref{Dominatefg}). Let us fix $\lambda\in\mathfrak{W}_k$. If $\lambda_i\geq x_{i}^{T}+m+1$ or $\lambda_i<x^T_{i}$ for some $i\in\{1,2,\dots,k\}$ we know that 
$$ \det\left(e_{\lambda_{i}-x_{j}^{T}-i+j}(1^m)\right)_{1\leq i,j\leq k} = 0.$$
To see this, observe that if $\lambda_s\geq x_{s}^{T}+m+1$ then the top-right $s \times (k-s)$-th block in the matrix consists of zeros (since $e_N(1^m) = 0$ for $N \geq m+1$). Thus if $A$ and $B$ are the top-left $(s-1) \times (s-1)$ submatrix and bottom-right $(k-s+1) \times (k-s + 1)$ submatrix we would have that 
$$\det\left(e_{\lambda_{i}-x_{j}^{T}-i+j}(1^m)\right)_{1\leq i,j\leq k} = \det A \cdot \det B,$$
but then $\det B = 0$ since its top row consists of $0$'s. Similar arguments show that the determinant is $0$ if $\lambda_s<x^T_{s}$ for some $s\in\{1,2,\dots,k\}$, where now we would get a block of $0$'s in the bottom left corner using $e_N(1^m) = 0$ for $N < 0$. From the definition of $B_\lambda(T)$ we conclude that $B_\lambda(T) = 0$ if $\lambda_i\geq x_{i}^{T}+m+1$ or $\lambda_i<x^T_{i}$ . Similarly, we have that $B_\lambda(T) = 0$ if $y_{i}^T\geq \lambda_i+n+1$ or $y_{i}^T< \lambda_i$ for some $i\in\{1,2,\dots,k\}$, using that
$$ \det\left(e_{y_{i}^{T} - \lambda_{j} -i+j}(1^n)\right)_{1\leq i,j\leq k} = 0$$
in this case. Overall, we conclude that $B_{\lambda}(T)=0$ unless
$$m\geq \lambda_i-x^{T}_i\geq 0\text{ and } n\geq y^{T}_i-\lambda_i\geq 0\text{ for all } i\in\{1,\dots,k\}$$
which implies that for all large enough $T$ we have
\begin{equation}{\label{BLambda0unless}}
	\begin{split}
		B_{\lambda}(T)=0 \text{, unless }|\lambda_{i}-x_{j}^T+j-i|\leq (1+p)m \text{ and }|y_{i}^T-\lambda_{j}+j-i|\leq (1+p)n
	\end{split}
\end{equation}
for all $i,j\in\{1,\cdots, k\}$. To see the latter, suppose that there exist $i,j$ such that $(1+p)m<|\lambda_i-x^{T}_{j}+j-i|$. Then we have 
\begin{equation*}
	\begin{split}
		(1+p)m<|\lambda_i-x^{T}_{j}+j-i|\leq |\lambda_i - x^{T}_i| + k + |x^T_i - x^T_j| = |\lambda_i - x^{T}_i| + O(\sqrt{T}).
	\end{split}
\end{equation*}
When $T$ is sufficiently large, the above inequality implies $\lambda_{i}-x_i^T \not \in [0,m]$ so that $B_{\lambda}(T)=0$, and similar result holds for $y_{i}^T-\lambda_{j}+j-i$, which justifies (\ref{BLambda0unless}). From the definition of $B_{\lambda}(T)$ we know
\begin{equation}{\label{BHC}}
	\begin{split}
		& B_{\lambda}(T)=C_{T}\cdot \det[E(\lambda_i-x_j^{T}+j-i,m)]_{i,j=1}^{k}\cdot\det[E(y_i^{T}-\lambda_j+j-i,n)]_{i,j=1}^{k}\text{, where}\\
		& E(N,n)=e_{N}(1^n)\cdot\exp\left(-N\log\left(\frac{1-p}{p}\right)+n\log(1-p)+(1/2)\log n\right)\text{, and}\\
		& C_T=(\sqrt{2\pi})^k (p(1-p))^{k/2}\cdot\exp(k\log T-(k/2)\log n-(k/2)\log m).
	\end{split}
\end{equation}
Notice that $C_{T}$ is uniformly bounded for all $T$ large enough, because
\begin{equation}
\begin{split}
	k\log T-\frac{k}{2}\log n-\frac{k}{2}\log m=\frac{k}{2}\log\left(\frac{T^2}{\lfloor tT \rfloor \cdot (T-\lfloor tT \rfloor)}\right)=-\frac{k}{2}\log(t(1-t))+O\left(T^{-1}\right)
\end{split}
\end{equation}
and $O\left(T^{-1}\right)$ is uniformly bounded.

In view of (\ref{ebounded}) we know that we can find constants $C_1$, $c_1>0$ such that for all large enough $T$ and $N_{1}\in[0,m]$ and $N_2\in[0,n]$ we have 
\begin{equation}{\label{H1H2}}
	\begin{split}
		E(N_1,m)\leq C_1\exp(-c_1 m^{-1}(N_1-pm)^2) \text{ and } E(N_2,n)\leq C_1\exp(-c_1 n^{-1}(N_2-pn)^2)
	\end{split}
\end{equation}
Observing that $e_{r}(1^{n})=0$ for $r>n$ or $r<0$, we know that (\ref{H1H2}) also holds for all $N_1\in[-(1+p)m,(1+p)m]$ and $N_2\in[-(1+p)n,(1+p)n]$. Combining (\ref{BLambda0unless}), (\ref{BHC}) and (\ref{H1H2}) we see that for all $\lambda\in\mathfrak{W}_k$ and $T$ sufficiently large
\begin{equation}{\label{BLambdaBounded}}
	\begin{split}
		0\leq B_{\lambda}(T)\leq \widetilde{C}\hspace{-2mm}\sum_{\sigma ,\tau \in S_{k}} \prod_{i = 1}^k\mathbf{1}\{|\lambda_i-x_j^{T}+j-i|\leq(1+p)m\}\cdot \mathbf{1}\{|y_{i}^{T}-\lambda_{j}+j-i|\leq(1+p)n\} \cdot \\
		 \exp\left(-\widetilde{c}T^{-1}\left[(\lambda_i-\sqrt{T}a_{\sigma(i)}-ptT)^2+(\sqrt{T}b_{i}-\lambda_{\tau(i)}+ptT)^2\right]\right)
	\end{split}
\end{equation}
where $\widetilde{c}$, $\widetilde{C}>0$ depend on $p,t,k$ but not on $T$ provided that it is sufficiently large.

In particular, we see that if $\vec{z} \in\mathbb{R}^{k}$ then either $\vec{z} \not\in Q_{\lambda}(T)$ for any $\lambda\in\mathfrak{W}_{k}$ in which case $f_{T}(\vec{z})=0$ or $\vec{z} \in Q_{\lambda}(T)$ for some $\lambda\in\mathfrak{W}_k$ in which case (\ref{BLambdaBounded}) implies
\begin{equation}{\label{DominatefT}}
	\begin{split}
		0\leq f_T(\vec{z})\leq C\sum_{\sigma, \tau \in S_{k}}\prod_{i = 1}^k\exp\left(-c((z_i-a_{\sigma(i)})^2+(b_i-z_{\tau(i)})^2)\right)
	\end{split}
\end{equation}
where $C$, $c>0$ depend on $p,t,k$ but not on $T$ provided that it is sufficiently large. We finally see that (\ref{Dominatefg}) holds with $g$ being equal to the right side of (\ref{DominatefT}), which is clearly integrable.\\

{\raggedleft \bf Step 3.} Our work in Steps 1 and 2 implies that the density $\rho(\vec{z})$ we want to prove to be the weak limit of $Z^T$ has the form
\begin{equation}\label{rhoDist}
\begin{split}
&\rho(\vec{z}) =Z_c^{-1} \cdot \det\left[e^{c_1(t,p)a_i z_j}\right]_{i,j=1}^{k} \cdot \det\left[e^{c_2(t,p)b_i z_j}\right]_{i,j=1}^{k} \cdot \prod_{i=1}^{k}e^{-c_{3}(t,p)z_{i}^{2}}, \mbox{ where } \\ 
& Z_c = (2\pi)^{\frac{k}{2}}(p(1-p)t(1-t))^{\frac{k}{2}}\cdot e^{\frac{c_{1}(t,p)}{2}\sum_{i=1}^{k}a_{i}^{2}}\cdot e^{\frac{c_{2}(t,p)}{2}\sum_{i=1}^{k}b_{i}^{2}}\det\left[e^{-\frac{1}{2p(1-p)}(b_{i}-a_{j})^{2}}\right]_{i,j=1}^{k}. 
\end{split}
\end{equation}

We fix a compact set $K \subset W_k^{\circ}$ and for $\vec{z} \in K$ we define $\lambda^T(\vec{z}) \in \mathfrak{W}_k$ through 
$$\lambda^T_i(\vec{z}) =  \lfloor ptT + z_i T^{1/2} \rfloor \mbox{ for $i = 1, \dots, k $}.$$
In this step we prove that 
\begin{equation}\label{UniformConvSlice}
\lim_{T\rightarrow \infty} T^{k/2} \cdot \mathbb{P}_{avoid, Ber}^{0,T,\vec{x}^{T},\vec{y}^{T}}(L^T_{1}(m) = \lambda^T_{1}(\vec{z}), \cdots, L^T_{k}(m) = \lambda^T_{k}(\vec{z})) = \rho(\vec{z}),
\end{equation}
where the convergence is uniform over $K$. Combining (\ref{ProbMassFunc}), (\ref{DefBLambda}),(\ref{SkewSchur3}), (\ref{BLambdafinal}), (\ref{DetIdentities}) we get 
\begin{equation*}
\begin{split}
&T^{k/2} \cdot \mathbb{P}_{avoid, Ber}^{0,T,\vec{x}^{T},\vec{y}^{T}}(L^T_{1}(m) = \lambda^T_{1}(\vec{z}), \cdots, L^T_{k}(m)= \lambda^T_{k}(\vec{z}))  = [1 + O(T^{-1/2})] (2\pi)^{-k/2}\cdot \\
&\det\left[e^{c_1(t,p)a_i z_j}\right]  \cdot \det\left[e^{c_2(t,p)b_i z_j}\right]\cdot \exp(-(k/2)\log(p(1-p))-(k/2)\log(t(1-t))  ) \cdot \\
& \prod_{i=1}^{k}\exp\left(-\frac{c_1(t,p)}{2}(a_i^2+z_i^2)-\frac{c_2(t,p)}{2}(b_i^2+z_i^2)\right) \cdot \det\left[\exp\left(-\frac{(y_{i}^{T}-x_{j}^{T}+j-i-pT)^{2}}{2(1-p)pT}\right)\right]^{-1},
\end{split}
\end{equation*}
where the constants in the big $O$ notation are uniform over $K$. Using that 
$$\det\left[\exp\left(-\frac{(y_{i}^{T}-x_{j}^{T}+j-i-pT)^{2}}{2(1-p)pT}\right)\right] = \det\left[e^{-\frac{1}{2p(1-p)}(b_{i}-a_{j})^{2}}\right] \cdot [1 + o(1)],$$
where the constant in the little $o$ notation does not depend on $K$ and (\ref{rhoDist}) we see that 
$$T^{k/2} \cdot \mathbb{P}_{avoid, Ber}^{0,T,\vec{x}^{T},\vec{y}^{T}}(L^T_{1}(m) = \lambda^T_{1}(\vec{z}), \cdots, L^T_{k}(m)= \lambda^T_{k}(\vec{z})) = [1 + O(T^{-1/2})]  [1 + o(1)] \cdot \rho(\vec{z}),$$
which implies (\ref{UniformConvSlice}).\\

\noindent \textbf{Step 4.} In this step, we prove that for any compact rectangle $R=[u_{1},v_{1}]\times\cdots\times[u_{k},v_{k}]\subset {W}_{k}^{\circ}$
\begin{align}\label{ConvRectangle}
\lim_{T\rightarrow\infty}\mathbb{P}(Z^T\in R)=\int_{R}\rho(\vec{z})d\vec{z},
\end{align}
where we have written $\mathbb{P}$ in place of $\mathbb{P}_{avoid, Ber}^{0,T,\vec{x}^{T},\vec{y}^{T}}$ to ease the notation.

Define $m_{i}^{T}=\lceil u_{i}\sqrt{T}+ptT\rceil$ and $M_{i}^{T}=\lfloor v_{i}\sqrt{T}+ptT\rfloor$. Then we have:
\begin{align*}
&\mathbb{P}\left(Z^T \in R\right) =\mathbb{P}\left(u_{i}\sqrt{T}+ptT\leq L^T_{i}(\lfloor tT\rfloor) \leq v_{i}\sqrt{T}+ptT, i=1,\dots, k\right)\\
&=\sum_{\lambda_{1}=m_{1}^{T}}^{M_{1}^{T}}\cdots\sum_{\lambda_{k}=m_{k}^{T}}^{M_{k}^{T}}\mathbb{P}(L^T_{1}(\lfloor tT \rfloor)=\lambda_{1},\dots,L^T_{k}(\lfloor tT \rfloor)=\lambda_{k})\\
&=\sum_{\lambda_{1}=m_{1}^{T}}^{M_{1}^{T}}\dots\sum_{\lambda_{k}=m_{k}^{T}}^{M_{k}^{T}}T^{-k/2} \cdot T^{k/2} \cdot \mathbb{P}(L^T_{1}(\lfloor tT \rfloor)=\lambda_{1},\dots,L^T_{k}(\lfloor tT \rfloor)=\lambda_{k}) = \int_{\mathbb{R}^k} h_T(\vec{z}) d\vec{z},
\end{align*}
where $h_T(\vec{z})$ is the step function
$$h_T(\vec{z}) = \sum_{\lambda_{1}=m_{1}^{T}}^{M_{1}^{T}}\dots\sum_{\lambda_{k}=m_{k}^{T}}^{M_{k}^{T}} \mathbf{1}_{Q_{\lambda}(T)}(\vec{z}) \cdot  T^{k/2} \cdot \mathbb{P}(L^T_{1}(\lfloor tT \rfloor)=\lambda_{1},\dots,L^T_{k}(\lfloor tT \rfloor)=\lambda_{k}) ,$$
where as in Step 1, $Q_{\lambda}(T)$ is the cube $[\lambda_{1}T^{-1/2}-pt\sqrt{T},(\lambda_{1}+1)T^{-1/2}-pt\sqrt{T})\times \cdots\times[\lambda_{k}T^{-1/2}-pt\sqrt{T},(\lambda_{k}+1)T^{-1/2}-pt\sqrt{T})$. The last equation and (\ref{UniformConvSlice}) together imply that 
$$\mathbb{P}\left(Z^T \in R\right)  = [1 + o(1)] \cdot \int_{R} \rho(\vec{z}) d\vec{z}.$$
 Letting $T \rightarrow \infty$ in the last equation we obtain (\ref{ConvRectangle}).\\

\noindent \textbf{Step 5.} In this step, we conclude the proof of the proposition. By \cite[Theorem 3.10.1]{Durrett} to prove the weak convergence of $Z^T$ to $\rho$ it suffices to show that 
for any open set $U \subset W^{\circ}_k$ we have 
\begin{equation}\label{OpenConv}
\liminf_{T\rightarrow\infty}\mathbb{P}( Z^T\in U) \geq \int_U \rho(z) dz.
\end{equation}
In the remainder we fix an open set $U \subset W^{\circ}_k$ and prove (\ref{OpenConv}).

From \cite[Theorem 1.4]{Stein} we know that we can write $U=\cup_{i=1}^{\infty}R_{i}$, where $R_{i}=[u_{1}^{i},v_{1}^{i}]\times\cdots\times[u_{k}^{i},v_{k}^{i}]$ are rectangles with pairwise disjoint interiors. Let us fix $n \in \mathbb{N}$ and $\epsilon > 0$ and put $R_{i}^{\epsilon}=[u_{1}^{i}+\epsilon,v_{1}^{i}-\epsilon]\times\cdots\times[u_{k}^{i}+\epsilon,v_{k}^{i}-\epsilon]$. By finite additivity of $\mathbb{P}$ and (\ref{ConvRectangle}) we know  
\begin{align*}
	&\liminf_{T\rightarrow\infty}\mathbb{P}(Z^T\in U)\geq\liminf_{T\rightarrow\infty}\mathbb{P}(Z^T\in \cup_{i=1}^{n}R_{i}^{\epsilon}) = \liminf_{T\rightarrow\infty}\sum_{i=1}^{n}\mathbb{P}(Z^T\in R_{i}^{\epsilon})=\sum_{i=1}^{n}\int_{R_{i}^{\epsilon}}\hspace{-3mm}\rho(\vec{z})d\vec{z} =\int_{\cup_{i=1}^{n}R_{i}^{\epsilon}}\hspace{-9mm}\rho(\vec{z})d\vec{z}.
\end{align*}
We can now let $\epsilon \rightarrow 0+$ and $n \rightarrow \infty$ above and apply the monotone convergence theorem to conclude that the right side converges to $\int_U \rho(z)dz$. Here we use that $\rho$ is continuous and non-negative. Doing this brings us to (\ref{OpenConv}) and thus we conclude the statement of the proposition.

%
\subsection{Proof of Proposition \ref{S9PropH} for any $\vec{a}, \vec{b} \in W_k$}\label{Section9.5} 
In this section, we give the proof of Proposition \ref{S9PropH} for any $\vec{a}, \vec{b} \in W_k$. In what follows we assume that $\vec{a}, \vec{b}$ have the form in (\ref{Block}), which we recall here for the reader's convenience. 
\begin{equation}{\label{Block1}}
\begin{split}
	&\vec{a}=(a_{1},\cdots,a_{k})=(\underbrace{\alpha_{1},\cdots,\alpha_{1}}_{m_{1}},\cdots,\underbrace{\alpha_{p},\cdots,\alpha_{p}}_{m_{p}})\\
	&\vec{b}=(b_{1},\cdots,b_{k})=(\underbrace{\beta_{1},\cdots,\beta_{1}}_{n_{1}},\cdots,\underbrace{\beta_{q},\cdots,\beta_{q}}_{n_{q}})
\end{split}
\end{equation}
We recall that $\alpha_{1}>\alpha_{2}>\cdots>\alpha_{p}$, $\beta_{1}>\beta_{2}>\cdots>\beta_{q}$ and $\sum_{i=1}^{p}m_{i}=\sum_{i=1}^{q}n_{i}=k$. We denote $\vec{m}=(m_{1},\cdots,m_{p})$, $\vec{n}=(n_{1},\cdots,n_{q})$. If $\vec{a}, \vec{b}$ have the above form we recall from (\ref{PreDensity17}) that 
\begin{align}{\label{PreDensity172}}
	H(\vec{z})=\varphi(\vec{a},\vec{z},\vec{m}) \cdot \psi(\vec{b},\vec{z},\vec{n}) \cdot \prod_{i=1}^{k}e^{-c_{3}(t,p)z_{i}^{2}},
\end{align}
where $\varphi$ and $\psi$ are as in (\ref{TwoDet}). 

We next introduce some new notation that will be useful for our arguments. For any $\epsilon > 0$ we define the vectors $\vec{a}^+_{\epsilon}$ and $\vec{b}^+_{\epsilon}$ through 
\begin{equation}\label{vecpos}
\begin{split}
(a^+_{\epsilon})_{m_1 + \cdots + m_{i-1} + j} = \alpha_i + (m_i-j + 1)\epsilon &\mbox{ for $i = 1, \dots, p$ and $j =1 ,\dots, m_i$}, \\
(b^+_{\epsilon})_{n_1 + \cdots + n_{i-1} + j} = \beta_i + (n_i-j + 1)\epsilon &\mbox{ for $i = 1, \dots, q$ and $j =1 ,\dots, n_i$}. \\
\end{split}
\end{equation}
Similarly, we define the vectors $\vec{a}^-_{\epsilon}$ and $\vec{b}^-_{\epsilon}$ through
\begin{equation}\label{vecneg}
\begin{split}
(a^-_{\epsilon})_{m_1 + \cdots + m_{i-1} + j} = \alpha_i -j \epsilon &\mbox{ for $i = 1, \dots, p$ and $j =1 ,\dots, m_i$}, \\
(b^-_{\epsilon})_{n_1 + \cdots + n_{i-1} + j} = \beta_i - j\epsilon &\mbox{ for $i = 1, \dots, q$ and $j =1 ,\dots, n_i$}. \\
\end{split}
\end{equation}
We next let $H_\epsilon^+$, $H_\epsilon^-$ be as in (\ref{PreDensity172}) for the vectors $\vec{a}_\epsilon^+, \vec{b}^+_\epsilon$ and $\vec{a}_\epsilon^-, \vec{b}^-_\epsilon$ respectively. In particular, 
\begin{equation}\label{HPM}
H^{\pm}_\epsilon(\vec{z}) = \det\left[e^{c_1(t,p)(a^\pm_\epsilon)_i z_j}\right]_{i,j=1}^{k} \cdot \det\left[e^{c_2(t,p)(b^\pm_\epsilon)_i  z_j}\right]_{i,j=1}^{k} \cdot \prod_{i=1}^{k}e^{-c_{3}(t,p)z_{i}^{2}}.
\end{equation}
Observe that by construction we have $\vec{a}^{\pm}_\epsilon, \vec{b}^{\pm}_\epsilon \in W_k^{\circ}$ for all $\epsilon \in (0,1)$ that are sufficiently small, which we implicity assume in the sequel. It follows from our work in Section \ref{Section9.3} that 
$$Z_\epsilon^{\pm} = \int_{W_k} H^{\pm}_\epsilon(z) dz \in (0,\infty)$$
and so the functions 
\begin{equation}\label{HPM}
\rho^{\pm}_\epsilon(\vec{z}) = [Z_\epsilon^{\pm} ]^{-1} \cdot H^{\pm}_\epsilon(\vec{z}) 
\end{equation}
are well-defined densities on $W_k$. 

We next recall some basic notation for multivariate Taylor series, following \cite[Chapter 3]{CJ}. Suppose $\sigma = (\sigma_{1},\dots,\sigma_{k})$ is a multi-index of \emph{length} $k$. In our context, we require $\sigma_{1},\dots,\sigma_{k}$ be all non-negative integers (some of them might be equal). We define $|\sigma|=\sum_{i=1}^{k}\sigma_{i}$ as the \emph{order} of $\sigma$. Suppose $\tau=(\tau_{1},\dots,\tau_{k})$ is another multi-index of length $n$. We say $\tau\leq\sigma$ if $\tau_{i}\leq \sigma_{i}$ for $i=1,\cdots,k$. We say $\tau<\sigma$ if $\tau\leq \sigma$ and there exists at least one index $i$ such that $\tau_{i}<\sigma_{i}$. Then, define the partial derivative with respect to the multi-index $\sigma$:
$$D^{\sigma}f(x_{1},\cdots,x_{k})=\frac{\partial^{|\sigma|}f(x_{1},\cdots,x_{k})}{\partial x_{1}^{\sigma_{1}}\partial x_{2}^{\sigma_{2}}\cdots \partial x_{k}^{\sigma_{k}}}.$$ 

We also have the Taylor expansion for multi-variable functions:
\begin{equation}{\label{MultiTaylor}}
	f(x_{1},\cdots,x_{k})=\sum_{|\sigma|\leq r}\frac{1}{\sigma!}D^{\sigma}f(\vec{x}_{0})(\vec{x}-\vec{x}_{0})^{\sigma}+R^f_{r+1}(\vec{x},\vec{x}_{0})
\end{equation}
 
In the equation, $\sigma!=\sigma_{1}!\sigma_{2}!\cdots\sigma_{k}!$ is the factorial with respect to the multi-index $\sigma$, $\vec{x}_{0}=(x_{1}^{0},\cdots,x_{k}^{0})$ is a constant vector at which we expand the function $f$, $(\vec{x}-\vec{x}_{0})^{\sigma}$ stands for $(x_{1}-x_{1}^{0})^{\sigma_{1}}\cdots(x_{k}-x_{k}^{0})^{\sigma_{k}}$, and 
$$R^f_{r+1}(\vec{x},\vec{x}_{0})=\sum_{\sigma:|\sigma|=r+1}\frac{1}{\sigma!}D^{\sigma}f(\vec{x}_{0}+\theta(\vec{x}-\vec{x}_{0}))(\vec{x}-\vec{x}_{0})^{\sigma}$$ is the remainder, where $\theta\in (0,1)$.

We also need some notation for \emph{permutations}. Suppose $s_{n}$ is a permutation of  $\{1,\dots,n\}$, and $s_{n}(i)$ represents the $i$-$th$ element in the permutation $s_{n}$. We define \emph{the number of inversions} of $s_{n}$ by $I(s_{n})=\sum_{i=1}^{n-1}\sum_{j=i+1}^{n}\mathbf{1}_{\{s_{n}(i)>s_{n}(j)\}}$. For example, the permutation $s_{n}=(1,\dots,n)$ has $0$ number of inversions, while the permutation $s_{5}=(3,2,5,1,4)$ has number of inversions equal to $2+1+2+0+0 = 5$. Define the sign of permutation $s_{n}$ by $sgn(s_{n})=(-1)^{I(s_{n})}$. For instance, $sgn((1,\dots,n))=1$ and $sgn(s_{5})=-1$ in the previous example.

We now turn to the proof of Proposition \ref{S9PropH}.

\begin{proof}(of Proposition \ref{S9PropH}) For clarity we split the proof into two steps.

{\bf \raggedleft Step 1.} In this step we prove that for every $\vec{z} \in W_k$ we have
\begin{equation}\label{limitH}
\begin{split}
&\lim_{\epsilon \rightarrow 0+} \epsilon^{- \sum_{i = 1}^p \binom{m_i}{2} - \sum_{i = 1}^p \binom{n_i}{2}} H_{\epsilon}^{\pm} (\vec{z}) = C(\vec{m}) \cdot C(\vec{n})\cdot  H(\vec{z}), \mbox{ where }\\
&C(\vec{m}) =  \prod_{i = 1}^p \frac{1}{m_i!} \cdot \prod_{1 \leq j_1 < j_2 \leq m_i} (j_2 - j_1)  \mbox{ and } C(\vec{n}) =  \prod_{i = 1}^q \frac{1}{n_i!} \cdot \prod_{1 \leq j_1 < j_2 \leq n_i} (j_2 - j_1).
\end{split}
\end{equation}
As the cases are very similar we only show (\ref{limitH}) for $H^{+}_\epsilon$. Equation (\ref{limitH}) would follow if we can show that 
\begin{equation}\label{limitH2}
\begin{split}
&\lim_{\epsilon \rightarrow 0+} \epsilon^{- \sum_{i = 1}^p \binom{m_i}{2} }\cdot \det\left[e^{c_1(t,p)(a^\pm_\epsilon)_i z_j}\right]_{i,j=1}^{k} =C(\vec{m})  \cdot \varphi(\vec{a},\vec{z},\vec{m}), \\
& \lim_{\epsilon \rightarrow 0+} \epsilon^{ - \sum_{i = 1}^p \binom{n_i}{2}}\cdot \det\left[e^{c_2(t,p)(b^\pm_\epsilon)_i  z_j}\right]_{i,j=1}^{k}  =C(\vec{n}) \cdot \psi(\vec{b},\vec{z},\vec{n}).
\end{split}
\end{equation}
Let us put $f(\vec{c}, \vec{z}) = \det\left[e^{c_1(t,p) c_i z_j}\right]_{i,j=1}^{k},$ $g(\vec{c}, \vec{z}) =  \det\left[e^{c_2(t,p)c_i  z_j}\right]_{i,j=1}^{k} $, $u =  \sum_{i = 1}^p \binom{m_i}{2}$, $v = \sum_{i = 1}^p \binom{n_i}{2}$. By the multi-variable Taylor series expansion (\ref{MultiTaylor}) we know that 
\begin{equation}{\label{Expandf}}
\begin{split}
f(\vec{a}_{\epsilon}^+,\vec{z})&=\sum_{|\sigma|\leq u}\frac{D^{\sigma}f(\vec{a},\vec{z})}{\sigma !}(\vec{a}^+_{\epsilon}-\vec{a})^{\sigma}+R^f_{u+1}(\vec{a}_{\epsilon}^+,\vec{a},\vec{z}).
\end{split}
\end{equation}
where 
\begin{equation}{\label{remainder}}
R^f_{u+1}(\vec{a}_{\epsilon}^+,\vec{a},\vec{z})=\sum_{\sigma:|\sigma|=u+1}\frac{1}{\sigma!}D^{\sigma}f(\vec{a}+\theta(\vec{a}_{\epsilon}^+-\vec{a}),\vec{z})(\vec{a}_\epsilon^+-\vec{a})^{\sigma}.
\end{equation} 
and $\theta \in (0,1)$. We also observe, by basic linear algebra, that for any multi-index $\alpha = (\alpha_1, \dots, \alpha_k)$
\begin{equation}\label{DiffDet}
D^{\alpha} f(\vec{x}, \vec{z}) =  \det\left[(c_1(t,p) z_j)^{\alpha_i} e^{c_1(t,p) x_i z_j}\right]_{i,j=1}^{k}.
\end{equation}

We note that if $|\sigma| < u$ then there exist $i \in \{1, \dots, p\}$ and $j_1, j_2 \in \{1, \dots, m_i\}$ such that $j_1 \neq j_2$ and $\sigma_{m_1 + \cdots + m_{i-1} + j_1} = \sigma_{m_1 + \cdots + m_{i-1} + j_2}$. The latter implies that $D^{\sigma}f(\vec{a},\vec{z}) = 0$ since by (\ref{DiffDet}) the latter is the determinant of a matrix with two equal rows. An analogous argument shows that $D^{\sigma}f(\vec{a},\vec{z}) = 0$ unless $|\sigma| = u$ and $\{ \sigma_{m_1 + \cdots + m_{i-1}+ j}: j = 1, \dots, m_i\} = \{0, 1, \dots, m_{i} - 1\}$ for all $i \in \{1, \dots, p\}$. 

On the other hand, if $|\sigma| = u$ and $\{ \sigma_{m_1 + \cdots + m_{i-1}+ j}: j = 1, \dots, m_i\} = \{0, 1, \dots, m_{i} - 1\}$ for all $i \in \{1, \dots, p\}$ we have that 
$$\sigma = (\sigma^1, \sigma^2, \dots, \sigma^p),$$
where $\sigma^i \in S_{m_i}$ (the permutation group of $\{0, 1, \dots, m_i- 1\}$). Using the multi-linearity of the determinant we obtain for all such $\sigma$ that  
$$\frac{D^{\sigma}f(\vec{a},\vec{z})}{\sigma!}(\vec{a}^+_{\epsilon}-\vec{a})^{\sigma} = \epsilon^{u} \cdot F(\sigma) \cdot \varphi(\vec{a},\vec{z},\vec{m}), \mbox{ where }F(\sigma) =  \prod_{i = 1}^p \frac{sgn(\sigma^i) \cdot \prod_{r= 1}^{m_i} r^{\sigma^i(j) }}{m_i!}.$$
Summing over all $\sigma$ we conclude that 
$$\sum_{|\sigma|\leq u}\frac{D^{\sigma}f(\vec{a},\vec{z})}{\sigma !}(\vec{a}^+_{\epsilon}-\vec{a})^{\sigma} = \epsilon^{u} \cdot  \varphi(\vec{a},\vec{z},\vec{m}) \cdot  \prod_{i = 1}^p \frac{1}{m_i!} \cdot \prod_{1 \leq j_1 < j_2 \leq m_i} (j_2 - j_1) = \epsilon^{u} \cdot  \varphi(\vec{a},\vec{z},\vec{m})  \cdot C(\vec{m}),$$
where in deriving the above we used the formula for a Vandermonde determinant, cf. \cite[pp. 40]{Mac}. Combining the latter with (\ref{Expandf}) and (\ref{remainder}) we conclude that 
\begin{equation}{\label{Expandf2}}
\begin{split}
\left| \epsilon^{- u} f(\vec{a}_{\epsilon}^+,\vec{z}) - C(\vec{m})  \cdot \varphi(\vec{a},\vec{z},\vec{m})\right|  \leq | \epsilon^{- u} R^f_{u+1}(\vec{a}_{\epsilon}^+,\vec{a},\vec{z})| = O(\epsilon).
\end{split}
\end{equation}
Analogous arguments show that 
\begin{equation}{\label{Expandg2}}
\begin{split}
\left| \epsilon^{- v} g(\vec{b}_{\epsilon}^+,\vec{z}) - C(\vec{n})  \cdot \psi(\vec{b},\vec{z},\vec{n})\right|  \leq | \epsilon^{- v} R^g_{v+1}(\vec{b}_{\epsilon}^+,\vec{b},\vec{z})| = O(\epsilon).
\end{split}
\end{equation}
Combining (\ref{Expandf2}) and (\ref{Expandg2}) we conclude (\ref{limitH2}).\\

{\raggedleft \bf Step 2.} In this step we conclude the proof of the proposition. In view of (\ref{limitH}) and the fact that $H_\epsilon^+(\vec{z}) > 0$ for $\vec{z} \in W_k^{\circ}$ (we proved this in Section \ref{Section9.3}) we conclude that $H(\vec{z}) \geq 0$ for $\vec{z} \in W_k^{\circ}$. Also by Lemma \ref{NonVanish} we know that $H(\vec{z}) \neq 0$ for $\vec{z} \in W_k^{\circ}$ and so indeed, $H(\vec{z}) > 0$ for $\vec{z} \in W_k^{\circ}$. Furthermore, we know that $H(\vec{z}) = 0$ for $\vec{z} \in W_k \setminus W_k^{\circ}$ since the determinants in the definition of $H(\vec{z})$ vanish due to equal columns when $\vec{z} \in W_k \setminus W_k^{\circ}$. Finally, we observe that by (\ref{Expandf2}) and (\ref{Expandg2}) we know that 
there exist positive constants $D, d > 0$ independent of $\epsilon$ provided it is sufficiently small such that 
\begin{equation}\label{DomFunH}
|\epsilon^{-u - v} \cdot H_\epsilon(\vec{z})| \leq D \cdot \exp \left(d \|\vec{z}\| - c_3(t,p) \|\vec{z}\|^2 \right), 
\end{equation}
where as usual $\|\vec{z} \|^2 = \sum_{i = 1}^k z_i^2$. In view of (\ref{DomFunH}) and the dominating convergence theorem, we conclude that $H(\vec{z})$ is integrable and since it is continuous and positive on $W_k^{\circ}$ we conclude that $Z_c \in (0, \infty)$ as desired.
\end{proof}

The above proof essentially shows the following statement. 

\begin{corollary}\label{WeakConvRho}
Let $\vec{a}, \vec{b} \in W_k$. Let $\rho^{\pm}_{\epsilon}$ be as (\ref{HPM}), and let $\rho$ be as in Proposition \ref{S9PropH} for the two vectors $\vec{a}, \vec{b}$. Then $\rho^{\pm}_\epsilon$ weakly converge to $\rho$ as $\epsilon \rightarrow 0+$. 
\end{corollary}
\begin{proof} We use the same notation as in the proof of Proposition \ref{S9PropH} above. As the proofs are analogous we only show that $\rho_{\epsilon}^+$ weakly converges to $\rho$. We claim that for any Borel set $B \subset W_k$ 
\begin{equation}\label{limitBorel}
\lim_{\epsilon \rightarrow 0+} \int_B \epsilon^{-u - v}  H_\epsilon^+(z) dz =C(\vec{m})  \cdot C(\vec{n})  \cdot   \int_B H(z) dz.
\end{equation}

Assuming the validity of (\ref{limitBorel}) we see that for any Borel set $B \subset W_k$ we have 
$$\int_B \rho(z) dz = \frac{\int_B H(z) dz}{\int_{W_k} H(z) dz} = \lim_{\epsilon \rightarrow 0+} \frac{C(\vec{m})  \cdot C(\vec{n}) \cdot  \int_B H^+_\epsilon(z) dz}{C(\vec{m})  \cdot C(\vec{n})  \cdot \int_{W_k} H^+_\epsilon(z) dz} = \lim_{\epsilon \rightarrow 0+}  \int_B \rho^+_\epsilon(z) dz,$$
which proves the weak convergence we wanted. Thus we only need to show (\ref{limitBorel}).

In view of (\ref{limitH}) we know that $\epsilon^{-u - v}  H_\epsilon^+(z)$ converges pointwise to $C(\vec{m})  \cdot C(\vec{n}) \cdot H(z)$ and then (\ref{limitBorel}) follows from the dominated convergence theorem, once we invoke (\ref{DomFunH}). 
\end{proof}

%
\subsection{Proof of Proposition \ref{WeakConvDistinct} for any $\vec{a}, \vec{b} \in W_k$}{\label{Section9.6}}
We fix the same notation as in Section \ref{Section9.5} and suppose that $\epsilon > 0$ is sufficiently small so that $\vec{a}^{\pm}_\epsilon, \vec{b}^{\pm}_\epsilon \in W_k^{\circ}$. To prove the proposition it suffices to show that for any $\vec{c} \in \mathbb{R}^k$ we have that 
\begin{equation}\label{S9ETS}
\lim_{T\rightarrow \infty}\mathbb{P}^{0,T,\vec{x}^{T},\vec{y}^{T}}_{avoid,Ber} \left( Z_1^T \leq c_1, \dots, Z_k^T \leq c_k \right) = \int_{W_k \cap R} \rho(\vec{z})d\vec{z},
\end{equation}
where $R = (-\infty, c_1] \times \cdots \times (-\infty, c_k]$. 

We define the vectors $\vec{x}^+_{\epsilon,T}$ and $\vec{y}^+_{\epsilon,T}$ through 
\begin{equation*}\label{vecposXY}
\begin{split}
(x^+_{\epsilon,T})_{m_1 + \cdots + m_{i-1} + j} = x^T_{m_1 + \cdots + m_{i-1} + j} + \lfloor \sqrt{T} (m_i-j + 1)\epsilon \rfloor &\mbox{ for $i = 1, \dots, p$ and $j =1 ,\dots, m_i$}, \\
(y^+_{\epsilon, T})_{n_1 + \cdots + n_{i-1} + j} =y^T_{n_1 + \cdots + n_{i-1} + j} + \lfloor \sqrt{T} (n_i-j + 1)\epsilon \rfloor &\mbox{ for $i = 1, \dots, q$ and $j =1 ,\dots, n_i$}. \\
\end{split}
\end{equation*}
Similarly, we define the vectors $\vec{x}^-_{\epsilon,T}$ and $\vec{y}^-_{\epsilon,T}$ through
\begin{equation*}\label{vecnegXY}
\begin{split}
(x^-_{\epsilon,T})_{m_1 + \cdots + m_{i-1} + j} =x^T_{m_1 + \cdots + m_{i-1} + j} - \lfloor \sqrt{T} j \epsilon \rfloor &\mbox{ for $i = 1, \dots, p$ and $j =1 ,\dots, m_i$}, \\
(y^-_{\epsilon,T})_{n_1 + \cdots + n_{i-1} + j} = y^T_{n_1 + \cdots + n_{i-1} + j}  - \lfloor \sqrt{T} j\epsilon \rfloor &\mbox{ for $i = 1, \dots, q$ and $j =1 ,\dots, n_i$}. \\
\end{split}
\end{equation*}

It follows from Lemma \ref{MCLxy} that 
\begin{equation}\label{S9ETS2}
\begin{split}
&\mathbb{P}^{0,T,\vec{x}^+_{\epsilon,T},\vec{y}^+_{\epsilon,T}}_{avoid,Ber} \left( Z_1^T \leq c_1, \dots, Z_k^T \leq c_k \right)  \leq \mathbb{P}^{0,T,\vec{x}^{T},\vec{y}^{T}}_{avoid,Ber} \left( Z_1^T \leq c_1, \dots, Z_k^T \leq c_k \right)  \leq \\
&\mathbb{P}^{0,T,\vec{x}^-_{\epsilon,T},\vec{y}^-_{\epsilon,T}}_{avoid,Ber} \left( Z_1^T \leq c_1, \dots, Z_k^T \leq c_k \right).
\end{split}
\end{equation}
Taking the limit as $T \rightarrow \infty$ in (\ref{S9ETS2}) and applying our result from Section \ref{Section9.4} we obtain
\begin{equation}\label{S9ETS3}
\begin{split}
&\int_{W_k \cap R} \rho^+(\vec{z})d\vec{z}  \leq \liminf_{T \rightarrow \infty} \mathbb{P}^{0,T,\vec{x}^{T},\vec{y}^{T}}_{avoid,Ber} \left( Z_1^T \leq c_1, \dots, Z_k^T \leq c_k \right)  \leq \\
& \limsup_{T \rightarrow \infty} \mathbb{P}^{0,T,\vec{x}^{T},\vec{y}^{T}}_{avoid,Ber} \left( Z_1^T \leq c_1, \dots, Z_k^T \leq c_k \right)  \leq\int_{W_k \cap R} \rho^-(\vec{z})d\vec{z}   .
\end{split}
\end{equation}
Taking the $\epsilon \rightarrow 0+$ limit in (\ref{S9ETS3}) and invoking Corollary \ref{WeakConvRho} we arrive at (\ref{S9ETS}). This suffices for the proof.

\bibliographystyle{amsplain}

\end{document}